\documentclass{amsart}
\usepackage{amsmath,amssymb,amsthm}
\usepackage{amsfonts,amssymb}

\usepackage{natbib}

\usepackage{tikz}
\usepackage{tikz-cd}
\tikzcdset{scale cd/.style={every label/.append style={scale=#1},
    cells={nodes={scale=#1}}}}
\tikzset{overcross/.style={double, line width=1.5, white, double=#1, double distance=\knotlinewidth},
    overcross/.default={black},
    knot/.style={line width=\knotlinewidth, baseline=-.5ex}}
\newcommand{\knotlinewidth}{.7pt}
\usetikzlibrary{positioning,shapes.multipart, fit,backgrounds}

\usepackage{float}
\usepackage{graphicx}
\usepackage{caption}
\usepackage{soul}

\let\emptyset\varnothing
\usepackage{bbm}
\usepackage{enumitem}

\usepackage{hyperref}

\usepackage{xcolor}

\newtheorem{theorem}{Theorem}[section]
\newtheorem{corollary}[theorem]{Corollary}
\newtheorem{proposition}[theorem]{Proposition}
\newtheorem{lemma}[theorem]{Lemma}

\theoremstyle{definition}
\newtheorem{definition}[theorem]{Definition}
\newtheorem{example}[theorem]{Example}

\theoremstyle{remark}
\newtheorem*{remark}{Remark}

\newcommand{\abs}[1]{\left| #1 \right|}
\newcommand{\op}{\mathrm{op}}

\makeatletter
\newcommand{\xRightarrow}[2][]{\ext@arrow 0359\Rightarrowfill@{#1}{#2}}
\makeatother

\makeatletter 
\def\l@subsection{\@tocline{2}{0pt}{2pc}{6pc}{}} 
\makeatother

\tikzset{
    partial ellipse/.style args={#1:#2:#3}{
        insert path={+ (#1:#3) arc (#1:#2:#3)}
    }
}

\DeclareRobustCommand{\nobrackbinom}{\genfrac{}{}{0pt}{}}

\newcommand{\nocontentsline}[3]{}
\let\origcontentsline\addcontentsline
\newcommand\stoptoc{\let\addcontentsline\nocontentsline}
\newcommand\resumetoc{\let\addcontentsline\origcontentsline}

\usepackage{geometry}

\title{Hochschild (co)homology for odd Khovanov arc algebras}
\author{Dean Spyropoulos}
\date{}

\begin{document}

\begin{abstract}
We extend Hochschild homology and cohomology to quasi-associative algebras, which were defined initially by Albuquerque and Majid and generalized by Naisse and Putyra via grading categories. As an application, we use our construction to give an alternative proof of up-to-sign functoriality for odd Khovanov homology, which was recently proven by Migdail and Wehrli. Our proof generalizes an argument of Khovanov: in doing so, we also give the first proof that the tangle theory of Naisse and Putyra is functorial with respect to tangle cobordisms, up to unit.
\end{abstract}

\maketitle

\tableofcontents

\section{Introduction}

Hochschild homology is a homological invariant of associative algebras. Initiated by an observation of Przytycki \cite{MR2657644}, Hochschild homology has been used frequently as a tool to study Khovanov's link homology theory and its relatives (see, for example, \cite{MR2339573, rozansky2010categorification, MR3171092, MR4332675, MR4611197, stoffregen2024joneswenzlprojectorskhovanovhomotopy}). Speaking informally, we will show that Khovanov-like theories ``give back'': we present a way in which variants of Khovanov homology motivate a generalization of Hochschild homology to a particular class of non-associative algebras.

\emph{Khovanov homology} is a homological link invariant which assembles to a functor
\[
\mathbf{K}: \left\{ \nobrackbinom{\text{Links}~L\subset \mathbb{R}^3}{\text{Cobordisms}~ L_1 \xrightarrow{\Sigma} L_2~ \text{in}~ \mathbb{R}^3 \times I \big/\text{isotopy}} \right\}
\rightarrow
\left\{
\nobrackbinom{\text{Bigraded chain complexes}~\mathbf{K}(L)}{\text{Chain maps}~ \mathbf{K}(L_1) \xrightarrow{\mathbf{K}(\Sigma)} \mathbf{K}(L_2)\big/ \text{homotopy}}
\right\}.
\]
The foundations of this theory are thanks to its namesake: in a trio of papers \cite{10.1215/S0012-7094-00-10131-7, MR1928174, MR2171235}, Khovanov described the \emph{Khovanov complex} of a link $\mathbf{K}(L)$, extended his work to tangles, and (via the tangle theory) associated to a cobordism of tangles a chain map $\mathbf{K}(T_1) \to \mathbf{K}(T_2)$ which was well-defined up to sign. Via direct computation, Jacobsson \cite{MR2113903} proved that the sign ambiguity was necessary (see also \cite{MR2174270}), launching a search for a fully functorial model for Khovanov homology. Since then, a number of researchers have independently provided several such models \cite{MR2443094, MR2496052, MR2647055, MR4376719, MR4598808}. 

In this paper, we study a variant called \emph{odd Khovanov homology}. Odd Khovanov homology was developed by Ozsv\'ath, Rasmussen, and Szab\'o \cite{MR3071132} in order to provide a Khovnaov-like theory which was more proximate to various Floer theories (see, for example, Section 1.1 of \cite{ozsvath2003heegaard}), especially in integral coefficients. Indeed, it is now known that odd Khovanov homology is the $E_2$-page of spectral sequences in $\mathbb{Z}$-coefficients converging to plane Floer homology \cite{daemi2015abelian} and framed instanton Floer homology \cite{MR3394316}. To the author's knowledge, Conjecture 1.9 from \cite{MR3071132} (\textit{i.e.}, that there is a spectral sequence from reduced odd Khovanov homology to the Heegaard-Floer homology of the branched double cover) is still open.

Odd Khovanov homology resembles the original (frequently called ``even'' for differentiation) theory in some ways. Despite differing in general (see, for example, \cite{MR2777025}), they agree over $\mathbb{Z}/ 2\mathbb{Z}$-coefficients and their \emph{reduced} theories agree over $\mathbb{Z}$-coefficients for non-split alternating links. In contrast, many ostensibly foundational results have proven elusive for the odd theory. For example, Migdail and Wehrli recently gave the first proof that odd Khovanov homology is functorial up to a plus or minus sign \cite{migdail2024functoriality}. Impressively, their argument is ``short'' in the sense that they do not perform a brute-force computation (which entails over thirty separate checks) yet they do not appeal to a tangle theory. One reason for this may be that, until recently, a tangle theory for odd Khovanov homology was unknown. Now there are two, likely equivalent, theories due to Naisse and Putyra \cite{naisse2020odd} and Schelstraete and Vaz \cite{MR4190457, schelstraete2023odd}.

The motivation for this paper is to provide a second proof that odd Khovanov homology is functorial up to sign by updating Khovanov's original argument \cite{MR2171235} to the Naisse-Putyra tangle theory. There are a few reasons to desire such a result. First, it proves that the Naisse-Putyra odd tangle theory is functorial up to sign, generalizing the Migdail-Wehrli result to tangle cobordisms. Moreover, the Naisse-Putyra tangle theory is actually stated for \emph{unified Khovanov homology}, a variant with coefficients in the ring of truncated polynomials
\[
R := \mathbb{Z}[X, Y, Z^{\pm 1}] \big/ (X^2 = Y^2 = 1).
\]
Unified Khovanov homology has the extremely useful quality of bridging the gap between the odd and even theories, in the sense that the even (resp. odd) Khovanov complex of a link can be obtained from the unified one by specializing coefficients $X, Y, Z \mapsto 1$ (resp. $X, Z \mapsto 1$, $Y \mapsto -1$). Thus, our strategy will be to show a stronger statement: that the Naisse-Putyra unified tangle theory is functorial with respect to tangle cobordisms up to unit of $R$.

There is an essential problem which makes our line of argument interesting, boiling down to the same reasons which made developing an odd tangle theory arduous. The central objects of Khovanov's tangle theory \cite{MR1928174} are his \emph{arc algebras} $\{H^n\}_{n\ge 0}$; for example, flat $(m, n)$-tangles $t$ are associated to $(H^m, H^n)$-bimodules $\mathcal{F}(t)$, and tangles with crossings $T$ are associated to complexes of such bimodules $\mathcal{F}(T)$. However, Naisse and Vaz  \cite{naisse2017odd} showed that the odd (and the unified) analogs of these algebras are problematic: they are non-associative, non-unital, and their product structure fails to preserve the usual grading structure.

Failure of associativity poses a problem for our argument. We would like to argue, like Khovanov, that if $M$ is an invertible complex (see Subsection \ref{ss:ccenter}) of bimodules over an algebra $A$, then there are ring isomorphisms
\[
\mathrm{End}_{\mathcal{K}(A)}(M) \cong 
\mathrm{End}_{\mathrm{Bim}(A)}(A) \cong \mathsf{HH}^0(A)
\]
for $\mathsf{HH}^\bullet(A)$ the \emph{Hochschild cohomology} of $A$ (note that $\mathsf{HH}^0(A)$ is more frequently called the \emph{(derived) center} of $A$). This does not make sense if $A$ is non-associative---in particular, Hochschild homology and cohomology are invariants of associative algebras, but are not defined for non-associative algebras.

The key observation, first made by Putyra and Shumakovitch in unpublished work and then by Naisse and Putyra in \cite{naisse2020odd}, is that the odd arc algebras (while non-associative) are \emph{quasi-associative}, in a sense generalizing the same definition of Albuquerque and Majid \cite{albuquerque1998quasialgebra}. See Section \ref{s:algebra} for a complete definition: we call an algebra $A$ quasi-associative if there exists a category $\mathcal{C}$ along with a choice of 3-cocycle $\alpha$ on the nerve $N(\mathcal{C})$ such that $A$ admits a grading by the morphisms of $\mathcal{C}$ and for each homogeneous $x, y, z\in A$,
\[
\mu_A(\mu_A(x, y), z) = \alpha(\abs{x}, \abs{y}, \abs{z}) \mu_A(x, \mu_A(y,z)). 
\]
We call the pair $(\mathcal{C}, \alpha)$ a \emph{grading category}. In most situations, we want to consider a quasi-associative algebra with respect to a particular grading category, and thus call it a $(\mathcal{C}, \alpha)$-graded algebra. Outside of the introduction, we almost always say ``$\mathcal{C}$-graded,'' with the assumption that the 3-cocycle $\alpha$ remains fixed unless stated otherwise.

With our motivation from link homology, a large portion of this paper is devoted to proving that there exists a notion of Hochschild homology and cohomology for quasi-associative algebras, extending the classical theory. The Hochschild homology in our setting requires an additional choice of a particular function denoted throughout this paper by $\varepsilon$---we omit a description for now.

\begin{theorem}[Hochschild homology for $(\mathcal{C}, \alpha)$-graded algebras]
\label{thm1}
Suppose $A$ is a $(\mathcal{C}, \alpha)$-graded algebra, and that $M$ is a $(\mathcal{C}, \alpha)$-graded $(A, A)$-bimodule (see Definition \ref{def:bims}). Furthermore, assume that $(\mathcal{C}, \alpha)$ admits a function $\varepsilon$ as in Definition \ref{def:looper}. Then, fixing the choice of the triple $(\mathcal{C}, \alpha, \varepsilon)$, there is a \emph{$(\mathcal{C}, \alpha,\varepsilon)$-Hochschild homology} of $A$ with coefficients in $M$, denoted $\mathsf{HH}^{(\mathcal{C}, \alpha, \varepsilon)}_\bullet(A, M)$, extending classical Hochschild homology. That is, if $A$ is an associative, $G$-graded algebra for some group $G$, then 
\[
\mathsf{HH}^{(BG, 1, 1)}_\bullet(A, M) \cong \mathsf{HH}_\bullet(A, M).
\]
\end{theorem}

We frequently drop the superscript, assuming that the choice of $(\mathcal{C}, \alpha, \varepsilon)$ is understood. A key observation is that the $\mathcal{C}$-grading does not survive when taking Hochschild homology; it is replaced with a grading by the \emph{universal trace} (see, for example, \cite{MR3910068}) of $\mathcal{C}$,
\[
\mathrm{Tr}(\mathcal{C}) = \coprod_{X\in \mathrm{Ob}(\mathcal{C})} \mathrm{End}_\mathcal{C}(X) \big / g\circ f \sim f \circ g.
\]
Note that this does not contradict classical Hoschschild homology since $\mathrm{Tr}(BG) = G$.

While the need for an additional witness $\varepsilon$ is initially perplexing, it has a satisfying resolution. Indeed, it is classically the case that $\{\mathsf{HH}_0(A, -)\}_{A}$ is a \emph{shadow} on the bicategory of bimodules \cite{MR3095324}. We will show that the choice of $\varepsilon$ actually fixes a choice of the natural isomorphism
\[
\theta_{M, N}: \mathsf{HH}_0(A, M \otimes_B N) \to \mathsf{HH}_0(B, N \otimes_A M)
\]
turning $\{\mathsf{HH}^{(\mathcal{C}, \alpha, \varepsilon)}_0(A, -)\}_{A}$ into a shadow on the bicategory of $(\mathcal{C}, \alpha)$-graded bimodules, with target category the category of modules graded by the trace of $\mathcal{C}$. 

In \cite{naisse2020odd}, Naisse and Putyra showed that each unified arc algebra $H^n$ admits a quasi-associative structure by a grading category denoted by $(\mathcal{G}, \alpha)$. In this paper, we construct a function $\varepsilon$ compatible with their grading category; consequently, we can study $\mathsf{HH}_\bullet^{(\mathcal{G}, \alpha, \varepsilon)} (H^n)$ in future work. The story is a little more complicated for $\mathsf{HH}_\bullet^{(\mathcal{G}, \alpha, \varepsilon)} (H^n, \mathcal{F}(T))$ when $T$ is an $(2n, 2n)$-tangle.

A curious anomaly presents itself once we study the dual object, Hochschild cohomology. On one hand, the Hochschild cohomology for $(\mathcal{C}, \alpha)$-graded algebras \emph{does not} require an additional choice, unlike the Hochschild homology (which required a choice of $\varepsilon$). On the other hand, the bicategory of $(\mathcal{C}, \alpha)$-graded bimodules is \emph{not} closed in the usual sense, thus Hochschild cohomology is not realizable as a coshadow (see \cite{MR4768952}) in our setting. Nevertheless, this theory does extend the classical one.

\begin{theorem}[Hochschild cohomology for $(\mathcal{C}, \alpha)$-graded algebras]
\label{thm2}
Suppose $A$ is a $(\mathcal{C}, \alpha)$-graded algebra, and that $M$ is a $(\mathcal{C}, \alpha)$-graded $(A, A)$-bimodule. Then there is a \emph{$(\mathcal{C}, \alpha)$-Hochschild cohomology} of $A$ with coefficients in $M$, denoted $\mathsf{HH}_{(\mathcal{C}, \alpha)}^\bullet(A, M)$, extending classical Hochschild cohomology. That is, if $A$ is an associative, $G$-graded algebra for some group $G$, then 
\[
\mathsf{HH}_{(BG, 1)}^\bullet(A, M) \cong \mathsf{HH}^\bullet(A, M).
\]
\end{theorem}

Again, we will drop the subscript often. The $\mathcal{C}$-grading does not survive when taking cohomology, but is replaced with a grading by a \emph{monoid of grading shifts} for $\mathcal{C}$. Again, this generalizes the classical scenario, though it takes a bit more background (summarized in Subsection \ref{ss:nonhomoHH}) to see why. Basically, when $\mathcal{C} = BG$, the monoid of grading shifts can be taken to be the center of $G$, so that the Hochschild cohomology has the typical $\mathbb{Z} \times Z(G)$ graded structure.

We pay particular attention to the ``grading preserving'' parts of the Hochschild cohomology of a $(\mathcal{C}, \alpha)$-graded algebra and show that in homological degree zero, this Hochschild cohomology is a generalization of the center:
\[
Z^\mathcal{C}(A) = \{z\in A : \mu(a, z) = \zeta(\abs{a}) \mu(z,a)\}
\]
where $\zeta(\abs{\alpha})$ is a product of values related to the unitors of the category of $(\mathcal{C}, \alpha)$-graded modules. It is straightforward to show that this notion of center satisfies our uses: if $M$ is an invertible complex of $(\mathcal{C}, \alpha)$-graded $(A, A)$-bimodules, we prove that
\[
\mathrm{End}_{\mathcal{K}^\mathcal{C}(A)} (M) \cong \mathrm{End}_{\mathrm{Bim}(A,A)^\mathcal{C}}(A) \cong Z^\mathcal{C}(A).
\]
Specializing to the case when $\mathcal{C} = \mathcal{G}$ and $A = H^n$, we prove that
\[
Z^\mathcal{G} (H^n) \cong R
\]
for each $n\ge 1$. As in \cite{MR2171235}, this observation is the first of a sequence of observations concluding in a proof of the fact that the Naisse-Putyra tangle invariant is functorial with respect to tangle cobordisms up to units of $R$.

\begin{theorem}[Naisse-Putyra tangle invariant is functorial]
\label{thm3}
Let $\mathbf{Tang}$ denote the chronological 2-tangle 2-category (see Section \ref{s:utu_fun}), and $\check{\mathbb{K}}$ denote the 2-category of complexes of $\mathcal{G}$-graded $(H^m, H^n)$-bimodules, whose 2-morphisms are homotopy equivalences up to multiplication by a unit of $R$ (described fully in Subsection \ref{ss:NPinvariant}). The Naisse-Putyra tangle invariant extends to a 2-functor $\mathbf{Tang} \to \check{\mathbb{K}}$. See Section \ref{ss:NPinvariant} for the definition. Consequently,
\begin{enumerate}
	\item The Naisse-Putyra odd tangle invariant \cite{naisse2020odd} is functorial with respect to tangle cobordisms up to sign;
	\item Putyra's unified Khovanov homology \cite{putyra20152categorychronologicalcobordismsodd} is functorial with respect to link cobordisms up to unit of $R$;
	\item Odd Khovanov homology is functorial with respect to link cobordisms up to sign.
\end{enumerate}
\end{theorem}

The main caveat lies with the category $\check{\mathbb{K}}$. Namely, the Naisse-Putyra tangle invariant is not an invariant of tangles in a $\mathcal{G}$-graded sense (it is a bit too sensitive), but rather in a category obtained by reducing the $\mathcal{G}$-degree to an integral, quantum grading. See Section \ref{ss:NPinvariant} for a more complete description. Note that this does not have any impact on the corollary to odd Khovanov homology for links, since the $\mathcal{G}$-degree already collapses to the quantum degree in that case.

Of course, the $\mathcal{G}$-graded center is not the sole tool for proving functoriality---for example, there are several movie moves relating automorphisms of non-invertible complexes. For these, we appeal to a series of adjunctions, generalizing those used by Khovanov (see Proposition 3 of \cite{MR2171235}). The existence of these adjunctions continues a trend, first observed in \cite{spyropoulos2024}, that adjunctions frequently used in even Khovanov homology (\textit{e.g.}, see \cite{MR4193863, https://doi.org/10.48550/arxiv.1405.2574}) have satisfying lifts to the $\mathcal{G}$-grading setting.

\stoptoc

\subsection*{Impending investigations}

We list two immediate follow-up questions.

\subsubsection*{Unification of the cohomologies for $(n,n)$-Springer varieties}

In analogy with $Z(H^n) \cong \mathbb{Z}$, we prove that $Z^\mathcal{G}(UH^n) \cong R$ for $UH^n$ the unified arc algebra (in the rest of this paper, we do not use notation to distinguish between $UH^n$ and $H^n$). However, Khovanov also remarks that this notion of center is isomorphic to the \emph{degree preserving} endomorphisms of $H^n$ as a bimodule over itself. If we allow for homogeneous endomorphisms of possibly nonzero degree, we get something much larger: in \cite{MR2078414}, Khovanov proves that
\[
\mathsf{HH}^0(H^n) \cong H^\bullet(\mathfrak{B}_{n, n}; \mathbb{Z})
\]
for $\mathfrak{B}_{n, n}$ the $(n, n)$-Springer variety. Interestingly, in \cite{MR3257552} Lauda and Russell define an ``odd'' analog for the cohomology of $\mathfrak{B}_{n,n}$, building on work of Ellis-Khovanov-Lauda which showed that the ring of \emph{odd} symmetric functions is isomorphic to the space of skew polynomials fixed by the action of the Hecke algebra at $q = -1$ \cite{MR3168401}. While this work introduced further parallels between the ``even'' and odd nikHecke algebras, it also established that both the even and odd Khovanov homologies are supported by similar algebraic frameworks. 

Indeed, in \cite{naisse2017odd}, Naisse and Vaz provide an ad hoc definition of the ``odd center'' $OZ$ for the odd arc algebra $OH^n$ in order to prove that
\[
OZ(OH^n) \cong OH^\bullet(\mathfrak{B}_{n, n}; \mathbb{Z})
\]
for $OH^\bullet$ the odd cohomology of \cite{MR3257552}. It is likely that the zeroth Hochschild cohomology of $(\mathcal{G}, \alpha)$-graded algebras, introduced in this paper, recovers the Naisse-Vaz odd center. More accurately, we claim that $\mathsf{HH}^0_{(\mathcal{G}, \alpha)}(UH^n)$ provides a \emph{unified cohomology} for $\mathfrak{B}_{n,n}$, from which we can obtain $OH^\bullet(\mathfrak{B}_{n, n}; \mathbb{Z})$ or $H^\bullet(\mathfrak{B}_{n, n}; \mathbb{Z})$ upon specialization.

\subsubsection*{A odd invariant of links in $S^2 \times S^1$}

Khovanov homology has a well studied relationship with Hochschild homology, originating with observations by Przytycki \cite{MR2657644} and Khovanov \cite{MR2339573}. Namely, Przytycki proved that the stable Khovanov homology of the $(2,n)$-torus link as $n\to \infty$ is isomorphic to the Hochschild homology of the first arc algebra $\mathsf{HH}_\bullet(H^1)$. In the odd setting, $OH^1$ is still associative, and it is easy to see that Przytycki's result still holds if one replaces Khovanov homology with odd Khovanov homology (see \cite{spyropoulos2024}). However, for $n\ge 2$, Rozansky proved in \cite{rozansky2010categorification} that 
\[
\mathsf{HH}_{-i}(H^n) \cong Kh_{-i} (T_{n, nm})
\]
for $i \le 2m-2$. Thus, one would hope to prove that this lifts to the unified setting, and thus to a result for the odd Khovanov homology of torus links, by showing that
\[
\mathsf{HH}_\bullet^{(\mathcal{G}, \alpha, \varepsilon)}(UH^n) \cong OKh_\bullet(T_{n, nm})
\]
for some range.

In the same paper, Rozansky also proved that Hochschild homology provides an interesting invariant of tangle closures in $S^2 \times S^1$: if $T$ is a $(2n,2n)$-tangle and $\mathcal{F}(T)$ is Khovanov's $(H^n, H^n)$-bimodule, then Rozansky showed that $\mathsf{HH}_\bullet(H^n, \mathcal{F}(T))$ is an invariant of tangle closures $\widetilde{T} \subset S^2 \times S^1$. In \cite{MR4332675}, Willis generalized this result to links in $\#^r (S^2 \times S^1)$. Hence, given our notion of Hochschild homology for quasi-associative algebras, one would hope that we recover similar unified and odd invariants by studying
\[
\mathsf{HH}_\bullet^{(\mathcal{G}, \alpha, \varepsilon)}(UH^n, \mathcal{F}(T)). 
\]
There are a few barriers to a result like this. Technically, as we mentioned before, the $\mathcal{G}$-graded $(H^n, H^n)$-bimodule $\mathcal{F}(T)$ is not \emph{quite} an invariant of $T$; the $\mathcal{G}$-grading is, in a precise sense, too sensitive. However, a simple procedure of reducing the $\mathcal{G}$-grading to an integral (quantum) grading yields an honest tangle invariant. This means that we must adjust the definition of $\mathsf{HH}_\bullet^{(\mathcal{G}, \alpha, \varepsilon)}$ if we hope to obtain an odd analog of the Rozansky-Willis invariant. Unfortunately, comparisons between the odd categorified projectors from \cite{spyropoulos2024} (see also \cite{Sch_tz_2022}) and Stoffregen and Willis's spectral projectors \cite{stoffregen2024joneswenzlprojectorskhovanovhomotopy} make the existence of such an invariant doubtful; namely, Stoffregen-Willis proved that $\mathsf{THH}_\bullet(\mathcal{H}^n, \mathcal{X}(T))$ is \emph{not} an invariant of $\widetilde{T}$, where $\mathsf{THH}$ is \emph{topological Hochschild homology}, and $\mathcal{H}^n$ and $\mathcal{X}$ are the spectral versions of $H^n$ and $\mathcal{F}$ respectively.

\subsection*{Organization of the paper}

This paper is organized as follows.

In Section \ref{s:algebra}, we review grading categories and relevant algebraic constructions. This includes a new study of the relationship between $(\mathcal{C}, \alpha)$-graded algebras and their opposites in Subsection \ref{ss:algebras_and_ops}. One of our main goals is to provide a correspondence between $(\mathcal{C}, \alpha)$-graded $(A, B)$-bimodules and $(\mathcal{C}, \alpha)$-graded left $(A \otimes_\Bbbk B^\op)$-modules; this is given in Subsection \ref{ss:mods_and_bimods}. We conclude the section by describing a bar construction for $(\mathcal{C}, \alpha)$-graded algebras (which has the structure of a $(\mathcal{C}, \alpha)$-graded DG-$(A, A)$-bimodule) in Subsection \ref{ss:hochschild}.

Section \ref{s:hochschild} is devoted to the proof of Theorem \ref{thm1}, which follows from the construction of a theory of Hochschild homology for $(\mathcal{C}, \alpha)$-graded algebras with coefficients in $(\mathcal{C}, \alpha)$-graded bimodules. In Subsection \ref{ss:cotrace}, we reveal that the Hochschild homology of such objects is naturally $\mathbb{Z} \times \mathrm{Tr}(\mathcal{C})$-graded, for $\mathrm{Tr}(\mathcal{C})$ the universal trace of $\mathcal{C}$, though existence of a Hochschild homology depends on the existence of a particular function $\varepsilon$; moreover, if $\varepsilon$ exists, the choice of such a function must be fixed. In Subsection \ref{ss:hochcomplex}, we perform the typical elementary computations, obtaining a notion of the module of coinvariants. If $A$ is a ``commutative $(\mathcal{C}, \alpha, \varepsilon)$-graded algebra,'' we also show that there is a notion of K\"ahler differentials $\Omega_A^1$ and that $\mathsf{HH}_1(A) \cong \Omega_A^1$. However, we are unable to provide an example of such a non-associative algebra, and satiate ourselves with a list of potential candidates. Finally, in Subsection \ref{ss:shadows}, we prove that the choice of $\varepsilon$ fixes a choice of natural isomorphism guaranteeing that the zeroth Hochschild homology is a shadow on the bicategory of $(\mathcal{C}, \alpha)$-graded bimodules.

In Section \ref{s:cohochschild}, we see that a definition of Hochschild cohomology for quasi-associative algebras is far more straightforward, and follows immediately from the existence of a bar construction for such algebras; thus, the proof of Theorem \ref{thm2} follows basically from definition. In Subsection \ref{ss:ccenter}, we study the center, in a $(\mathcal{C},\alpha)$-graded sense, of quasi-associative algebras. This serves as the main input for the proof of up-to-unit functoriality in Section \ref{s:utu_fun}. As in the classical case, the Hochschild cocomplex for quasi-associative algebras can be reduced, replacing the $n$th chain group $\mathrm{Hom}(A^{\otimes n+2}, M)$ with $\mathrm{Hom}(A^{\otimes n}, M)$ in a certain sense; we prove this fact in Subsection \ref{ss:reduction}. This allows us to perform the usual computations, obtaining notions of the center, module of invariants, and module of derivations for quasi-associative structures. We also note that Hochschild cohomology, while easier to define, does not possess the structure of a coshadow (simply by the fact that the bicategory of $\mathcal{C}$-graded bimodules is not closed). All the work of Section \ref{s:cohochschild} prior to Subsection \ref{ss:nonhomoHH} occurs in the setting of $\mathcal{C}$-grading preserving morphisms; in Subsection \ref{ss:nonhomoHH}, we show that Hochschild cohomology for quasi-associative algebras exists in the more general setting of $\mathcal{C}$-homogeneous morphisms, though it requires background on $\mathcal{C}$-shifting systems.

Section \ref{s:odd} consists mostly of background. In Subsection \ref{ss:chroncob}, we start by introducing Putyra's category of chronological cobordisms $\mathbf{ChCob}$ and all other background necessary for studying unified Khovanov homology. From Putyra's \emph{chronological TQFT}, we introduce the unified (and odd) arc algebras in Subsection \ref{ss:oddarcs}. This background can be found in much greater detail in either of \cite{putyra20152categorychronologicalcobordismsodd} or \cite{naisse2020odd}. In Subsection \ref{ss:tracecat}, we conclude with a description of the grading category $(\mathcal{G}, \alpha)$; for more on this grading category, and generalizations of it, see \cite{naisse2020odd} and \cite{spyropoulos2024}. The only new result of Section \ref{s:odd} is that $\mathcal{G}$ admits a function $\varepsilon$ coherent with Naisse and Putyra's choice of associator which (by the work in Section \ref{s:hochschild}) is enough to conclude that there exists a notion of Hochschild homology for $(\mathcal{G}, \alpha)$-graded algebras.

We conclude this paper with the proof of Theorem \ref{thm3}, showing that the Naisse-Putyra tangle invariant is functorial up to unit of $R$ in Section \ref{s:utu_fun}. This section follows Khovanov's paper \cite{MR2171235} closely. In Subsection \ref{ss:unicent}, we prove that the $(\mathcal{G}, \alpha)$-graded center of any unified arc algebra is $R$. In Subsection \ref{ss:geobims}, we define generalizations of Khovanov's geometric bimodules, giving a brief outline of Naisse and Putyra's system of $\mathcal{G}$-grading shifts in the process. In Subsection \ref{ss:khosadjuns}, we provide a few ``$(\mathcal{G}, \alpha)$-lifts'' of familiar adjunctions provided by Khovanov. Finally, we define the Naisse-Putyra tangle invariant and the 2-functor candidate in Subsection \ref{ss:NPinvariant}, and prove up-to-unit functoriality in Subsection \ref{ss:conclusion}.
\medskip

For readers in a hurry, we list two possible routes with different objectives.

\subsubsection*{Proof of functoriality for Naisse-Putyra tangle invariant route}

Only the basic definitions of Section \ref{s:algebra} are needed. Section \ref{s:hochschild} can be skipped altogether. The majority of Subsection \ref{ss:ccenter} is needed, and we suggest that the reader refer back to Subsection \ref{ss:nonhomoHH} only as needed. Subsections \ref{ss:chroncob} and \ref{ss:oddarcs} are needed in their entirety, and only the definition of the grading category $(\mathcal{G}, \alpha)$ is needed from Subsection \ref{ss:tracecat}. Then, read Section \ref{s:utu_fun}.

\subsubsection*{Hochschild (co)homology for quasi-associative algebras route}

Read Sections \ref{s:algebra}, \ref{s:hochschild}, and \ref{s:cohochschild} only.

\subsection*{Acknowledgement}

The author would like to thank Matthew Stoffregen for his guidance and support, as well as Keegan Boyle, David Chan, Teena Gerhardt, Matt Hogancamp, Louisa Liles, and Stephan Wehrli for their interest and helpful discussions.

\resumetoc

\section{$\mathcal{C}$-graded algebras, bimodules, and bar resolution}
\label{s:algebra}

In this section, we define grading categories as introduced by Naisse-Putyra in \cite{naisse2020odd} and develop an expanded theory. This includes a careful study of algebras and modules graded by a grading category in Subsections \ref{ss:algebras_and_ops} and \ref{ss:mods_and_bimods}.  We also write down an explicit bar construction for such objects in Subsection \ref{ss:hochschild}. The bar construction will be essential to our definitions of Hochschild homology and cohomology for quasi-associative algebras in Sections \ref{s:hochschild} and \ref{s:cohochschild}.

Throughout, $\mathcal{C}$ is a small category, and $\Bbbk$ is a unital commutative ring whose units are denoted by $\Bbbk^\times$. The notation $\mathcal{C}^{[n]}$ denotes the set of paths of morphisms of length $n$ in $\mathcal{C}$.

\begin{definition}
	A \emph{grading category} is a pair $(\mathcal{C}, \alpha)$ where $\alpha: \mathcal{C}^{[3]} \to \Bbbk^\times$ is a 3-cocycle on the nerve of $\mathcal{C}$; that is, for any path of length four $f_4 \circ f_3 \circ f_2 \circ f_1 \in \mathcal{C}^{[4]}$, 
	\[
	d\alpha(f_1, f_2, f_3, f_4) := \alpha(f_1, f_2, f_3) \alpha(f_1, f_2, f_4 \circ f_3)^{-1} \alpha(f_1, f_3 \circ f_2, f_4) \alpha(f_2 \circ f_1, f_3, f_4)^{-1} \alpha(f_2, f_3, f_4) =1.
	\]
\end{definition}

A \emph{$(\mathcal{C}, \alpha)$-graded} (or, more commonly, \emph{$\mathcal{C}$-graded})  \emph{$\Bbbk$-module} is defined as a $\Bbbk$-module $M$ admitting decomposition
\[
	M = \bigoplus_{g\in \mathrm{Mor}(\mathcal{C})} M_g
\]
where $\mathrm{Mor}(\mathcal{C})$ denotes the set of morphisms in $\mathcal{C}$. If $x\in M_g$, we will write $\left| x \right|_{\mathcal{C}} = g$, or simply $\left| x \right| = g$ when clear. A graded map $f: M \to N$ between two $\mathcal{C}$-graded modules $M$ and $N$ is a $\Bbbk$-linear map satisfying $f(M_g) \subset N_g$ for all $g \in \mathrm{Mor}(\mathcal{C})$. Let $\mathrm{Mod}^{(\mathcal{C}, \alpha)}$ (or, more commonly, $\mathrm{Mod}^\mathcal{C}$) denote the category whose objects are $\mathcal{C}$-graded $\Bbbk$-modules and whose morphisms are graded maps.

The utility of the 3-cocyle $\alpha$ is that it endows $\mathrm{Mod}^\mathcal{C}$ with a monoidal structure. Define a monoidal product $M \otimes N$ by
\[
M \otimes N := \bigoplus_{g \in \mathrm{Mor}(\mathcal{C})} (M \otimes N)_g \qquad \text{where} \qquad (M \otimes N)_g:= \bigoplus_{g = g_2 \circ g_1} M_{g_1} \otimes N_{g_2}.
\]
Then, the coherence isomorphism is induced by the associator: fix $\alpha: (M_1 \otimes M_2) \otimes M_3 \to M_1 \otimes (M_2 \otimes M_3)$ by
\[
(x \otimes y) \otimes z \mapsto \alpha(\abs{x}, \abs{y}, \abs{z}) x \otimes (y \otimes z)
\]
for homogeneous elements $x$, $y$, and $z$. The fact that $\alpha$ satisfies the pentagon relation follows directly from the 3-cocycle condition of the grading category. The unit object is given by
\[
I_\mathcal{C} := \bigoplus_{X \in \mathrm{Ob}(\mathcal{C})} \Bbbk_{\mathrm{Id}_X}
\]
where $\mathrm{Id}_X$ denotes the identity morphism on $X \in \mathrm{Ob}(\mathcal{C})$. In general, left- and right-unitors $\mathcal{L}: I_\mathcal{C} \otimes M \to M$ and $\mathcal{R}: M \otimes I_\mathcal{C} \to M$ are given by any isomorphisms which satisfy the triangle relation:
\[
\begin{tikzcd}
(M \otimes I_\mathcal{C}) \otimes N \arrow[rr, "\alpha"] \arrow[dr, "\mathcal{R} \otimes \mathbbm{1}_{N}"'] & & M \otimes (I_\mathcal{C} \otimes N) \arrow[dl, "\mathbbm{1}_M \otimes \mathcal{L}"] \\
& M \otimes N &
\end{tikzcd}
\]
When needed, we will denote the chosen unitors for $\mathrm{Mod}^\mathcal{C}$ by $\mathcal{L}_\mathcal{C}$ and $\mathcal{R}_\mathcal{C}$. Indeed, the unitors can be chosen to be induced by the associator. For example, one can take
\begin{itemize}
	\item $\mathcal{L}: I_\mathcal{C} \otimes M \to M$ by $(k \otimes m ) \mapsto \mathcal{L}(\abs{k}, \abs{m}) km$, fixing
	\begin{equation}
	\label{eq:typicalL}
		\mathcal{L}(\abs{k}, \abs{m}) := \alpha(\mathrm{Id}_X, \mathrm{Id}_X, \abs{m})^{-1},
	\end{equation}
	and
	\item $\mathcal{R}: M \otimes I_\mathcal{C} \to M$ by $(m \otimes k ) \mapsto \mathcal{R}(\abs{m}, \abs{k}) km$, fixing
	\begin{equation}
	\label{eq:typicalR}
		\mathcal{R}(\abs{m}, \abs{k}) := \alpha(\abs{m}, \mathrm{Id}_Y, \mathrm{Id}_Y),
	\end{equation}
\end{itemize}
where $\abs{m} : X \to Y$. To see that the triangle relation is satisfied, notice that for $X \xrightarrow{g} Y \xrightarrow{h} Z$,
\[
1 = d\alpha(g, \mathrm{Id}_Y, \mathrm{Id}_Y, h) = \alpha(g, \mathrm{Id}_Y, \mathrm{Id}_Y) \alpha(g, \mathrm{Id}_Y, h)^{-1} \alpha(\mathrm{Id}_Y, \mathrm{Id}_Y, h).
\]
We will call the unitors given by equations (\ref{eq:typicalL}) and (\ref{eq:typicalR}) the \emph{typical unitors induced by $\alpha$}.

In general, the cocycle relation implies $\alpha(g, g, g)=1$ for any idempotent loop morphism $g: X \to X$. If the typical unitors induced by $\alpha$ are chosen, this means that whenever $\abs{m} = \mathrm{Id}_X$ for any $X\in \mathrm{Ob}(\mathcal{C})$, we have that $\mathcal{L}(k \otimes m) = km = \mathcal{R}(m \otimes k)$. Notice that, in general, the requirement that
\[
(\mathbbm{1}_M \otimes \mathcal{L}) \circ \alpha = \mathcal{R} \otimes \mathbbm{1}_N
\]
implies only that the values associated to $\mathcal{L}$ and $\mathcal{R}$ agree on $m\in M$ with $\abs{m} = \mathrm{Id}_X$. When the typical unitors are taken, we have that $\mathcal{L} \equiv 1$ and $\mathcal{R} \equiv 1$ on any elements $m \in M_{\mathrm{Id}_X} \subset M$ for any $X \in \mathrm{Ob}(\mathcal{C})$.

In conclusion, we list a few quick computations regarding the associator which we will want to have in our back-pocket.

\begin{lemma}
\label{lem:assocomps}
Let $g, h\in \mathrm{Mor}(\mathcal{C})$ and $g: X\to Y$ and $h: Y \to Z$. We have the following equivalences, with their paths pictured.
\begin{enumerate}[label = (\roman*)]
	\item $\alpha(\mathrm{Id}_X, g, \mathrm{Id}_Y) = 1$.
	\[
	\begin{tikzcd}
	X \arrow[r, "g"] \arrow[loop above, "\mathrm{Id}_X"] & Y \arrow[loop above, "\mathrm{Id}_Y"]
	\end{tikzcd}
	\]
	\item $\alpha(\mathrm{Id}_X, \mathrm{Id}_X, h \circ g) = \alpha(\mathrm{Id}_X, g, h) \alpha(\mathrm{Id}_X, \mathrm{Id}_X, g)$.
	\[
	\begin{tikzcd}
	X \arrow[r, "g"] \arrow[loop above, "\mathrm{Id}_X"] & Y \arrow[r, "h"] & Z
	\end{tikzcd}
	\]
	\item $\alpha(h \circ g, \mathrm{Id}_Z, \mathrm{Id}_Z) = \alpha(g, h, \mathrm{Id}_Z) \alpha(h, \mathrm{Id}_Z, \mathrm{Id}_Z)$.
	\[
	\begin{tikzcd}
	X \arrow[r, "g"] & Y \arrow[r, "h"] & Z \arrow[loop above, "\mathrm{Id}_Z"]
	\end{tikzcd}
	\]
	\item $\alpha(g, \mathrm{Id}_Y, h) = \alpha(g, \mathrm{Id}_Y, \mathrm{Id}_Y) \alpha(\mathrm{Id}_Y, \mathrm{Id}_Y, h)$.
	\[
	\begin{tikzcd}
	X \arrow[r, "g"] & Y \arrow[r, "h"] \arrow[loop above, "\mathrm{Id}_Y"] & Z
	\end{tikzcd}
	\]
	\item $\alpha(f, f, f) = 1$ whenever $f$ is an idempotent morphism (that is, a loop $f: X \to X$ with $f \circ f = f$).
\end{enumerate}
\end{lemma}

\begin{proof}
Each of these are routine; we will prove (ii) as demonstration. We have
\[
1 = d\alpha (\mathrm{Id}_X, \mathrm{Id}_X, g, h) = \alpha(\mathrm{Id}_X,\mathrm{Id}_X, g) \alpha(\mathrm{Id}_X,\mathrm{Id}_X, h \circ g)^{-1} \alpha(\mathrm{Id}_X, g, h)
\]
as desired.
\end{proof}

\begin{lemma}
\label{lem:genassocomps}
Suppose $f, g, h \in \mathrm{Mor}(\mathcal{C})$ form the following directed graph:
\[
\begin{tikzcd}
W \arrow[loop above, "\mathrm{Id}_W"] \arrow[r, "f"] & X \arrow[r, "g"] & Y \arrow[r, "h"] & Z \arrow[loop above, "\mathrm{Id}_Z"].
\end{tikzcd}
\]
Then
\[
\alpha(\mathrm{Id}_W, f, h\circ g) = \alpha(\mathrm{Id}_W, f, g) \alpha(\mathrm{Id}_W, g \circ f, h)
\qquad \text{and} \qquad
\alpha(g \circ f, h, \mathrm{Id}_Z) = \alpha(g, h, \mathrm{Id}_Z) \alpha(f, h\circ g, \mathrm{Id}_Z)
\]
generalizing (ii) and (iii) of Lemma \ref{lem:assocomps}
\end{lemma}

\begin{remark}
We separate Lemmas \ref{lem:assocomps} and \ref{lem:genassocomps} only because we frequently use the specializations (ii) and (iii) of Lemma \ref{lem:assocomps}, and hardly ever the more general statement of Lemma \ref{lem:genassocomps}.
\end{remark}

\subsection{$\mathcal{C}$-graded algebras and their opposites}
\label{ss:algebras_and_ops}

Granted a monoidal structure on $\mathrm{Mod}^\mathcal{C}$, it is simple to define the usual algebraic objects. 

\begin{definition}
\label{def:algebras}
	A \emph{$\mathcal{C}$-graded algebra} is a $\mathcal{C}$-graded $\Bbbk$-module $A = \bigoplus_{g\in \mathrm{Mor}(\mathcal{C})} A_g$ endowed with a $\Bbbk$-linear multiplication $\mu_A: A \otimes A \to A$ and unit element $1_X \in A_{\mathrm{Id}_X}$ for each $X \in \mathrm{Ob} (\mathcal{C})$ for which each of the following hold.
	\begin{enumerate}[label=(A.\Roman*), align=left]
		\item $\mu_A$ is a $\mathcal{C}$-graded map; that is, for each homogeneous $x, y \in A$,  $\abs{\mu_A(x, y)} = \abs{y} \circ \abs{x}$.
		\item $\mu_A$ is graded associative; that is, for each homogeneous $x, y, z \in A$, 
		\[
			\mu_A(\mu_A(x, y), z) = \alpha(\abs{x}, \abs{y}, \abs{z}) \mu_A(x, \mu_A(x,y)). 
		\]
		\item For each homogeneous $x \in A$,
		\[
			\mu_{A}(1_X, x) = \mathcal{L}(\mathrm{Id}_X, \abs{x}) x \qquad \text{and} \qquad \mu_{A}(x, 1_Y) = \mathcal{R}(\abs{x}, \mathrm{Id}_Y) x
		\]
		where $\abs{x}: X \to Y$. 
 	\end{enumerate}
\end{definition}

Notice that if our choice of unitors in $\mathrm{Mod}^\mathcal{C}$ is typical, we have that $\mu_A(1_X, 1_X) = 1_X$. It will be necessary to pick typical unitors induced by $\alpha$ later on (\textit{e.g.}, the proof of Proposition \ref{prop:bimod=mod} and thereafter, but not before). Also, notice that the ``units'' of a $\mathcal{C}$-graded algebra are not units in the usual sense, but should be thought of as a set of mutually orthogonal idempotents; the non-homogeneous element
\[
1 := \sum_{X \in \mathrm{Ob}(\mathcal{C})} 1_X
\]
is a more natural ``unit'' for $A$.

\begin{example}
\label{eg:trivgc}
Consider the category $\mathcal{Z} := B\mathbb{Z}$ with a single object $\star$ and $\mathrm{Hom}_\mathcal{Z}(\star,\star) = \mathbb{Z}$. Extend $\mathcal{Z}$ to a grading category trivially: that is, take $\alpha \equiv 1$. Thus, for the grading category $(\mathcal{Z}, 1)$, a $\mathcal{Z}$-graded object is the same thing as a $\mathbb{Z}$-graded object. In general, if $BG$ denotes the category with a single object $\star$ and $\mathrm{Hom}_{BG}(\star,\star) = G$ for $G$ a group, then we recover grading by arbitrary groups, as defined by Albequerque and Majid \cite{albuquerque1998quasialgebra}, as well as their notion of quasi-associativity.
\end{example}

Some of the usual operations performed on small categories can be extended to operations on grading categories. For motivation, suppose $A$ is a $\mathcal{C}$-graded algebra, and consider $A^\op$. Recall that $A^\op$ is simply $A$ but with multiplication defined by
\[
	\mu_{A^{\op}} (x, y) := \mu_A(y, x).
\]
Then, notice that $A^\op$ fails to be a $\mathcal{C}$-graded algebra: gradings are reversed, and the associator is inverted. However, $A^\op$ has a natural description as a $\mathcal{C}^\op$-graded algebra. Recall that the category \emph{opposite} $\mathcal{C}$, denoted $\mathcal{C}^\op$, is the category with
\begin{itemize}
	\item $\mathrm{Ob}(\mathcal{C}^\op) = \mathrm{Ob}(\mathcal{C})$, and
	\item $\mathrm{Hom}_{\mathcal{C}^\op}(X, Y) = \mathrm{Hom}_\mathcal{C}(Y, X)$.
\end{itemize}
Notice that, if $X \xrightarrow{f} Y \xrightarrow{g} Z$ is a sequence of morphisms in $\mathcal{C}$, then $(X \xrightarrow{f} Y \xrightarrow{g} Z)^\op = Z \xrightarrow{g^\op} Y \xrightarrow{f^\op} X$. That is, the functor $\op: \mathcal{C} \to \mathcal{C}^\op$ is contravariant, and $(\mathcal{C}^\op)^\op = \mathcal{C}$.

\begin{definition}
	Let $(\mathcal{C}, \alpha)$ be a grading category. Let $(\mathcal{C}, \alpha)^\op := (\mathcal{C}^\op, \alpha^\op)$ denote the \emph{opposite grading category}, with $\alpha^\op: (\mathcal{C}^\op)^{[3]} \to \Bbbk^\times$ defined by
	\[
	\alpha^\op(f_3^\op, f_2^\op, f_1^\op) := \alpha(f_1, f_2, f_3)^{-1}.
	\]
\end{definition}

\begin{proposition}
Assume $(\mathcal{C}, \alpha)$ is a grading category and $A$ is a $\mathcal{C}$-graded algebra. Then $(\mathcal{C}, \alpha)^\op$ is a grading category, and $A^\op$ is a $\mathcal{C}^\op$-graded algebra.
\end{proposition}

\begin{proof}
For the first claim, note that $d(\alpha^\op)(f_4^\op, f_3^\op, f_2^\op, f_1^\op) = d\alpha(f_1, f_2, f_3, f_4)^{-1}$, so the result follows from the assumption that $(\mathcal{C}, \alpha)$ is a grading category. For the second claim, given a decomposition $A = \bigoplus_{g\in \mathrm{Mor}(\mathcal{C})} A_g$, choose the decomposition $A^\op = \bigoplus_{g^\op \in \mathrm{Mor}(\mathcal{C}^\op)}$. Requirement (A.I) is satisfied since
\[
\abs{\mu_{A^\op}(x, y)}_{\mathcal{C}^\op} = \left(\abs{\mu_A(y,x)}_\mathcal{C}\right)^\op = \left(\abs{x}_\mathcal{C} \circ \abs{y}_\mathcal{C}\right)^\op = \abs{y}_\mathcal{C^\op} \circ \abs{x}_\mathcal{C^\op}
\]
using the fact that $(\abs{x}_\mathcal{C})^\op = \abs{x}_{\mathcal{C}^\op}$. Requirement (A.II) is similar:
\begin{align*}
\mu_{A^\op}(\mu_{A^\op}(x,y), z) &= \mu_A(z, \mu_A(y,x)) \\ &= \alpha(\abs{z}_\mathcal{C}, \abs{y}_\mathcal{C}, \abs{x}_\mathcal{C})^{-1}\mu_A(\mu_A(z,y), x)\\ &= \alpha^\op(\abs{x}_{\mathcal{C}^\op}, \abs{y}_{\mathcal{C}^\op}, \abs{z}_{\mathcal{C}^\op}) \mu_{A^\op} (x, \mu_{A^\op}(y,z)).
\end{align*}
(Notice that this is why we must invert the associator.) Finally, for (A.III), notice that the unit object $I_{\mathcal{C}^\op}$ is exactly $I_{\mathcal{C}}$. Then, sufficient unitors for $\mathrm{Mod}^{\mathcal{C}^\op}$ are provided by fixing $\mathcal{L}_{\mathcal{C}^\op} = \mathcal{R}_\mathcal{C}$ and $\mathcal{R}_{\mathcal{C}^\op} = \mathcal{L}_\mathcal{C}$.
\end{proof}

Now, suppose $A$ and $B$ are $\mathcal{C}$-graded and $\mathcal{D}$-graded algebras respectively. Abusing notation, we will write $A \otimes_\Bbbk B$ to denote the tensor product of $A$ and $B$ as $\Bbbk$-modules, rather than the monoidal product in either of $\mathrm{Mod}^\mathcal{C}$ or $\mathrm{Mod}^\mathcal{D}$. Confusing notation further, we sometimes write $\otimes$ for $\otimes_\Bbbk$ when context makes the meaning clear. The graded structure on $A$ and $B$ induces one on $A\otimes_\Bbbk B$ as follows. Recall that the product category $\mathcal{C} \times \mathcal{D}$ of two categories $\mathcal{C}$ and $\mathcal{D}$ is the one with
\begin{itemize}
	\item $\mathrm{Ob}(\mathcal{C} \times \mathcal{D}) = \mathrm{Ob}(\mathcal{C}) \times \mathrm{Ob}(\mathcal{D})$,
	\item $\mathrm{Hom}_{\mathcal{C} \times \mathcal{D}}((X_1, X_2), (Y_1, Y_2)) = \mathrm{Hom}_\mathcal{C}(X_1, Y_1) \times \mathrm{Hom}_\mathcal{D}(X_2, Y_2)$,
	\item composition defined by $(f_2, g_2) \circ (f_1, g_1) = (f_2\circ f_1, g_2 \circ g_1)$, and
	\item identity morphisms $\mathrm{Id}_{(X, Y)} = (\mathrm{Id}_X, \mathrm{Id}_Y)$.
\end{itemize}

\begin{definition}
Given grading categories $(\mathcal{C}, \alpha)$ and $(\mathcal{D}, \beta)$, define the \emph{product grading category} $(\mathcal{C}, \alpha) \times (\mathcal{D}, \beta) := (\mathcal{C} \times \mathcal{D}, \alpha \times \beta)$ where
\[
(\alpha \times \beta) ((f_1, g_1), (f_2, g_2), (f_3, g_3)) := \alpha(f_1, f_2, f_3) \beta(g_1, g_2, g_3).
\]
\end{definition}

\begin{proposition}
If $(\mathcal{C}, \alpha)$ and $(\mathcal{D}, \beta)$ are grading categories, then so is $(\mathcal{C} \times \mathcal{D}, \alpha \times \beta)$. Moreover, if $A$ is a $(\mathcal{C}, \alpha)$-graded algebra and $B$ is a $(\mathcal{D}, \beta)$-graded algebra, then $A\otimes_\Bbbk B$ is a $(\mathcal{C} \times \mathcal{D}, \alpha \times \beta)$-graded algebra.
\end{proposition}

\begin{proof}
The first claim is immediate and the second is routine: in general, we interpret $A \otimes_\Bbbk B$ as a $(\mathcal{C} \times \mathcal{D})$-graded algebra by taking $\abs{a \otimes b}_{\mathcal{C} \times \mathcal{D}} := (\abs{a}_\mathcal{C}, \abs{b}_\mathcal{D})$ and defining the multiplication $\mu_{A \otimes_\Bbbk B}: (A \otimes_\Bbbk B) \otimes (A \otimes_\Bbbk B) \to A \otimes_\Bbbk B$ as
\[
\mu_{A \otimes_\Bbbk B} (a_1 \otimes b_1, a_2 \otimes b_2):= \mu_A(a_1, a_2) \otimes \mu_B(b_1, b_2).
\]
Then, for example, check (A.I) by computing
\begin{align*}
	\abs{\mu_{A\otimes B}(a_1 \otimes b_1, a_2 \otimes b_2)}_{\mathcal{C} \times \mathcal{D}} &= \abs{\mu_A(a_1, a_2) \otimes \mu_B(b_1, b_2)}_{\mathcal{C} \times \mathcal{D}}
	\\& = \left( \abs{\mu_A(a_1, a_2)}_\mathcal{C}, \abs{\mu_B(b_1, b_2)}_\mathcal{D} \right) 
	\\& = \left(\abs{a_2}_\mathcal{C} \circ \abs{a_1}_\mathcal{C}, \abs{b_2}_\mathcal{D} \circ \abs{b_1}_\mathcal{D} \right) 
	\\& = (\abs{a_2}_\mathcal{C}, \abs{b_2}_\mathcal{D}) \circ (\abs{a_1}_\mathcal{C}, \abs{b_1}_\mathcal{D}) =: \abs{a_2 \otimes b_2}_{\mathcal{C} \times \mathcal{D}} \circ \abs{a_1 \otimes b_1}_{\mathcal{C} \times \mathcal{D}}.
\end{align*}
Checking (A.II) is also routine. To check (A.III), we note that, as $\mathrm{Mod}^{\mathcal{C} \times \mathcal{D}}$ inherits its coherence isomorphism from $\mathrm{Mod}^\mathcal{C}$ and $\mathrm{Mod}^\mathcal{D}$, its unitors may also be chosen from these categories, defining $\mathcal{L}_{\mathcal{C} \times \mathcal{D}} := \mathcal{L}_\mathcal{C} \times \mathcal{L}_\mathcal{D}$, and similarly for the right unitor $\mathcal{R}_{\mathcal{C} \times \mathcal{D}}$. Also fix unit elements $1_{(X, Y)} \in A_{\mathrm{Id}_{(X, Y)}}$ to be $1_X \otimes 1_Y$, recalling that, by definition, $\mathrm{Id}_{(X, Y)} = (\mathrm{Id}_X, \mathrm{Id}_Y)$. Then the checks required for (A.III) are also straightforward: for example,
\begin{align*}
	\mu_{A \otimes_\Bbbk B} (1_{(X,Y)}, a \otimes b) &= \mu_A(1_X, a) \otimes \mu_B(1_Y, b)
	\\& = \mathcal{L}_\mathcal{C}(\mathrm{Id}_X, \abs{a}_\mathcal{C}) \mathcal{L}_\mathcal{D}(\mathrm{Id}_Y, \abs{b}_\mathcal{D}) a \otimes b
	\\& = \mathcal{L}_{\mathcal{C} \times \mathcal{D}} ((\mathrm{Id}_X, \mathrm{Id}_Y), (\abs{a}_\mathcal{C}, \abs{b}_\mathcal{D})) a \otimes b
	\\& = \mathcal{L}_{\mathcal{C} \times \mathcal{D}} (\mathrm{Id}_{(X,Y)}, \abs{a \otimes b}_{\mathcal{C}\times \mathcal{D}}) a \otimes b.
\end{align*}
The check for $\mathcal{R}_{\mathcal{C} \times \mathcal{D}}$ is totally analagous.
\end{proof}

\subsection{Modules and bimodules}
\label{ss:mods_and_bimods}

In this subsection, we study $\mathcal{C}$-graded modules and bimodules. Our goal is to describe in what way $(A, A)$-bimodules and left $A^e := A \otimes_\Bbbk A^\op$-modules coincide in the $\mathcal{C}$-graded context.

\begin{definition}
\label{def:bims}
Suppose $A$ and $B$ are $\mathcal{C}$-graded algebras. We define a \emph{$\mathcal{C}$-graded $(A, B)$-module} as a $\mathcal{C}$-graded $\Bbbk$-module with graded, $\Bbbk$-linear actions
\[
\rho_L: A \otimes M \to M \qquad \text{and} \qquad \rho_R: M \otimes B \to M
\]
which satisfy the following axioms for each homogeneous $a, a' \in A$, $b, b'\in B$, and $m\in M$.
\begin{enumerate}[label=(B.\Roman*), align=left]
	\item $\rho_L(\mu_A(a,a'), m) = \alpha(\abs{a}, \abs{a'}, \abs{m}) \rho_L(a, \rho_L(a', m))$;
	\item $\rho_R(\rho_R(m,b), b') = \alpha(\abs{m}, \abs{b}, \abs{b'}) \rho_R(m, \mu_A(b,b'))$;
	\item $\rho_R(\rho_L(a,m), b) = \alpha(\abs{a}, \abs{m}, \abs{b}) \rho_L(a, \rho_R(m,b))$;
	\item $\rho_L(1_X, m) = \mathcal{L}(\mathrm{Id}_X, \abs{m}) m$ and $\rho_R(m, 1_Y) = \mathcal{R}(\abs{m}, \mathrm{Id}_Y) m$.
\end{enumerate}
We define a \emph{$\mathcal{C}$-graded left $A$-module (resp. right $B$-module)} as a $\mathcal{C}$-graded $(A, I_\mathcal{C})$-bimodule (resp. $(I_\mathcal{C}, B)$-bimodule).
\end{definition}

Evidently, we can think of a left (resp. right) $\mathcal{C}$-graded $A$-module as a $\mathcal{C}$-graded $\Bbbk$-module with a single graded, $\Bbbk$-linear action $\rho_L$ (resp. $\rho_R$) satisfying (B.I) (resp. (B.II)) and the first (resp. second) half of (B.IV).

\begin{proposition}
$M$ is a $\mathcal{C}$-graded left (resp. right) $A$-module if and only if it is a $\mathcal{C}^\op$-graded right (resp. left) $A^\op$-module.
\end{proposition}

\begin{proof}
Assuming $M$ is a $\mathcal{C}$-graded left $A$-module means that it has a left action $\rho_L: A \otimes M \to M$ which satisfies
\[
\rho_L(\mu_A(x,y), m) = \alpha(\abs{x}_\mathcal{C}, \abs{y}_\mathcal{C}, \abs{m}_\mathcal{C}) \rho_L(x, \rho_L(y,m))
\]
and
\[
\rho_L(1_Y, m) = \mathcal{L}(\mathrm{Id}_Y, \abs{m}_\mathcal{C}) m.
\]
We want to show that $M$ has a natural definition as a $\mathcal{C}^\op$-graded right $A^\op$-module. First, if $M = \bigoplus_{g\in \mathrm{Mor}(\mathcal{C})} M_g$, we reverse arrows as before to get an induced grading by $\mathcal{C}^\op$; \textit{i.e.}, $M = \bigoplus_{g^\op \in \mathrm{Mor}(\mathcal{C}^\op)} M_{g^\op}$. Then, define $\rho_R^\op: M \otimes A^\op \to M$ by $\rho_R^\op (m, a) := \rho_L(a, m)$. We compute
\begin{align*}
	\rho_R^\op (\rho_R^\op(m,x), y) &= \rho_L(y, \rho_L(x, m)) 
	\\& = \alpha(\abs{y}_\mathcal{C}, \abs{x}_\mathcal{C}, \abs{m}_\mathcal{C})^{-1} \rho_L(\mu_A(y,x),m) 
	\\& = \alpha^\op (\abs{m}_{\mathcal{C}^\op}, \abs{x}_{\mathcal{C}^\op}, \abs{y}_{\mathcal{C}^\op}) \rho_R^\op(m, \mu_{A^\op}(x,y))
\end{align*}
and
\[
\rho_R^\op(m, 1_X) = \rho_L(1_X, m) = \mathcal{L}_\mathcal{C}(\mathrm{Id}_X, \abs{m}_\mathcal{C}) m = \mathcal{R}_{\mathcal{C}^\op} (\abs{m}_{\mathcal{C}^\op}, \mathrm{Id}_X) m
\]
as desired. The other checks are similar.
\end{proof}

Assume $A$ and $B$ are both $\mathcal{C}$-graded algebras. To conclude this section, we want there to be an equivalence between $\mathcal{C}$-graded $(A, B)$-bimodules and $\mathcal{C}$-graded left $A\otimes_\Bbbk B^\op$-modules. The problem is that our current definition of modules assumes that the algebra and the module share the same grading category---in the latter instance, $A\otimes_\Bbbk B^\op$ is a $\mathcal{C} \times \mathcal{C}^\op$-graded (rather than $\mathcal{C}$-graded) algebra. This prompts the following definition.

\begin{definition}
\label{def:left-Ae-modules}
Fix $\mathcal{C}$-graded algebras $A$ and $B$. Define a \emph{$\mathcal{C}$-graded left $A \otimes_\Bbbk B^\op$-module} to be a $\mathcal{C}$-graded $\Bbbk$-module $M$ with a left, $\Bbbk$-bilinear action map
\[
\rho_L^e: (A \otimes_\Bbbk B^\op) \times M \to M
\]
which is graded in the sense that $\abs{\rho_L^e(a \otimes b, m)} = \abs{b}_\mathcal{C} \circ \abs{m}_\mathcal{C} \circ \abs{a}_\mathcal{C}$, and the following hold.
\begin{enumerate}[label=(E.\Roman*), align=left]
	\item For $a_1, a_2 \in A$, $b_1, b_2\in B^\op$, and $m\in M$ homogeneous,
	\[
	\rho_L^e (\mu_{A \otimes B^\op} (a_1 \otimes b_1, a_2 \otimes b_2), m) = \Delta(\abs{a_1 \otimes b_1}_{\mathcal{C}\times \mathcal{C}^\op}, \abs{a_2 \otimes b_2}_{\mathcal{C}\times\mathcal{C}^\op}, \abs{m}_\mathcal{C}) \rho_L^e(a_1 \otimes b_1, \rho_L^e(a_2 \otimes b_2, m));
	\]
	\item for $(X, Y) \in \mathrm{Ob}(\mathcal{C} \times \mathcal{C}^\op)$, 
	\[
	\rho_L^e(1_{(X, Y)}, m) = \mathcal{L}_\mathcal{C} (\mathrm{Id}_X, \abs{m}_\mathcal{C}) \mathcal{R}_\mathcal{C} (\abs{m}_\mathcal{C}, \mathrm{Id}_Y) m
	\]
\end{enumerate}
where 
\[
\Delta(\abs{a_1 \otimes b_1}_{\mathcal{C}\times \mathcal{C}^\op}, \abs{a_2 \otimes b_2}_{\mathcal{C}\times\mathcal{C}^\op}, \abs{m}_\mathcal{C}) := \alpha(\abs{a_1}, \abs{a_2}, \abs{m}) \alpha(\abs{m} \circ \abs{a_2} \circ \abs{a_1}, \abs{b_2}, \abs{b_1})^{-1} \alpha(\abs{a_1}, \abs{m} \circ \abs{a_2}, \abs{b_2})
\]
with all gradings on the right-hand side taken in $\mathcal{C}$, fixing $\abs{b}_\mathcal{C} := (\abs{b}_{\mathcal{C}^\op})^\op$. When $B = A$, we write $A^e := A \otimes_\Bbbk A^\op$.
\end{definition}

Note that under the canonical identification $\abs{m}_{\mathcal{C}^\op} := (\abs{m}_\mathcal{C})^\op$, 
\[
\mathcal{L}_\mathcal{C} (\mathrm{Id}_X, \abs{m}_\mathcal{C}) \mathcal{R}_\mathcal{C} (\abs{m}_\mathcal{C}, \mathrm{Id}_Y) = \mathcal{L}_{\mathcal{C} \times \mathcal{C}^\op} (\mathrm{Id}_{(X,Y)}, \abs{m}_{\mathcal{C} \times \mathcal{C}^\op}).
\]
In addition, notice that the value denoted $\Delta$ can be obtained in many different ways. Indeed, $\Delta$ is the value associated to \emph{any} path between the red-labeled vertices of the associahedron pictured below.
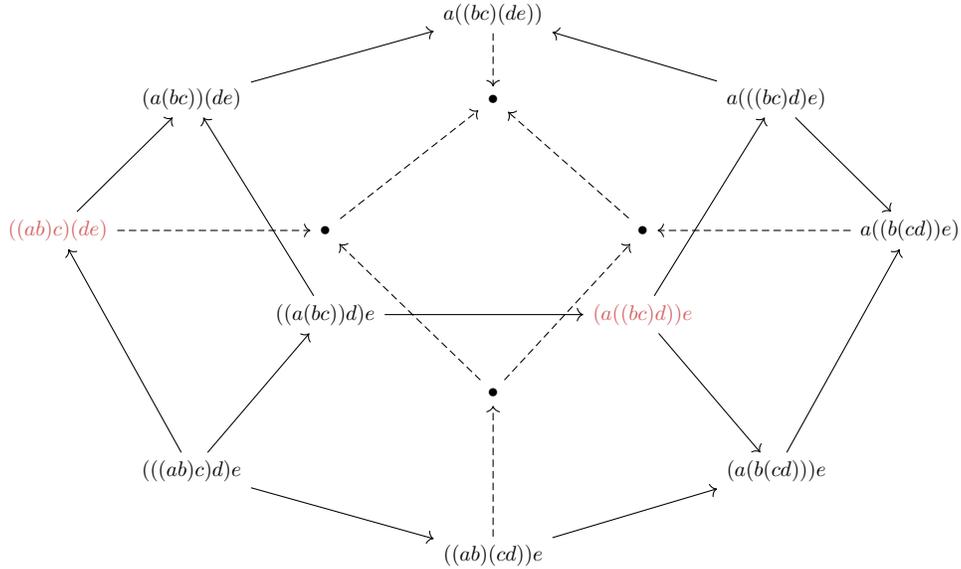
\[
\begin{tikzcd}[column sep = tiny, row sep = normal, scale cd=0.8]
	&&&& {a((bc)(de))} \\
	& {(a(bc))(de)} &&& \bullet &&& {a(((bc)d)e)} \\
	\\
	\textcolor{rgb,255:red,214;green,92;blue,92}{{((ab)c)(de)}} && \bullet & {} &&& \bullet && {a((b(cd))e)} \\
	&& {((a(bc))d)e} &&&& \textcolor{rgb,255:red,214;green,92;blue,92}{{(a((bc)d))e}} \\
	&&&& \bullet \\
	& {(((ab)c)d)e} &&&&&& {(a(b(cd)))e} \\
	&&&& {((ab)(cd))e}
	\arrow[dashed, from=1-5, to=2-5]
	\arrow[from=2-2, to=1-5]
	\arrow[from=2-8, to=1-5]
	\arrow[from=2-8, to=4-9]
	\arrow[from=4-1, to=2-2]
	\arrow[dashed, from=4-1, to=4-3]
	\arrow[dashed, from=4-3, to=2-5]
	\arrow[dashed, from=4-7, to=2-5]
	\arrow[dashed, from=4-9, to=4-7]
	\arrow[from=5-3, to=2-2]
	\arrow[from=5-3, to=5-7]
	\arrow[from=5-7, to=2-8]
	\arrow[from=5-7, to=7-8]
	\arrow[dashed, from=6-5, to=4-3]
	\arrow[dashed, from=6-5, to=4-7]
	\arrow[from=7-2, to=4-1]
	\arrow[from=7-2, to=5-3]
	\arrow[from=7-2, to=8-5]
	\arrow[from=7-8, to=4-9]
	\arrow[dashed, from=8-5, to=6-5]
	\arrow[from=8-5, to=7-8]
\end{tikzcd}
\]
The edges of the associahedron are understood to be labeled by $\alpha$. To see that any path determines the same value, notice that each pentagonal face commutes by the 3-cocycle relation. The square faces also commute: for example,
\[
\begin{tikzcd}
	& ((a(bc))d)e \arrow[dr, "\alpha(c \circ b \circ a\text{,}\, d\text{,}\, e)"] & 
	\\
	(((ab)c)d)e \arrow[dr, "\alpha(c \circ b \circ a\text{,}\, d\text{,}\, e)"'] \arrow[ur, "\alpha(a\text{,}\, b\text{,}\, c)"] & & (a(bc)) (de)
	\\
	& ((ab)c) (de) \arrow[ur, "\alpha(a\text{,}\, b\text{,}\, c)"'] &
\end{tikzcd}
\]
commutes by the well-definedness of $\alpha$ and the fact that $\alpha$ takes values in a commutative ring. We call this property \emph{distant commutativity} for $\alpha$. These observations imply that the associahedron commutes, as asserted.

\begin{lemma}
\label{lem:bimod_to_mod}
If $A$ and $B$ are $\mathcal{C}$-graded algebras, then any $\mathcal{C}$-graded $(A, B)$-bimodule is a $\mathcal{C}$-graded left $A\otimes_\Bbbk B^\op$-module.
\end{lemma}

\begin{proof}
This is rigged to work. Given a $\mathcal{C}$-graded $(A, B)$-bimodule $M$, we give it the structure of a $\mathcal{C}$-graded left $A \otimes_\Bbbk B^\op$-module by defining $\rho_L^e: (A \otimes_\Bbbk B^\op) \otimes M \to M$ to be
\[
\rho_L^e(a \otimes b, m) := \rho_R(\rho_L(a,m), b).
\]
To verify (E.I), we compute
\begin{align*}
	\rho_L^e(\mu_{A \otimes_\Bbbk B^\op}(a_1 \otimes b_1, a_2 \otimes b_2), m) &= \rho_R( \rho_L(\mu_A(a_1, a_2), m), \mu_A(b_2, b_1))
		\\ & = \alpha(\abs{a_1}, \abs{a_2}, \abs{m}) 
			\rho_R( \rho_L(a_1, \rho_L(a_2, m)), \mu_A(b_2, b_1))
		\\ & = \alpha(\abs{a_1}, \abs{a_2}, \abs{m}) 
			\alpha(\abs{m} \circ \abs{a_2} \circ \abs{a_1}, \abs{b_1}, \abs{b_2})^{-1}
			\\ & \phantom{{}={}} \rho_R( \rho_R(\rho_L(a_1, \rho_L(a_2, m)), b_2), b_1)
		\\ & = \alpha(\abs{a_1}, \abs{a_2}, \abs{m}) 
			\alpha(\abs{m} \circ \abs{a_2} \circ \abs{a_1}, \abs{b_1}, \abs{b_2})^{-1}
			\alpha(\abs{a_1}, \abs{m} \circ \abs{a_2}, \abs{b_2})
			\\ & \phantom{{}={}} \rho_R( \rho_L(a_1, \rho_R(\rho_L(a_2, m), b_2)), b_1)
		\\ & = \Delta(\abs{a_1 \otimes b_1}, \abs{a_2 \otimes b_2}, \abs{m})
			\rho_L^e(a_1 \otimes b_1, \rho_L^e(a_2 \otimes b_2, m))
\end{align*}
as desired. For (E.II), we fix $\abs{m}: X\to Y$ and compute that
\[
\rho_L^e(1_{(X, Y)}, m) = \rho_L^e(1_X \otimes 1_Y, m) = \mathcal{L}(\mathrm{Id}_X, \abs{m}) \mathcal{R}(\abs{m}, \mathrm{Id}_Y) m
\]
as well.
\end{proof}

Interestingly, for the other direction to work, we need to assume that the unitors of $\mathrm{Mod}^\mathcal{C}$ are typical (see equations (\ref{eq:typicalL}) and (\ref{eq:typicalR})).

\begin{proposition}
\label{prop:bimod=mod}
Suppose that $A$ and $B$ are $\mathcal{C}$-graded algebras, and that the unitors of $\mathrm{Mod}^\mathcal{C}$ are the typical unitors induced by $\alpha$. Then, every $\mathcal{C}$-graded left $A\otimes_\Bbbk B^\op$-module can be given the structure of a $\mathcal{C}$-graded $(A, B)$-bimodule, and vice-versa.
\end{proposition}

\begin{proof}
The vice-versa statement is a corollary of Lemma \ref{lem:bimod_to_mod}. Otherwise, assume $M$ is a $\mathcal{C}$-graded $A\otimes_\Bbbk B^\op$-module. If $\abs{m}: X \to Y$, define
\[
\rho_L(a, m) := \mathcal{R}(\abs{m} \circ \abs{a}, \mathrm{Id}_Y)^{-1} \rho_L^e(a \otimes 1_Y, m) \qquad \text{and} \qquad \rho_R(m, b) := \mathcal{L}(\mathrm{Id}_X, \abs{m})^{-1} \rho_L^e(1_X \otimes b, m).
\]

First we check that the axioms of a $\mathcal{C}$-graded $(A, B)$-bimodule are satisfied. We take the time to perform the checks arduously as to not take the result for granted, although the entire proof might be a bit pedantic. To check (B.I), assume that $a_1, a_2 \in A$ and $m \in M$ are homogeneous so that
\[
\begin{tikzcd}
	W \arrow[r, "\abs{a_1}"] & X \arrow[r, "\abs{a_2}"] & Y \arrow[r, "\abs{m}"] & Z
\end{tikzcd}
\]
and compute
\begin{align*}
	\rho_L(\mu_A (a_1, a_2), m) &= \mathcal{R}(\abs{m} \circ \abs{a_2} \circ \abs{a_1}, \mathrm{Id}_z)^{-1} \rho_L^e(\mu_A(a_1, a_2) \otimes 1_Z, m)
		\\ & = \mathcal{R}(\abs{m} \circ \abs{a_2} \circ \abs{a_1}, \mathrm{Id}_z)^{-1} 
			\rho_L^e (\mu_A(a_1, a_2) \otimes \mu_{B^{\op}}(1_Z, 1_Z), m)
		\\ & = \mathcal{R}(\abs{m} \circ \abs{a_2} \circ \abs{a_1}, \mathrm{Id}_z)^{-1} 
			\rho_L^e (\mu_{A \otimes_\Bbbk B^{\op}}(a_1 \otimes 1_Z, a_2 \otimes 1_Z), m)
		\\ & = \mathcal{R}(\abs{m} \circ \abs{a_2} \circ \abs{a_1}, \mathrm{Id}_z)^{-1}
			\Delta(\abs{a_1 \otimes 1_Z}, \abs{a_2 \otimes 1_Z}, \abs{m})
			\rho_L^e(a_1 \otimes 1_Z, \rho_L^e(a_2 \otimes 1_Z, m))
		\\ & = \mathcal{R}(\abs{m} \circ \abs{a_2} \circ \abs{a_1}, \mathrm{Id}_z)^{-1}
			\Delta(\abs{a_1 \otimes 1_Z}, \abs{a_2 \otimes 1_Z}, \abs{m})
			\mathcal{R}(\abs{m} \circ \abs{a_2}, \mathrm{Id}_Z)
			\\ & \phantom{{}={}} \mathcal{R}(\abs{\rho_L^e(a_2 \otimes 1_Z, m)} \circ \abs{a_1}, \mathrm{Id}_Z) \rho_L(a_1, \rho_L(a_2, m)).
\end{align*}
Notice that the second equality assumes that the unitors are typical. The first and the last term written as a function of $\mathcal{R}$ cancel each other since $\abs{\rho_L^e(a_1 \otimes 1_Z, m)} = \abs{m} \circ \abs{a_2}$. Expanding the remaining terms, $\Delta(\abs{a_1 \otimes 1_Z}, \abs{a_2 \otimes 1_Z}, \abs{m})$ and $\mathcal{R}(\abs{m} \circ \abs{a_2}, \mathrm{Id}_Z)$, in terms of $\alpha$ (using the fact that the right unitor is the typical one induced by $\alpha$), we obtain
\[
\alpha(\abs{a_1}, \abs{a_2}, \abs{m}) \underbrace{\alpha(\abs{m} \circ \abs{a_2} \circ \abs{a_1}, \mathrm{Id}_Z, \mathrm{Id}_Z)^{-1} \alpha(\abs{a_1}, \abs{m}\circ \abs{a_2}, \mathrm{Id}_Z) \alpha(\abs{m} \circ \abs{a_2}, \mathrm{Id}_Z, \mathrm{Id}_Z)}_{(*)}.
\]
Then, the terms labeled $(*)$ cancel by (iii) of Lemma \ref{lem:assocomps}, and we have that 
\[
\rho_L(\mu_A(a_1, a_2), m) = \alpha(\abs{a_1}, \abs{a_2}, \abs{m}) \rho_L(a_1, \rho_L(a_2, m))
\]
as desired.

The argument for axiom (B.II) is very similar. Assume that $b_1, b_2 \in B$ and $m\in M$ are homogeneous so that
\[
\begin{tikzcd}
	W \arrow[r, "\abs{m}"] & X \arrow[r, "\abs{b_2}"] & Y \arrow[r, "\abs{b_1}"] & Z.
\end{tikzcd}
\]
We leave it to the reader to verify that
\begin{align*}
\rho_R(\rho_R(m, b_2), b_1) & = \mathcal{L}(\mathrm{Id}_W, \abs{m})^{-1}
						\mathcal{L}(\mathrm{Id}_W, \abs{b_2} \circ \abs{m})^{-1}
						\Delta(\abs{1_W \otimes b_1}, \abs{1_W \otimes b_2}, \abs{m})^{-1}
						\mathcal{L}(\mathrm{Id}_W, \abs{m})
						\\ & \phantom{{}={}} \rho_R(m, \mu_A(b_2, b_1)).
\end{align*}
The first and the last term which appear as a function of $\mathcal{L}$ cancel. Then, expanding the rest in terms of $\alpha$ gives
\[
\underbrace{\alpha(\mathrm{Id}_W, \mathrm{Id}_W, \abs{b_2} \circ \abs{m}) \alpha(\mathrm{Id}_W, \mathrm{Id}_W, \abs{m})^{-1}}_{(*)} \alpha(\abs{m}, \abs{b_2}, \abs{b_1}) \underbrace{\alpha(\mathrm{Id}_W, \abs{m}, \abs{b_2})^{-1}}_{(*)}.
\]
The terms labeled $(*)$ cancel by (ii) of Lemma \ref{lem:assocomps}, so we are left with the desired result.

Axiom (B.III) also follows from similar ideas, but requires a little more computation. Now, pick $a\in A$, $b\in B$, and $m\in M$ homogeneous so that
\[
\begin{tikzcd}
	W \arrow[r, "\abs{a}"] & X \arrow[r, "\abs{m}"] & Y \arrow[r, "\abs{b}"] & Z.
\end{tikzcd}
\]
We want to show that 
\[
\rho_R(\rho_L(a,m), b) = \alpha(\abs{a}, \abs{m}, \abs{b}) \rho_L(a, \rho_R(m,b)).
\]
On one hand, we compute
\begin{align*}
	\rho_R(\rho_L(a,m), b) & = \mathcal{R}(\abs{m} \circ \abs{a}, \mathrm{Id}_Y)^{-1} 
						\mathcal{L}(\mathrm{Id}_W, \abs{m} \circ \abs{a})^{-1}
						\rho_L^e(1_W \otimes b, \rho_L^e(a \otimes 1_Y, m))
					\\ & = \mathcal{R}(\abs{m} \circ \abs{a}, \mathrm{Id}_Y)^{-1} 
						\mathcal{L}(\mathrm{Id}_W, \abs{m} \circ \abs{a})^{-1}
						\Delta(\abs{1_W \otimes b}, \abs{a \otimes 1_Y}, \abs{m})^{-1}
						\\ & \phantom{{}={}} \rho_L^e(\mu_{A \otimes_\Bbbk B^\op}(1_W \otimes b, a \otimes 1_Y), m)
					\\ & = \mathcal{R}(\abs{m} \circ \abs{a}, \mathrm{Id}_Y)^{-1} 
						\mathcal{L}(\mathrm{Id}_W, \abs{m} \circ \abs{a})^{-1}
						\Delta(\abs{1_W \otimes b}, \abs{a \otimes 1_Y}, \abs{m})^{-1}
						\\ & \phantom{{}={}} \rho_L^e(\mu_A(1_W, a) \otimes \mu_A(1_Y, b)), m)
					\\ & = \mathcal{R}(\abs{m} \circ \abs{a}, \mathrm{Id}_Y)^{-1} 
						\mathcal{L}(\mathrm{Id}_W, \abs{m} \circ \abs{a})^{-1}
						\Delta(\abs{1_W \otimes b}, \abs{a \otimes 1_Y}, \abs{m})^{-1}
						\mathcal{L}(\mathrm{Id}_W, \abs{a}) \mathcal{L}(\mathrm{Id}_Y, \abs{b})
						\\ & \phantom{{}={}} \rho_L^e(a \otimes b, m).
\end{align*}
Expanding the values on the last line in terms of $\alpha$, we are left with the product
\begin{align*}
	&\underbrace{\alpha(\abs{m} \circ \abs{a}, \mathrm{Id}_Y, \mathrm{Id}_Y)^{-1}}_{(*)} \underbrace{\alpha(\mathrm{Id}_W, \mathrm{Id}_W, \abs{m} \circ \abs{a})}_{(**)}
	\\ & \underbrace{\alpha(\mathrm{Id}_W, \abs{a}, \abs{m})^{-1}}_{(**)}\underbrace{\alpha(\abs{m} \circ \abs{a}, \mathrm{Id}_Y, \abs{b})}_{(*)} \underbrace{\alpha(\mathrm{Id}_W, \abs{m} \circ \abs{a}, \mathrm{Id}_Y)^{-1}}_{(***)}
	\\ & \underbrace{\alpha(\mathrm{Id}_W, \mathrm{Id}_W, \abs{a})^{-1}}_{(**)} \underbrace{\alpha(\mathrm{Id}_Y, \mathrm{Id}_Y, \abs{b})^{-1}}_{(*)}.
\end{align*}
The terms marked by $(*)$ cancel by (iv) of Lemma \ref{lem:assocomps}, those marked by $(**)$ cancel by (ii), and the term marked by $(***)$ is trivial by (i). On the other hand, one can verify in the same way that
\begin{align*}
\rho_L(a, \rho_R(m,b)) & = \mathcal{L}(\mathrm{Id}_X, \abs{m})^{-1}
					\mathcal{R}(\abs{b} \circ \abs{m} \circ \abs{a}, \mathrm{Id}_Z)^{-1}
					\Delta(\abs{a \otimes 1_Z}, \abs{1_X \otimes b}, \abs{m})^{-1}
					\mathcal{R}(\abs{a}, \mathrm{Id}_X)
					\mathcal{R}(\abs{b}, \mathrm{Id}_Z)
					\\ & \phantom{{}={}} \rho_L^e(a \otimes b, m).
\end{align*}
Then, expanding in terms of $\alpha$, we have
\begin{align*}
	& \underbrace{\alpha(\mathrm{Id}_X, \mathrm{Id}_X, \abs{m})}_{(*)} \underbrace{\alpha(\abs{b} \circ \abs{m} \circ \abs{a}, \mathrm{Id}_Z, \mathrm{Id}_Z)^{-1}}_{(**)}
	\\ & \underbrace{\alpha(\abs{a}, \mathrm{Id}_X, \abs{m})^{-1}}_{(*)} \underbrace{\alpha(\abs{m} \circ \abs{a}, \abs{b}, \mathrm{Id}_Z)}_{(**)}\underbrace{\alpha(\abs{a}, \abs{m} \abs{b})^{-1}}_{(***)}
	\\ & \underbrace{\alpha(\abs{a}, \mathrm{Id}_X, \mathrm{Id}_X)}_{(*)} \underbrace{\alpha(\abs{b}, \mathrm{Id}_Z, \mathrm{Id}_Z)}_{(**)}.
\end{align*}
The terms marked by $(*)$ cancel by (iv) and the terms marked by $(**)$ cancel by (iii) of Lemma \ref{lem:assocomps}. The term marked by $(***)$ remains, and we are left with the desired equality.

Axiom (B.IV) is quickly verified. If $\abs{m}: X\to Y$, recall that $1_X \otimes 1_Y = 1_{(X,Y)}$ and 
\[
\rho_L^e(1_{(X, Y)}, m) = \mathcal{L}(\mathrm{Id}_X, \abs{m}) \mathcal{R}(\abs{m}, \mathrm{Id}_Y)m.
\]
Then
\[
\rho_L(1_X, m) = \mathcal{R}(\abs{m}, \mathrm{Id}_Y)^{-1} \rho_L(1_X \otimes 1_Y, m) = \mathcal{L}(\mathrm{Id}_X, \abs{m}) m
\]
and
\[
\rho_R(m, 1_Y) = \mathcal{L}(\mathrm{Id}_X, \abs{m})^{-1} \rho_L^e(1_X \otimes 1_Y, m) = \mathrm{R}(\abs{m}, \mathrm{Id}_Y)m
\]
as desired.

Finally, we check that this assignment is inverse to the one $\rho_L^e(a \otimes b, m):= \rho_R(\rho_L(a,m), b)$ in the proof of Lemma \ref{lem:bimod_to_mod}. Per usual, one direction is rigged to work: we have
\[
\rho_L(a,m) = \mathcal{R}(\abs{m} \circ \abs{a}, \mathrm{Id}_Y)^{-1} \rho_L^e(a \otimes 1_Y, m) = \mathcal{R}(\abs{m} \circ \abs{a}, \mathrm{Id}_Y)^{-1} \rho_R(\rho_L(a,m), 1_Y) = \rho_L(a,m),
\]
since $\abs{\rho_L(a, m)} = \abs{m} \circ \abs{a}$, and
\[
\rho_R(m, b) = \mathcal{L}(\mathrm{Id}_X, \abs{m})^{-1} \rho_L^e(1_X \otimes b, m) = \mathcal{L}(\mathrm{Id}_X, \abs{m})^{-1} \rho_R(\rho_L(1_X, m), b) = \rho_R(m, b).
\]
For the other direction, we have to assume that the unitors are the typical ones induced by $\alpha$. We assume the relevant gradings fit into the diagram 
\[
\begin{tikzcd}
W \arrow[r, "\abs{a}"] \arrow[loop above, "\mathrm{Id}_W"] & X \arrow[r, "\abs{m}"] & Y \arrow[r, "\abs{a'}"] \arrow[loop above, "\mathrm{Id}_Y"] & Z.
\end{tikzcd}
\]
First, we compute
\begin{align*}
	\rho_L^e(a \otimes b, m) &= \rho_R(\rho_L(a,m), b) 
		\\ &= \mathcal{L}(\mathrm{Id}_W, \abs{m} \circ \abs{a})^{-1}
		\rho_L^e(1_W \otimes b, \rho_L(a,m))
		\\ &= \mathcal{L}(\mathrm{Id}_W, \abs{m} \circ \abs{a})^{-1}
			\mathcal{R}(\abs{m} \circ \abs{a}, \mathrm{Id}_Y)^{-1}
		\rho_L^e(1_W \otimes b, \rho_L^e(a \otimes 1_Y, m))
		\\ &= \mathcal{L}(\mathrm{Id}_W, \abs{m} \circ \abs{a})^{-1}
			\mathcal{R}(\abs{m} \circ \abs{a}, \mathrm{Id}_Y)^{-1}
			\Delta(\abs{1_W \otimes b}, \abs{a \otimes 1_Y}, \abs{m})^{-1}
			\\ & \phantom{{}={}} \rho_L^e(\mu_{A \otimes_\Bbbk B^{\op}} (1_W \otimes b, a \otimes 1_Y), m)
		\\ &= \mathcal{L}(\mathrm{Id}_W, \abs{m} \circ \abs{a})^{-1}
			\mathcal{R}(\abs{m} \circ \abs{a}, \mathrm{Id}_Y)^{-1}
			\Delta(\abs{1_W \otimes b}, \abs{a \otimes 1_Y}, \abs{m})^{-1}
			\mathcal{L}(\mathrm{Id}_W, \abs{a})
			\mathcal{L}(\mathrm{Id}_Y, \abs{b})
			\\ & \phantom{{}={}} \rho_L^e(a \otimes b, m)
\end{align*}
where all gradings, apart from the first two entries of $\Delta$, are taken in $\mathcal{C}$. Now we rewrite all the terms of the last line in terms of the associator to get the product
\begin{align*}
&\underbrace{\alpha(\mathrm{Id}_W, \mathrm{Id}_W, \abs{m} \circ \abs{a})}_{(*)} \underbrace{\alpha(\abs{m} \circ \abs{a}, \mathrm{Id}_Y, \mathrm{Id}_Y)^{-1}}_{(**)}\\
&\underbrace{\alpha(\mathrm{Id}_W, \abs{a}, \abs{m})^{-1}}_{(*)} \underbrace{\alpha(\abs{m} \circ \abs{a}, \mathrm{Id}_Y, \abs{b})}_{(**)} \underbrace{\alpha(\mathrm{Id}_W, \abs{m} \circ \abs{a}, \mathrm{Id}_Y)^{-1}}_{(***)}\\
&\underbrace{\alpha(\mathrm{Id}_W, \mathrm{Id}_W, \abs{a})^{-1}}_{(*)} \underbrace{\alpha(\mathrm{Id}_Y, \mathrm{Id}_Y, \abs{b})^{-1}}_{(**)}
\end{align*}
Then, the terms labeled $(*)$ cancel by (ii) of Lemma \ref{lem:assocomps}, the terms labeled $(**)$ cancel by (iv) of Lemma \ref{lem:assocomps}, and the term labeled $(***)$ is trivial by (i) of Lemma \ref{lem:assocomps}.
\end{proof}

One can also define $\mathcal{C}$-graded \emph{right $A \otimes_\Bbbk B^\op$-modules}, as well as $A^\op \otimes_\Bbbk B$-modules. Notice that, for a right $A\otimes_\Bbbk B^\op$-module with action map
	\[
	\rho_R^e(m, a\otimes b^\op) := \rho_L(b, \rho_R(m,a)),
	\]
the $\Delta$-value of axiom (E.I) in Definition \ref{def:left-Ae-modules} is replaced by the value assoicated to any path
\[
\begin{tikzcd}
b_2 (( b_1 ( m a_1)) a_2) \arrow[r] & (b_2 b_1) (m (a_1 a_2)).
\end{tikzcd}
\]
Also note that, in the case of $A^\op \otimes_\Bbbk B$-modules, the grading of the action map is reversed. Adjusting Definition \ref{def:left-Ae-modules} accordingly, the following is straightforward.

\begin{proposition}
If $A$ and $B$ are $\mathcal{C}$-graded algebras and $M$ is a $\mathcal{C}$-graded $(A, B)$-bimodule, then $M$ also has the structure of a
\begin{itemize}
	\item $\mathcal{C}$-graded left $A \otimes_\Bbbk B^\op$-module by
	\[
	\rho_L^e(a \otimes b^\op, m) := \rho_R(\rho_L(a, m), b);
	\]
	\item $\mathcal{C}$-graded right $A^\op \otimes_\Bbbk B$-module by
	\[
	\rho_R^e(m, a^\op \otimes b) := \rho_L(a, \rho_R(m,b));
	\]
	\item $\mathcal{C}$-graded left $A^\op \otimes_\Bbbk B$-module by
	\[
	\rho_L^e(a^\op \otimes b, m) := \rho_R(\rho_L(b,m), a);
	\]
	\item $\mathcal{C}$-graded right $A \otimes_\Bbbk B^\op$-module by
	\[
	\rho_R^e(m, a\otimes b^\op) := \rho_L(b, \rho_R(m,a)).
	\]
\end{itemize}
If unitors are the typical ones induced by $\alpha$, then each of the structures listed above give rise to a $\mathcal{C}$-graded $(A, B)$-bimodule.
\end{proposition}

\subsection{$\mathcal{C}$-graded bar resolution}
\label{ss:hochschild}

We will provide a definition of Hochschild homology for $\mathcal{C}$-graded algebras (potentially with coefficients in $\mathcal{C}$-graded $(A,A)$-bimodules $M$) which generalizes the usual notion if $\mathcal{C}$ is replaced by the grading category $(\mathcal{Z}, 1)$ or $(BG, 1)$ of Example \ref{eg:trivgc}. The first step toward this goal is defining a bar construction for $\mathcal{C}$-graded algebras, which we provide in this subsection.

\begin{definition}
A \emph{$\mathcal{C}$-graded DG-$(A, B)$-bimodule} is a pair $(M, \partial_M)$ of a $\mathcal{Z} \times \mathcal{C}$-graded $(A,B)$-bimodule $M = \bigoplus_{n\in \mathbb{Z}, g\in \mathrm{Mor}(\mathcal{C})} M_g^n$ and a $\Bbbk$-linear map $\partial_M: M \to M$, called the \emph{differential}, satisfying the following:
\begin{enumerate}[label=(DG.\Roman*), align=left]
	\item $\partial_M(M_g^n)\subset M_g^{n-1}$;
	\item $\partial_M(\rho_L(a, m)) = \rho_L(a, \partial_M(m))$;
	\item $\partial_M(\rho_R(m, b)) = \rho_R(\partial_M(m), b)$;
	\item $\partial_M \circ \partial_M = 0$,
\end{enumerate}
for each $a\in A$, $b\in B$, and $m\in M$. If $m\in M$ is homogeneous with $\abs{m} = (\abs{m}_{\mathcal{Z}}, \abs{m}_\mathcal{C})$, we call $\abs{m}_\mathcal{Z} \in \mathbb{Z}$ the \emph{homological degree} of $m$. We call $(M, \partial_M)$ a \emph{$\mathcal{C}$-graded chain complex} if $A = B = I_\mathcal{C}$, so that the left- and right-actions are just scalar multiplication.

A \emph{$\mathcal{C}$-graded left DG-$A\otimes_\Bbbk B^\op$-module} is a pair $(M, \partial_M)$ of a $\mathcal{Z} \times \mathcal{C}$-graded left $A\otimes B^\op$-module which is defined exactly the same way, except that axioims (DG.II) and (DG.III) are replaced by the single axiom
\begin{enumerate}[label=(DG.II'), align=left]
	\item $\partial_M(\rho^e_L(a\otimes b, m)) = \rho^e_L(a\otimes b, \partial_M(m))$.
\end{enumerate}
\end{definition}

Axiom (DG.I) says that the differential decreases homological degree by 1 and preserves $\mathcal{C}$-degree. For clarity, we note that we could have just as easily defined $\mathcal{C}$-graded DG-$(A, B)$-bimodules where axiom (DG.I) is replaced with the requirement that $\partial_M(M_g^n) \subset M_g^{n+1}$ (see, for example, Definition 4.24 of \cite{naisse2020odd}). Finally, note that the homology $H(M,\partial_M) = \ker(\partial_M)/\mathrm{im}(\partial_M)$ of a $\mathcal{C}$-graded DG-$(A, B)$-bimodule $(M, \partial_M)$ is a $\mathcal{Z} \times \mathcal{C}$-graded bimodule.

\begin{proposition}
Suppose $A$ and $B$ are $\mathcal{C}$-graded algebras, and that the unitors of $\mathrm{Mod}^\mathcal{C}$ are the typical unitors induced by $\alpha$. Then every $\mathcal{C}$-graded left DG-$A\otimes_\Bbbk B^\op$-module can be given the structure of a $\mathcal{C}$-graded DG-$(A, B)$-bimodule, and vice-versa.
\end{proposition}

\begin{proof}
This is a direct consequence of Proposition \ref{prop:bimod=mod} and the proof thereof. It is an easy exercise, left to the reader, to verify that the actions defined there satisfy the conditions of a DG-(bi)module.
\end{proof}

Let $A$ be a $\mathcal{C}$-graded algebra. The \emph{bar construction} $\mathcal{B}(A)$ of $A$ is a primary example of a $\mathcal{C}$-graded DG-$(A,A)$-bimodule. As a complex, it takes the form
\[
\mathcal{B}(A) := 
\begin{tikzcd}
	\cdots \arrow[r] & ((A \otimes A) \otimes A) \otimes A \arrow[r] & (A \otimes A) \otimes A \arrow[r] & A \otimes A \arrow[r] & 0
\end{tikzcd}
\]
with differential $\partial: A^{\otimes n+2} \to A^{\otimes n+1}$ given by
\[
\partial (a_0 \otimes a_1 \otimes \cdots \otimes a_{n+1}) = \sum_{i=0}^n (-1)^i \alpha(\abs{a_{i-1}} \circ \cdots \circ \abs{a_0}, \abs{a_i}, \abs{a_{i+1}}) a_0 \otimes \cdots \otimes \mu_A(a_i, a_{i+1}) \otimes \cdots \otimes a_{n+1}
\]
where we fix $\alpha(\emptyset, \abs{a_0}, \abs{a_1}) = 1$ in the $i=0$ summand. The tensor product in $A^{\otimes n}$ is the monoidal product of $\mathrm{Mod}^\mathcal{C}$; in particular, $A^{\otimes n}$ is $\mathcal{C}$-graded. We will view $\mathcal{B}(A)$ as $\mathcal{Z} \times \mathcal{C}$-graded taking
\[
\abs{a_0 \otimes a_1 \otimes \cdots \otimes a_{n+1}}_{\mathcal{Z} \times \mathcal{C}} = (n, \abs{a_{n+1}}_\mathcal{C} \circ \cdots \circ \abs{a_1}_\mathcal{C} \circ \abs{a_0}_\mathcal{C}).
\]
Then we have that $(\mathcal{B}(A), \partial)$ satisfies (DG.I) clearly.

\begin{lemma}
\label{lem:barcomplex}
If $A$ is a $\mathcal{C}$-graded algebra, $\mathcal{B}(A)$ is a chain complex; that is, $\partial \circ \partial = 0$.
\end{lemma}

\begin{proof}
Consider $\partial(\partial(a_0 \otimes \cdots \otimes a_{n+1})$. We will denote summands in the ensuing expansion by pairs $(i, j)$ for $i = 0, 1, \ldots, n$ coming from the first differential and $j = 0, 1, \ldots n-1$ coming from the second. Then, fixing $i \le j$, observe that in the proof that the regular bar complex is a chain complex, the $(i, j)$ summand cancels with the $(j+1, i)$ summand (see, for example, Lemma 1.1.2 of \cite{MR1600246}). We claim that this is also how terms cancel in the $\mathcal{C}$-graded setting. Thus, since the signs are as they appear in the original setting, we do not need to keep track of them. There are three cases to consider.

The first is when $(i,j) = (0,0)$. This term is always
\[
\mu(\mu(a_0, a_1), a_2) \otimes a_3 \otimes \cdots \otimes a_{n+1}
\]
and it clearly cancels with the $(j+1, i) = (1, 0)$ term
\[
\alpha(\abs{a_0}, \abs{a_1}, \abs{a_2}) \mu(a_0, \mu(a_1, a_2)) \otimes a_3 \otimes \cdots \otimes a_{n+1}.
\]

For the second case, assume that $i < j$. Then the $(i,j)$ term is
\[
\alpha(\abs{a_{i-1}} \circ \cdots \circ \abs{a_0}, \abs{a_i}, \abs{a_{i+1}})
\alpha(\abs{a_j} \circ \cdots \circ \abs{\mu(a_i, a_{i+1})} \circ \cdots \circ \abs{a_0}, \abs{a_{j+1}}, \abs{a_{j+2}})
\]
times $a_0 \otimes \cdots \otimes \mu(a_{i}, a_{i+1}) \otimes \cdots \otimes \mu(a_{j+1}, a_{j+2}) \otimes \cdots \otimes a_{n+1}$. The $(j+1, i)$ term is clearly alike, with coefficient
\[
\alpha(\abs{a_j} \circ \cdots \circ \abs{a_0}, \abs{a_{j+1}}, \abs{a_{j+2}})
\alpha(\abs{a_{i-1}} \circ \cdots \circ \abs{a_0}, \abs{a_i}, \abs{a_{i+1}})
\]
Thus, these two terms cancel, as $\abs{\mu(a_i, a_{i+1})} = \abs{a_{i+1}} \circ \abs{a_i}$.

Finally, suppose that $i = j > 0$. Then, the $(i,i)$-term is 
\[
\alpha(\abs{a_{i-1}} \circ \cdots \circ \abs{a_0}, \abs{a_i}, \abs{a_{i+1}})
\alpha(\abs{a_{i-1}} \circ \cdots \circ \abs{a_0}, \abs{\mu(a_i, a_{i+1})}, \abs{a_{i+2}})
\]
times $a_0 \otimes \cdots \otimes \mu(\mu(a_i, a_{i+1}), a_{i+2}) \otimes \cdots \otimes a_{n+1}$, and the $(i+1, i)$-term is 
\[
\alpha(\abs{a_i} \circ \abs{a_{i-1}} \circ \cdots \circ \abs{a_0}, \abs{a_{i+1}}, \abs{a_{i+2}})
\alpha(\abs{a_{i-1}} \circ \cdots \circ \abs{a_0}, \abs{a_i}, \abs{\mu(a_{i+1}, a_{i+2})})
\]
times $a_0 \otimes \cdots \otimes \mu(a_i, \mu(a_{i+1}, a_{i+2})) \otimes \cdots \otimes a_{n+1}$. Write $f = \abs{a_{i-1}} \circ \cdots \circ \abs{a_0}$, $g = \abs{a_i}$, $h = \abs{a_{i+1}}$ and $\ell = \abs{a_{i+2}}$. Then, the cocycle relation $d\alpha(f, g, h, \ell) = 1$ implies that these two terms are equivalent, since
\[
\mu(\mu(a_i, a_{i+1}), a_{i+2}) = 
\alpha(\abs{a_i}, \abs{a_{i+1}}, \abs{a_{i+2}})
\mu(a_i, \mu(a_{i+1}, a_{i+2})).
\]
This concludes the proof.
\end{proof}

Suppose $a, a_0, a_1,\ldots, a_{n+1} \in A$. We define the following values: let
\[
\Phi(\abs{a}, \abs{a_0}, \abs{a_1}, \ldots, \abs{a_{n+1}}) := \prod_{i=1}^{n+1} \alpha(\abs{a}, \abs{a_{i-1}} \circ \cdots \circ \abs{a_0}, \abs{a_i})^{-1}
\]
and
\[
\Psi(\abs{a_0}, \ldots, \abs{a_n}, \abs{a_{n+1}}, \abs{a}) := \alpha(\abs{a_n} \circ \cdots \circ \abs{a_0}, \abs{a_{n+1}}, \abs{a}).
\]

\begin{proposition}
\label{prop:A-DGA}
If $A$ is a $\mathcal{C}$-graded algebra, $(\mathcal{B}(A), \partial)$ is a $\mathcal{C}$-graded DG-$(A,A)$-bimodule, with left-action
\[
\rho_L(a, a_0 \otimes a_1 \otimes \cdots \otimes a_{n+1}) := \Phi(\abs{a}, \abs{a_0}, \abs{a_1}, \ldots, \abs{a_{n+1}}) \mu_A(a, a_0) \otimes a_1 \otimes \cdots \otimes a_{n+1}
\]
and right-action
\[
\rho_R(a_0 \otimes a_1 \otimes \cdots \otimes a_{n+1}, a) := \Psi(\abs{a_0}, \ldots, \abs{a_n}, \abs{a_{n+1}}, \abs{a}) a_0 \otimes a_1 \otimes \cdots \otimes \mu_A(a_{n+1}, a).
\]
\end{proposition}

\begin{proof}
After Lemma \ref{lem:barcomplex}, we need to verify axioms (B.I)--(B.IV) and axioms (DG.II) and (DG.III). Like many proofs to this point, the argument is straightforward, but tedious. We'll verify the more difficult (B.I), (B.III), and (DG.II), leaving the rest to the reader. These three are more tedious to check because of the involvement of the left action.

For (B.I), we must show that
\[
\rho_L(\mu(a, a'), a_0 \otimes a_1 \otimes \cdots \otimes a_{n+1}) = \Phi(\abs{\mu(a,a')}, \abs{a_0}, \ldots, \abs{a_{n+1}}) \mu(\mu(a,a'), a_0) \otimes a_1 \otimes \ldots \otimes a_{n+1}
\]
is equal to
\begin{align*}
&\alpha(\abs{a}, \abs{a'}, \abs{a_0 \otimes \cdots \otimes a_{n+1}}) \rho_L(a, \rho_L(a', a_0 \otimes \ldots \otimes a_{n+1}))
\\ & =
\alpha(\abs{a}, \abs{a'}, \abs{a_0 \otimes \cdots \otimes a_{n+1}}) \Phi(\abs{a}, \abs{\mu(a', a_0)}, \abs{a_1}, \ldots, \abs{a_{n+1}}) \Phi(\abs{a'}, \abs{a_0},\ldots, \abs{a_{n+1}})
\\ & \phantom{{}={}}\mu(a, \mu(a', a_0)) \otimes a_1 \otimes \cdots \otimes a_{n+1}.
\end{align*}
Thus, it suffices to prove that
\begin{align*}
\alpha(\abs{a},& \abs{a'}, \abs{a_0 \otimes \cdots \otimes a_{n+1}}) \times \\&\Phi(\abs{\mu(a,a')}, \abs{a_0}, \ldots, \abs{a_{n+1}})^{-1} \Phi(\abs{a}, \abs{\mu(a', a_0)}, \abs{a_1}, \ldots, \abs{a_{n+1}}) \Phi(\abs{a'}, \abs{a_0},\ldots, \abs{a_{n+1}})
\end{align*}
is equal to $\alpha(\abs{a}, \abs{a'}, \abs{a_0})$. This can be seen via an iterative process. Start with the ``$n+1$'' terms from the expansions of each of the $\Phi$ products. These look like
\[
\alpha(\abs{a'} \circ \abs{a}, \abs{a_n} \circ \cdots \circ \abs{a_0}, \abs{a_{n+1}})
\alpha(\abs{a}, \abs{a_n}\circ \cdots \circ \abs{a_0} \circ \abs{a'}, \abs{a_{n+1}})^{-1}
\alpha(\abs{a'}, \abs{a_n} \circ \cdots \circ \abs{a_0}, \abs{a_{n+1}})^{-1}.
\]
Taking $f = \abs{a}$, $g = \abs{a'}$, $h = \abs{a_n} \circ \cdots \circ \abs{a_0}$, and $\ell = \abs{a_{n+1}}$, and inspecting the cocycle relation for $d\alpha(f, g, h ,\ell)$, we see that the above product is equal to
\[
\alpha(\abs{a}, \abs{a'}, \abs{a_{n+1}} \circ \abs{a_n} \circ \cdots \circ \abs{a_0})^{-1}
\alpha(\abs{a}, \abs{a'}, \abs{a_n} \circ \abs{a_{n-1}} \circ  \cdots \circ \abs{a_0}).
\]
On one hand, the first term cancels with the original $\alpha(\abs{a}, \abs{a'}, \abs{a_0 \otimes \cdots \otimes a_{n+1}})$ term. On the other, consider the product of the second term with the ``$n$'' terms from the $\Phi$-expansions: these look like
\[
\alpha(\abs{a'} \circ \abs{a}, \abs{a_n} \circ \cdots \circ \abs{a_0}, \abs{a_{n+1}})
\alpha(\abs{a}, \abs{a_n}\circ \cdots \circ \abs{a_0} \circ \abs{a'}, \abs{a_{n+1}})^{-1}
\alpha(\abs{a'}, \abs{a_n} \circ \cdots \circ \abs{a_0}, \abs{a_{n+1}})^{-1}.
\]
Then, taking $f = \abs{a}$, $g = \abs{a'}$, $h = \abs{a_{n-1}} \circ \cdots \circ \abs{a_0}$, and $\ell = \abs{a_{n}}$, the cocycle relation for $d\alpha(f, g, h ,\ell)$ tells us that this product is equal to
\[
\alpha(\abs{a}, \abs{a'}, \abs{a_{n-1}} \circ \abs{a_{n-2}} \circ \cdots \circ \abs{a_0}).
\]
To conclude the proof, iterate this process, which terminates with the leftover term being $\alpha(\abs{a}, \abs{a'}, \abs{a_0})$

The proof of (B.II) is far easier given that $\Psi$ is expressed by only one $\alpha$ term---it follows by only one application of the cocycle relation. Similarly, though there are more terms, the proof of (B.III) requires only one application of the cocycle relation. Both are left to the reader.

The proof of the first part of (B.IV) requires an iteration. By definition, we have
\begin{align*}
\rho_L(1_X, a_0 \otimes \cdots \otimes a_{n+1}) &= \Phi(\mathrm{Id}_X, \abs{a_0}, \ldots, \abs{a_{n+1}})
\mu(1_X, a_0) \otimes a_1 \otimes \ldots \otimes a_{n+1}
\\&= \Phi(\mathrm{Id}_X, \abs{a_0}, \ldots, \abs{a_{n+1}}) \mathcal{L}(\mathrm{Id}_X, \abs{a_0}) a_0 \otimes \cdots \otimes a_{n+1}.
\end{align*}
By assumption, $\mathcal{L}(\mathrm{Id}_X, \abs{a_0}) = \alpha(\mathrm{Id}_X, \mathrm{Id}_X, \abs{a_0})^{-1}$. Expanding $\Phi$, we have
\begin{align*}
	\alpha(\mathrm{Id}_X, & \abs{a_n} \circ \cdots \circ \abs{a_0}, \abs{a_{n+1}})^{-1} \cdots \alpha(\mathrm{Id}_X, \abs{a_1} \circ \abs{a_0}, \abs{a_2})^{-1}\alpha(\mathrm{Id}_X, \abs{a_0}, \abs{a_1})^{-1} \alpha(\mathrm{Id}_X, \mathrm{Id}_X, \abs{a_0})^{-1}
	\\ &= \alpha(\mathrm{Id}_X, \abs{a_n} \circ \cdots \circ \abs{a_0}, \abs{a_{n+1}})^{-1} \cdots \alpha(\mathrm{Id}_X, \abs{a_1} \circ \abs{a_0}, \abs{a_2})^{-1} \alpha(\mathrm{Id}_X, \mathrm{Id}_X, \abs{a_1} \circ \abs{a_0})
	\\ & \phantom{{}..{}}\vdots
	\\ & = \alpha(\mathrm{Id}_X, \mathrm{Id}_X, \abs{a_{n+1}} \circ \cdots \circ \abs{a_0})^{-1}
	\\ & = \mathcal{L}(\mathrm{Id}_X, \abs{a_0} \otimes \cdots \otimes \abs{a_{n+1}})
\end{align*}
by iterative applications of (ii) from Lemma \ref{lem:assocomps}. Similarly, the proof of the second half of (B.IV) follows from a single application of (iii) from Lemma \ref{lem:assocomps}

We proceed to proving the DG-axioms. As noted earlier, (DG.I) is immediate, and (DG.IV) is Lemma \ref{lem:barcomplex}. Checking (DG.II) directly, we compute that
\begin{align*}
&\partial(\rho_L(a, a_0 \otimes a_1 \otimes \cdots \otimes a_{n+1}))
	\\ & = \alpha(\abs{a}, \abs{a_n} \circ \cdots \circ \abs{a_0}, \abs{a_{n+1}})^{-1} 
		\alpha(\abs{a}, \abs{a_{n-1}} \circ \cdots \circ \abs{a_0}, \abs{a_n})^{-1}
		\cdots
		\alpha(\abs{a}, \abs{a_0}, \abs{a_1})^{-1}
		\\ & \phantom{{}={}} \partial(\mu(a, a_1) \otimes a_1 \otimes \cdots \otimes{a_{n+1}})
	\\ & = \alpha(\abs{a}, \abs{a_n} \circ \cdots \circ \abs{a_0}, \abs{a_{n+1}})^{-1} 
		\alpha(\abs{a}, \abs{a_{n-1}} \circ \cdots \circ \abs{a_0}, \abs{a_n})^{-1}
		\cdots
		\alpha(\abs{a}, \abs{a_0}, \abs{a_1})^{-1}
		\\ & \phantom{{}={}} \mu(\mu(a, a_0), a_1) \otimes a_2 \otimes \cdots \otimes a_{n+1}
		\\ & \phantom{{}={}} + \alpha(\abs{a}, \abs{a_n} \circ \cdots \circ \abs{a_0}, \abs{a_{n+1}})^{-1} 
		\alpha(\abs{a}, \abs{a_{n-1}} \circ \cdots \circ \abs{a_0}, \abs{a_n})^{-1}
		\cdots
		\alpha(\abs{a}, \abs{a_0}, \abs{a_1})^{-1}
		\\ & \phantom{{}={}} \sum_{i=1}^{n} (-1)^i \alpha(\abs{a_{i-1}} \circ \cdots \circ \abs{a_0} \circ \abs{a}, \abs{a_i}, \abs{a_{i+1}}) \mu(a, a_0) \otimes \cdots \otimes \mu(a_i, a_{i+1}) \otimes \cdots \otimes a_{n+1}.
\end{align*}
On the other hand,
\begin{align*}
& \rho_L(a, \partial(a_0 \otimes a_1 \otimes \cdots \otimes a_{n+1}))
	\\ & = \rho_L(a, \mu(a_0, a_1) \otimes \cdots \otimes a_{n+1})
		\\ & \phantom{{}={}} + \sum_{i=1}^n (-1)^i \alpha(\abs{a_{i-1}} \circ \cdots \circ \abs{a_0}, \abs{a_i}, \abs{a_{i+1}}) \rho_L(a, a_0 \otimes \cdots \otimes \mu(a_i, a_{i+1}) \otimes \cdots \otimes a_{n+1})
	\\ & = \alpha(\abs{a}, \abs{a_n} \circ \cdots \circ \abs{a_0}, \abs{a_{n+1}})^{-1} 
		\alpha(\abs{a}, \abs{a_{n-1}} \circ \cdots \circ \abs{a_0}, \abs{a_n})^{-1}
		\cdots
		\alpha(\abs{a}, \abs{a_1} \circ \abs{a_0}, \abs{a_2})^{-1}
		\\ & \phantom{{}={}} \mu(a, \mu(a_0, a_1)) \otimes a_2 \otimes \cdots \otimes a_{n+1}
		\\ & \phantom{{}={}} + \sum_{i=1}^n (-1)^i \alpha(\abs{a_{i-1}} \circ \cdots \circ \abs{a_0}, \abs{a_i}, \abs{a_{i+1}}) \alpha(\abs{a}, \abs{a_n} \circ \cdots \circ \abs{a_0}, \abs{a_{n+1}})^{-1} \cdots
		\\ & \phantom{{}={}} \alpha(\abs{a}, \abs{a_{i+1}} \circ \abs{a_i} \circ \abs{a_{i-1}} \circ \cdots \circ \abs{a_0}, \abs{a_{i+2}})^{-1}
		\alpha(\abs{a}, \abs{a_{i-1}} \circ \cdots \abs{a_0}, \abs{a_{i+1}} \circ \abs{a_i})^{-1} 
		\\ & \phantom{{}={}}\alpha(\abs{a}, \abs{a_{i-2}} \circ \cdots \circ \abs{a_0}, \abs{a_{i-1}})^{-1}
		\cdots
		\alpha(\abs{a}, \abs{a_0}, \abs{a_1})^{-1}
		\mu(a, a_0) \otimes \cdots \otimes \mu(a_i, a_{i+1}) \otimes \cdots \otimes a_{n+1}
\end{align*}
There are two cases to consider. First, observe that the coefficients leading the $i=0$ summands in both expansions are exactly the same outside of the $\alpha(\abs{a}, \abs{a_0}, \abs{a_1})^{-1}$ appearing in front of the first, but not the second. But this is as we hoped, as
\[
\mu(\mu(a, a_0), a_1) \otimes a_2 \otimes \cdots \otimes a_{n+1}
=
\alpha(\abs{a}, \abs{a_0}, \abs{a_1})
\mu(a, \mu(a_0, a_1)) \otimes a_2 \otimes \cdots \otimes a_{n+1}.
\]
In the second case, we can consider any of the summands when $i \ge 1$. The coefficients of these summands are exactly the same outside of the appearance of the terms
\[
\alpha(\abs{a_{i-1}} \circ \cdots \circ \abs{a_0}, \abs{a_i}, \abs{a_{i+1}})
\alpha(\abs{a}, \abs{a_i} \circ \abs{a_{i-1}} \circ \cdots \circ \abs{a_0}, \abs{a_{i+1}})^{-1}
\alpha(\abs{a}, \abs{a_{i-1}} \circ \cdots \circ \abs{a_0}, \abs{a_i})^{-1}
\]
appearing in the first expansion, and the terms
\[
\alpha(\abs{a_{i-1}} \circ \cdots \circ \abs{a_0}, \abs{a_i}, \abs{a_{i+1}})
\alpha(\abs{a}, \abs{a_{i-1}} \circ \cdots \circ \abs{a_0}, \abs{a_{i+1}} \circ \abs{a_i})^{-1}
\]
appearing in the second. However, these values are equivalent by the cocycle condition. The proof of (DG.III) is similar but much less tedious, and is left to the reader.
\end{proof}

\begin{remark}
The $\Phi$ and $\Psi$ terms are decided naturally by the following processes: the $\Phi$ term is chosen according to the path
\[
\begin{tikzcd}
a (((a_0 a_1) a_2) \cdots a_n) \arrow[r, "\alpha^{-1}"] \arrow[drrr, "\Phi"', bend right=10]
 & (a ((a_0 a_1) a_2) \cdots) a_n \arrow[r, "\alpha^{-1}"]
 & \cdots \arrow[r, "\alpha^{-1}"]
 & (((a (a_0 a_1)) a_2) \cdots a_n) \arrow[d, "\alpha^{-1}"]
 \\
 &&& ((((a a_0) a_1) a_2) \cdots) a_n
\end{tikzcd}
\]
and, similarly, the much simpler $\Psi$ term is chosen according to the path
\[
\begin{tikzcd}
\Psi: ((((a_0 a_1) a_2) \cdots) a_n) a' \arrow[r, "\alpha"]
& (((a_0 a_1) a_2) \cdots) (a_n a').
\end{tikzcd}
\]
Hunting for ``$\alpha$-correction terms'' generally consists of this process.
\end{remark}

\section{Hochschild homology for $\mathcal{C}$-graded algebras and bimodules}
\label{s:hochschild}

Now we will use the concepts of the previous section to give a first definition of Hochschild homology for $\mathcal{C}$-graded algebras via the bar construction of Subsection \ref{ss:hochschild}. Notice that this gives an extension of Hochschild homology to the class of quasi-associative algebras. However, unlike the classical setting, this takes some preparation.

Recall that, in general, the Hochschild homology of an algebra $A$ with coefficients in an $(A, A)$-bimodule $M$ can be taken as the homology of the complex
\[
M \otimes_{A \otimes A^\op} \mathcal{B}(A).
\]
Naisse and Putyra \cite{naisse2020odd} describe the tensor product of two $\mathcal{C}$-graded modules over an intermediary $\mathcal{C}$-graded algebra. Suppose $A$, $B$, and $C$ are $\mathcal{C}$-graded algebras, and that $M$ is a $\mathcal{C}$-graded $(A, B)$-bimodule and $N$ is a $\mathcal{C}$-graded $(B,C)$-bimodule. We view $M \otimes N$ as a $\mathcal{C}$-graded $(A, C)$-bimodule by defining actions 
\[
\begin{tikzcd}[column sep = tiny]
	A \otimes (M \otimes N) \arrow[dr, "\alpha^{-1}"'] \arrow[rr, "\rho_L^{M \otimes N}"] && M \otimes N
	\\
	& (A \otimes M) \otimes N \arrow[ur, "\rho_L^M \otimes \mathbbm{1}_N"'] &
\end{tikzcd}
\qquad
\text{and}
\qquad
\begin{tikzcd}[column sep = tiny]
	(M \otimes N) \otimes C \arrow[dr, "\alpha"'] \arrow[rr, "\rho_R^{M \otimes N}"] && M \otimes N
	\\
	& M \otimes (N \otimes C) \arrow[ur, "\mathbbm{1}_M \otimes \rho_R^N"'] &
\end{tikzcd}
\]
We define the tensor product of $M$ and $N$ over the intermediary algebra $B$ as
\begin{equation}
\label{eq:Acoequal}
M \otimes_B N := M \otimes N 
{ /} 
\left( \rho_R^M(m, b) \otimes n - \alpha(\abs{m}, \abs{b}, \abs{n}) m \otimes \rho_L^N(b, n)\right)
\end{equation}
for any homogeneous $m\in M$, $b\in B$, and $n\in N$, which is the coequalizer of the following diagram.
\[
\begin{tikzcd}[row sep = small, column sep = large]
	(M \otimes B) \otimes N \arrow[dd, "\alpha"'] \arrow[dr, "\rho_R \otimes \mathbbm{1}_N"] & 
	\\
	& M \otimes N
	\\
	M \otimes (B \otimes N) \arrow[ur, "\mathbbm{1}_M \otimes \rho_L"'] &
\end{tikzcd}
\]
The $\mathcal{C}$-graded $(A, C)$-bimodule structure on $M \otimes N$ induces one on $M \otimes_B N$. Finally, if $M$ and $N$ are $\mathcal{C}$-graded DG-bimodules, we define their tensor product over $B$ as
\[
(M, \partial_M) \otimes_B (N, \partial_N) := (M \otimes_B N, \partial_\otimes)
\]
where
\[
\partial_\otimes(m \otimes n) := \partial_M(m) \otimes n + (-1)^{\abs{m}_\mathcal{Z}} m \otimes \partial_N(n).
\]

The issue for us is that $\otimes_{A \otimes_\Bbbk A^\op}$ is not defined, as $A \otimes_\Bbbk A^\op$ is not canonically $\mathcal{C}$-graded, but rather $\mathcal{C} \times \mathcal{C}^\op$-graded. Explicitly, to define a tensor product over $A \otimes_\Bbbk A^\op$, we would like to take the coequalizer of the following diagram, where $M$ (resp. $N$) is a $\mathcal{C}$-graded right (resp. left) $A\otimes_\Bbbk A^\op$-module.
\[
\begin{tikzcd}[row sep = small, column sep = large]
	\left(M \times (A \otimes_\Bbbk A^\op)\right) \otimes N \arrow[dr, "\rho_R^e \otimes \mathbbm{1}_N"] \arrow[dd, dotted, "\Theta"'] & 
	\\
	& M \otimes N
	\\
	M \otimes \left((A \otimes_\Bbbk A^\op) \times N\right) \arrow[ur, "\mathbbm{1}_M \otimes \rho_L^e"'] &
\end{tikzcd}
\]
However, the connecting map $\Theta$ cannot be as simple as $\alpha$: in the tensor product over $A \otimes_\Bbbk A^\op$, we hope to identify
\[
\rho_R^e(m, a \otimes a') \otimes n \sim m \otimes \rho_L^e(a \otimes a', n)
\]
up to a unit determined by $\Theta$. However the former has grading 
\[
\abs{n} \circ \abs{a} \circ \abs{m} \circ \abs{a'}
\]
while the latter has grading
\[
\abs{a'} \circ \abs{n} \circ \abs{a} \circ \abs{m}.
\]

We see this as having two consequences. First, this means that the gradings of the elements involved must form a loop of length four in $\mathcal{C}$:
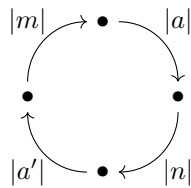
\[
\tikz[scale=1]{
	\node(N) at (0,1) {$\bullet$};
	\node(E) at (1,0) {$\bullet$};
	\node(S) at (0,-1) {$\bullet$};
	\node(W) at (-1,0) {$\bullet$};
	\draw[->] (N) to[out=0, in=90] (E);
		\node at (1,1) {$\abs{a}$};
	\draw[->] (E) to[out=270, in=0] (S);
		\node at (1,-1) {$\abs{n}$};
	\draw[->] (S) to[out=180, in=270] (W);
		\node at (-1,-1) {$\abs{a'}$};
	\draw[->] (W) to[out=90, in=180] (N);
		\node at (-1,1) {$\abs{m}$};
}
\]
else they are killed in the tensor over $A\otimes_\Bbbk A^\op$. More interestingly, this also means that $M \otimes_{A \otimes_\Bbbk A^\op} N$, if it is definable, is \emph{not} $\mathcal{C}$-graded, but rather graded by the \emph{universal trace of $\mathcal{C}$}:
\[
\mathrm{Tr}(\mathcal{C}) = \coprod_{X\in \mathrm{Ob}(\mathcal{C})} \mathrm{End}_\mathcal{C}(X) \big / g\circ f \sim f \circ g
\]
which is also referred to by some authors as the \emph{zeroth Hochschild homology} of $\mathcal{C}$.

\begin{remark}
The first of the two consequences is interesting, as it means that the ``size'' of the tensor product over $A \otimes_\Bbbk A^\op$ (and, thus, the Hochschild homology) in the $\mathcal{C}$-graded setting depends largely on the abundance of loops in $\mathcal{C}$. Notice that this doesn't have any impact on the $\mathbb{Z}$- or $G$-graded settings, as all paths are loops in $BG$, thus nothing ``extra'' dies in the tensor.
\end{remark}

We describe a well-defined tensor product $- \otimes_{A\otimes_\Bbbk A^\op} -$ by providing the connecting morphism $\Theta$ in Subsection \ref{ss:cotrace}; Theorem \ref{thm1} follows with little work. The key player is a new witnessing function $\varepsilon$. The defining characteristic of $\varepsilon$ works in tandem with the pentagon relation of $\alpha$ to show that $\Theta$ is well-defined. Moreover, in Subsection \ref{ss:hochcomplex}, we give an explicit description of the Hochschild complex for $\mathcal{C}$-graded pair $(A, M)$. Indeed, the defining characteristic of $\epsilon$ is exactly the missing piece necessary for the last canceling pair in $b^2$, where $b$ is the differential on the Hochschild complex. We provide a few elementary computations as well. In conclusion, Subsection \ref{ss:shadows} is devoted to investigating hidden implications of $\varepsilon$; among other things, we show that the choice of $\varepsilon$ fixes a choice of natural isomorphism turning the zeroth Hochschild homology into a shadow on the bicategory of $\mathcal{C}$-graded bimodules.

\subsection{Coherence in the universal trace of a grading category}
\label{ss:cotrace}

Fix a grading category $(\mathcal{C}, \alpha)$. Recall that the \emph{trace} of a small category is defined as the coend of its $\mathrm{Hom}$-functor $\mathcal{C}(-,-): \mathcal{C}^\op \times \mathcal{C} \to \mathbf{Set}$,
\[
\mathrm{Tr}(\mathcal{C}) := \int^{X\in\mathrm{Ob}(\mathcal{C})} \mathcal{C}(X,X).
\]
This means that $\mathrm{Tr}(\mathcal{C})$ is a set together with a set of canonical maps
\[
\{\mathrm{tr}_X: \mathrm{End}_\mathcal{C}(X) \to \mathrm{Tr}(\mathcal{C})\}_{X \in \mathrm{Ob}(\mathcal{C}}
\]
satisfying the relation that, for each $f:X \rightleftarrows Y: g$, 
\[
\mathrm{tr}_X(g\circ f) = \mathrm{tr}_Y(f \circ g);
\]
\textit{i.e.}, the diagram
\begin{equation}
\label{eq:cohotrace}
\begin{tikzcd}
	\mathrm{Hom}_\mathcal{C}(X, Y) \widetilde{\times} \mathrm{Hom}_\mathcal{C}(Y, X) \arrow[rr, equal] \arrow[d, "\circ"] & & \mathrm{Hom}_{\mathcal{C}}(Y, X) \widetilde{\times} \mathrm{Hom}_\mathcal{C}(X, Y) \arrow[d, "\circ"]
	\\
	\mathrm{End}_\mathcal{C}(X) \arrow[dr, "\mathrm{tr}_X"'] & & \mathrm{End}_\mathcal{C}(Y) \arrow[dl, "\mathrm{tr}_Y"]
	\\
	& \mathrm{Tr}(\mathcal{C}) & 
\end{tikzcd}
\end{equation}
(where $\widetilde{\times}$ denotes unordered cartesian product) commutes. 

\begin{remark}
While we write 
\[
\mathrm{Tr}(\mathcal{C}) = \coprod_{X\in \mathrm{Ob}(\mathcal{C})} \mathrm{End}_\mathcal{C}(X) \big / g\circ f \sim f \circ g
\]
it is worth noting that the relation ``$u \sim v$ for $u\in \mathrm{End}_\mathcal{C}(X)$ and $v\in \mathrm{End}_\mathcal{C}(Y)$ if and only if there exists pairs of maps $f: A \rightleftarrows B: g$ such that $u = g \circ f$ and $v = g \circ f$'' is just, \emph{a priori}, a reflexive and symmetric relation. It is an equivalence relation when $\mathcal{C}$ is, \textit{e.g.}, a groupoid, but we do not assume this to be the case (in particular,  the grading category $(\mathcal{G}, \alpha)$ of Section \ref{s:odd} is not a groupoid). We refer interested readers to \cite{faro2008trace}.
\end{remark}

For grading by the trace of a grading category to make sense, we need a witness to diagram (\ref{eq:cohotrace}), serving as a sort of extension of the role played by the associator $\alpha$. Let $\Omega_n\mathcal{C}$ denote paths of length $n$ in $\mathcal{C}$ which form loops.

\begin{definition}
\label{def:looper}
A \emph{looper} for a grading category $(\mathcal{C}, \alpha)$ is a function $\varepsilon: \Omega_2\mathcal{C} \to \Bbbk^\times$ which satisfies
	\begin{enumerate}[label=(\roman*)]
	\item $\varepsilon(f, g)^{-1} = \varepsilon(g,f)$, and
	\item $\varepsilon$ is \emph{coherent} with $\alpha$; that is, if $h \circ g \circ f$ is a loop of length three in $\mathcal{C}$, then
	\begin{equation}
	\label{eq:a-e_coherence}
		\alpha(f, g, h) \varepsilon(f, h \circ g)
		\alpha(g, h, f) \varepsilon(g, f \circ h)
		\alpha(h, f, g) \varepsilon(h, g \circ f)
		=1.
	\end{equation}
	\end{enumerate}
	If such an $\varepsilon$ exists, we say that $(\mathcal{C}, \alpha)$ \emph{admits a looper}.
\end{definition}

\begin{lemma}
\label{lem:eploops}
Suppose that $\ell \in \mathrm{End}_\mathcal{C}(X)$ is a loop in $\mathcal{C}$. Then
\[
\varepsilon(\ell, \mathrm{Id}_X) = \alpha(\mathrm{Id}_X, \mathrm{Id}_X, \ell) \alpha(\ell, \mathrm{Id}_X, \mathrm{Id}_X).
\]
\end{lemma}

\begin{proof}
More generally, consider the loop of length three $h \circ \mathrm{Id}_X \circ f$, pictured 
$\begin{tikzcd}
	\bullet \arrow[r, bend left, "f"] & \bullet \arrow[l, bend left, "h"] \arrow[loop right, "\mathrm{Id}_X"]
\end{tikzcd}$. Condition (ii) of Definition \ref{def:looper} tells us that 
\[
		\alpha(f, \mathrm{Id}_X, h) \varepsilon(f, h)
		\alpha(\mathrm{Id}_X, h, f) \varepsilon(\mathrm{Id}_X, f \circ h)
		\alpha(h, f, \mathrm{Id}_X) \varepsilon(h, f)
		=1.
\]
Condition (i) of Definition \ref{def:looper} tells us that $\varepsilon(f, h) \varepsilon(h, f) = 1$. Then, apply (ii) and (iii) of Lemma \ref{lem:assocomps} in addition to condition (i) to obtain
\[
\varepsilon(f \circ h, \mathrm{Id}_X) = \alpha(\mathrm{Id}_X, \mathrm{Id}_X, f \circ h) \alpha(f \circ h, \mathrm{Id}_X, \mathrm{Id}_X).
\]
\end{proof}

\begin{remark}
Notice that Lemma \ref{lem:eploops} implies that for a $\mathcal{C}$-graded module $M$ with gradings supported entirely in loops, the diagram
\[
\begin{tikzcd}
M \otimes I_\mathcal{C} \arrow[rr, "\varepsilon"] \arrow[dr, "\mathcal{R}"'] & & I_\mathcal{C} \otimes M\arrow[dl, "\mathcal{L}"]
\\
& M & 
\end{tikzcd}
\]
commutes.
\end{remark}

The formula (\ref{eq:a-e_coherence}) is called \emph{$\alpha$-$\varepsilon$ coherence}. It comes from the observation that the choices in ``smoothing'' a loop should be witnessed: consider the following diagram.
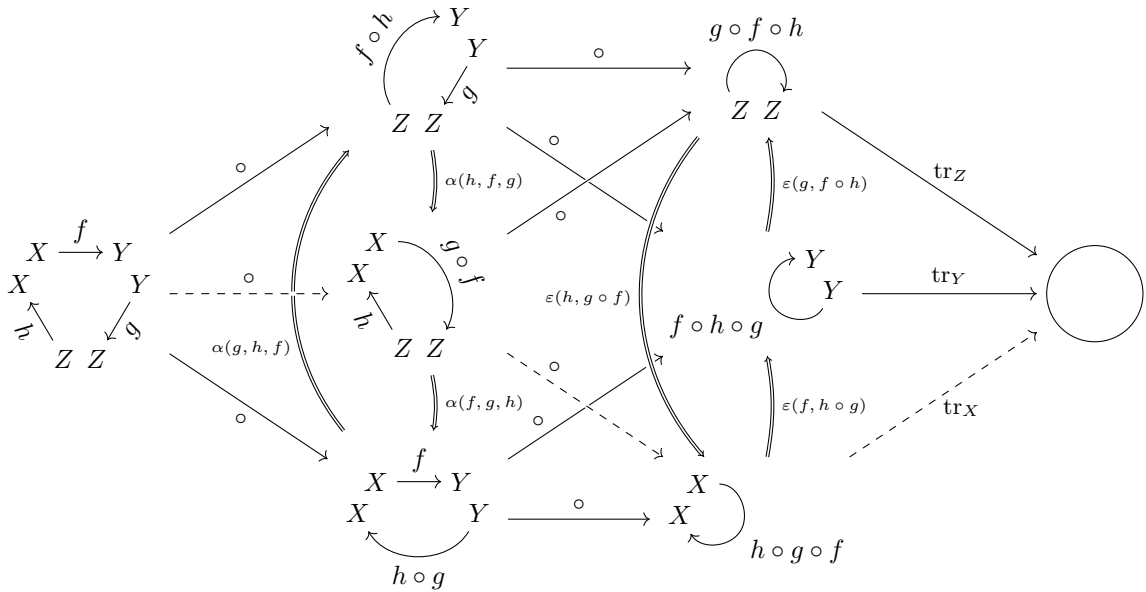
\[
\tikz[]{
\node(000) at (0,0) {$
	\tikz[scale=0.8]{
		%
			\node(Zr) at (0+.27,-1) {$Z$};
			\node(Zl) at (0-.27,-1) {$Z$};
			\node(Yl) at (0.71, 0.77024) {$Y$};
			\node(Yr) at (1.01, 0.25063) {$Y$};
			\node(Xr) at (-0.71, 0.77024) {$X$};
			\node(Xl) at (-1.01, 0.25063) {$X$};
		\draw[->] (Yr) to node[below, sloped]{$g$} (Zr);
		\draw[->] (Xr) to node[above]{$f$} (Yl);
		\draw[->] (Zl) to node[below, sloped]{\\ $h$} (Xl);
		}
	$};
\node(100) at (4.5,3) {$
	\tikz[scale=0.8]{
			\node(Zr) at (0+.27,-1) {$Z$};
			\node(Zl) at (0-.27,-1) {$Z$};
			\node(Yl) at (0.71, 0.77024) {$Y$};
			\node(Yr) at (1.01, 0.25063) {$Y$};
		\draw[->] (Yr) to node[below, sloped]{$g$} (Zr);
		\draw[->] (Zl) to[out=120, in=180] node[above, sloped]{$f \circ h$} (Yl);
		}
	$};
\node(010) at (4.5,0) {$
	\tikz[scale=0.8]{
			\node(Zr) at (0+.27,-1) {$Z$};
			\node(Zl) at (0-.27,-1) {$Z$};
			\node(Xr) at (-0.71, 0.77024) {$X$};
			\node(Xl) at (-1.01, 0.25063) {$X$};
		\draw[->] (Xr) to[out=0, in=60] node[above, sloped]{$g \circ f$} (Zr);
		\draw[->] (Zl) to node[below, sloped]{$h$} (Xl);
		}
	$};
\node(001) at (4.5,-3) {$
	\tikz[scale=0.8]{
			\node(Yl) at (0.71, 0.77024) {$Y$};
			\node(Yr) at (1.01, 0.25063) {$Y$};
			\node(Xr) at (-0.71, 0.77024) {$X$};
			\node(Xl) at (-1.01, 0.25063) {$X$};
		\draw[->] (Xr) to node[above]{$f$} (Yl);
		\draw[->] (Yr) to[out=240, in=300] node[below, sloped]{$h \circ g$} (Xl);
		}
	$};
\node(110) at (9, 3) {$
	\tikz[scale=0.8]{
			\node(Zr) at (0+.27,-1) {$Z$};
			\node(Zl) at (0-.27,-1) {$Z$};
		\draw[->] (Zl) to[out=120, in=180] (0,0) to[out=0, in=60] (Zr);
		\node[anchor=south] at (0,0) {$g \circ f \circ h$};
		}
	$};
\node(101) at (9,0) {$
	\tikz[scale=0.8]{
			\node(Yl) at (0.71, 0.77024) {$Y$};
			\node(Yr) at (1.01, 0.25063) {$Y$};
		\draw[->] (Yr) to[out=240, in=300] (0,0) to[out=120, in=180] (Yl);
		\node[anchor=north east] at (0,0) {$f \circ h \circ g$};
		}
	$};
\node(011) at (9,-3) {$
	\tikz[scale=0.8]{
			\node(Xr) at (-0.71, 0.77024) {$X$};
			\node(Xl) at (-1.01, 0.25063) {$X$};
		\draw[->] (Xr) to[out=0, in=60] (0,0) to[out=-120, in=-60] (Xl);
		\node[anchor=north west] at (0,0) {$h \circ g \circ f$};
		}
	$};
\node(111) at (13.5,0) {$
	\tikz[scale=0.8]{
		\draw (0,0) circle [radius=.8];
		}
	$};
\draw[double, -{Implies[]}, bend left=40] (001) to node[left, pos=0.3]{\tiny$\alpha(g,h,f)$} (100);
\draw[->] (000) to node[above, sloped]{\small$\circ$} (100);
\draw[draw=white, ultra thick] (000) to (010);
\draw[->, dashed] (000) to node[above, sloped]{\small$\circ$} (010);
\draw[->] (000) to node[below,sloped]{\small$\circ$} (001);
\draw[->] (100) to node[above,sloped]{\small$\circ$} (110);
\draw[->] (100) to node[near start, above, sloped]{\small$\circ$} (101);
\draw[->] (001) to node[near start, above, sloped]{\small$\circ$} (101);
\draw[draw=white, ultra thick] (010) to (110);
\draw[->] (010) to node[near start, below, sloped]{\small$\circ$} (110);
\draw[draw=white, ultra thick] (010) to (011);
\draw[->, dashed] (010) to node[near start, above, sloped]{\small$\circ$} (011);
\draw[->] (001) to node[above, sloped]{\small$\circ$} (011);
\draw[->] (110) to node[above=1mm, pos=0.6]{\small$\mathrm{tr}_Z$} (111);
\draw[->] (101) to node[above, sloped]{\small$\mathrm{tr}_Y$} (111);
\draw[->, dashed] (011) to node[below=1mm, pos=0.6]{\small$\mathrm{tr}_X$} (111);
\draw[double, -{Implies[]}, bend left=10] (010) to node[right]{\tiny$\alpha(f,g,h)$} (001);
\draw[double, -{Implies[]}, bend left=10] (100) to node[right]{\tiny$\alpha(h,f,g)$} (010);
\draw[double, -{Implies[]}, bend right=10] (101) to node[right]{\tiny$\varepsilon(g, f\circ h)$} (110);
\draw[double, -{Implies[]}, bend right=10] (011) to node[right]{\tiny$\varepsilon(f, h\circ g)$} (101);
\draw[draw=white, ultra thick, double,-{Implies[]}, bend right=40] (110) to (011);
\draw[double, -{Implies[]}, bend right=40] (110) to node[left]{\tiny$\varepsilon(h, g \circ f)$} (011);
}
\]
The idea is to take formula (\ref{eq:a-e_coherence}) as a map which pushes the dotted path around the cube and back onto itself. As a relation between composites in $\mathcal{C}$, $\alpha$-$\varepsilon$ coherence states that the following hexagon commutes:
\[
\begin{tikzcd}[column sep=large, row sep=large]
	& (h \circ g) \circ f \arrow[r, "\varepsilon \text{(} f\text{,}\, h \circ g\text{)}"] & f \circ (h \circ g) \arrow[dr, "\alpha\text{(}g\text{,}\,h\text{,}\,f\text{)}"] & 
	\\
	h \circ (g \circ f) \arrow[ur, "\alpha\text{(}f\text{,}\,g\text{,}\,h\text{)}"]  & & & (f \circ h) \circ g \arrow[dl, "\varepsilon\text{(}g\text{,}\, f \circ h\text{)}"]
	\\
	& (g \circ f) \circ h \arrow[ul, "\varepsilon\text{(}h\text{,}\, g \circ f\text{)}"] & g \circ (f \circ h) \arrow[l, "\alpha\text{(}h\text{,}\, f\text{,}\, g\text{)}"] &
\end{tikzcd}
\]
It is very useful to index the binary matchings encoded above via pictures, as so:
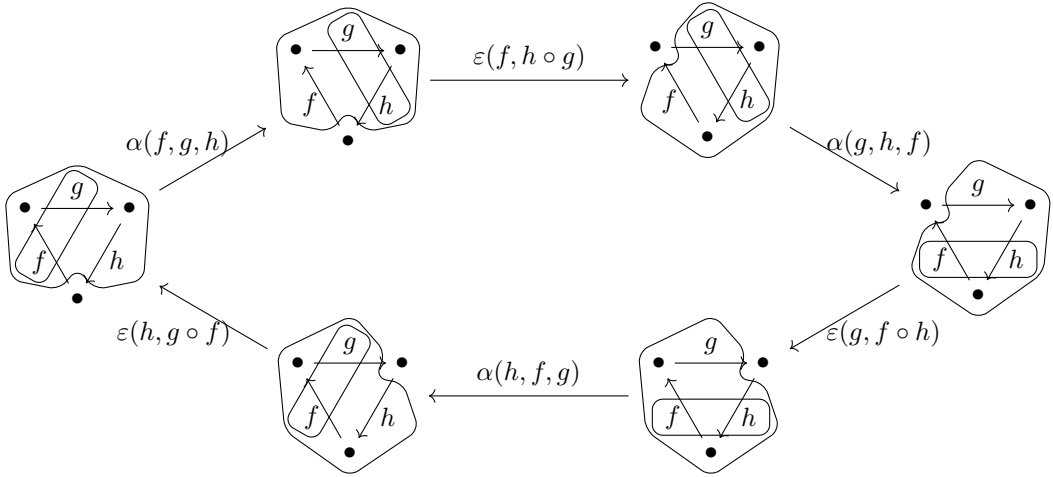
\[
\tikz[x=1.2cm, y=0.7cm]{
	\node(1) at (0,0) {$
	\tikz[scale=0.8,x=1cm, y=1cm]{
	\node(x) at (0,-1) {$\bullet$};
	\node(z) at (1.73205/2, 0.5) {$\bullet$};
	\node(y) at (-1.73205/2, 0.5) {$\bullet$};
	\draw[->] (x) to node(f)[left, pos=0.35]{$f$} (y);
	\draw[->] (y) to node(g)[above]{$g$} (z);
	\draw[->] (z) to node(h)[right, pos=0.6]{$h$} (x);
	%
	\draw[rounded corners] (0.3,-0.85) -- (0,-0.5) -- (-0.3,-0.85) -- (-1.1, -0.69) -- (-1.2, 0.7) -- (0,1.25) -- (1.2,0.7) --  (1.1, -0.69) -- cycle;
	\draw[rounded corners, rotate around={240:(0,0)}] (-.975, -.7) rectangle (.975, -.1);
	}
	$};
	\node(2) at (3,3) {$
	\tikz[scale=0.8,x=1cm, y=1cm]{
	\node(x) at (0,-1) {$\bullet$};
	\node(z) at (1.73205/2, 0.5) {$\bullet$};
	\node(y) at (-1.73205/2, 0.5) {$\bullet$};
	\draw[->] (x) to node(f)[left, pos=0.35]{$f$} (y);
	\draw[->] (y) to node(g)[above]{$g$} (z);
	\draw[->] (z) to node(h)[right, pos=0.65]{$h$} (x);
	%
	\draw[rounded corners] (0.3,-0.85) -- (0,-0.5) -- (-0.3,-0.85) -- (-1.1, -0.69) -- (-1.2, 0.7) -- (0,1.25) -- (1.2,0.7) --  (1.1, -0.69) -- cycle;
	\draw[rounded corners, rotate around={120:(0,0)}] (-.975, -.7) rectangle (.975, -.1);
		}
	$};
	\node(3) at (7,3) {$
	\tikz[scale=0.8,x=1cm, y=1cm]{
	\node(x) at (0,-1) {$\bullet$};
	\node(z) at (1.73205/2, 0.5) {$\bullet$};
	\node(y) at (-1.73205/2, 0.5) {$\bullet$};
	\draw[->] (x) to node(f)[left, pos=0.35]{$f$} (y);
	\draw[->] (y) to node(g)[above]{$g$} (z);
	\draw[->] (z) to node(h)[right, pos=0.6]{$h$} (x);
	%
	\draw[rounded corners, rotate around={240:(0,0)}] (0.3,-0.85) -- (0,-0.5) -- (-0.3,-0.85) -- (-1.1, -0.69) -- (-1.2, 0.7) -- (0,1.25) -- (1.2,0.7) --  (1.1, -0.69) -- cycle;
	\draw[rounded corners, rotate around={120:(0,0)}] (-.975, -.7) rectangle (.975, -.1);
		}
	$};
	\node(4) at (10,0) {$
	\tikz[scale=0.8,x=1cm, y=1cm]{
	\node(x) at (0,-1) {$\bullet$};
	\node(z) at (1.73205/2, 0.5) {$\bullet$};
	\node(y) at (-1.73205/2, 0.5) {$\bullet$};
	\draw[->] (x) to node(f)[left, pos=0.35]{$f$} (y);
	\draw[->] (y) to node(g)[above]{$g$} (z);
	\draw[->] (z) to node(h)[right, pos=0.65]{$h$} (x);
	%
	\draw[rounded corners, rotate around={240:(0,0)}] (0.3,-0.85) -- (0,-0.5) -- (-0.3,-0.85) -- (-1.1, -0.69) -- (-1.2, 0.7) -- (0,1.25) -- (1.2,0.7) --  (1.1, -0.69) -- cycle;
	\draw[rounded corners] (-.975, -.7) rectangle (.975, -.1);
		}
	$};
	\node(5) at (7,-3) {$
	\tikz[scale=0.8,x=1cm, y=1cm]{
	\node(x) at (0,-1) {$\bullet$};
	\node(z) at (1.73205/2, 0.5) {$\bullet$};
	\node(y) at (-1.73205/2, 0.5) {$\bullet$};
	\draw[->] (x) to node(f)[left, pos=0.35]{$f$} (y);
	\draw[->] (y) to node(g)[above]{$g$} (z);
	\draw[->] (z) to node(h)[right, pos=0.65]{$h$} (x);
	%
	\draw[rounded corners, rotate around={120:(0,0)}] (0.3,-0.85) -- (0,-0.5) -- (-0.3,-0.85) -- (-1.1, -0.69) -- (-1.2, 0.7) -- (0,1.25) -- (1.2,0.7) --  (1.1, -0.69) -- cycle;
	\draw[rounded corners] (-.975, -.7) rectangle (.975, -.1);
		}
	$};
	\node(6) at (3,-3) {$
	\tikz[scale=0.8,x=1cm, y=1cm]{
	\node(x) at (0,-1) {$\bullet$};
	\node(z) at (1.73205/2, 0.5) {$\bullet$};
	\node(y) at (-1.73205/2, 0.5) {$\bullet$};
	\draw[->] (x) to node(f)[left, pos=0.35]{$f$} (y);
	\draw[->] (y) to node(g)[above]{$g$} (z);
	\draw[->] (z) to node(h)[right, pos=0.65]{$h$} (x);
	%
	\draw[rounded corners, rotate around={120:(0,0)}] (0.3,-0.85) -- (0,-0.5) -- (-0.3,-0.85) -- (-1.1, -0.69) -- (-1.2, 0.7) -- (0,1.25) -- (1.2,0.7) --  (1.1, -0.69) -- cycle;
	\draw[rounded corners, rotate around={240:(0,0)}] (-.975, -.7) rectangle (.975, -.1);
		}
	$};
	\draw[->] (1) to node[left, near end]{$\alpha(f, g, h)$} (2);
	\draw[->] (2) to node[above]{$\varepsilon(f, h\circ g)$} (3);
	\draw[->] (3) to node[right, near start]{$\alpha(g, h, f)$} (4);
	\draw[->] (4) to node[right, near end]{$\varepsilon(g, f\circ h)$} (5);
	\draw[->] (5) to node[above]{$\alpha(h, f, g)$} (6);
	\draw[->] (6) to node[left, near start]{$\varepsilon(h, g\circ f)$} (1);
}
\]

\begin{definition}
We will call an element of $\Omega_n\mathcal{C}$ an \emph{$n$-partitioning} of a loop it represents in $\mathrm{Tr}(\mathcal{C})$. A presentation of an $n$-partition with a choice of $n-1$ binary matchings (depicted as above) is called a \emph{topography} on the $n$-partitioned loop. Denote the set of topographies on an arbitrary $n$-partitioning by $T(n)$.
\end{definition}

\begin{lemma}
\label{lem:topo_counting}
The number of topographies on an $n$-partitioned loop is the $n$th central binomial coefficient:
\[
\abs{T(n)} 
= 
\binom{2(n-1)}{n-1}.
\]
\end{lemma}

\begin{proof}
Choose a basepoint of an $n$-partitioned loop (there are $n$ choices). Doing so represents that loop as an element of $\mathrm{End}_\mathcal{C}(X)$ for some $X \in \mathrm{Ob}(\mathcal{C})$. Then, a choice of binary matchings after this first choice is equivalent to the number of distinct full binary trees on $n$ leaves, which is equal to the $(n-1)$st Catalan number. Thus $\abs{T(n)} = n \cdot C_{n-1} = \binom{2(n-1)}{n-1}$, as desired.
\end{proof}

Recall that $\alpha$ satisfies a relation called \emph{distant commutativity}, which says that 
\[
\alpha(a, b, c) \alpha(c \circ b \circ a, d, e) = \alpha(c \circ b \circ a, d, e)\alpha(a, b, c).
\]
Similarly, there is $\alpha$-$\varepsilon$ distant commutativity; for a loop partitioned into enough morphisms (at least four), commutative diagrams of the following form start to appear.
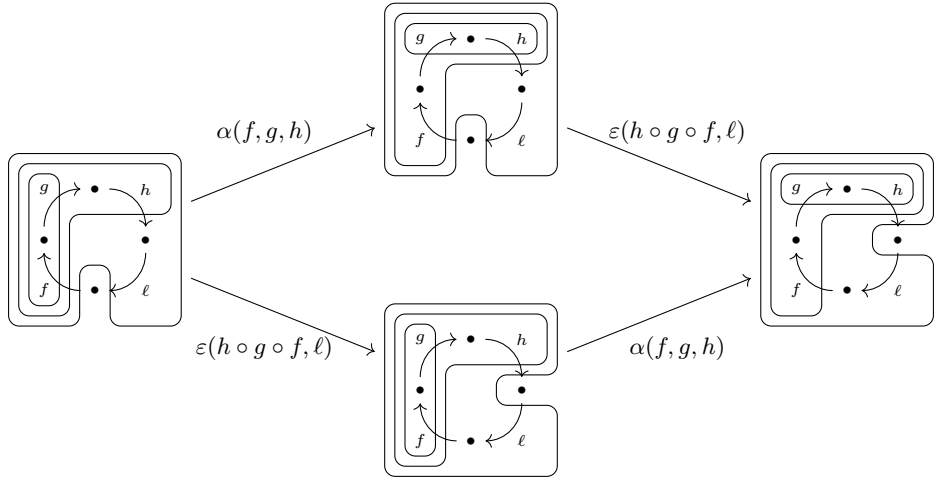
\[
\tikz{
\node(a) at (-1,0) {$
	\tikz[baseline={([yshift=-.5ex]current bounding box.center)}, scale=0.675, x=1cm, y=1cm]{
		\node(N) at (0,1) {\tiny$\bullet$};
		\node(E) at (1,0) {\tiny$\bullet$};
		\node(S) at (0,-1) {\tiny$\bullet$};
		\node(W) at (-1,0) {\tiny$\bullet$};
		\draw[->] (N) to[out=0, in=90] (E);
			\node at (1,1) {\tiny$h$};
		\draw[->] (E) to[out=270, in=0] (S);
			\node at (1,-1) {\tiny$\ell$};
		\draw[->] (S) to[out=180, in=270] (W);
			\node at (-1,-1) {\tiny$f$};
		\draw[->] (W) to[out=90, in=180] (N);
			\node at (-1,1) {\tiny$g$};
		\draw[rounded corners] (-1.3, -1.3) rectangle (-0.7, 1.3);
		\draw[rounded corners] (-1.5, 0) -- (-1.5, -1.5) -- (-0.5, -1.5) -- (-0.5, 0.5) -- (1.5, 0.5) -- (1.5, 1.5) -- (-1.5, 1.5) -- (-1.5, 0);
		\draw[rounded corners] (-1.7, 0) -- (-1.7, -1.7) -- (-0.3, -1.7) -- (-0.3, -0.5) -- (0.3, -0.5) -- (0.3, -1.7) -- (1.7, -1.7) -- (1.7, 1.7) -- (-1.7, 1.7) -- (-1.7, 0);
	}
	$};
\node(b) at (4,2) {$
	\tikz[baseline={([yshift=-.5ex]current bounding box.center)}, scale=0.675, x=1cm, y=1cm]{
		\node(N) at (0,1) {\tiny$\bullet$};
		\node(E) at (1,0) {\tiny$\bullet$};
		\node(S) at (0,-1) {\tiny$\bullet$};
		\node(W) at (-1,0) {\tiny$\bullet$};
		\draw[->] (N) to[out=0, in=90] (E);
			\node at (1,1) {\tiny$h$};
		\draw[->] (E) to[out=270, in=0] (S);
			\node at (1,-1) {\tiny$\ell$};
		\draw[->] (S) to[out=180, in=270] (W);
			\node at (-1,-1) {\tiny$f$};
		\draw[->] (W) to[out=90, in=180] (N);
			\node at (-1,1) {\tiny$g$};
		\draw[rounded corners] (-1.3, 0.7) rectangle (1.3, 1.3);
		\draw[rounded corners] (-1.5, 0) -- (-1.5, -1.5) -- (-0.5, -1.5) -- (-0.5, 0.5) -- (1.5, 0.5) -- (1.5, 1.5) -- (-1.5, 1.5) -- (-1.5, 0);
		\draw[rounded corners] (-1.7, 0) -- (-1.7, -1.7) -- (-0.3, -1.7) -- (-0.3, -0.5) -- (0.3, -0.5) -- (0.3, -1.7) -- (1.7, -1.7) -- (1.7, 1.7) -- (-1.7, 1.7) -- (-1.7, 0);
	}
	$};
\node(c) at (4,-2) {$
	\tikz[baseline={([yshift=-.5ex]current bounding box.center)}, scale=0.675, x=1cm, y=1cm]{
		\node(N) at (0,1) {\tiny$\bullet$};
		\node(E) at (1,0) {\tiny$\bullet$};
		\node(S) at (0,-1) {\tiny$\bullet$};
		\node(W) at (-1,0) {\tiny$\bullet$};
		\draw[->] (N) to[out=0, in=90] (E);
			\node at (1,1) {\tiny$h$};
		\draw[->] (E) to[out=270, in=0] (S);
			\node at (1,-1) {\tiny$\ell$};
		\draw[->] (S) to[out=180, in=270] (W);
			\node at (-1,-1) {\tiny$f$};
		\draw[->] (W) to[out=90, in=180] (N);
			\node at (-1,1) {\tiny$g$};
		\draw[rounded corners] (-1.3, -1.3) rectangle (-0.7, 1.3);
		\draw[rounded corners] (-1.5, 0) -- (-1.5, -1.5) -- (-0.5, -1.5) -- (-0.5, 0.5) -- (1.5, 0.5) -- (1.5, 1.5) -- (-1.5, 1.5) -- (-1.5, 0);
		\draw[rounded corners] (-1.7, 0) -- (-1.7, 1.7) -- (1.7, 1.7) -- (1.7, 0.3) -- (0.5, 0.3) -- (0.5, -0.3) -- (1.7, -0.3) -- (1.7, -1.7) -- (-1.7, -1.7) -- (-1.7,0);
	}
	$};
\node(d) at (9,0) {$
	\tikz[baseline={([yshift=-.5ex]current bounding box.center)}, scale=0.675, x=1cm, y=1cm]{
		\node(N) at (0,1) {\tiny$\bullet$};
		\node(E) at (1,0) {\tiny$\bullet$};
		\node(S) at (0,-1) {\tiny$\bullet$};
		\node(W) at (-1,0) {\tiny$\bullet$};
		\draw[->] (N) to[out=0, in=90] (E);
			\node at (1,1) {\tiny$h$};
		\draw[->] (E) to[out=270, in=0] (S);
			\node at (1,-1) {\tiny$\ell$};
		\draw[->] (S) to[out=180, in=270] (W);
			\node at (-1,-1) {\tiny$f$};
		\draw[->] (W) to[out=90, in=180] (N);
			\node at (-1,1) {\tiny$g$};
		\draw[rounded corners] (-1.3, 0.7) rectangle (1.3, 1.3);
		\draw[rounded corners] (-1.5, 0) -- (-1.5, -1.5) -- (-0.5, -1.5) -- (-0.5, 0.5) -- (1.5, 0.5) -- (1.5, 1.5) -- (-1.5, 1.5) -- (-1.5, 0);
		\draw[rounded corners] (-1.7, 0) -- (-1.7, 1.7) -- (1.7, 1.7) -- (1.7, 0.3) -- (0.5, 0.3) -- (0.5, -0.3) -- (1.7, -0.3) -- (1.7, -1.7) -- (-1.7, -1.7) -- (-1.7,0);
	}
	$};
\draw[->] (a) to node[above=2.5mm, pos=0.4]{\small$\alpha(f, g, h)$} (b);
\draw[->] (a) to node[below=2.5mm, pos=0.4]{\small$\varepsilon(h \circ g\circ f, \ell)$} (c);
\draw[->] (b) to node[above=2.5mm, pos=0.6]{\small$\varepsilon(h \circ g\circ f, \ell)$} (d);
\draw[->] (c) to node[below=2.5mm, pos=0.6]{\small$\alpha(f, g, h)$} (d);
}
\]
We refer to both properties
\[
\alpha(a, b, c) \alpha(c \circ b \circ a, d, e) = \alpha(c \circ b \circ a, d, e)\alpha(a, b, c)
\quad
\text{and}
\quad
\alpha(f, g, h) \varepsilon(h\circ g \circ f, \ell) = \varepsilon(h\circ g \circ f, \ell)  \alpha(f, g, h)
\]
ambiguously as \emph{distant commutativity}.

\begin{definition}
We denote by $\mathcal{T}(n)$ the \emph{space of topographies} associated to an arbitrary $n$-partitioned loop, defined as the following 2-dimensional CW-complex:
	\begin{enumerate}
		\item $\mathcal{T}(n)^0 := T(n)$;
		\item $\mathcal{T}(n)^1$ is an $(n-1)$-valent graph with $\abs{T(n)}$-many vertices corresponding to changing a single binary matching ($n-2$ edges correspond to a single application of $\alpha$, and one edge corresponds to a basepoint change, \textit{i.e.}, an application of $\varepsilon$);
		\item $\mathcal{T}(n)^2 = \mathcal{T}(n)$ is obtained by gluing 2-cells along all words corresponding to
		\begin{enumerate}
			\item the cocycle condition on $\alpha$,
			\item $\alpha$-$\varepsilon$ coherence, or
			\item distant commutativity.
		\end{enumerate}
	\end{enumerate}
\end{definition}

Note that $\mathcal{T}(2) \simeq \mathcal{T}(3) \simeq \mathbb{D}^2$ by definition.

\begin{lemma}
\label{lem:simpcon}
The space of topographies $\mathcal{T}(4)$ is simply connected.
\end{lemma}

\begin{proof}
We compute by hand that $\mathcal{T}(4) \simeq S^2$; see Figure \ref{fig:tensorpf}. Indeed, $\mathcal{T}(4)$ is a polyhedron with $\binom{6}{3} = 20$ vertices by Lemma \ref{lem:topo_counting} (each of valence 3) and twelve faces: four square, four pentagonal, and four hexagonal. Each square face commutes by $\alpha$-$\varepsilon$ distant commutativity, each pentagonal face commutes by the cocycle condition, and each hexagonal face commutes by $\alpha$-$\varepsilon$ coherence. 
\end{proof}

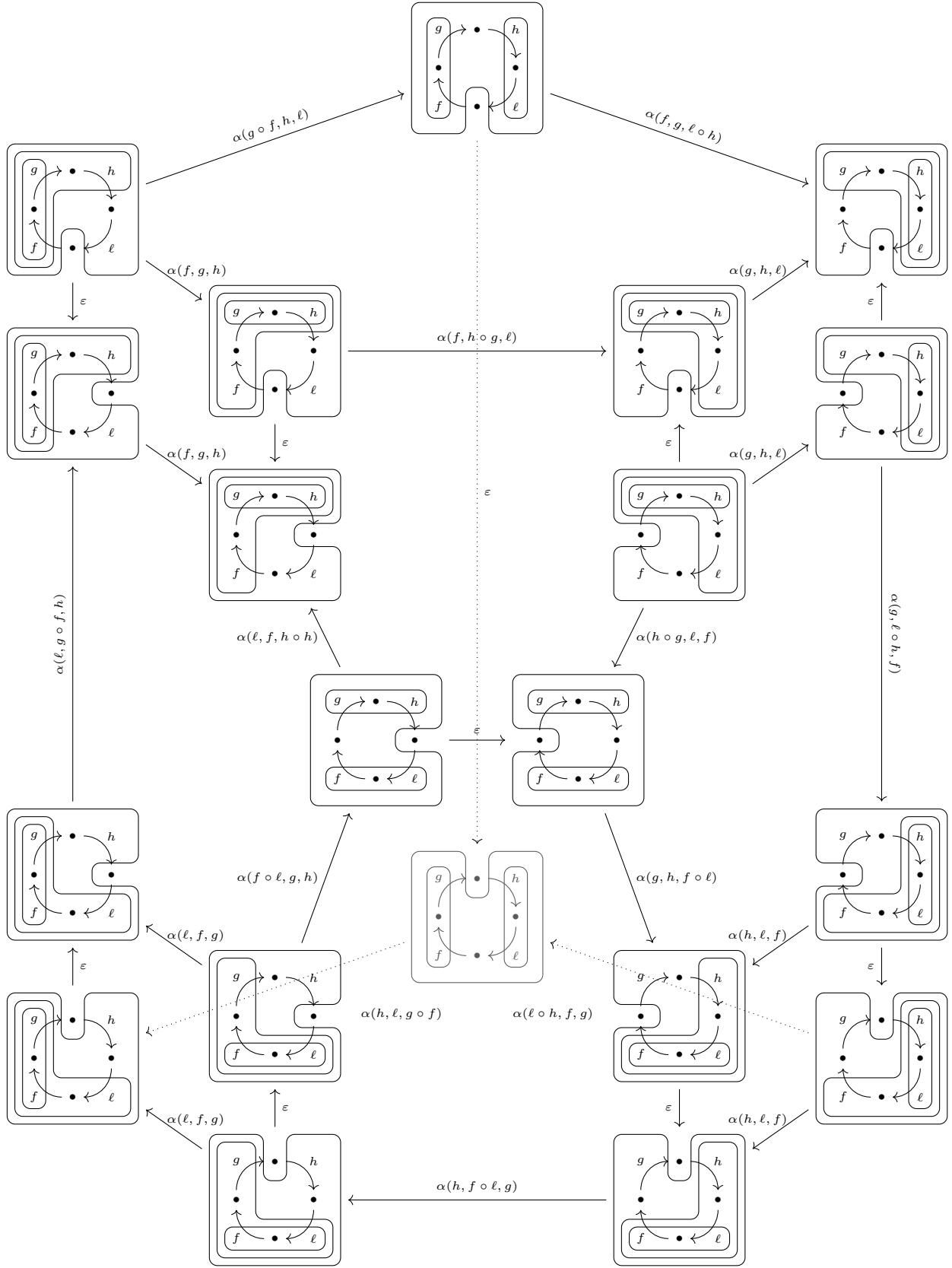
\begin{figure}[p]
\centering
\begin{tikzpicture}[x=2.5cm, y=1.75cm]
	\node(start) at (0, 8*1.4142) {$
	\tikz[baseline={([yshift=-.5ex]current bounding box.center)}, scale=0.675, x=1cm, y=1cm]{
		\node(N) at (0,1) {\tiny$\bullet$};
		\node(E) at (1,0) {\tiny$\bullet$};
		\node(S) at (0,-1) {\tiny$\bullet$};
		\node(W) at (-1,0) {\tiny$\bullet$};
		\draw[->] (N) to[out=0, in=90] (E);
			\node at (1,1) {\tiny$h$};
		\draw[->] (E) to[out=270, in=0] (S);
			\node at (1,-1) {\tiny$\ell$};
		\draw[->] (S) to[out=180, in=270] (W);
			\node at (-1,-1) {\tiny$f$};
		\draw[->] (W) to[out=90, in=180] (N);
			\node at (-1,1) {\tiny$g$};
		\draw[rounded corners] (-1.3, -1.3) rectangle (-0.7, 1.3);
		\draw[rounded corners] (1.3, -1.3) rectangle (0.7, 1.3);
		\draw[rounded corners] (-1.7, 0) -- (-1.7, -1.7) -- (-0.3, -1.7) -- (-0.3, -0.5) -- (0.3, -0.5) -- (0.3, -1.7) -- (1.7, -1.7) -- (1.7, 1.7) -- (-1.7, 1.7) -- (-1.7, 0);
	}
	$};
	\node(s1a) at (-2*1.4142, 7*1.4142) {$
	\tikz[baseline={([yshift=-.5ex]current bounding box.center)}, scale=0.675, x=1cm, y=1cm]{
		\node(N) at (0,1) {\tiny$\bullet$};
		\node(E) at (1,0) {\tiny$\bullet$};
		\node(S) at (0,-1) {\tiny$\bullet$};
		\node(W) at (-1,0) {\tiny$\bullet$};
		\draw[->] (N) to[out=0, in=90] (E);
			\node at (1,1) {\tiny$h$};
		\draw[->] (E) to[out=270, in=0] (S);
			\node at (1,-1) {\tiny$\ell$};
		\draw[->] (S) to[out=180, in=270] (W);
			\node at (-1,-1) {\tiny$f$};
		\draw[->] (W) to[out=90, in=180] (N);
			\node at (-1,1) {\tiny$g$};
		\draw[rounded corners] (-1.3, -1.3) rectangle (-0.7, 1.3);
		\draw[rounded corners] (-1.5, 0) -- (-1.5, -1.5) -- (-0.5, -1.5) -- (-0.5, 0.5) -- (1.5, 0.5) -- (1.5, 1.5) -- (-1.5, 1.5) -- (-1.5, 0);
		\draw[rounded corners] (-1.7, 0) -- (-1.7, -1.7) -- (-0.3, -1.7) -- (-0.3, -0.5) -- (0.3, -0.5) -- (0.3, -1.7) -- (1.7, -1.7) -- (1.7, 1.7) -- (-1.7, 1.7) -- (-1.7, 0);
	}
	$};
	\node(s2a) at (2*1.4142, 7*1.4142) {$
	\tikz[baseline={([yshift=-.5ex]current bounding box.center)}, scale=0.675, x=1cm, y=1cm]{
		\node(N) at (0,1) {\tiny$\bullet$};
		\node(E) at (1,0) {\tiny$\bullet$};
		\node(S) at (0,-1) {\tiny$\bullet$};
		\node(W) at (-1,0) {\tiny$\bullet$};
		\draw[->] (N) to[out=0, in=90] (E);
			\node at (1,1) {\tiny$h$};
		\draw[->] (E) to[out=270, in=0] (S);
			\node at (1,-1) {\tiny$\ell$};
		\draw[->] (S) to[out=180, in=270] (W);
			\node at (-1,-1) {\tiny$f$};
		\draw[->] (W) to[out=90, in=180] (N);
			\node at (-1,1) {\tiny$g$};
		\draw[rounded corners] (1.3, -1.3) rectangle (0.7, 1.3);
		\draw[rounded corners] (1.5, 0) -- (1.5, -1.5) -- (0.5, -1.5) -- (0.5, 0.5) -- (-1.5, 0.5) -- (-1.5, 1.5) -- (1.5, 1.5) -- (1.5, 0);
		\draw[rounded corners] (-1.7, 0) -- (-1.7, -1.7) -- (-0.3, -1.7) -- (-0.3, -0.5) -- (0.3, -0.5) -- (0.3, -1.7) -- (1.7, -1.7) -- (1.7, 1.7) -- (-1.7, 1.7) -- (-1.7, 0);
	}
	$};
	\node(s1b) at (-1*1.4142,6*1.4142) {$
	\tikz[baseline={([yshift=-.5ex]current bounding box.center)}, scale=0.675, x=1cm, y=1cm]{
		\node(N) at (0,1) {\tiny$\bullet$};
		\node(E) at (1,0) {\tiny$\bullet$};
		\node(S) at (0,-1) {\tiny$\bullet$};
		\node(W) at (-1,0) {\tiny$\bullet$};
		\draw[->] (N) to[out=0, in=90] (E);
			\node at (1,1) {\tiny$h$};
		\draw[->] (E) to[out=270, in=0] (S);
			\node at (1,-1) {\tiny$\ell$};
		\draw[->] (S) to[out=180, in=270] (W);
			\node at (-1,-1) {\tiny$f$};
		\draw[->] (W) to[out=90, in=180] (N);
			\node at (-1,1) {\tiny$g$};
		\draw[rounded corners] (-1.3, 0.7) rectangle (1.3, 1.3);
		\draw[rounded corners] (-1.5, 0) -- (-1.5, -1.5) -- (-0.5, -1.5) -- (-0.5, 0.5) -- (1.5, 0.5) -- (1.5, 1.5) -- (-1.5, 1.5) -- (-1.5, 0);
		\draw[rounded corners] (-1.7, 0) -- (-1.7, -1.7) -- (-0.3, -1.7) -- (-0.3, -0.5) -- (0.3, -0.5) -- (0.3, -1.7) -- (1.7, -1.7) -- (1.7, 1.7) -- (-1.7, 1.7) -- (-1.7, 0);
	}
	$};
	\node(s2b) at (1*1.4142,6*1.4142) {$
	\tikz[baseline={([yshift=-.5ex]current bounding box.center)}, scale=0.675, x=1cm, y=1cm]{
		\node(N) at (0,1) {\tiny$\bullet$};
		\node(E) at (1,0) {\tiny$\bullet$};
		\node(S) at (0,-1) {\tiny$\bullet$};
		\node(W) at (-1,0) {\tiny$\bullet$};
		\draw[->] (N) to[out=0, in=90] (E);
			\node at (1,1) {\tiny$h$};
		\draw[->] (E) to[out=270, in=0] (S);
			\node at (1,-1) {\tiny$\ell$};
		\draw[->] (S) to[out=180, in=270] (W);
			\node at (-1,-1) {\tiny$f$};
		\draw[->] (W) to[out=90, in=180] (N);
			\node at (-1,1) {\tiny$g$};
		\draw[rounded corners] (-1.3, 0.7) rectangle (1.3, 1.3);
		\draw[rounded corners] (1.5, 0) -- (1.5, -1.5) -- (0.5, -1.5) -- (0.5, 0.5) -- (-1.5, 0.5) -- (-1.5, 1.5) -- (1.5, 1.5) -- (1.5, 0);
		\draw[rounded corners] (-1.7, 0) -- (-1.7, -1.7) -- (-0.3, -1.7) -- (-0.3, -0.5) -- (0.3, -0.5) -- (0.3, -1.7) -- (1.7, -1.7) -- (1.7, 1.7) -- (-1.7, 1.7) -- (-1.7, 0);
	}
	$};
	\node(s1c) at (-2*1.4142,5.7*1.4142) {$
	\tikz[baseline={([yshift=-.5ex]current bounding box.center)}, scale=0.675, x=1cm, y=1cm]{
		\node(N) at (0,1) {\tiny$\bullet$};
		\node(E) at (1,0) {\tiny$\bullet$};
		\node(S) at (0,-1) {\tiny$\bullet$};
		\node(W) at (-1,0) {\tiny$\bullet$};
		\draw[->] (N) to[out=0, in=90] (E);
			\node at (1,1) {\tiny$h$};
		\draw[->] (E) to[out=270, in=0] (S);
			\node at (1,-1) {\tiny$\ell$};
		\draw[->] (S) to[out=180, in=270] (W);
			\node at (-1,-1) {\tiny$f$};
		\draw[->] (W) to[out=90, in=180] (N);
			\node at (-1,1) {\tiny$g$};
		\draw[rounded corners] (-1.3, -1.3) rectangle (-0.7, 1.3);
		\draw[rounded corners] (-1.5, 0) -- (-1.5, -1.5) -- (-0.5, -1.5) -- (-0.5, 0.5) -- (1.5, 0.5) -- (1.5, 1.5) -- (-1.5, 1.5) -- (-1.5, 0);
		\draw[rounded corners] (-1.7, 0) -- (-1.7, 1.7) -- (1.7, 1.7) -- (1.7, 0.3) -- (0.5, 0.3) -- (0.5, -0.3) -- (1.7, -0.3) -- (1.7, -1.7) -- (-1.7, -1.7) -- (-1.7,0);
	}
	$};
	\node(s2c) at (2*1.4142,5.7*1.4142) {$
	\tikz[baseline={([yshift=-.5ex]current bounding box.center)}, scale=0.675, x=1cm, y=1cm]{
		\node(N) at (0,1) {\tiny$\bullet$};
		\node(E) at (1,0) {\tiny$\bullet$};
		\node(S) at (0,-1) {\tiny$\bullet$};
		\node(W) at (-1,0) {\tiny$\bullet$};
		\draw[->] (N) to[out=0, in=90] (E);
			\node at (1,1) {\tiny$h$};
		\draw[->] (E) to[out=270, in=0] (S);
			\node at (1,-1) {\tiny$\ell$};
		\draw[->] (S) to[out=180, in=270] (W);
			\node at (-1,-1) {\tiny$f$};
		\draw[->] (W) to[out=90, in=180] (N);
			\node at (-1,1) {\tiny$g$};
		\draw[rounded corners] (1.3, -1.3) rectangle (0.7, 1.3);
		\draw[rounded corners] (1.5, 0) -- (1.5, -1.5) -- (0.5, -1.5) -- (0.5, 0.5) -- (-1.5, 0.5) -- (-1.5, 1.5) -- (1.5, 1.5) -- (1.5, 0);
		\draw[rounded corners] (1.7, 0) -- (1.7, 1.7) -- (-1.7, 1.7) -- (-1.7, 0.3) -- (-0.5, 0.3) -- (-0.5, -0.3) -- (-1.7, -0.3) -- (-1.7, -1.7) -- (1.7, -1.7) -- (1.7,0);
	}
	$};
	\node(s1d) at (-1*1.4142,4.7*1.4142) {$
	\tikz[baseline={([yshift=-.5ex]current bounding box.center)}, scale=0.675, x=1cm, y=1cm]{
		\node(N) at (0,1) {\tiny$\bullet$};
		\node(E) at (1,0) {\tiny$\bullet$};
		\node(S) at (0,-1) {\tiny$\bullet$};
		\node(W) at (-1,0) {\tiny$\bullet$};
		\draw[->] (N) to[out=0, in=90] (E);
			\node at (1,1) {\tiny$h$};
		\draw[->] (E) to[out=270, in=0] (S);
			\node at (1,-1) {\tiny$\ell$};
		\draw[->] (S) to[out=180, in=270] (W);
			\node at (-1,-1) {\tiny$f$};
		\draw[->] (W) to[out=90, in=180] (N);
			\node at (-1,1) {\tiny$g$};
		\draw[rounded corners] (-1.3, 0.7) rectangle (1.3, 1.3);
		\draw[rounded corners] (-1.5, 0) -- (-1.5, -1.5) -- (-0.5, -1.5) -- (-0.5, 0.5) -- (1.5, 0.5) -- (1.5, 1.5) -- (-1.5, 1.5) -- (-1.5, 0);
		\draw[rounded corners] (-1.7, 0) -- (-1.7, 1.7) -- (1.7, 1.7) -- (1.7, 0.3) -- (0.5, 0.3) -- (0.5, -0.3) -- (1.7, -0.3) -- (1.7, -1.7) -- (-1.7, -1.7) -- (-1.7,0);
	}
	$};
	\node(s2d) at (1*1.4142,4.7*1.4142) {$
	\tikz[baseline={([yshift=-.5ex]current bounding box.center)}, scale=0.675, x=1cm, y=1cm]{
		\node(N) at (0,1) {\tiny$\bullet$};
		\node(E) at (1,0) {\tiny$\bullet$};
		\node(S) at (0,-1) {\tiny$\bullet$};
		\node(W) at (-1,0) {\tiny$\bullet$};
		\draw[->] (N) to[out=0, in=90] (E);
			\node at (1,1) {\tiny$h$};
		\draw[->] (E) to[out=270, in=0] (S);
			\node at (1,-1) {\tiny$\ell$};
		\draw[->] (S) to[out=180, in=270] (W);
			\node at (-1,-1) {\tiny$f$};
		\draw[->] (W) to[out=90, in=180] (N);
			\node at (-1,1) {\tiny$g$};
		\draw[rounded corners] (-1.3, 0.7) rectangle (1.3, 1.3);
		\draw[rounded corners] (1.5, 0) -- (1.5, -1.5) -- (0.5, -1.5) -- (0.5, 0.5) -- (-1.5, 0.5) -- (-1.5, 1.5) -- (1.5, 1.5) -- (1.5, 0);
		\draw[rounded corners] (1.7, 0) -- (1.7, 1.7) -- (-1.7, 1.7) -- (-1.7, 0.3) -- (-0.5, 0.3) -- (-0.5, -0.3) -- (-1.7, -0.3) -- (-1.7, -1.7) -- (1.7, -1.7) -- (1.7,0);
	}
	$};
	\node(end') at (-1*1.4142/2,3.25*1.4142) {$
	\tikz[baseline={([yshift=-.5ex]current bounding box.center)}, scale=0.675, x=1cm, y=1cm]{
		\node(N) at (0,1) {\tiny$\bullet$};
		\node(E) at (1,0) {\tiny$\bullet$};
		\node(S) at (0,-1) {\tiny$\bullet$};
		\node(W) at (-1,0) {\tiny$\bullet$};
		\draw[->] (N) to[out=0, in=90] (E);
			\node at (1,1) {\tiny$h$};
		\draw[->] (E) to[out=270, in=0] (S);
			\node at (1,-1) {\tiny$\ell$};
		\draw[->] (S) to[out=180, in=270] (W);
			\node at (-1,-1) {\tiny$f$};
		\draw[->] (W) to[out=90, in=180] (N);
			\node at (-1,1) {\tiny$g$};
		\draw[rounded corners] (-1.3, -0.7) rectangle (1.3, -1.3);
		\draw[rounded corners] (-1.3, 0.7) rectangle (1.3, 1.3);
		\draw[rounded corners] (-1.7, 0) -- (-1.7, 1.7) -- (1.7, 1.7) -- (1.7, 0.3) -- (0.5, 0.3) -- (0.5, -0.3) -- (1.7, -0.3) -- (1.7, -1.7) -- (-1.7, -1.7) -- (-1.7,0);
	}
	$};
	\node(end) at (1*1.4142/2,3.25*1.4142) {$
	\tikz[baseline={([yshift=-.5ex]current bounding box.center)}, scale=0.675, x=1cm, y=1cm]{
		\node(N) at (0,1) {\tiny$\bullet$};
		\node(E) at (1,0) {\tiny$\bullet$};
		\node(S) at (0,-1) {\tiny$\bullet$};
		\node(W) at (-1,0) {\tiny$\bullet$};
		\draw[->] (N) to[out=0, in=90] (E);
			\node at (1,1) {\tiny$h$};
		\draw[->] (E) to[out=270, in=0] (S);
			\node at (1,-1) {\tiny$\ell$};
		\draw[->] (S) to[out=180, in=270] (W);
			\node at (-1,-1) {\tiny$f$};
		\draw[->] (W) to[out=90, in=180] (N);
			\node at (-1,1) {\tiny$g$};
		\draw[rounded corners] (-1.3, -0.7) rectangle (1.3, -1.3);
		\draw[rounded corners] (-1.3, 0.7) rectangle (1.3, 1.3);
		\draw[rounded corners] (1.7, 0) -- (1.7, 1.7) -- (-1.7, 1.7) -- (-1.7, 0.3) -- (-0.5, 0.3) -- (-0.5, -0.3) -- (-1.7, -0.3) -- (-1.7, -1.7) -- (1.7, -1.7) -- (1.7,0);
	}
	$};
	\node(start') at (0,2*1.4142) {$
	\tikz[baseline={([yshift=-.5ex]current bounding box.center)}, scale=0.675, x=1cm, y=1cm, black!30!gray]{
		\node(N) at (0,1) {\tiny$\bullet$};
		\node(E) at (1,0) {\tiny$\bullet$};
		\node(S) at (0,-1) {\tiny$\bullet$};
		\node(W) at (-1,0) {\tiny$\bullet$};
		\draw[->] (N) to[out=0, in=90] (E);
			\node at (1,1) {\tiny$h$};
		\draw[->] (E) to[out=270, in=0] (S);
			\node at (1,-1) {\tiny$\ell$};
		\draw[->] (S) to[out=180, in=270] (W);
			\node at (-1,-1) {\tiny$f$};
		\draw[->] (W) to[out=90, in=180] (N);
			\node at (-1,1) {\tiny$g$};
		\draw[rounded corners] (-1.3, -1.3) rectangle (-0.7, 1.3);
		\draw[rounded corners] (1.3, -1.3) rectangle (0.7, 1.3);
		\draw[rounded corners] (-1.7, 0) -- (-1.7, 1.7) -- (-0.3, 1.7) -- (-0.3, 0.5) -- (0.3, 0.5) -- (0.3, 1.7) -- (1.7, 1.7) -- (1.7, -1.7) -- (-1.7, -1.7) -- (-1.7, 0);
	}
	$};
	\node(s3a) at (-2*1.4142,2.3*1.4142) {$
	\tikz[baseline={([yshift=-.5ex]current bounding box.center)}, scale=0.675, x=1cm, y=1cm]{
		\node(N) at (0,1) {\tiny$\bullet$};
		\node(E) at (1,0) {\tiny$\bullet$};
		\node(S) at (0,-1) {\tiny$\bullet$};
		\node(W) at (-1,0) {\tiny$\bullet$};
		\draw[->] (N) to[out=0, in=90] (E);
			\node at (1,1) {\tiny$h$};
		\draw[->] (E) to[out=270, in=0] (S);
			\node at (1,-1) {\tiny$\ell$};
		\draw[->] (S) to[out=180, in=270] (W);
			\node at (-1,-1) {\tiny$f$};
		\draw[->] (W) to[out=90, in=180] (N);
			\node at (-1,1) {\tiny$g$};
		\draw[rounded corners] (-1.3, -1.3) rectangle (-0.7, 1.3);
		\draw[rounded corners] (-1.5, 0) -- (-1.5, 1.5) -- (-0.5, 1.5) -- (-0.5, -0.5) -- (1.5, -0.5) -- (1.5, -1.5) -- (-1.5, -1.5) -- (-1.5, 0);
		\draw[rounded corners] (-1.7, 0) -- (-1.7, 1.7) -- (1.7, 1.7) -- (1.7, 0.3) -- (0.5, 0.3) -- (0.5, -0.3) -- (1.7, -0.3) -- (1.7, -1.7) -- (-1.7, -1.7) -- (-1.7,0);
	}
	$};
	\node(s4a) at (2*1.4142,2.3*1.4142) {$
	\tikz[baseline={([yshift=-.5ex]current bounding box.center)}, scale=0.675, x=1cm, y=1cm]{
		\node(N) at (0,1) {\tiny$\bullet$};
		\node(E) at (1,0) {\tiny$\bullet$};
		\node(S) at (0,-1) {\tiny$\bullet$};
		\node(W) at (-1,0) {\tiny$\bullet$};
		\draw[->] (N) to[out=0, in=90] (E);
			\node at (1,1) {\tiny$h$};
		\draw[->] (E) to[out=270, in=0] (S);
			\node at (1,-1) {\tiny$\ell$};
		\draw[->] (S) to[out=180, in=270] (W);
			\node at (-1,-1) {\tiny$f$};
		\draw[->] (W) to[out=90, in=180] (N);
			\node at (-1,1) {\tiny$g$};
		\draw[rounded corners] (1.3, -1.3) rectangle (0.7, 1.3);
		\draw[rounded corners] (1.5, 0) -- (1.5, 1.5) -- (0.5, 1.5) -- (0.5, -0.5) -- (-1.5, -0.5) -- (-1.5, -1.5) -- (1.5, -1.5) -- (1.5, 0);
		\draw[rounded corners] (1.7, 0) -- (1.7, 1.7) -- (-1.7, 1.7) -- (-1.7, 0.3) -- (-0.5, 0.3) -- (-0.5, -0.3) -- (-1.7, -0.3) -- (-1.7, -1.7) -- (1.7, -1.7) -- (1.7,0);
	}
	$};
	\node(s3b) at (-1*1.4142,1.3*1.4142) {$
	\tikz[baseline={([yshift=-.5ex]current bounding box.center)}, scale=0.675, x=1cm, y=1cm]{
		\node(N) at (0,1) {\tiny$\bullet$};
		\node(E) at (1,0) {\tiny$\bullet$};
		\node(S) at (0,-1) {\tiny$\bullet$};
		\node(W) at (-1,0) {\tiny$\bullet$};
		\draw[->] (N) to[out=0, in=90] (E);
			\node at (1,1) {\tiny$h$};
		\draw[->] (E) to[out=270, in=0] (S);
			\node at (1,-1) {\tiny$\ell$};
		\draw[->] (S) to[out=180, in=270] (W);
			\node at (-1,-1) {\tiny$f$};
		\draw[->] (W) to[out=90, in=180] (N);
			\node at (-1,1) {\tiny$g$};
		\draw[rounded corners] (-1.3, -1.3) rectangle (1.3, -0.7);
		\draw[rounded corners] (-1.5, 0) -- (-1.5, 1.5) -- (-0.5, 1.5) -- (-0.5, -0.5) -- (1.5, -0.5) -- (1.5, -1.5) -- (-1.5, -1.5) -- (-1.5, 0);
		\draw[rounded corners] (-1.7, 0) -- (-1.7, 1.7) -- (1.7, 1.7) -- (1.7, 0.3) -- (0.5, 0.3) -- (0.5, -0.3) -- (1.7, -0.3) -- (1.7, -1.7) -- (-1.7, -1.7) -- (-1.7,0);
	}
	$};
	\node(s4b) at (1*1.4142,1.3*1.4142) {$
	\tikz[baseline={([yshift=-.5ex]current bounding box.center)}, scale=0.675, x=1cm, y=1cm]{
		\node(N) at (0,1) {\tiny$\bullet$};
		\node(E) at (1,0) {\tiny$\bullet$};
		\node(S) at (0,-1) {\tiny$\bullet$};
		\node(W) at (-1,0) {\tiny$\bullet$};
		\draw[->] (N) to[out=0, in=90] (E);
			\node at (1,1) {\tiny$h$};
		\draw[->] (E) to[out=270, in=0] (S);
			\node at (1,-1) {\tiny$\ell$};
		\draw[->] (S) to[out=180, in=270] (W);
			\node at (-1,-1) {\tiny$f$};
		\draw[->] (W) to[out=90, in=180] (N);
			\node at (-1,1) {\tiny$g$};
		\draw[rounded corners] (-1.3, -1.3) rectangle (1.3, -0.7);
		\draw[rounded corners] (1.5, 0) -- (1.5, 1.5) -- (0.5, 1.5) -- (0.5, -0.5) -- (-1.5, -0.5) -- (-1.5, -1.5) -- (1.5, -1.5) -- (1.5, 0);
		\draw[rounded corners] (1.7, 0) -- (1.7, 1.7) -- (-1.7, 1.7) -- (-1.7, 0.3) -- (-0.5, 0.3) -- (-0.5, -0.3) -- (-1.7, -0.3) -- (-1.7, -1.7) -- (1.7, -1.7) -- (1.7,0);
	}
	$};
	\node(s3c) at (-2*1.4142,1*1.4142) {$
	\tikz[baseline={([yshift=-.5ex]current bounding box.center)}, scale=0.675, x=1cm, y=1cm]{
		\node(N) at (0,1) {\tiny$\bullet$};
		\node(E) at (1,0) {\tiny$\bullet$};
		\node(S) at (0,-1) {\tiny$\bullet$};
		\node(W) at (-1,0) {\tiny$\bullet$};
		\draw[->] (N) to[out=0, in=90] (E);
			\node at (1,1) {\tiny$h$};
		\draw[->] (E) to[out=270, in=0] (S);
			\node at (1,-1) {\tiny$\ell$};
		\draw[->] (S) to[out=180, in=270] (W);
			\node at (-1,-1) {\tiny$f$};
		\draw[->] (W) to[out=90, in=180] (N);
			\node at (-1,1) {\tiny$g$};
		\draw[rounded corners] (-1.3, -1.3) rectangle (-0.7, 1.3);
		\draw[rounded corners] (-1.5, 0) -- (-1.5, 1.5) -- (-0.5, 1.5) -- (-0.5, -0.5) -- (1.5, -0.5) -- (1.5, -1.5) -- (-1.5, -1.5) -- (-1.5, 0);
		\draw[rounded corners] (-1.7, 0) -- (-1.7, 1.7) -- (-0.3, 1.7) -- (-0.3, 0.5) -- (0.3, 0.5) -- (0.3, 1.7) -- (1.7, 1.7) -- (1.7, -1.7) -- (-1.7, -1.7) -- (-1.7, 0);
	}
	$};
	\node(s4c) at (2*1.4142,1*1.4142) {$
	\tikz[baseline={([yshift=-.5ex]current bounding box.center)}, scale=0.675, x=1cm, y=1cm]{
		\node(N) at (0,1) {\tiny$\bullet$};
		\node(E) at (1,0) {\tiny$\bullet$};
		\node(S) at (0,-1) {\tiny$\bullet$};
		\node(W) at (-1,0) {\tiny$\bullet$};
		\draw[->] (N) to[out=0, in=90] (E);
			\node at (1,1) {\tiny$h$};
		\draw[->] (E) to[out=270, in=0] (S);
			\node at (1,-1) {\tiny$\ell$};
		\draw[->] (S) to[out=180, in=270] (W);
			\node at (-1,-1) {\tiny$f$};
		\draw[->] (W) to[out=90, in=180] (N);
			\node at (-1,1) {\tiny$g$};
		\draw[rounded corners] (1.3, -1.3) rectangle (0.7, 1.3);
		\draw[rounded corners] (1.5, 0) -- (1.5, 1.5) -- (0.5, 1.5) -- (0.5, -0.5) -- (-1.5, -0.5) -- (-1.5, -1.5) -- (1.5, -1.5) -- (1.5, 0);
		\draw[rounded corners] (-1.7, 0) -- (-1.7, 1.7) -- (-0.3, 1.7) -- (-0.3, 0.5) -- (0.3, 0.5) -- (0.3, 1.7) -- (1.7, 1.7) -- (1.7, -1.7) -- (-1.7, -1.7) -- (-1.7, 0);
	}
	$};
	\node(s3d) at (-1*1.4142,0) {$
	\tikz[baseline={([yshift=-.5ex]current bounding box.center)}, scale=0.675, x=1cm, y=1cm]{
		\node(N) at (0,1) {\tiny$\bullet$};
		\node(E) at (1,0) {\tiny$\bullet$};
		\node(S) at (0,-1) {\tiny$\bullet$};
		\node(W) at (-1,0) {\tiny$\bullet$};
		\draw[->] (N) to[out=0, in=90] (E);
			\node at (1,1) {\tiny$h$};
		\draw[->] (E) to[out=270, in=0] (S);
			\node at (1,-1) {\tiny$\ell$};
		\draw[->] (S) to[out=180, in=270] (W);
			\node at (-1,-1) {\tiny$f$};
		\draw[->] (W) to[out=90, in=180] (N);
			\node at (-1,1) {\tiny$g$};
		\draw[rounded corners] (-1.3, -1.3) rectangle (1.3, -0.7);
		\draw[rounded corners] (-1.5, 0) -- (-1.5, 1.5) -- (-0.5, 1.5) -- (-0.5, -0.5) -- (1.5, -0.5) -- (1.5, -1.5) -- (-1.5, -1.5) -- (-1.5, 0);
		\draw[rounded corners] (-1.7, 0) -- (-1.7, 1.7) -- (-0.3, 1.7) -- (-0.3, 0.5) -- (0.3, 0.5) -- (0.3, 1.7) -- (1.7, 1.7) -- (1.7, -1.7) -- (-1.7, -1.7) -- (-1.7, 0);
	}
	$};
	\node(s4d) at (1*1.4142,0) {$
	\tikz[baseline={([yshift=-.5ex]current bounding box.center)}, scale=0.675, x=1cm, y=1cm]{
		\node(N) at (0,1) {\tiny$\bullet$};
		\node(E) at (1,0) {\tiny$\bullet$};
		\node(S) at (0,-1) {\tiny$\bullet$};
		\node(W) at (-1,0) {\tiny$\bullet$};
		\draw[->] (N) to[out=0, in=90] (E);
			\node at (1,1) {\tiny$h$};
		\draw[->] (E) to[out=270, in=0] (S);
			\node at (1,-1) {\tiny$\ell$};
		\draw[->] (S) to[out=180, in=270] (W);
			\node at (-1,-1) {\tiny$f$};
		\draw[->] (W) to[out=90, in=180] (N);
			\node at (-1,1) {\tiny$g$};
		\draw[rounded corners] (-1.3, -1.3) rectangle (1.3, -0.7);
		\draw[rounded corners] (1.5, 0) -- (1.5, 1.5) -- (0.5, 1.5) -- (0.5, -0.5) -- (-1.5, -0.5) -- (-1.5, -1.5) -- (1.5, -1.5) -- (1.5, 0);
		\draw[rounded corners] (-1.7, 0) -- (-1.7, 1.7) -- (-0.3, 1.7) -- (-0.3, 0.5) -- (0.3, 0.5) -- (0.3, 1.7) -- (1.7, 1.7) -- (1.7, -1.7) -- (-1.7, -1.7) -- (-1.7, 0);
	}
	$};
	\draw[->] (s1a) to node[above, sloped]{\tiny$\alpha(g\circ f, h, \ell)$} (start);
	\draw[->] (start) to node[above, sloped]{\tiny$\alpha(f, g, \ell \circ h)$} (s2a);
	\draw[->] (s2b) to node[left, near end]{\tiny$\alpha(g, h, \ell)$} (s2a);
	\draw[->] (s1a) to node[right, near start]{\tiny$\alpha(f, g, h)$} (s1b);
	\draw[->] (s1b) to node[above]{\tiny$\alpha(f, h \circ g, \ell)$} (s2b);
	\draw[->] (s1a) to node[right]{\tiny$\varepsilon$} (s1c);
	\draw[->] (s1c) to node[right, near start]{\tiny$\alpha(f, g, h)$} (s1d);
	\draw[->] (s1b) to node[right]{\tiny$\varepsilon$} (s1d);
	\draw[->] (s2d) to node[left]{\tiny$\varepsilon$} (s2b);
	\draw[->] (s2c) to node[left]{\tiny$\varepsilon$} (s2a);
	\draw[->] (s2d) to node[left, near end]{\tiny$\alpha(g, h, \ell)$} (s2c);
	\draw[->] (s3b) to node[right, near end]{\tiny$\alpha(\ell, f, g)$} (s3a);
	\draw[->] (s3a) to node[above, sloped]{\tiny$\alpha(\ell, g \circ f, h)$} (s1c);
	\draw[->] (s3b) to node[left]{\tiny$\alpha(f \circ \ell, g, h)$} (end');
	\draw[->] (end') to node[left]{\tiny$\alpha(\ell, f, h \circ h)$} (s1d);
	\draw[->] (s2c) to node[above, sloped]{\tiny$\alpha(g, \ell\circ h, f)$} (s4a);
	\draw[->] (s4a) to node[left, near start]{\tiny$\alpha(h, \ell, f)$} (s4b);
	\draw[->] (s2d) to node[right]{\tiny$\alpha(h\circ g, \ell, f)$} (end);
	\draw[->] (end) to node[right]{\tiny$\alpha(g, h, f \circ \ell)$} (s4b);
	\draw[->] (s3d) to node[right, near end]{\tiny$\alpha(\ell, f, g)$} (s3c);
	\draw[->] (s3d) to node[right]{\tiny$\varepsilon$} (s3b);
	\draw[->] (s3c) to node[right]{\tiny$\varepsilon$} (s3a);
	\draw[->] (s4a) to node[left]{\tiny$\varepsilon$} (s4c);
	\draw[->] (s4c) to node[left, near start]{\tiny$\alpha(h, \ell, f)$} (s4d);
	\draw[->] (s4b) to node[left]{\tiny$\varepsilon$} (s4d);
	\draw[->] (s4d) to node[above]{\tiny$\alpha(h, f \circ \ell, g)$} (s3d);
	\draw[dotted, ->] (s4c) to node[below=10mm, pos=0.99]{\tiny$\alpha(\ell \circ h, f, g)$} (start');
	\draw[dotted, ->] (start') to node[below=10mm, pos=0.01]{\tiny$\alpha(h, \ell, g \circ f)$} (s3c);
	\draw[dotted, ->] (start) to node[right]{\tiny$\varepsilon$} (start');
	\draw[->] (end') to node[above]{\tiny$\varepsilon$} (end);
\end{tikzpicture}
\caption{The space of topographies for a $4$-partitioned loop, $\mathcal{T}(4)$.}
\label{fig:tensorpf}
\end{figure}

\begin{theorem}
\label{thm:Aetensor}
Assume that $(\mathcal{C}, \alpha)$ is a grading category admitting a looper $\varepsilon$. Fix the triple $(\mathcal{C}, \alpha, \varepsilon)$ and suppose that $A$ is a $\mathcal{C}$-graded algebra and $M$ and $N$ are $\mathcal{C}$-graded $(A, A)$-bimodules, interpreting $M$ as a right $\mathcal{C}$-graded $A\otimes_\Bbbk A^\op$-module and $N$ as a left $\mathcal{C}$-graded $A\otimes_\Bbbk A^\op$-module. Let $\Theta(\abs{m}, \abs{a\otimes a'}_{\mathcal{C} \times \mathcal{C}^\op}, \abs{n})$ denote the value associated to any path
	\[
	\abs{n} \circ \left( \left(\abs{a} \circ \abs{m} \right) \circ \abs{a'} \right) \to \left( \abs{a'} \circ \left( \abs{n} \circ \abs{a} \right) \right) \circ \abs{m}
	\]
	or, in terms of topographies,
	\[
	\tikz{
	\node(a) at (0,0) {$
	\Theta(\abs{m}, \abs{a\otimes a'}_{\mathcal{C} \times \mathcal{C}^\op}, \abs{n}): 
	\tikz[baseline={([yshift=-.5ex]current bounding box.center)}, scale=0.675]{
	\node(N) at (0,1) {$\bullet$};
	\node(E) at (1,0) {$\bullet$};
	\node(S) at (0,-1) {$\bullet$};
	\node(W) at (-1,0) {$\bullet$};
	\draw[->] (N) to[out=0, in=90] (E);
		\node at (1,1) {$\abs{a}$};
	\draw[->] (E) to[out=270, in=0] (S);
		\node at (1,-1) {$\abs{n}$};
	\draw[->] (S) to[out=180, in=270] (W);
		\node at (-1,-1) {$\abs{a'}$};
	\draw[->] (W) to[out=90, in=180] (N);
		\node at (-1,1) {$\abs{m}$};
	\draw[rounded corners] (-1.3, 0.7) rectangle (1.3, 1.3);
	\draw[rounded corners] (-1.5, 0) -- (-1.5, -1.5) -- (-0.5, -1.5) -- (-0.5, 0.5) -- (1.5, 0.5) -- (1.5, 1.5) -- (-1.5, 1.5) -- (-1.5, 0);
	\draw[rounded corners] (-1.7, 0) -- (-1.7, -1.7) -- (-0.3, -1.7) -- (-0.3, -0.5) -- (0.3, -0.5) -- (0.3, -1.7) -- (1.7, -1.7) -- (1.7, 1.7) -- (-1.7, 1.7) -- (-1.7, 0);
}$
};
	\node(b) at (7,0) {$
	\tikz[baseline={([yshift=-.5ex]current bounding box.center)}, scale=0.675]{
	\node(N) at (0,1) {$\bullet$};
	\node(E) at (1,0) {$\bullet$};
	\node(S) at (0,-1) {$\bullet$};
	\node(W) at (-1,0) {$\bullet$};
	\draw[->] (N) to[out=0, in=90] (E);
		\node at (1,1) {$\abs{a}$};
	\draw[->] (E) to[out=270, in=0] (S);
		\node at (1,-1) {$\abs{n}$};
	\draw[->] (S) to[out=180, in=270] (W);
		\node at (-1,-1) {$\abs{a'}$};
	\draw[->] (W) to[out=90, in=180] (N);
		\node at (-1,1) {$\abs{m}$};
	\draw[rounded corners] (0.7, -1.3) rectangle (1.3, 1.3);
	\draw[rounded corners] (1.5, 0) -- (1.5,1.5) -- (0.5, 1.5) -- (0.5, -0.5) -- (-1.5, -0.5) -- (-1.5,-1.5) -- (1.5, -1.5) -- (1.5,0);
	\draw[rounded corners] (1.7, 0) -- (1.7, 1.7) -- (-1.7, 1.7) -- (-1.7, 0.3) -- (-0.5, 0.3) -- (-0.5, -0.3) -- (-1.7, -0.3) -- (-1.7, -1.7) -- (1.7, -1.7) -- (1.7,0);
}$
};
\draw[->] (a) to (b);
}
	\]
Then, $M\otimes_{A \otimes_\Bbbk A^\op} N$ is a $\mathrm{Tr}(\mathcal{C})$-graded module, where
	\begin{equation}
	\label{eq:Aedefrel}
	M \otimes_{A \otimes_\Bbbk A^\op} N : = M \otimes N \big/ \left( \rho_R^e(m, a \otimes a') \otimes n - \Theta(\abs{m}, \abs{a\otimes a'}_{\mathcal{C} \times \mathcal{C}^\op}, \abs{n}) m \otimes \rho_L^e(a \otimes a', n)\right).
	\end{equation}
\end{theorem}

\begin{proof}
The result holds as long as the value $\Theta$ is well-defined, but this is an obvious consequence of Lemma \ref{lem:simpcon}.
\end{proof}

To be clear, we mean that $M \otimes_{A \otimes_\Bbbk A^\op} N$ is $\mathrm{Tr}(\mathcal{C})$-graded in the usual sense of being graded by a set.

\begin{corollary}
Assume that $(\mathcal{C}, \alpha)$ admits a looper and fix the triple $(\mathcal{C}, \alpha, \varepsilon)$. If $A$ is a $\mathcal{C}$-graded algebra and $M$ is a $\mathcal{C}$-graded $(A,A)$-bimodule, then the Hochschild complex
\[
\mathsf{HC}_\bullet (A, M) :=  M \otimes_{A \otimes_\Bbbk A^\op} \mathcal{B}(A)
\]
is a well-defined $(\mathbb{Z} \times \mathrm{Tr}(\mathcal{C}))$-graded chain complex.
\end{corollary}

Thus, Theorem \ref{thm1} follows, noting that $\mathrm{Tr}(BG) = G$ for any group $G$. Notice that Hochschild complexes for $\mathcal{C}$-graded algerbas depend on both of the choices of $\alpha$ and $\varepsilon$. It is interesting to ask how much the complex, or its homology, depends on these choices, though.

\begin{definition}
The \emph{Hochschild homology} of a $\mathcal{C}$-graded algebra $A$ with coefficients in the $\mathcal{C}$-graded $(A, A)$-bimodule $M$ is the homology of the Hochschild complex $\mathsf{HC}_\bullet(A, M)$, and is denoted $\mathsf{HH}_\bullet(A, M)$. When $M = A$, we write denoted the Hochschild homology by $\mathsf{HH}_\bullet(A)$. 
\end{definition}

\subsection{The Hochschild complex}
\label{ss:hochcomplex}

Using the results of the previous section, it is easy to give an explicit model for the Hochschild complex $\mathsf{HC}_\bullet$. Suppose $A$ is a $\mathcal{C}$-graded algebra and $M$ is a $\mathcal{C}$-graded $(A, A)$-bimodule. Interpreting $A^{\otimes n}$ as an $(A, A)$-bimodule, consider the chain complex $(C_n(A, M), b_n)$ where
\[
C_n(A, M) := M \otimes A^{\otimes n}
\]
via the monoidal product of $\mathrm{Mod}^\mathcal{C}$, and $b_n = \sum_{i=0}^n (-1)^i d_i$ where
\begin{align*}
	d_0 (m \otimes a_1 \otimes \cdots \otimes a_n) & = \rho_R(m, a_1) \otimes a_2 \otimes \cdots \otimes a_n
	\\
	d_i (m \otimes a_1 \otimes \cdots \otimes a_n) &= \alpha(\abs{a_{i-1}} \circ \cdots \circ \abs{a_1} \circ \abs{m}, \abs{a_i}, \abs{a_{i+1}}) m \otimes a_1 \otimes \cdots \otimes \mu(a_i, a_{i+1}) \otimes \cdots \otimes a_n
	\\
	d_n (m \otimes a_1 \otimes \cdots \otimes a_n) &= \Omega(\abs{m}, \abs{a_1}, \ldots, \abs{a_n}) \rho_L(a_n, m) \otimes a_1 \otimes \cdots \otimes a_{n-1}
\end{align*}
and
\[
\Omega(\abs{m}, \abs{a_1}, \ldots, \abs{a_n}) :=
\varepsilon(\abs{a_{n-1}} \circ \cdots \circ \abs{a_1} \circ \abs{m}, \abs{a_n}) 
\alpha(\abs{a_n}, \abs{m}, \abs{a_1})^{-1}
\prod_{i=1}^{n-2}
\alpha(\abs{a_n}, \abs{a_i} \circ \cdots \circ \abs{m}, \abs{a_{i+1}})^{-1}
\]
on homogeneous elements $a_1, \ldots, a_n\in A$ and $m\in M$. Notice that $b_n$ is only nontrivial when the gradings $\abs{a_n} \cdots \abs{a_1} \circ \abs{m}$ form a path, and that the summand $d_n$ is only nontrivial if $\abs{m} \circ \abs{a_n}$ is defined (that is, if this path forms a loop of length $n+1$). We leave the proof of the following to the reader.

\begin{proposition}
Fix the triple of a grading category admitting a looper $(\mathcal{C}, \alpha, \varepsilon)$. If $A$ is a $\mathcal{C}$-graded algebra and $M$ is a $\mathcal{C}$-graded $(A, A)$-bimodule, then
$(C_n(A, M), b_n)$ is a $(\mathrm{Tr}(\mathcal{C}) \times\mathbb{Z})$-graded chain complex, and $\mathsf{HC}_\bullet(A, M) \cong (C_n(A, M), b_n)$ as chain complexes.
\end{proposition}

This description makes the usual elementary computations easy. In particular, we can define the \emph{module of coinvaraints} in the usual way, declaring
\[
M^\mathcal{C}_A := \mathsf{HH}_0(A, M) = M / \{\rho_R(m, a) - \varepsilon(\abs{m}, \abs{a}) \rho_L(a, m) : a \in A, m\in M\}.
\]
Similarly, when $M = A$, we can define the \emph{commutator sub-module} $[A, A]_\varepsilon^\mathcal{C}$ generated by $[a_1, a_2]_\varepsilon := \mu(a_1, a_2) - \varepsilon(\abs{a_1}, \abs{a_2}) \mu(a_2, a_1)$. It follows that
\[
\mathsf{HH}_0(A) = A / [A, A]_\varepsilon^\mathcal{C}.
\]

Typically, the next easiest computation (see, for example, Section 1.1.9 of \cite{MR1600246}) requires restricting to the case of commutative algebras. Define a \emph{commutative $(\mathcal{C}, \alpha, \varepsilon)$-graded algebra} $A$ to be a $\mathcal{C}$-graded algebra which satisfies $\mu(x, y) = \varepsilon(\abs{x}, \abs{y}) \mu(y, x)$. This added assumption is very restrictive; in particular, it implies that the gradings of $A$ are supported entirely in some subset of $\mathrm{Mor}(\mathcal{C})$ for which $\abs{y} \circ \abs{x} = \abs{x} \circ \abs{y}$ for each $x, y\in A$. Thus, for each homogeneous $a \in A$, $\abs{a} \in Z(\mathrm{End}_\mathcal{C}(X))$ for some $X \in \mathrm{Ob}(\mathcal{C})$, where $Z$ denotes the center. Notice that this implies that $A$ has exactly one unit in the sense of Definition \ref{def:algebras}. In summary, a commutative $\mathcal{C}$-graded algebra is a quasi-associative algebra whose gradings are supported in a commutative sub-monoid of $\mathcal{C}$.

Now, the ``commutativity up to $\varepsilon$'' assumption $\mu(x, y) = \varepsilon(\abs{x}, \abs{y}) \mu(y, x)$ implies that $b_2: A \otimes A \to A$ is the zero map, therefore $\mathsf{HH}_1(A)$ is the quotient of $A \otimes A$ by the relation
\begin{equation}
\label{eq:relh1}
\mu(a, b) \otimes c - \alpha(\abs{a}, \abs{b}, \abs{c}) a \otimes \mu(b,c) + \varepsilon(\abs{b} \circ \abs{a}, \abs{c}) \alpha(\abs{c}, \abs{a}, \abs{b})^{-1} \mu(c, a) \otimes b
=0.
\end{equation}
Notice that $\mathsf{HH}_1(A)$ inherits the structure of an $A$-module, defining $\rho_L(x, a \otimes b) := \alpha(\abs{x}, \abs{a}, \abs{b})^{-1} \mu(x, a) \otimes b$.

\begin{lemma}
\label{lem:comalg}
Assume $A$ is a commutative $(\mathcal{C}, \alpha, \varepsilon)$-graded algebra. For any homogeneous $a, b\in A$, the following hold for elements of $\mathsf{HH}_1(A)$:
\begin{enumerate}[label=(\roman*)]
	\item $a \otimes 1 = 0$;
	\item $a \otimes b + \varepsilon(\abs{a}, \abs{b}) \, b \otimes a = \mathcal{L}(\mathrm{Id}, \abs{b} \circ \abs{a})^{-1} \, 1 \otimes \mu(a, b)$.
\end{enumerate}
\end{lemma}

\begin{proof}
To explain notation, since $A$ is graded by a commutative monoid, there is a unique object in the grading category, thus the unique unit of $A$ is denoted unambiguously by 1. Similarly, there is a unique identity morphism, denoted by $\mathrm{Id}$. By way of (\ref{eq:relh1}), we compute
\begin{align*}
	a \otimes 1 &= \mathcal{R}(\abs{a}, \mathrm{Id})^{-1} \mu(a, 1) \otimes 1
	\\
	&= \mathcal{R}(\abs{a}, \mathrm{Id})^{-1} \left(
	\alpha(\abs{a}, \mathrm{Id}, \mathrm{Id}) a \otimes \mu(1,1)
	-
	\varepsilon(\abs{a}, \mathrm{Id}) \alpha(\mathrm{Id}, \abs{a}, \mathrm{Id})^{-1} \mu(1, a) \otimes 1
	\right)
	\\
	&= a \otimes 1 - a \otimes 1 = 0.
\end{align*}
The third equality follows from parts (i) and (iv) of Lemma \ref{lem:assocomps} and Lemma \ref{lem:eploops}. Similarly, applying (\ref{eq:relh1}) to the expression $\mu(1, a) \otimes b$ yields
\[
\mu(1, a) \otimes b - \alpha(\mathrm{Id}, \abs{a}, \abs{b}) \, 1 \otimes \mu(a,b) + \varepsilon(\abs{a}, \abs{b}) \alpha(\abs{b}, \mathrm{Id}, \abs{a})^{-1} \, \mu(b, 1) \otimes a = 0.
\]
Then (ii) follows from the usual arguments via Lemma \ref{lem:assocomps}.
\end{proof}

Let $\Omega_A^1$ denote the $(A, A)$-bimodule of K\"ahler differentials, which is generated by symbols $da$ for each $a\in A$ where $d: A \to A$ is a $\mathcal{C}$-grading preserving $\Bbbk$-linear map satisfying
\begin{equation}
\label{eq:relkahler}
d\mu(a, b) = \mu(a, db) + \mu(da, b).
\end{equation}
It is a $\mathcal{C}$-graded $(A, A)$-bimodule in a trivial sense, setting $\rho_L(a, db) = \mu(a, db)$ and $\rho_R(da, b) = \mu(da, b)$. This bimodule is definable for any $\mathcal{C}$-graded algebra, and takes on the graded structure of $A$; in particular, when $A$ is commutative, then $\Omega_A^1$ has gradings supported entirely in a commutative monoid. Notice that for $\mathcal{M}$ a commutative monoid, $\mathrm{Tr}(\mathcal{M}) = \mathcal{M}$.

\begin{proposition}
Assume $\mathcal{M}$ is a commutative sub-monoid of $(\mathcal{C}, \alpha, \varepsilon)$. If $A$ is a commutative $\mathcal{M}$-graded algebra, then $\mathsf{HH}_1(A) \cong \Omega_A^1$ as $\mathcal{M}$-graded $A$-modules.
\end{proposition}

\begin{proof}
Notice that $\Omega_A^1$ has the structure of an $A$-module, as the usual left- and right-actions differ up to unit of $\Bbbk$ when $A$ is commutative. Define a map $\mathsf{HH}_1(A) \to \Omega_A^1$ by $a \otimes b \mapsto \mu(a, db)$. This map is well-defined: the value
\[
\mu(a, b) \otimes c - \alpha(\abs{a}, \abs{b}, \abs{c}) a \otimes \mu(b,c) + \varepsilon(\abs{b} \circ \abs{a}, \abs{c}) \alpha(\abs{c}, \abs{a}, \abs{b})^{-1} \mu(c, a) \otimes b
\]
is sent to
\[
\mu(\mu(a, b), dc) - \alpha(\abs{a}, \abs{b}, \abs{c}) \mu(a, d\mu(b,c)) + \varepsilon(\abs{b} \circ \abs{a}, \abs{c}) \alpha(\abs{c}, \abs{a}, \abs{b})^{-1} \mu(\mu(c, a), db).
\]
Well-definedness follows since 
\begin{align*}
\mu(\mu(a, b), dc) + \mu(\mu(a, db), c) &= \alpha(\abs{a}, \abs{b}, \abs{c}) \mu(a, d\mu(b,c)) 
\\
&= \alpha(\abs{a}, \abs{b}, \abs{c}) \left(
\mu(a,\mu (b, dc) ) + \mu(a, \mu(db, c))
\right).
\end{align*}
It is easy to check that this map is an $A$-module homomorphism:
\[
\rho_L(x, a\otimes b) = \alpha(\abs{x}, \abs{a}, \abs{b})^{-1} \mu(x, a) \otimes b \mapsto \alpha(\abs{x}, \abs{a}, \abs{b})^{-1} \mu(\mu(x,a), db) = \mu(x, \mu(a, db)).
\]
Conversely, define a map $\Omega_A^1 \to \mathsf{HH}_1(A)$ by $da \mapsto \mathcal{L}(\mathrm{Id}, \abs{a})^{-1} 1 \otimes a$. The map is well defined since the relation (\ref{eq:relkahler}) is sent to (ii) of Lemma \ref{lem:comalg}. It is an $A$-module homomorphism, sending $\mu(a, db) \mapsto a \otimes b$ since
\[
\alpha(\abs{a}, \mathrm{Id}, \abs{b})^{-1} \mathcal{L}(\mathrm{Id}, \abs{b})^{-1} \mathcal{R}(\abs{a}, \mathrm{Id}) = 1.
\]
These two maps are inverse to one another, establishing the isomorphism.
\end{proof}

\begin{remark}
The main examples in this paper of non-associative but $\mathcal{C}$-graded algebras are the unified versions of Khovanov's arc aglebras, denoted $H^n$ for $n\ge 1$. The arc algebras are noncommutative, and the grading category required has many objects. Therefore, we still need an example of a $\mathcal{M}$-graded commutative algebra in order for the result on $\mathsf{HH}_1(A)$ to be interesting. One can generate non-associative, commutative algebras by considering the \emph{Jordan algebras}, especially the Jordan algebra associated to some underlying non-commutative associative algebra, with mutliplication
\[
x \bullet y := \cfrac{xy+yx}{2}.
\]
However, the failure of associativity of this family is not easily controlled. The simplest example of a non-associative Jordan algebra is the one associated to the quaternions $\mathbb{H}$. This example does not work since, for example, the values $(\mathbf{i} \bullet \mathbf{i}) \bullet \mathbf{j} = -\mathbf{j}$ and $\mathbf{i} \bullet (\mathbf{i} \bullet \mathbf{j}) = 0$ do not differ up to unit. We will not investigate more examples in this paper.
\end{remark}

\subsection{Generalities on $\mathrm{Mod}^{\mathrm{Tr}(\mathcal{C})}$ and Hochschild shadows}
\label{ss:shadows}

Suppose $\mathcal{C}$ is a small category. Recall that a $\mathrm{Tr}(\mathcal{C})$-graded $\Bbbk$-module is simply a $\Bbbk$-module which admits a grading by the set $\mathrm{Tr}(\mathcal{C})$. Let $\mathrm{Mod}^{\mathrm{Tr}(\mathcal{C})}$ denote the category of $\mathrm{Tr}(\mathcal{C})$-graded $\Bbbk$-modules whose morphisms are grading preserving $\Bbbk$-linear maps, and let $\mathrm{End}(\mathcal{C}) = \coprod_{X \in \mathrm{Ob}(\mathcal{C})} \mathrm{End}_\mathcal{C}(X)$. There is an essentially surjective functor
\[
(-)^{\circ}: \mathrm{Mod}^\mathcal{C} \to \mathrm{Mod}^{\mathrm{Tr}(\mathcal{C})}
\]
which sends the $\mathcal{C}$-graded $\Bbbk$-module $M = \bigoplus_{g \in \mathrm{Mor}(\mathcal{C})} M_g$ to the $\Bbbk$-module $M^\circ = \bigoplus_{\gamma \in \mathrm{Tr}(\mathcal{C})} M_\gamma$ with
\[
M_\gamma := \bigoplus_{\substack{g\in \mathrm{End}(\mathcal{C}) \\ \mathrm{tr}_X(g) = \gamma}} M_g.
\]
In the process, summands $M_g$ vanish when $g$ does not represent an element of $\mathrm{Tr}(\mathcal{C})$.

It is interesting to study $\mathrm{Mod}^{\mathrm{Tr}(\mathcal{C})}$ for grading categories with fixed loopers $(\mathcal{C}, \alpha, \varepsilon)$. In particular, we noticed in Section \ref{s:algebra} that the choice of $\alpha$ is the main ingredient to give $\mathrm{Mod}^\mathcal{C}$ the structure of a monoidal category. What does the choice of $\varepsilon$ do? It is tempting to claim that this choice provides a witness for the swap map in a symmetric monoidal category, but this is not the case; for example, notice that the $\alpha$-$\varepsilon$ coherence relation of Definition \ref{def:looper} is \emph{not} quite the hexagon relation for a symmetric monoidal category. We will show, among other properties, that the choice of $\varepsilon$ is not only necessary to define Hochschild homology for quasi-associative algebras, but that the choice makes $\mathsf{HH}_0$ into a shadow with target category $\mathrm{Mod}^{\mathrm{Tr}(\mathcal{C})}$.

Start by fixing the grading category $(\mathcal{C}, \alpha)$ and let $\mathrm{Bim}(A, B)^\mathcal{C}$ denote the category of $\mathcal{C}$-graded $(A, B)$-bimodules, whose morphisms are $\mathcal{C}$-grading preserving $\Bbbk$-linear maps. Notice that $(\mathrm{Bim}(A, A)^\mathcal{C}, \otimes_A)$ is also a (graded) monoidal category. As for $\mathrm{Mod}^\mathcal{C}$, we can pick the coherence isomorphism 
\[
(M \otimes_A N) \otimes_A L \xrightarrow{\sim} M \otimes_A (N \otimes_A L)
\]
induced by the associator via $(m \otimes n) \otimes \ell \mapsto \alpha(\abs{m}, \abs{n} \abs{\ell}) \, m \otimes (n \otimes \ell)$. The unit object of this monoidal category is given by the algebra $A$.

\begin{proposition}
\label{prop:AAmonoidal}
Assume that the unitors of $\mathrm{Mod}^\mathcal{C}$ are typical. Then $(\mathrm{Bim}(A, A)^\mathcal{C}, \otimes_A)$ is a monoidal category with coherence isomorphism $\alpha$ and trivial unitors. That is, $A$ is the unit object, and the unitors are 
\[
\begin{tikzcd}[column sep = small, row sep = tiny]
A \otimes_A M \arrow[r] & M
\\
a \otimes m \arrow[r, mapsto] & \rho_L(a, m)
\end{tikzcd}
\qquad \text{and} \qquad
\begin{tikzcd}[column sep = small, row sep = tiny]
M \otimes_A A \arrow[r] & M
\\
m \otimes a \arrow[r, mapsto] & \rho_R(m,a)
\end{tikzcd}
\]
\end{proposition}

We highlight this observation to point out that the category of $\mathcal{C}$-graded $A$-bimodules requires no additional witnesses when compared to the classical setting (unlike $\mathrm{Mod}^\mathcal{C}$). The proposition follows from the following tautological lemma.

\begin{lemma}
\label{lem:tautological}
Let $A$ be a $\mathcal{C}$-graded algebra, and denote its units by $\{1_X^A \in A_{\mathrm{Id}_X} : X \in \mathrm{Ob}(\mathcal{C})\}$. Then
\[
\mathrm{Span}_\Bbbk\{1_X^A \in A_{\mathrm{Id}_X} : X \in \mathrm{Ob}(\mathcal{C})\} \cong I_\mathcal{C}
\]
as $\mathcal{C}$-graded algebras under the canonical identification $1_X^A \mapsto 1_X \in \Bbbk_{\mathrm{Id}_X}$.
\end{lemma}

\begin{proof}
This is immediate by definition. Notice that $I_\mathcal{C}$ is a $\mathcal{C}$-graded algebra with multiplication induced by $\Bbbk$ when $\abs{x} = \abs{y}$, fixing $\mu_{I_\mathcal{C}}(x, y) = 0$ whenever $\abs{x} \not= \abs{y}$.
\end{proof}

\begin{proof}[Proof of Proposition \ref{prop:AAmonoidal}]
Pick homogeneous elements $a\in A$ and $m\in M$; to reduce notation, we will write $\mathrm{dom}(\abs{a}) = X$. Any element $a \otimes m$ of $A \otimes_A M$ can be rewritten as 
\[
\mathcal{L}(\mathrm{Id}_X, \abs{a})^{-1} \mu(1_X, a) \otimes m
\sim
\mathcal{L}(\mathrm{Id}_X, \abs{a})^{-1} \alpha(\mathrm{Id}_X, \abs{a}, \abs{m}) 1_X \otimes \rho_L(a, m)
\]
by equation (\ref{eq:Acoequal}). By assumption, 
\begin{align*}
\mathcal{L}(\mathrm{Id}_X, \abs{a})^{-1} \alpha(\mathrm{Id}_X, \abs{a}, \abs{m}) &= \alpha(\mathrm{Id}_X, \mathrm{Id}_X, \abs{a}) \alpha(\mathrm{Id}_X, \abs{a}, \abs{m})
\\ &= \alpha(\mathrm{Id}_X, \mathrm{Id}_X, \abs{m} \circ \abs{a})
\\ &= \mathcal{L}(\mathrm{Id}_X, \abs{m} \circ \abs{a})^{-1}
\end{align*}
where the second equality comes from (ii) of Lemma \ref{lem:assocomps}. By Lemma \ref{lem:tautological}, this establishes an isomorphism
\[
A \otimes_A M \xrightarrow{\mathcal{L}^{-1}} I_\mathcal{C} \otimes M \xrightarrow{\mathcal{L}} M
\]
by $a \otimes m \mapsto \rho_L(a, m)$, as desired. The same argument can be used to show that the right unitor has the stated structure. Then the triangle diagram
\[
\begin{tikzcd}[column sep = small]
(M \otimes_A A) \otimes_A N \arrow[dr] \arrow[rr, "\alpha"] & & M \otimes_A (A \otimes_A N) \arrow[dl]
\\
& M \otimes_A N &
\end{tikzcd}
\]
commutes exactly by the definition of $M \otimes_A N$, as $\rho_R^M(m, a) \otimes n \sim \alpha(\abs{m}, \abs{a}, \abs{n}) m \otimes \rho_L^N(a, n)$.
\end{proof}

Now, assume that $(\mathcal{C}, \alpha)$ admits a looper, and fix the triple $(\mathcal{C}, \alpha, \varepsilon)$. Let $A^e := A \otimes_\Bbbk A^\op$; next, we will study
\[
- \otimes_{A^e} - : \mathrm{Bim}(A, A)^\mathcal{C} \times \mathrm{Bim}(A, A)^\mathcal{C} \to \mathrm{Mod}^{\mathrm{Tr}(\mathcal{C})}.
\]
It does not take much work beyond Theorem \ref{thm:Aetensor} to prove that this assignment extends to a functor. One may interpret Definition \ref{def:looper} as establishing some properties of this functor---for example, it has a natural symmetry:

\begin{proposition}
\label{prop:Aesymmetry}
Assuming unitors are typical, there is an isomorphism 
\[
\varepsilon_{M, N}: M \otimes_{A^e} N \to N \otimes_{A^e} M
\]
given by $m \otimes n \mapsto \varepsilon(\abs{m}, \abs{n}) \, n \otimes m$.
\end{proposition}

\begin{proof}
Note that the defining relation of $M\otimes_{A^e} N$ can be written 
\[
\rho_L(a', \rho_R(m,a)) \otimes n = \Theta(\abs{m}, \abs{a \otimes a'}_{\mathcal{C} \times \mathcal{C}^\op}, \abs{n}) m \otimes \rho_R(\rho_L(a, n), a').
\]
Under $\varepsilon_{M, N}$, this is sent to
\[
\varepsilon(\abs{a} \circ \abs{m} \circ \abs{a'}, \abs{n}) n \otimes \rho_L(a', \rho_R(m, a))
=
\Theta(\abs{m}, \abs{a \otimes a'}_{\mathcal{C} \times \mathcal{C}^\op}, \abs{n})
\varepsilon(\abs{m}, \abs{a'} \circ \abs{n} \circ \abs{a})
\rho_L(a, \rho_R(n,a')) \otimes m
\]
or, equivalently,
\[
n \otimes \rho_R(\rho_L(a', m), a)
=
(\star) \,
\rho_L(a, \rho_R(n, a')) \otimes m
\]
for 
\[
(\star) = \varepsilon(\abs{n}, \abs{a} \circ \abs{m} \circ \abs{a'}) \alpha(\abs{a'}, \abs{m}, \abs{a})
\Theta(\abs{m}, \abs{a \otimes a'}_{\mathcal{C} \times \mathcal{C}^\op}, \abs{n})
\varepsilon(\abs{m}, \abs{a'} \circ \abs{n} \circ \abs{a})
\alpha(\abs{a}, \abs{n}, \abs{a'}).
\]
The value $(\star)$ describes a path from the topography $\abs{m} \circ ((\abs{a'} \circ \abs{n}) \circ \abs{a})$ to $((\abs{a'} \circ \abs{m}) \circ \abs{a}) \circ \abs{n}$. By Lemma \ref{lem:simpcon}, this is equal to $\Theta(\abs{n}, \abs{a' \otimes a}_{\mathcal{C} \times \mathcal{C}^\op}, \abs{m})$, thus $\varepsilon_{M, N}$ is well-defined. Similarly $\varepsilon_{N, M}$ is well-defined, and these maps are inverse to each other as the diagram
\[
\begin{tikzcd}
M \otimes_{A^e} N \arrow[rr, equals] \arrow[dr, "\varepsilon_{M, N}"'] & & M \otimes_{A^e} N
\\
& N \otimes_{A^e} M \arrow[ur, "\varepsilon_{N, M}"']&
\end{tikzcd}
\]
commutes by (i) of Definition \ref{def:looper}.
\end{proof}

Secondly, $\alpha$ extends to yet another type of coherence:

\begin{proposition}
\label{prop:Aeassoc}
Assuming unitors are typical, there is an isomorphism
\[
\alpha_{M, N, L}: (M \otimes_A N) \otimes_{A^e} L \to M \otimes_{A^e} (N \otimes_A L)
\]
given by $(m \otimes n) \otimes \ell \mapsto \alpha(\abs{m}, \abs{n}, \abs{\ell}) \, m \otimes (n \otimes \ell)$.
\end{proposition}

\begin{proof}
As before, the main difficulty is in proving that this map is well-defined. If so, then the map $m \otimes (n \otimes \ell) \mapsto \alpha(\abs{m}, \abs{n}, \abs{\ell})^{-1} (m \otimes n) \otimes \ell$ is similarly well-defined, and provides an inverse, concluding the proof. To prove that $\alpha_{M, N, L}$ is well-defined, we show that the defining relation
\begin{equation}
\label{eq:proofAeassoc}
\rho_L(a', \rho_R(m\otimes n, a)) \otimes \ell = \underbrace{\Theta(\abs{m \otimes n}_\mathcal{C}, \abs{a \otimes a'}_{\mathcal{C} \times \mathcal{C}^\op}, \abs{\ell}_\mathcal{C})}_{\Theta_1} (m \otimes n) \otimes \rho_R(\rho_L(a, \ell), a')
\end{equation}
(compare with (\ref{eq:Aedefrel})) is preserved. The left-hand side of (\ref{eq:proofAeassoc}) can be rewritten as
\[
\alpha(\abs{m}, \abs{n}, \abs{a}) \alpha(\abs{a'}, \abs{m}, \abs{a} \circ \abs{n})^{-1} \left( \rho_L(a', m) \otimes \rho_R(n, a)\right) \otimes \ell
\]
which is sent by $\alpha_{M, N, L}$ to 
\[
\alpha(\abs{m}, \abs{n}, \abs{a}) \alpha(\abs{a'}, \abs{m}, \abs{a} \circ \abs{n})^{-1} \alpha(\abs{m} \circ \abs{a'}, \abs{a} \circ \abs{n}, \abs{\ell})
\rho_L(a', m) \otimes \left(\rho_R(n, a) \otimes \ell\right).
\]
We rewrite this as
\[
\underbrace{\alpha(\abs{m}, \abs{n}, \abs{a}) \alpha(\abs{a'}, \abs{m}, \abs{a} \circ \abs{n})^{-1} \alpha(\abs{m} \circ \abs{a'}, \abs{a} \circ \abs{n}, \abs{\ell})}_{A_1} \mathcal{R}(\abs{m}, \mathrm{Id}_X)^{-1}
\rho_L(a', \rho_R(m, 1_X)) \otimes \left( \rho_R(n, a) \otimes \ell \right)
\]
fixing $X = \mathrm{codom}(\abs{m}) = \mathrm{dom}(\abs{n})$.

Similarly, the right-hand side of (\ref{eq:proofAeassoc}), omitting the $\Theta$-term, is sent to
\[
\alpha(\abs{m}, \abs{n}, \abs{a'} \circ \abs{\ell} \circ \abs{a})
m \otimes \left(n \otimes \rho_R(\rho_L(a, \ell), a')\right)
\]
by $\alpha_{M, N, L}$. Then, since the terms within the larger parentheses occur in $N \otimes_A L$, we can rewrite this value as
\[
\underbrace{
\alpha(\abs{m}, \abs{n}, \abs{a'} \circ \abs{\ell} \circ \abs{a})
\alpha(\abs{a}, \abs{\ell}, \abs{a'})
\alpha(\abs{n}, \abs{a}, \abs{a'} \circ \abs{\ell})^{-1}
\alpha(\abs{a} \circ \abs{n}, \abs{\ell}, \abs{a'})^{-1}}_{A_2}
m \otimes \rho_R(\rho_R(n,a) \otimes \ell, a')
\]
which we then finally rewrite as 
\[
A_2 \, \mathcal{L}(\mathrm{Id}_X, \abs{\ell} \circ \abs{a} \circ \abs{n})^{-1}
m \otimes \left(\rho_L(1_X, n) \otimes \rho_R(\rho_L(a, \ell), a') \right).
\]

Now, the defining relation of $M \otimes_{A^e} (N \otimes_A L)$ tells us that 
\[
\rho_L(a', \rho_R(m, 1_X)) \otimes \left( \rho_R(n, a) \otimes \ell \right)
=
\underbrace{\Theta(\abs{m}_\mathcal{C}, \abs{1_X \otimes a'}_{\mathcal{C} \times \mathcal{C}^\op}, \abs{\rho_R(n, a)\otimes \ell}_\mathcal{C})}_{\Theta_2}
m \otimes \rho_R(\rho_L(1_X, \rho_R(n, a) \otimes \ell), a').
\]
Thus, the result follows as long as
\[
\Theta_2 = A_1^{-1} \, \mathcal{R}(\abs{m}, \mathrm{Id}_X) \, \Theta_1 \, A_2 \, \mathcal{L}(\mathrm{Id}_X, \abs{\ell} \circ \abs{a} \circ \abs{n})^{-1}.
\]
This is true and follows as a routine, though tedious, computation. After expanding the $\Theta_1$ and $\Theta_2$ terms as the shortest path between relevant topographies in Figure \ref{fig:tensorpf}, the $\varepsilon$ terms cancel immediately, and three terms cancel from one application of (iv) from Lemma \ref{lem:assocomps}. Then, the result follows from three applications of the 3-cocycle relation for $\alpha$; we leave this verification to the reader.
\end{proof}

Propositions \ref{prop:Aesymmetry} and \ref{prop:Aeassoc} immediately compile to say that the following diagram commutes, thanks to (ii) of Definition \ref{def:looper}.
\[
\begin{tikzcd}[column sep=large]
& M\otimes_{A^e} (N \otimes_{A} L) \arrow[r, "\varepsilon_{M, N \otimes_A L}"] & (N \otimes_A L) \otimes_{A^e} M \arrow[dr, "\alpha_{N, L, M}"] & 
\\
(M\otimes_A N) \otimes_{A^e} L \arrow[ur, "\alpha_{M, N, L}"] & & & N \otimes_{A^e} (L \otimes_A M) \arrow[dl, "\varepsilon_{N, L\otimes_A M}"]
\\
& L \otimes_{A^e} (M \otimes_A N) \arrow[ul, "\varepsilon_{L, M \otimes_A N}"] & (L \otimes_A M) \otimes_{A^e} N \arrow[l, "\alpha_{L, M, N}"] &
\end{tikzcd}
\]

As a penultimate observation, we note some relationships between $- \otimes_{A^e} -$ and $(- )^\circ \circ (- \otimes_{A} -)$.

\begin{proposition}
Assume that the unitors of $\mathrm{Mod}^\mathcal{C}$ are typical. If $M$ and $N$ are $\mathcal{C}$-graded $(A, A)$-bimodules, there is a canonical surjection
\[
(M \otimes_A N)^\circ \twoheadrightarrow M \otimes_{A^e} N.
\]
\end{proposition}

\begin{proof}
For $m \otimes n \in (M \otimes_A N)^\circ$, fix $\mathrm{dom}(\abs{m}) = \mathrm{codom}(\abs{n}) = X$. Notice that the identity
\[
\rho_R^e(m, a \otimes a') \otimes n = \Theta(\abs{m}, \abs{a \otimes a'}, \abs{n}) m \otimes \rho_L^e(a \otimes a', n)
\]
can be rewritten in the case that $a' = 1_X$ as 
\[
\mathcal{L}(\mathrm{Id}_X, \abs{a} \circ \abs{m})
 \rho_R(m,a) \otimes n
=
\Theta(\abs{m}, \abs{a \otimes 1_X}_{\mathcal{C} \times \mathcal{C}^\op}, \abs{n}) 
\mathcal{R}(\abs{n} \circ \abs{a}, \mathrm{Id}_X)
m \otimes \rho_L(a,n).
\]
Thus the claim follows as long as 
\[
\Theta(\abs{m}, \abs{a \otimes 1_X}_{\mathcal{C} \times \mathcal{C}^\op}, \abs{n}) 
\mathcal{R}(\abs{n} \circ \abs{a}, \mathrm{Id}_X)
\mathcal{L}(\mathrm{Id}_X, \abs{a} \circ \abs{m})^{-1}
=
\alpha(\abs{m}, \abs{a}, \abs{n}).
\]
To see this, we expand $\Theta(\abs{m}, \abs{a \otimes 1_X}_{\mathcal{C} \times \mathcal{C}^\op}, \abs{n})$ as
\[
\alpha(\mathrm{Id}_X, \abs{a} \circ \abs{m}, \abs{n}) \varepsilon(\mathrm{Id}_X, \abs{n} \circ \abs{a} \circ \abs{m}) \alpha(\abs{m}, \abs{a}, \abs{n}) \alpha(\abs{m}, \abs{n} \circ \abs{a}, \mathrm{Id}_X)
\]
using the shortest path between the relevant vertices in Figure \ref{fig:tensorpf}. Since $\abs{n} \circ \abs{a} \circ \abs{m}$ is a loop by assumption, we can expand the $\varepsilon(\mathrm{Id}_X, \abs{n} \circ \abs{a} \circ \abs{m})$ term via Lemma \ref{lem:eploops}. Then, writing $\mathcal{L}$ and $\mathcal{R}$ in terms of $\alpha$, it is apparent that after one application of each of (ii) and (iii) of Lemma \ref{lem:assocomps} the only surviving term is $\alpha(\abs{m}, \abs{a}, \abs{n})$.
\end{proof}

We also have a more familiar relationship between $-\otimes_{A^e}-$ and the module of coinvaraints.

\begin{proposition}
For $A$ a $\mathcal{C}$-graded algebra and $M$ a $\mathcal{C}$-graded $(A, A)$-bimodule,
\[
M \otimes_{A^e} A \cong M_A^\mathcal{C} = \mathsf{HH}_0(A, M).
\]
\end{proposition}

\begin{proof}
The first step in this proof is similar in spirit to the proof of Proposition \ref{prop:AAmonoidal}; we claim that any element $m \otimes a$ of $M \otimes_{A^e} A$ can be written in the form $\mathcal{R}(\abs{a} \circ \abs{m}, \mathrm{Id}_X)^{-1} \rho_R(m, a) \otimes 1_X$, working with the assumption that gradings involved take the form $\begin{tikzcd} X \arrow[r, bend left, "\abs{m}"] & Y \arrow[l, bend left, "\abs{a}"] \end{tikzcd}$. Indeed, 
\begin{align*}
	m \otimes a &= \left(\mathcal{R}(\abs{a}, \mathrm{Id}_X)^{-1}\right)^{2} m \otimes \rho_R(\rho_L(a, 1_X), 1_X) 
		\\
		&= \Theta(\abs{m}, \abs{a \otimes 1_X}_{\mathcal{C} \times \mathcal{C}^\op}, \mathrm{Id}_X)^{-1}
			\left(\mathcal{R}(\abs{a}, \mathrm{Id}_X)^{-1}\right)^{2}
			\rho_L(1_X, \rho_R(m, a)) \otimes 1_X
		\\
		&= \Theta(\abs{m}, \abs{a \otimes 1_X}_{\mathcal{C} \times \mathcal{C}^\op}, \mathrm{Id}_X)^{-1}
			\left(\mathcal{R}(\abs{a}, \mathrm{Id}_X)^{-1}\right)^{2}
			\mathcal{L}(\mathrm{Id}_X, \abs{a} \circ \abs{m})
			\rho_R(m, a) \otimes 1_X.
\end{align*}
Then, it is an easy exercise to verify using Lemmas \ref{lem:assocomps} and \ref{lem:eploops} that 
\[
\Theta(\abs{m}, \abs{a \otimes 1_X}_{\mathcal{C} \times \mathcal{C}^\op}, \mathrm{Id}_X)^{-1}
			\left(\mathcal{R}(\abs{a}, \mathrm{Id}_X)^{-1}\right)^{2}
			\mathcal{L}(\mathrm{Id}_X, \abs{a} \circ \abs{m})
=
\mathcal{R}(\abs{a} \circ \abs{m}, \mathrm{Id}_X)^{-1}.
\]
Now, let's rewrite the defining relation (\ref{eq:Aedefrel}) of $M\otimes_{A^e} A$ as
\[
\rho_L(a_2, \rho_R(m, a_1)) \otimes a = \Theta(\abs{a}, \abs{a_1 \otimes a_2}_{\mathcal{C} \times \mathcal{C}^\op}, \abs{a}) m \otimes \rho_R(\rho_L(a_1, a), a_2)
\]
using this identification. Omitting the input values for the $\mathcal{R}$ terms (written $\mathcal{R}_1$ and $\mathcal{R}_2$ respectively), the left-hand side can be rewritten as
\[
\mathcal{R}_1^{-1}
\underbrace{\alpha(\abs{a_2}, \abs{a_1} \circ \abs{m}, \abs{a})
\alpha(\abs{m}, \abs{a_1}, \abs{a})}_{(*)}
\rho_L(a_2, \rho_R(m, \mu(a_1, a))) \otimes 1_X
\]
and the right-hand side can be rewritten as
\[
\underbrace{\Theta(\abs{m}, \abs{a_1 \otimes a_2}_{\mathcal{C} \times \mathcal{C}^\op}, \abs{a})}_{(*)}
\mathcal{R}_2^{-1}
\underbrace{\alpha(\abs{m}, \abs{a} \circ \abs{a_1}, \abs{a_2})^{-1}}_{(*)}
\rho_R(\rho_R(m, \mu(a_1, a)), a_2) \otimes 1_Y.
\]
Expanding the $\Theta$ term using the shortest path between relevant topographies in Figure \ref{fig:tensorpf}, we see that the terms marked by $(*)$ cancel up to the value $\varepsilon(\abs{a_2}, \abs{a} \circ \abs{a_1} \circ \abs{m})$. To conclude the proof, use Lemma \ref{lem:tautological} to send values $\mathcal{R}^{-1} \rho_R(m, a) \otimes 1_X$ isomorphically to $\rho_R(m, a)$. This establishes an isomorphism
\[
M \otimes_{A^e} A \cong M / \rho_L(a', m') \sim \varepsilon(\abs{a'}, \abs{m'}) \rho_R(m', a')
\]
taking $a':= a_2$ and $m':= \rho_R(m, \mu(a_1, a_2))$. This is isomorphic to $M_A^\mathcal{C}$ by $\varepsilon_{A, M}$, concluding the proof.
\end{proof}

Finally, we conclude with the observation that $\varepsilon$ endows $\mathsf{HH}_0$ with the structure of a symmetric shadow on $\mathrm{Mod}^{\mathrm{Tr}(\mathcal{C})}$. It is probably also true that $\varepsilon$ provides a more general Hochschild shadow (see, for example, Proposition 3.1.5 of \cite{MR4644963} or Proposition 6.2 of \cite{MR2928988}, or the Dennis-Waldhausen Lemma, p. 319 of \cite{10.1007/BFb0088094}); we omit this deliberation. Recall that a \emph{bicategory} is a category $\mathcal{B}$ enriched in categories together with a composition functor
\[
- \odot - : \mathrm{Hom}_\mathcal{B} (X, Y) \times \mathrm{Hom}_\mathcal{B}(Y, Z) \to \mathrm{Hom}_\mathcal{B}(X, Z)
\]
and natural isomorphisms
\[
\mathfrak{a}: (M \odot N) \odot P \xrightarrow{\cong} M \odot (N \odot P)
\]
\[
\mathfrak{l}: \mathrm{Id}_X \odot M \xrightarrow{\cong} M
\]
\[
\mathfrak{r}: M \odot \mathrm{Id}_Y \xrightarrow{\cong} M
\]
which satisfy the same coherence axioms as those for a monoidal category. For a more precise definition, see Definition 1.2.1 of \cite{MR2056094}. 

\begin{example}
\label{eg:bimc}
Fix grading category $(\mathcal{C}, \alpha)$. Define the category $\mathrm{Bim}^\mathcal{C}$ whose objects are $\mathcal{C}$-graded algebras, and set
\[
\mathrm{Hom}_{\mathrm{Bim}^\mathcal{C}}( A, B) = \mathrm{Bim}(A, B)^\mathcal{C}.
\]
Then, $- \odot - :  \mathrm{Bim}(A, B)^\mathcal{C} \times  \mathrm{Bim}(B, C)^\mathcal{C} \to  \mathrm{Bim}(A, C)^\mathcal{C}$ is given by $ - \otimes_B -$, $\mathfrak{a} = \alpha$, $\mathrm{Id}_A = A$ viewed as an $(A, A)$-bimodule, and $\mathfrak{l}$ and $\mathfrak{r}$ are the (trivial) unitors of Proposition \ref{prop:AAmonoidal}.
\end{example}

\begin{definition}
A \emph{shadow} for a bicategory $\mathcal{B}$ is a category $\mathbf{T}$ and a collection of functors
\[
\left\{ \left \langle \! \left \langle - \right \rangle \! \right \rangle_X: \mathrm{End}_\mathcal{B}(X) \to \mathbf{T} \right\}_{X \in \mathrm{Ob}(\mathcal{B})}
\]
equipped with natural isomorphisms
\[
\theta_{f,g}: \left \langle \! \left \langle  f \odot g\right \rangle \! \right \rangle_X \xrightarrow{\cong} \left \langle \! \left \langle  g \odot f \right \rangle \! \right \rangle_Y
\]
whenever $f \in \mathrm{Ob}(\mathrm{Hom}_\mathcal{B}(X, Y))$ and $g \in \mathrm{Ob}(\mathrm{Hom}_\mathcal{B}(Y, X))$, such that the diagrams
\begin{equation}
\label{eq:shadow1}
\begin{tikzcd}
\left \langle \! \left \langle (f \odot g) \odot h \right \rangle \! \right \rangle_X
\arrow[r, "\left \langle \! \left \langle\mathfrak{a}\right \rangle \! \right \rangle"] 
\arrow[d, "\theta_{f \odot g, h}"']
&
\left \langle \! \left \langle f \odot (g \odot h) \right \rangle \! \right \rangle_X
\arrow[r, "\theta_{f, g \odot h}"]
&
\left \langle \! \left \langle (g \odot h) \odot f \right \rangle \! \right \rangle_Y
\arrow[d, "\left \langle \! \left \langle\mathfrak{a}\right \rangle \! \right \rangle"]
\\
\left \langle \! \left \langle h \odot (f \odot g) \right \rangle \! \right \rangle_Z
\arrow[r, "\left \langle \! \left \langle\mathfrak{a}^{-1}\right \rangle \! \right \rangle"]
&
\left \langle \! \left \langle (h \odot f) \odot g \right \rangle \! \right \rangle_Z
\arrow[r, "\theta_{h \odot f, g}"]
&
\left \langle \! \left \langle g \odot (h \odot f) \right \rangle \! \right \rangle_Y
\end{tikzcd}
\end{equation}
and
\begin{equation}
\label{eq:shadow2}
\begin{tikzcd}
\left \langle \! \left \langle f \odot \mathrm{Id}_X \right \rangle \! \right \rangle_X
\arrow[r, "\theta_{f, \mathrm{Id}_X}"]
\arrow[dr, "\left \langle \! \left \langle\mathfrak{r}\right \rangle \! \right \rangle"']
&
\left \langle \! \left \langle \mathrm{Id}_X \odot f \right \rangle \! \right \rangle_X
\arrow[r, "\theta_{\mathrm{Id}_X, f}"]
\arrow[d, "\left \langle \! \left \langle\mathfrak{l}\right \rangle \! \right \rangle"]
&
\left \langle \! \left \langle f \odot \mathrm{Id}_X \right \rangle \! \right \rangle_X
\arrow[dl, "\left \langle \! \left \langle\mathfrak{r}\right \rangle \! \right \rangle"]
\\
&
\left \langle \! \left \langle f \right \rangle \! \right \rangle_X
&
\end{tikzcd}
\end{equation}
commute whenever they make sense. A shadow is called \emph{symmetric} if $\theta_{f, g}^{-1} = \theta_{g, f}$.
\end{definition}

Note the similarity between the first diagram and the $\alpha$-$\varepsilon$ coherence relation of Definition \ref{def:looper}.

\begin{proposition}
Suppose $(\mathcal{C}, \alpha)$ is a grading category admitting a looper $\varepsilon$, and fix the triple $(\mathcal{C}, \alpha, \varepsilon)$. Then the zeroth Hochschild homology
\[
\left \langle \! \left \langle - \right \rangle \! \right \rangle_A := \mathsf{HH}_0(A, -)
\]
is a symmetric shadow with target category $\mathbf{T} = \mathrm{Mod}^{\mathrm{Tr}(\mathcal{C})}$ on the bicategory $\mathcal{B} = \mathrm{Bim}^\mathcal{C}$.
\end{proposition}

\begin{proof}
For $M$ (respectively, $N$) a $\mathcal{C}$-graded $(A, B)$ (respectively $(B, A)$)-bimodule, define 
\[
\varepsilon_{M, N}: \left \langle \! \left \langle M \odot N \right \rangle \! \right \rangle_A \to \left \langle \! \left \langle N \odot M \right \rangle \! \right \rangle_B
\]
by $m \otimes n \mapsto \varepsilon(\abs{m}, \abs{n}) n \otimes m$, noting that
\begin{align*}
 \left \langle \! \left \langle M \odot N \right \rangle \! \right \rangle_A &= \mathsf{HH}_0(A, M \otimes_B N) \\ & = M \otimes_B N / \left ( \rho_R(m \otimes n, a) - \varepsilon(\abs{m\otimes n}, \abs{a}) \rho_L(a, m\otimes n)\right )
\end{align*}
We claim that this map is an isomorphism. Again, if this map is well-defined, then it is an isomorphism with inverse $\varepsilon_{M, N}^{-1} = \varepsilon_{N, M}$. This also means that if we verify commutativity of Diagrams (\ref{eq:shadow1}) and (\ref{eq:shadow2}), then the shadow is automatically symmetric. 

Well-definedness of $\varepsilon_{M, N}$ follows from $\alpha$-$\varepsilon$ coherence. By definition, recall that 
\[
\rho_R(m \otimes n, a) - \varepsilon(\abs{n} \circ \abs{m}, \abs{a}) \rho_L(a, m\otimes n)
=
\alpha(\abs{m}, \abs{n}, \abs{a})
m \otimes (n \cdot a)
-
\varepsilon(\abs{n} \circ \abs{m}, \abs{a})
\alpha(\abs{a}, \abs{m}, \abs{n})^{-1}
(a \cdot m) \otimes n.
\]
This is sent to 
\[
\alpha(\abs{m}, \abs{n}, \abs{a})
\varepsilon(\abs{m}, \abs{a} \circ \abs{n})
(n \cdot a) \otimes m
-
\varepsilon(\abs{n} \circ \abs{m}, \abs{a})
\alpha(\abs{a}, \abs{m}, \abs{n})^{-1}
\varepsilon(\abs{m} \circ \abs{a}, \abs{n})
n \otimes (a \cdot m)
\]
by $\varepsilon_{M, N}$, which we can rewrite using the definition of $N \otimes_A M$ as
\[
\alpha(\abs{m}, \abs{n}, \abs{a})
\varepsilon(\abs{m}, \abs{a} \circ \abs{n})
\alpha(\abs{n}, \abs{a}, \abs{m})
n \otimes (a \cdot m)
-
\varepsilon(\abs{n} \circ \abs{m}, \abs{a})
\alpha(\abs{a}, \abs{m}, \abs{n})^{-1}
\varepsilon(\abs{m} \circ \abs{a}, \abs{n})
n \otimes (a \cdot m)
\]
which is equal to zero by $\alpha$-$\varepsilon$ coherence.

In conclusion, Diagram (\ref{eq:shadow1}) commutes exactly by $\alpha$-$\varepsilon$-coherence. Diagram (\ref{eq:shadow2}) is also seen to commute, where $\mathfrak{r}$ and $\mathfrak{l}$ are the unitors of Proposition \ref{prop:AAmonoidal}; we leave this verification to the reader.
\end{proof}

\section{Hochschild cohomology for $\mathcal{C}$-graded algebras and bimodules}
\label{s:cohochschild}

Notice that the left and right actions of Proposition \ref{prop:A-DGA} imply that for each $n\ge 0$, $A^{\otimes n+ 2}$ has the structure of a $\mathcal{C}$-graded $(A,A)$-bimodule. Given a $\mathcal{C}$-graded algebra $A$ and module $M$, we define the $k$th \emph{$\mathcal{C}$-graded Hochschild cohomology} $\mathsf{HH}^k(A; M)$ as the $k$th cohomology group of the cochain complex
\[
\begin{tikzcd}
	0 
	\arrow[r] 
	&
	\mathrm{Hom}_{\mathrm{Bim}(A,A)^\mathcal{C}} (A^{\otimes 2}, M)
	\arrow[r, "- \circ \partial_1"]
	&
	\mathrm{Hom}_{\mathrm{Bim}(A,A)^\mathcal{C}} (A^{\otimes 3}, M)
	\arrow[r, "- \circ \partial_2"]
	&
	\mathrm{Hom}_{\mathrm{Bim}(A,A)^\mathcal{C}} (A^{\otimes 4}, M)
	\arrow[r, "- \circ \partial_3"]
	&
	\cdots
\end{tikzcd}
\]
where $\partial_k$ is the differential of the $\mathcal{C}$-graded Bar complex $\mathcal{B}(A)$. Therefore, Hochschild cohomology is, in a sense, more natural that Hochschild homology for $\mathcal{C}$-graded structures, as it does not depend on the existence and choice of a looper $\varepsilon$. 

In this section, we prove that there is an isomorphism
\[
\mathrm{Hom}_{\mathrm{Bim}(I_\mathcal{C}, I_\mathcal{C})^\mathcal{C}}(A^{\otimes n}, M)
\cong
\mathrm{Hom}_{\mathrm{Bim}(A,A)^\mathcal{C}}(A^{\otimes n+2}, M)
\]
generalizing the usual scenario (recall that $I_\mathcal{C} := \bigoplus_{X \in \mathrm{Ob}(\mathcal{C})} \Bbbk_{\mathrm{Id}_X}$). Since the morphisms of $\mathrm{Bim}(A, A)^\mathcal{C}$ are $\mathcal{C}$-grading preserving, the equivalence is an isomorphism on the level of $\Bbbk$-modules.

Of particular interest is the zeroth cohomology $\mathsf{HH}^0(A; M)$ in the $\mathcal{C}$-graded setting---we hope that it is the center of $A$, but in what sense? In Subsection \ref{ss:ccenter}, we posit a guess and study its relationship with invertible complexes of $\mathcal{C}$-graded $(A,A)$-bimodules, mirroring Khovanov \cite{MR2171235}. In Subsection \ref{ss:reduction}, we prove the above isomorphism, and verify our guess for $\mathsf{HH}^0$, among other elementary computations. Finally, the Hochschild cohomology defined above is the ``$\mathcal{C}$-degree preserving'' part of a more general notion of Hochschild cohomology, which we describe in Subsection \ref{ss:nonhomoHH}; the zeroth Hochschild cohomology therefore provides a natural definition of a larger \emph{derived center} for $\mathcal{C}$-graded algebras.

\subsection{The center of $\mathcal{C}$-graded algebras}
\label{ss:ccenter}

We start this section with a meditation on the ring of endomorphisms of a $\mathcal{C}$-graded algebra viewed as a $\mathcal{C}$-graded $(A,A)$-bimodule. As is classically the case, we give $A$ this structure by declaring $\rho_L(x,a) = \mu(x,a)$ and $\rho_R(a,y) = \mu(a,y)$. It is well known that the endomorphisms of a unital, associative algebra as an $(A, A)$-bimodule are multiplications by elements of the center of $A$. However, in the $\mathcal{C}$-graded context, it is not obvious how one should define the center \textit{a priori}.

Take $f\in \mathrm{Hom}_{\mathrm{Bim}(A,A)^\mathcal{C}}(A,A)$; \textit{i.e.}, $f$ is a $\mathcal{C}$-graded endomorphism of $A$ as an $(A,A)$-bimodule. The classical isomorphism is obtained by sending $f \mapsto f(1)$ for $1$ the unit of $A$. In the $\mathcal{C}$-graded setting, the unit of $A$ is a sum of mutually orthogonal components
\[
1 = \sum_{X \in \mathrm{Ob}(\mathcal{C})} 1_X
\]
where $\{1_X \in A_{\mathrm{Id}_X: X \in \mathrm{Ob}(\mathcal{C})}\}$ are the units of $A$ in the sense of Definition \ref{def:algebras}. Now, it is clear that $f\in \mathrm{Hom}_{\mathrm{Bim}(A,A)^\mathcal{C}}(A,A)$ is determined by $f(1)$: indeed, for any $a\in A$, with $\abs{a}: X \to Y$, 
\[
f(a) = \mathcal{R}(\abs{a}, \mathrm{Id}_Y)^{-1} f(\mu(a, 1_Y)) = \mathcal{R}(\abs{a}, \mathrm{Id}_Y)^{-1} \rho_L(a, f(1))
\]
since $\rho_R = \mu = \rho_L$ and $f$ is a map of bimodules. On the other hand,
\[
f(a) = \mathcal{L}(\mathrm{Id}_X, \abs{a})^{-1} \rho_R(f(1), a).
\]
Thus, we define the \emph{center} of the $\mathcal{C}$-graded algebra $A$ as
\[
Z^\mathcal{C}(A):= \{z\in A: \mu(a,z) = \zeta(\abs{a}) \mu(z,a))~ \text{for each}~a\in A\}
\]
where
\[
\zeta(\abs{a}) := \alpha(\mathrm{Id}_X, \mathrm{Id}_X, \abs{a}) \alpha(\abs{a}, \mathrm{Id}_Y, \mathrm{Id}_Y).
\]
In particular, we must assume always that the unitors are the typical ones induced by $\alpha$.

\begin{remark}
Indeed, when $\abs{a}$ is a loop (say, based at $X$), then we have that $\zeta(\abs{a}) = \varepsilon(\abs{a}, \mathrm{Id}_X)$, thanks to Lemma \ref{lem:eploops}!
\end{remark}

\begin{proposition}
\label{prop:c-gradedcenter}
Assume unitors are typical. If $A$ is a $\mathcal{C}$-graded algebra, $\mathrm{Hom}_{\mathrm{Bim}(A,A)^\mathcal{C}}(A,A)$ is isomorphic to $Z^\mathcal{C}(A)$ as $\Bbbk$-modules.
\end{proposition}

Note that since the maps are $\mathcal{C}$-grading preserving, the graded structure of $\mathrm{Hom}_{\mathrm{Bim}(A,A)^\mathcal{C}}(A,A)$ is trivial: each element has non-homogeneous $\mathcal{C}$-grading given by $\prod_{X\in \mathcal{C}} \mathrm{Id}_X$. Implicit in the definition is that $Z^\mathcal{C}(A)$ bears the same, trivial $\mathcal{C}$-grading---it is for this reason that Proposition \ref{prop:c-gradedcenter} is stated on the level of $\Bbbk$-modules. Moreover, it makes sense to talk about the group
 \[
 Z_*^\mathcal{C}(A):= \{z\in Z^\mathcal{C}(A):\, \text{there exists}~z' \in Z^\mathcal{C}(A)~\text{such that}~ \mu(z, z
') = 1\}
 \]
 of invertible elements of $Z^\mathcal{C}(A)$ (which we can define without reference to a looper). Therefore, Proposition \ref{prop:c-gradedcenter} has the following immediate corollary.
  
 \begin{corollary}
 The group of automorphisms of $A$ as a $\mathcal{C}$-graded $(A,A)$-bimodule is isomorphic to $Z_*^\mathcal{C}(A)$.
 \end{corollary}

\begin{proof}[Proof of Proposition \ref{prop:c-gradedcenter}]
Define $\eta: \mathrm{Hom}_{\mathrm{Bim}(A,A)^\mathcal{C}}(A,A) \to Z^\mathcal{C}(A)$ by $\eta(f) := f(1)$ for $1 = \sum_{X \in \mathrm{Ob}(\mathcal{C})} 1_X \in A$, where $\{1_X\}_{X \in \mathrm{Ob}(\mathcal{C})}$ are the units of $A$. Since the domain is the set of bimodule endomorphisms of $A$,
\[
\eta(g\circ f) = (g\circ f)(1) = g(f(1)) = g(\rho_R(1, f(1)) = \mu(g(1), f(1)) = \mu(\eta(g), \eta(f)).
\]
Note that the third equality follows since gradings are supported entirely in $\prod_{X\in \mathrm{Ob}(\mathcal{C})} \mathrm{Id}_X$. By a similar argument, $\eta(f)$ is a central element of the $\mathcal{C}$-graded algebra $A$, as for any $a\in A$ with homogeneous grading $\abs{a}: X \to Y$,
\[
\mu(a, \eta(f)) = \rho_L(a, f(1)) = f(\mu(a, 1)) = \mathcal{R}(\abs{a}, \mathrm{Id}_Y) f(a) = \mathcal{R}(\abs{a}, \mathrm{Id}_Y) \mathcal{L}(\mathrm{Id}_X, \abs{a})^{-1} f(\mu(1,a))
\]
which is equal to $\zeta(\abs{a}) \mu(\eta(f), a)$. For bijectivity, on one hand, it is apparent that $\ker(\eta) = \{\mathrm{Id}_A\}$. On the other hand, for $z\in Z^\mathcal{C}(A)$, there is $f\in \mathrm{Hom}_{\mathrm{Bim}(A,A)^\mathcal{C}}(A,A)$ for which $\eta(f) = f(1) = z$---indeed, $f$ is the map
\[
f(a) = \mathcal{R}(\abs{a}, \mathrm{Id}_Y)^{-1} \mu(a, z).
\]
It remains to check that $f$ respects the bimodule action. Indeed,
\[
f(\rho_L(x,a)) = \mathcal{R}(\abs{a} \circ \abs{x}, \mathrm{Id}_Y)^{-1} \mu(\mu(x,a), z) =  \mathcal{R}(\abs{a} \circ \abs{x}, \mathrm{Id}_Y)^{-1} \alpha(\abs{x}, \abs{a}, \mathrm{Id}_Y) \mu(x, \mu(a, z))
\]
is equal to $\rho_L(x, f(a)) = \mathcal{R}(\abs{a}, \mathrm{Id}_Y)^{-1} \mu(x, \mu(a,z))$, appealing to (iii) of Lemma \ref{lem:assocomps}, assuming that unitors are typically chosen. The other check is similar, but slightly more interesting: we compute
\[
f(\rho_R(a,y)) = \mathcal{R}(\abs{y} \circ \abs{a}, \mathrm{Id}_Z)^{-1} \mu(\mu(a,y),z)
	= \mathcal{R}(\abs{y} \circ \abs{a}, \mathrm{Id}_Z)^{-1} \zeta(\abs{y} \circ \abs{a}) \mu(z, \mu(a,y))
\]
on one hand, which simplifies to $\alpha(\mathrm{Id}_X, \mathrm{Id}_X, \abs{y} \circ \abs{a})$ assuming typical unitors. On the other,
\[
\rho_R(f(a), y) = \mathcal{R}(\abs{a}, \mathrm{Id}_Y)^{-1} \mu(\mu(a,z), y) 
	= \mathcal{R}(\abs{a}, \mathrm{Id}_Y)^{-1} \zeta(\abs{a}) \mu(\mu(z,a), y)
\]
which simplifies to $\alpha(\mathrm{Id}_X, \mathrm{Id}_X, \abs{a})$, again assuming typical unitors are chosen. Then, equivalence is granted by (ii) of Lemma \ref{lem:assocomps}.
\end{proof}

Let $\mathcal{K}^\mathcal{C}(A)$ denote the category whose objects are bounded complexes of $\mathcal{C}$-graded $(A, A)$-bimodules with grading preserving differential, and whose morphisms are grading-preserving chain maps up to chain-homotopy. Later on, we will include a brief description of the theory of $\mathcal{C}$-grading shifts (for more on this theory, refer to \cite{naisse2020odd} or \cite{spyropoulos2024}), but we will only need to consider grading-preserving differentials and chain-maps at the current time.

We say that a complex of $\mathcal{C}$-graded $(A, A)$-bimodules $M\in \mathrm{Ob}(\mathcal{K}^\mathcal{C}(A))$ is \emph{invertible} if there exists $N \in \mathrm{Ob}(\mathcal{K}^\mathcal{C}(A))$ such that $N \otimes_A M \simeq A$ and $M \otimes_A N \simeq A$, for $A$ the complex $\begin{tikzcd}[column sep = small] 0 \arrow[r] & A \arrow[r] & 0\end{tikzcd}$. Thus, it is immediate that
\[
\mathrm{Hom}_{\mathcal{K}^\mathcal{C}(A)}(A, A) \cong \mathrm{Hom}_{\mathrm{Bim}(A,A)^\mathcal{C}}(A,A).
\]
Moreover, we have the following, as in \cite{MR2171235}.

\begin{proposition}
\label{prop:invcomplex}
If $M$ is invertible, then
\[
\mathrm{Hom}_{\mathcal{K}^\mathcal{C}(A)}(M, M) \cong  \mathrm{Hom}_{\mathrm{Bim}(A,A)^\mathcal{C}}(A,A)
\]
as $\Bbbk$-modules. Consequently, the group of automorphisms of $M$ in $\mathcal{K}^\mathcal{C}(A)$ is isomorphic to $Z_*^\mathcal{C}(A)$.
\end{proposition}

\begin{proof}
We claim that the map $\phi: \mathrm{End}_{\mathcal{K}^\mathcal{C}(A)}(M) \to \mathrm{End}_{\mathcal{K}^\mathcal{C}(A)}(M \otimes_A N) \cong \mathrm{End}_{\mathcal{K}^\mathcal{C}(A)} (A) \cong  \mathrm{End}_{\mathrm{Bim}(A,A)^\mathcal{C}}(A)$ given by $f \mapsto f \otimes \mathrm{id}_N$ is an isomorphism. It is easy to see that $\phi$ is injective, as it has a left-inverse given by the composite
\[
\mathrm{End}_{\mathcal{K}^\mathcal{C}(A)}(M \otimes_A N) 
\to 
\mathrm{End}_{\mathcal{K}^\mathcal{C}(A)}((M \otimes_A N) \otimes_A M)
\xrightarrow{\alpha}
\mathrm{End}_{\mathcal{K}^\mathcal{C}(A)}(M \otimes_A (N \otimes_A M))
\]
where the first arrow is the map $g \mapsto g \otimes \mathrm{id}_M$ (note that $\mathrm{End}_{\mathcal{K}^\mathcal{C}(A)}(M \otimes_A (N \otimes_A M))
\cong
\mathrm{End}_{\mathcal{K}^\mathcal{C}(A)} (M \otimes_A A)
\cong
\mathrm{End}_{\mathcal{K}^\mathcal{C}(A)}(M)$). To see that $\phi$ is surjective, notice that multiplication on the left by central elements gives each of the $\Bbbk$-modules above the structure of a $Z^\mathcal{C}(A)$-bimodule. Appealing to Proposition \ref{prop:c-gradedcenter}, $\mathrm{End}_{\mathcal{K}^\mathcal{C}(A)}(M \otimes_A N) \cong Z^\mathcal{C}(A)$ is generated by $\phi(\mathrm{id}_M) = \mathrm{id}_M \otimes \mathrm{id}_N$, thus $\phi$ is an isomorphism.
\end{proof}

To conclude our study of the center of $\mathcal{C}$-graded algebras, we note that the usual relationship with $\mathsf{HH}_\bullet(A, M)$ exists. Define an action of $Z^\mathcal{C}(A)$ on $C_n(A, M)$ by
\begin{equation}
\label{eq:centeracts}
z \cdot (m \otimes a_1 \otimes \cdots \otimes a_n) := \alpha(\mathrm{Id}_X, \mathrm{Id}_X, \abs{m}) \, \rho_L(z, m) \otimes a_1 \otimes \cdots \otimes a_n
\end{equation}
where $X = \mathrm{dom}(\abs{m}) = \mathrm{codom}(\abs{a_n})$.

\begin{proposition}
The Hochschild homology of a $\mathcal{C}$-graded algebra $\mathsf{HH}_\bullet(A, M)$ is a module over $Z^\mathcal{C}(A)$ by the action (\ref{eq:centeracts}).
\end{proposition}

\begin{proof}
We just have to show that the action (\ref{eq:centeracts}) of $Z^\mathcal{C}(A)$ on $C_n(A, M)$ commutes with the differential, so that the action is an endomorphism of the complex. Writing the differential of $\mathsf{HC}_\bullet$ as $b_n = \sum_{i=0}^n (-1)^i d_i$, it suffices to check that $z \cdot d_i(m \otimes a_1 \otimes \cdots \otimes a_n) = d_i(z \cdot m \otimes a_1 \otimes \cdots \otimes a_n)$ for each $i=0, \ldots, n$. For $i=0$,
\[
z \cdot d_0(m \otimes a_1 \otimes \cdots \otimes a_n) = \alpha(\mathrm{Id}_X, \mathrm{Id}_X, \abs{a_1} \circ \abs{m}) \, \rho_L(z, \rho_R(m, a_1)) \otimes a_2 \otimes \cdots \otimes a_n
\]
and
\begin{align*}
d_0(z \cdot m \otimes a_1 \otimes \cdots \otimes a_n) &= \alpha(\mathrm{Id}_X, \mathrm{Id}_X, \abs{m}) \, \rho_R(\rho_L(z, m), a_1) \otimes a_2 \otimes \cdots \otimes a_n
\\
&= \alpha(\mathrm{Id}_X, \mathrm{Id}_X, \abs{m}) \, 
\alpha(\mathrm{Id}_X, \abs{m}, \abs{a_1}) \, 
\rho_L(z, \rho_R(m, a_1)) \otimes a_2 \otimes \cdots \otimes a_n.
\end{align*}
Thus, equality follows from (ii) of Lemma \ref{lem:assocomps}. For $i=1, \ldots, n-1$ the result follows trivially. The only interesting case is when $i=n$. We omit the writing of the $\Omega(\abs{m}, \abs{a_1}, \ldots, \abs{a_n})$ term because, obviously,
\[
\Omega(\abs{m}, \abs{a_1}, \ldots, \abs{a_n})
=
\Omega(\abs{m} \circ \mathrm{Id}_X, \abs{a_1}, \ldots, \abs{a_n}).
\]
Let $W = \mathrm{dom}(\abs{a_n})$; we compute
\[
z \cdot d_n(m \otimes a_1 \otimes \cdot \otimes a_n) = \alpha(\mathrm{Id}_W, \mathrm{Id}_W, \abs{m} \circ \abs{a_n}) 
\rho_L(z, \rho_L(a_n, m)) \otimes a_1 \otimes \cdots \otimes a_{n-1}
\]
and
\begin{align*}
d_n(z \cdot m \otimes a_1 \otimes \cdots \otimes a_n) &= \alpha(\mathrm{Id}_X, \mathrm{Id}_X, \abs{m}) \rho_L(a_n, \rho_L(z, m)) a_1 \otimes \cdots \otimes a_{n-1}
\\
&= \alpha(\mathrm{Id}_X, \mathrm{Id}_X, \abs{m}) \,
\alpha(\abs{a_n}, \mathrm{Id}_X, \abs{m})^{-1} \, 
\rho_L(\mu(a_n, z), m) a_1 \otimes \cdots \otimes a_{n-1}
\\
&= \alpha(\mathrm{Id}_X, \mathrm{Id}_X, \abs{m}) \, 
\alpha(\abs{a_n}, \mathrm{Id}_X, \abs{m})^{-1} \, 
\zeta(\abs{a_n}) \, 
\rho_L(\mu(z, a_n), m) a_1 \otimes \cdots \otimes a_{n-1}
\\
&= \alpha(\mathrm{Id}_X, \mathrm{Id}_X, \abs{m}) \, 
\alpha(\abs{a_n}, \mathrm{Id}_X, \abs{m})^{-1} \, 
\zeta(\abs{a_n}) \, 
\alpha(\mathrm{Id}_W, \abs{a_n}, \abs{m}) \, 
\\ & \phantom{{}={}}
\rho_L(z, \rho_L(a_n, m)) a_1 \otimes \cdots \otimes a_{n-1}.
\end{align*}
Using (iv) of Lemma \ref{lem:assocomps}, $\alpha(\abs{a_n}, \mathrm{Id}_X, \abs{m})^{-1}$ cancels with $\alpha(\mathrm{Id}_X, \mathrm{Id}_X, \abs{m})$ and one part of $\zeta(\abs{a_n})$. The remaining terms are
\[
\alpha(\mathrm{Id}_W, \mathrm{Id}_W, \abs{a_n}) \alpha(\mathrm{Id}_W, \abs{a_n}, \abs{m}) = \alpha(\mathrm{Id}_W, \mathrm{Id}_W, \abs{m} \circ \abs{a_n})
\]
concluding the proof.
\end{proof}

\subsection{Reduction of the Hochschild cocomplex}
\label{ss:reduction}

In this subsection, we prove the following theorem, and note its consequences.

\begin{theorem}
\label{thm:cohored}
Suppose $A$ is a $\mathcal{C}$-graded algebra and $M$ is a $\mathcal{C}$-graded $(A,A)$-bimodule. Then
\[
\mathrm{Hom}_{\mathrm{Bim}(I_\mathcal{C}, I_\mathcal{C})^\mathcal{C}}(A^{\otimes n}, M)
\cong
\mathrm{Hom}_{\mathrm{Bim}(A,A)^\mathcal{C}}(A^{\otimes n+2}, M).
\]
\end{theorem}

Consider the following maps.
\begin{itemize}
	\item For $n=0$ and $\xi \in \mathrm{Hom}_{\mathrm{Bim}(I_\mathcal{C}, I_\mathcal{C})^\mathcal{C}}(I_\mathcal{C}, M)$,
	\[
	\psi_0(\xi)(a \otimes b) :=
	\mathcal{R}(\abs{a}, \mathrm{Id}_{\mathrm{codom}(\abs{a})})^{-1}
	\rho_R(\rho_L(a, \xi(1)), b)
	\]
	where $1 = \sum_{x \in \mathrm{Ob}(\mathcal{C})} 1_X$ for $1_X \in \Bbbk_{\mathrm{Id}_X} \subset I_\mathcal{C}$. Note that $\xi$ is determined by $\xi(1) = \sum_{x\in \mathrm{Ob}(\mathcal{C})} \xi(1_X)$.
	\item For $n\ge 1$ and $f\in \mathrm{Hom}_{\mathrm{Bim}(I_\mathcal{C}, I_\mathcal{C})^\mathcal{C}}(A^{\otimes n}, M)$,
	\[
	\psi(f)(a_0 \otimes a_1 \otimes \cdots \otimes a_n \otimes a_{n+1}) := 
	\prod_{i=1}^{n-1} \alpha(\abs{a_0}, \abs{a_i} \circ \cdots \circ \abs{a_1}, \abs{a_{i+1}})
	\rho_R(\rho_L(a_0, f(a_1 \otimes \ldots \otimes a_n)), a_{n+1}).
	\]
\end{itemize}
Note that the action maps $\rho_L$ and $\rho_R$ are the actions of $A$ on the $(A, A)$-bimodule $M$.

\begin{lemma}
\label{lem:psiact}
For each $n\ge 0$, $\psi_n(f) \in \mathrm{Hom}_{\mathrm{Bim}(A, A)^\mathcal{C}} (A^{\otimes n+2}, M)$.
\end{lemma}

\begin{proof}
It suffices to show that $\psi_n(f)$ preserves actions. We have to check the $n=0$ and $n\ge 1$ cases separately.

The $n=0$ case is understandably easier. We compute
\begin{align*}
	\psi_0(\xi)(x \cdot a \otimes b) & =
		\alpha(\abs{x}, \abs{a}, \abs{b})^{-1} \psi_0(\xi) (\mu(x, a) \otimes b)
		\\&=
		\alpha(\abs{x}, \abs{a}, \abs{b})^{-1} 
		\mathcal{R}(\abs{a} \circ \abs{x}, \mathrm{Id}_{\mathrm{codom}(\abs{a})})^{-1}
		(\mu(x, a) \cdot \xi(1)) \cdot b.
\end{align*}
On the other hand,
\begin{align*}
	x \cdot \psi_0(\xi)(a \otimes b) &=
		\mathcal{R}(\abs{a}, \mathrm{Id}_{\mathrm{codom}(\abs{a})})^{-1}
		 x \cdot ((a \cdot \xi(1)) \cdot b)
		\\&=
		\mathcal{R}(\abs{a}, \mathrm{Id}_{\mathrm{codom}(\abs{a})})^{-1}
		\alpha(\abs{x}, \abs{a}, \abs{b})^{-1}
		(x \cdot (a \cdot \xi(1))) \cdot b
		\\&=
		\mathcal{R}(\abs{a}, \mathrm{Id}_{\mathrm{codom}(\abs{a})})^{-1}
		\alpha(\abs{x}, \abs{a}, \abs{b})^{-1}
		\alpha(\abs{x}, \abs{a}, \mathrm{Id}_{\mathrm{codom}(\abs{a})})^{-1}
		(\mu(x,a) \cdot \xi(1)) \cdot b.
\end{align*}
Equivalence follows since
\[
\mathcal{R}(\abs{a} \circ \abs{x}, \mathrm{Id}_{\mathrm{codom}(\abs{a})})
=
\mathcal{R}(\abs{a}, \mathrm{Id}_{\mathrm{codom}(\abs{a})})
\alpha(\abs{x}, \abs{a}, \mathrm{Id}_{\mathrm{codom}(\abs{a})})
\]
by (iii) of Lemma \ref{lem:assocomps}. We leave it to the reader to verify that $\psi_0(\xi)(a\otimes b \cdot y) = \psi_0(\xi)(a \otimes b) \cdot y$ follows by definition.

For the $n\ge 1$ case, it is also immediate that 
\[
\psi_n(f)(a_0 \otimes \cdots \otimes a_{n+1} \cdot a) = \psi(f)(a_0 \otimes \cdots \otimes a_{n+1}) \cdot a.
\]
The check for the right action is tedious and requires iterative applications of the pentagon relation. We compute
\begin{align*}
	\psi_n(f)(\rho_L(a, a_0 \otimes a_1 \otimes \cdots \otimes a_{n+1})) &=
		\prod_{i=1}^{n+1} \alpha(\abs{a}, \abs{a_i-1} \circ \cdots \circ \abs{a_0}, \abs{a_i})^{-1}
		\psi_n(f)(\mu(a, a_0) \otimes a_1 \otimes \cdots \otimes a_{n+1})
		\\ & =
		\prod_{i=1}^{n+1} \alpha(\abs{a}, \abs{a_{i-1}} \circ \cdots \circ \abs{a_0}, \abs{a_i})^{-1}
			\\ & \phantom{{}={}}
		\prod_{i=1}^{n-1} \alpha(\abs{a} \circ \abs{a_0}, \abs{a_i} \circ \cdots \circ \abs{a_1}, \abs{a_{i+1}})
			\\ & \phantom{{}={}}
		\rho_R(\rho_L(\mu(a, a_0), f(a_1 \otimes \cdots \otimes a_n)), a_{n+1}).
\end{align*}
Conversely,
\begin{align*}
	\rho_L(a, \psi(f)(a_0 \otimes \cdots \otimes a_{n+1})) &=
		\prod_{i=1}^{n-1} \alpha(\abs{a_0}, \abs{a_i} \circ \cdots \circ \abs{a_1}, \abs{a_{i+1}})
		\rho_L(a, \rho_R(\rho_L(a_0, f(a_1\otimes \cdots \otimes a_n)), a_{n+1}))
		\\ & = 
		\prod_{i=1}^{n-1} \alpha(\abs{a_0}, \abs{a_i} \circ \cdots \circ \abs{a_1}, \abs{a_{i+1}})
		\alpha(\abs{a}, \abs{a_n} \circ \cdots \circ \abs{a_1} \circ \abs{a_0}, \abs{a_{n+1}})^{-1}
			\\ & \phantom{{}={}}
		\alpha(\abs{a}, \abs{a_0}, \abs{a_n} \circ \cdots \circ \abs{a_1})^{-1}
		\rho_R(\rho_L(\mu(a, a_0), f(a_1 \otimes \cdots \otimes a_n)), a_{n+1}).
\end{align*}
Thus, the lemma follows if the product
\begin{align*}
\alpha&(\abs{a}, \abs{a_0}, \abs{a_n}\circ \cdots \circ \abs{a_1})
\prod_{i=1}^{n-1} \alpha(\abs{a_0}, \abs{a_i} \circ \cdots \circ \abs{a_1}, \abs{a_{i+1}})^{-1}
\prod_{i=1}^{n-1} \alpha(\abs{a} \circ \abs{a_0}, \abs{a_i} \circ \cdots \circ \abs{a_1}, \abs{a_{i+1}})\\ &
\prod_{i=1}^{n} \alpha(\abs{a}, \abs{a_{i-1}} \circ \cdots \circ \abs{a_1} \circ \abs{a_0}, \abs{a_i})^{-1}
\end{align*}
equals one. This follows from an iteration of the pentagon relation: first, consider the ultimate term from each of the products. The pentagon relation $d\alpha(\abs{a}, \abs{a_0}, \abs{a_{n-1}} \circ \cdots \circ \abs{a_1}, \abs{a_n}) = 1$ states that
\begin{align*}
\alpha(\abs{a}, \abs{a_0}, \abs{a_n}\circ \cdots \circ \abs{a_1})&
\alpha(\abs{a_0}, \abs{a_i} \circ \cdots \circ \abs{a_1}, \abs{a_{i+1}})^{-1}
\alpha(\abs{a} \circ \abs{a_0}, \abs{a_i} \circ \cdots \circ \abs{a_1}, \abs{a_{i+1}}) \\ &
\alpha(\abs{a}, \abs{a_{i-1}} \circ \cdots \circ \abs{a_1} \circ \abs{a_0}, \abs{a_i})^{-1}
=
\alpha(\abs{a}, \abs{a_0}, \abs{a_{n-1}} \circ \cdots \circ \abs{a_1})
\end{align*}
This provides the term necessary to repeat this reduction via the pentagon $d\alpha(\abs{a}, \abs{a_0}, \abs{a_{n-2}} \circ \cdots \circ \abs{a_1}, \abs{a_{n-1}}) = 1$. The process terminates with the pentagon $1 = d\alpha(\abs{a}, \abs{a_0}, \abs{a_1}, \abs{a_2})$, providing
\[
\alpha(\abs{a}, \abs{a_0}, \abs{a_2} \circ \abs{a_1})
\alpha(\abs{a}, \abs{a_1} \circ \abs{a_0}, \abs{a_2})^{-1}
\alpha(\abs{a} \circ \abs{a_0}, \abs{a_1}, \abs{a_2})
\alpha(\abs{a_0}, \abs{a_1}, \abs{a_2})^{-1}
=
\alpha(\abs{a}, \abs{a_0}, \abs{a_1})
\]
which cancels with the extra term in the length-$n$ product, $\alpha(\abs{a}, \abs{a_0}, \abs{a_1})^{-1}$.
\end{proof}

\begin{proof}[Proof of Theorem \ref{thm:cohored}]
We claim that the maps $\{\psi_n\}_{n\ge 0}$ are the isomorphisms. Following Lemma \ref{lem:psiact}, it suffices to provide inverses $\{\varphi_n\}_{n\ge 0}$: let $g \in \mathrm{Hom}_{\mathrm{Bim}(A,A)^\mathcal{C}} (A^{\otimes n+2}, M)$.
\begin{itemize}
	\item For $n=0$, define $\varphi_0(g) \in \mathrm{Hom}_{\mathrm{Bim}(I_\mathcal{C}, I_\mathcal{C})^\mathcal{C}} (I_\mathcal{C}, M)$ by
	\[
	\varphi_0(g) (1) := \sum_{X \in \mathrm{Ob}(\mathcal{C})} g(1^A_X \otimes 1^A_X)
	\]
	where the $1 = \sum_{\mathrm{Ob}(\mathcal{C})} 1_X$ appearing on the left is the unit of $I_\mathcal{C} = \bigoplus (\Bbbk)_{\mathrm{Id}_X}$, and the units $\{1^A_X\}_{X\in\mathrm{Ob}(\mathcal{C})}$ appearing on the right are the units of $A$.
	\item For $n\ge 1$, define $\varphi_n(g) \in \mathrm{Hom}_{\mathrm{Bim}(I_\mathcal{C}, I_\mathcal{C})^\mathcal{C}} (A^{\otimes n}, M)$ by
	\begin{align*}
	\varphi_n(g)(a_1 \otimes \cdots \otimes a_n) & :=
	\mathcal{L}(\mathrm{Id}_{\mathrm{dom}(\abs{a_1})}, \abs{a_1 \otimes \cdots \otimes a_n})^{-1}
	\mathcal{R}(\abs{a_1 \otimes \cdots \otimes a_n}, \mathrm{Id}_{\mathrm{codom}(\abs{a_n})})^{-1}
	\\ & \phantom{{}={}}
	\prod_{i=1}^{n-1} \alpha(\mathrm{Id}_{\mathrm{dom}(\abs{a_1})}, \abs{a_i} \circ \cdots \circ \abs{a_1}, \abs{a_{i+1}})^{-1}
	g(1^A_{\mathrm{dom}(\abs{a_1})} \otimes a_1 \otimes \cdots \otimes a_n \otimes 1^A_{\mathrm{codom}(\abs{a_n})}).
	\end{align*}
	Notice that $g(1 \otimes a_1 \otimes \cdots \otimes a_n \otimes 1) = g(1^A_{\mathrm{dom}(\abs{a_1})} \otimes a_1 \otimes \cdots \otimes a_n \otimes 1^A_{\mathrm{codom}(\abs{a_n})})$.
\end{itemize}
The fact that these are maps of $\mathcal{C}$-graded $(I_\mathcal{C}, I_\mathcal{C})$-bimodules is apparent.

For the $n=0$ case, we compute
\begin{align*}
	\varphi_0(\psi_0(\xi)) (1) &= 
		\sum_{X \in \mathrm{Ob}(\mathcal{C})} \psi_0(\xi)(1^A_X \otimes 1^A_X)
		\\ &=
		\sum_{X \in \mathrm{Ob}(\mathcal{C})} \rho_R(\rho_L(1^A_X, \xi(1)), 1^A_X)
		\\ &=
		\xi(1)
\end{align*}
since $\mathcal{R}(\mathrm{Id}_X, \mathrm{Id}_X) = 1$ assuming unitors are typically chosen. On the other hand, assume that $a, b\in A$ are homogeneous satisfying $\mathrm{codom}(\abs{a}) = \mathrm{dom}(\abs{b}) = X$. We compute
\begin{align*}
	\psi_0(\varphi_0(g)) (a \otimes b) &=
		\mathcal{R}(\abs{a}, \mathrm{Id}_X)^{-1} \rho_R(\rho_L(a, \varphi_0(g)(1)), b)
		\\ & = 
		\mathcal{R}(\abs{a}, \mathrm{Id}_X)^{-1} g\left(\rho_R(\rho_L(a, 1^A_X \otimes 1^A_X), b)\right)
		\\ & = \mathcal{R}(\abs{a}, \mathrm{Id}_X)^{-1} \alpha(\abs{a}, \mathrm{Id}_X, \mathrm{Id}_X)^{-1} \mathcal{R}(\abs{a}, \mathrm{Id}_X)
			\\ & \phantom{{}={}} g\left(\rho_R(a \otimes 1^A_X, b)\right)
		\\ &= \mathcal{R}(\abs{a}, \mathrm{Id}_X)^{-1} \alpha(\abs{a}, \mathrm{Id}_X, \abs{b}) \mathcal{L}(\mathrm{Id}_X, \abs{b})
			\\ & \phantom{{}={}} g(a \otimes b)
\end{align*}
To explain, the second equality follows from grading considerations and the assumption that $g$ is a map of $(A, A)$-bimodules. The third equality follows from the definition of the left action of $A$ on $A \otimes A$, together with the fact that $\mu(a, 1_X^A) = \mathcal{R}(\abs{a}, \mathrm{Id}_X) a$. Since unitors are typically chosen, these two new terms cancel. The fourth equality follows from the definition of the right action of $A$ on $A\otimes A$, and the fact that $\mu(1_X^A, b) = \mathcal{L}(\mathrm{Id}_X, \abs{b}) b$. Finally, 
\[
\mathcal{R}(\abs{a}, \mathrm{Id}_X)^{-1} \alpha(\abs{a}, \mathrm{Id}_X, \abs{b}) \mathcal{L}(\mathrm{Id}_X, \abs{b}) = 1
\]
is the triangle identity.

The $n\ge 1$ case is slightly more tedious. Write $\varphi:= \varphi_n$ and $\psi:= \psi_n$. For $f \in \mathrm{Hom}_{\mathrm{Bim}(I_\mathcal{C}, I_\mathcal{C})^\mathcal{C}}(A^{\otimes n}, M)$, we leave it to the reader to verify that $\varphi(\psi(f)) (a_1 \otimes \cdots \otimes a_n) = f(a_1 \otimes \cdots \otimes a_n)$ follows strictly from definitions (it requires no applications of the pentagon relation). The rest of the proof is verification that $\psi(\varphi(g)) = g$. To reduce notation, we'll write $\mathrm{dom}(\abs{a_1})=X$ and $\mathrm{codom}(\abs{a_n})=Y$. We compute
\begin{align*}
	\psi(\varphi(g)) (a_0 \otimes a_1 \otimes \cdots \otimes a_n \otimes a_{n+1}) &= 
		\prod_{i=1}^{n-1} \alpha(\abs{a_0}, \abs{a_i} \circ \cdots \circ \abs{a_1}, \abs{a_{i+1}})
		\rho_R(\rho_L(a_0, \varphi(g)(a_1 \otimes \ldots \otimes a_n)), a_{n+1})
		\\ & =
		\underbrace{\prod_{i=1}^{n-1} \alpha(\abs{a_0}, \abs{a_i} \circ \cdots \circ \abs{a_1}, \abs{a_{i+1}})}_{(*)}
		\mathcal{L}(\mathrm{Id}_{X}, \abs{a_1 \otimes \cdots \otimes a_n})^{-1}
		\\ & \phantom{{}={}}
		\underbrace{\mathcal{R}(\abs{a_1 \otimes \cdots \otimes a_n}, \mathrm{Id}_{Y})^{-1}}_{(**)}
		\prod_{i=1}^{n-1} \alpha(\mathrm{Id}_{X}, \abs{a_i} \circ \cdots \circ \abs{a_1}, \abs{a_{i+1}})^{-1}
		\\ & \phantom{{}={}}
		\rho_R(\rho_L(a_0, g(1^A_{X} \otimes a_1 \otimes \cdots \otimes a_n \otimes 1^A_{Y})), a_{n+1})
\end{align*}
simply unraveling definitions. Continuing,
\begin{align*}
	\rho_R(\rho_L(a_0, g(1^A_X \otimes a_1 \otimes \cdots \otimes a_n \otimes 1^A_Y)), a_{n+1}) & = 
		g \left(\rho_R(\rho_L(a_0, 1^A_X \otimes a_1 \otimes \cdots \otimes a_n \otimes 1^A_Y), a_{n+1})\right)
		\\ & =
		\alpha(\abs{a_0}, \mathrm{Id}_X, \abs{a_1})^{-1}
		\prod_{i=1}^{n-1} \alpha(\abs{a_0}, \abs{a_i} \circ \cdots \circ \abs{a_1}, \abs{a_{i+1}})^{-1}
				\\ & \phantom{{}={}}
		\alpha(\abs{a_0}, \abs{a_n} \circ \cdots \circ \abs{a_1}, \mathrm{Id}_Y)^{-1}
		\mathcal{R}(\abs{a_0}, \mathrm{Id}_X)
				\\ & \phantom{{}={}}
		g \left(\rho_R(a_0 \otimes a_1 \otimes \cdots \otimes a_n \otimes 1^A_Y, a_{n+1})\right)
		\\ & =
		\alpha(\abs{a_0}, \mathrm{Id}_X, \abs{a_1})^{-1}
		\underbrace{\prod_{i=1}^{n-1} \alpha(\abs{a_0}, \abs{a_i} \circ \cdots \circ \abs{a_1}, \abs{a_{i+1}})^{-1}}_{(*)}
				\\ & \phantom{{}={}}
		\underbrace{\alpha(\abs{a_0}, \abs{a_n} \circ \cdots \circ \abs{a_1}, \mathrm{Id}_Y)^{-1}}_{(**)}
		\mathcal{R}(\abs{a_0}, \mathrm{Id}_X)
				\\ & \phantom{{}={}}
		\underbrace{\alpha(\abs{a_n} \circ \cdots \circ \abs{a_0}, \mathrm{Id}_Y, \abs{a_{n+1}})
		\mathcal{L}(\mathrm{Id}_Y, \abs{a_{n+1}})}_{(**)}
				\\ & \phantom{{}={}}
		g \left(a_0 \otimes a_1 \otimes \cdots \otimes a_n \otimes a_{n+1}\right).
\end{align*}
The terms labeled $(*)$ cancel obviously. Also, the terms labeled $(**)$ cancel by way of Lemma \ref{lem:assocomps}, since
\begin{align*}
\alpha(\abs{a_n} \circ \cdots \circ \abs{a_1}, \mathrm{Id}_Y, \mathrm{Id}_Y)
\alpha(\mathrm{Id}_Y, \mathrm{Id}_Y, \abs{a_{n+1}})^{-1}
& \stackrel{(iv)}{=}
\alpha(\abs{a_n} \circ \cdots \circ \abs{a_1} \circ \abs{a_0}, \mathrm{Id}_Y, \mathrm{Id}_Y)
\\ & \stackrel{(iii)}{=}
\alpha(\abs{a_n} \circ \cdots \circ \abs{a_1}, \mathrm{Id}_Y, \mathrm{Id}_Y)
\alpha(\abs{a_0}, \abs{a_n} \circ \cdots \circ \abs{a_1}, \mathrm{Id}_Y)
\end{align*}
So, the theorem follows as long as
\[
\mathcal{L}(\mathrm{Id}_X, \abs{a_1 \otimes \cdots \otimes a_n})^{-1}
\prod_{i=1}^{n-1} \alpha(\mathrm{Id}_X, \abs{a_i} \circ \cdots \circ \abs{a_1}, \abs{a_{i+1}})^{-1}
\alpha(\abs{a_0}, \mathrm{Id}_X, \abs{a_1})^{-1}
\mathcal{R}(\abs{a_0}, \mathrm{Id}_X) = 1.
\]
This follows from an iteration. First, notice that 
\[
\alpha(\abs{a_0}, \mathrm{Id}_X, \abs{a_1})^{-1}
\mathcal{R}(\abs{a_0}, \mathrm{Id}_X)
=
\alpha(\mathrm{Id}_X, \mathrm{Id}_X, \abs{a_1})^{-1}
\]
by (iv) of Lemma \ref{lem:assocomps}. Comparing this term with the first term in the product notation, we have
\[
\alpha(\mathrm{Id}_X, \mathrm{Id}_X, \abs{a_1})^{-1}
\alpha(\mathrm{Id}_X, \abs{a_1}, \abs{a_2})^{-1}
=
\alpha(\mathrm{Id}_X, \mathrm{Id}_X, \abs{a_2} \circ \abs{a_1})^{-1}
\]
by (ii) of Lemma \ref{lem:assocomps}. At this point the iteration becomes apparent. Finally, notice that the iteration terminates with the value $\alpha(\mathrm{Id}_X, \mathrm{Id}_X, \abs{a_n} \circ \cdots \circ \abs{a_1})^{-1}$, which cancels with the remaining term $\mathcal{L}(\mathrm{Id}_X, \abs{a_1 \otimes \cdots \otimes a_n})^{-1}$.
\end{proof}

\subsubsection{Center and derivations}

We start by forming an isomorphic chain complex 
\[
\left(\mathrm{Hom}_{\mathrm{Bim}(I_\mathcal{C}, I_\mathcal{C})^\mathcal{C}}(A^{\otimes n}, M), \mathsf{d}^n\right)_{n\ge 0}
\]
with differential
\[
\mathsf{d}^n := \varphi_{n+1} \circ (- \circ \partial_{n+1}) \circ \psi_n
\]
as pictured below.
\[
\begin{tikzcd}
0
\arrow[r]
&
\mathrm{Hom}_{\mathrm{Bim}(A, A)^\mathcal{C}}(A^{\otimes 2}, M)\
\arrow[r, "- \circ \partial_1"]
&
\mathrm{Hom}_{\mathrm{Bim}(A, A)^\mathcal{C}}(A^{\otimes 3}, M)
\arrow[r, "- \circ \partial_2"]
\arrow[d, dashrightarrow, "\varphi_1", bend left=20]
&
\mathrm{Hom}_{\mathrm{Bim}(A, A)^\mathcal{C}}(A^{\otimes 4}, M)
\arrow[r, "- \circ \partial_3"]
\arrow[d, dashrightarrow, "\varphi_2", bend left=20]
&
\cdots
\\
0
\arrow[r]
&
\mathrm{Hom}_{\mathrm{Bim}(I_\mathcal{C}, I_\mathcal{C})^\mathcal{C}}(I^\mathcal{C}, M)
\arrow[r, "\mathsf{d}_0"]
\arrow[u, "\psi_0"]
&
\mathrm{Hom}_{\mathrm{Bim}(I_\mathcal{C}, I_\mathcal{C})^\mathcal{C}}(A, M)
\arrow[r, "\mathsf{d}_1"]
\arrow[u, "\psi_1"]
&
\mathrm{Hom}_{\mathrm{Bim}(I_\mathcal{C}, I_\mathcal{C})^\mathcal{C}}(A^{\otimes 2}, M)
\arrow[r, "\mathsf{d}_2"]
\arrow[u, "\psi_2"]
&
\cdots
\end{tikzcd}
\]
It is easy to check that, consequently,
\[
\mathsf{d}^0(\xi)(a) = 
	\mathcal{R}(\abs{a}, \mathrm{Id}_Y)^{-1} \rho_L(a, \xi(1_Y))
	-
	\mathcal{L}(\mathrm{Id}_X, \abs{a})^{-1} \rho_R(\xi(1_X), a)
\]
for $\xi \in \mathrm{Hom}_{\mathrm{Bim}(I_\mathcal{C}, I_\mathcal{C})^\mathcal{C}}(I_\mathcal{C}, M)$ and $\abs{a}: X \to Y$, and for $n > 1$
\begin{align*}
\mathsf{d}^{n-1}(f)(a_1 \otimes \cdots \otimes a_n) & = 
	\prod_{i=2}^{n-1} \alpha(\abs{a_1}, \abs{a_i} \circ \cdots \circ \abs{a_2}, \abs{a_{i+1}})
	\rho_L(a_1, f(a_2 \otimes \cdots \otimes a_n))
		\\ & \phantom{{}={}}
	+ \sum_{i=1}^{n-1} (-1)^i \alpha(\abs{a_{i-1}} \circ \cdots \circ \abs{a_1}, \abs{a_i} \abs{a_{i+1}}) f(a_1 \otimes \cdots \otimes \mu(a_i, a_{i+1}) \otimes \cdots \otimes a_n)
		\\ & \phantom{{}={}}
	+ (-1)^n \rho_R(f(a_1 \otimes \cdots \otimes a_{n-1}), a_n)
\end{align*}
for $f\in \mathrm{Hom}_{\mathrm{Bim}(I_\mathcal{C}, I_\mathcal{C})^\mathcal{C}}(A^{\otimes n-1}, M)$, assuming $\mathrm{dom}(\abs{a_1}) = X$ and $\mathrm{codom}(\abs{a_n}) = Y$. 

We conclude by listing some consequences of Theorem \ref{thm:cohored}. Define the \emph{module of invariants} as
\[
M_\mathcal{C}^A := \{m \in M: \rho_L(a, m) = \zeta(\abs{a}) \rho_R(m, a)\}.
\]
Then we have the following immediate corollary.

\begin{corollary}
As long as unitors are chosen typically,
\[
\mathsf{HH}^0(A; M) = M_\mathcal{C}^A.
\]
In particular, when $M = A$,
\[
\mathsf{HH}^0(A) = Z^\mathcal{C}(A).
\]
\end{corollary}

Notice that the kernel of $\mathsf{d}^1$ consists of those maps $f: A \to M$ which satisfy
\[
f(\mu(a_1, a_2)) = \rho_L(a_1, f(a_2)) + \rho_R(f(a_1), a_2).
\]
That is, $\ker \mathsf{d}^1 = \mathrm{Der}(A, M)$ is the \emph{module of derivations from $A$ to $M$} in the usual sense.

\subsubsection{A note on Hochschild coshadows}

Before proceeding, we offer an explanation as to why the definition of Hochschild cohomology in the $\mathcal{C}$-graded setting requires no additional witnesses, despite the fact that a definition of Hochschild homology depended on a choice of $\varepsilon$. Recall that $\mathsf{HH}_0$ extends to a shadow with target category $\mathrm{Mod}^{\mathrm{Tr}(\mathcal{C})}$, but that the description depended on an isomorphism induced by $\varepsilon$. Classically, $\mathsf{HH}^0$ is an example of the dual concept, called a \emph{coshadow}. Outside of the observations related $\mathcal{C}$-graded structures, the information summarized below comes from Sections 3 and 4 of \cite{MR4768952}.

Unlike shadows, to define a coshadow on a bicategory, the bicategory must be assumed to be closed. A \emph{closed bicategory} is a bicategory $\mathcal{B}$ so that for each 1-morphism $f$ there is a right adjoint $f \triangleright -$ to the functor $- \odot f$ and similarly a left adjoint $- \triangleleft f$ to the functor $f \odot -$. Closed bicategories are known to come with natural isomorphisms
\[
t: (f \odot g) \triangleright h \xrightarrow{\cong} f \triangleright (g \triangleright h)
\]
called the \emph{tensor-hom adjunction}, and
\[
\overline{\mathfrak{a}}: (g \triangleright h) \triangleleft f \xrightarrow{\cong} g \triangleright (h \triangleleft f)
\qquad
\overline{\mathfrak{l}}: f \xrightarrow{\cong} f \triangleleft \mathrm{Id}_X
\qquad
\overline{\mathfrak{r}}: f \xrightarrow{\cong} \mathrm{Id}_X \triangleright f
\]
which are called the \emph{transposes} of the associator $\mathfrak{a}$, left unitor $\mathfrak{l}$, and right unitor $\mathfrak{r}$ of the bicategory $\mathcal{B}$.

\begin{definition}
A \emph{coshadow} for a closed bicategory $\mathcal{B}$ is the pair of a category $\mathbf{T}$ and a family of functors
\[
\left\{ \left \langle \! \left \langle - \right \langle \! \right \langle_X: \mathrm{End}_\mathcal{B}(X) \to \mathbf{T} \right\}_{X \in \mathrm{Ob}(\mathcal{B})}
\]
equipped with natural isomorphisms
\[
\vartheta_{f,g}: \left \langle \! \left \langle  f \triangleright g \right \langle \! \right \langle_X \xrightarrow{\cong} \left \langle \! \left \langle  g \triangleleft f  \right \langle \! \right \langle_Y
\]
whenever $f, g\in  \mathrm{Ob}(\mathrm{Hom}_\mathcal{B}(X, Y))$, such that the diagrams
\[
\begin{tikzcd}
\left \langle \! \left \langle (f \odot g) \triangleright h \right \langle \! \right \langle_X
\arrow[r, "\left \langle \! \left \langle t \right \langle \! \right \langle"] 
\arrow[d, "\vartheta"']
&
\left \langle \! \left \langle f \triangleright (g \triangleright h) \right \langle \! \right \langle_X
\arrow[r, "\vartheta"]
&
\left \langle \! \left \langle (g \triangleright h) \triangleleft f \right \langle \! \right \langle_Y
\arrow[d, "\left \langle \! \left \langle \overline{\mathfrak{a}} \right \langle \! \right \langle"]
\\
\left \langle \! \left \langle h \triangleleft (f \odot g) \right \langle \! \right \langle_Z
\arrow[r, "\left \langle \! \left \langle t \right \langle \! \right \langle"]
&
\left \langle \! \left \langle (h \triangleleft f) \triangleleft g \right \langle \! \right \langle_Z
\arrow[r, "\vartheta"]
&
\left \langle \! \left \langle g \triangleright (h \triangleleft f) \right \langle \! \right \langle_Y
\end{tikzcd}
\]
and
\[
\begin{tikzcd}
\left \langle \! \left \langle \mathrm{Id}_X \triangleright f \right \langle \! \right \langle_X
\arrow[r, "\vartheta"]
\arrow[dr, " \langle \!  \langle \overline{\mathfrak{r}}^{-1}  \langle \!  \langle"']
&
\left \langle \! \left \langle f \triangleright \mathrm{Id}_X \right \langle \! \right \langle_X
\arrow[r, "\vartheta"]
\arrow[d, " \langle \!  \langle \overline{\mathfrak{l}}^{-1}  \langle \!  \langle"]
&
\left \langle \! \left \langle \mathrm{Id}_X \triangleright f \right \langle \! \right \langle_X
\arrow[dl, "\langle \! \langle \overline{\mathfrak{r}}^{-1} \langle \! \langle"]
\\
&
\left \langle \! \left \langle f \right \langle \! \right \langle_X
&
\end{tikzcd}
\]
commute whenever they make sense.
\end{definition}

The most frequently used example of a bicategory, the bicategory of (non-$\mathcal{C}$-graded) bimodules over $\Bbbk$-algebras, is also an example of a closed bicategory: for bimodules ${}_AM_B$, ${}_BN_C$ and ${}_AP_C$, 
\[
P \triangleleft M := \mathrm{Hom}_{(A,-)}(M, P)
\qquad \text{and} \qquad
N \triangleright P := \mathrm{Hom}_{(-,C)}(N, P)
\]
are $(B, C)$- and $(A, B)$-bimodules respectively. With these choices, the natural isomorphism $t$ becomes the common tensor-hom adjunction
\[
t: \mathrm{Hom}_{(-, C)} (M \otimes_B N, P) \xrightarrow{\cong} \mathrm{Hom}_{(-,B)} (M, \mathrm{Hom}_{(-,C)}(N, P)).
\]
Indeed, since $\mathsf{HH}^0(A, M) \cong M^A \cong \mathrm{Hom}_{(A,A)} (A, M)$ for $M$ an $(A, A)$-bimodule, the zeroth Hochschild cohomology defines a coshadow $\left \langle \! \left \langle -\right \langle \! \right \langle_A:= \mathsf{HH}^0(A, -)$ on $\mathrm{Bim}$ with target category the category of $\Bbbk$-modules. Importantly, the isomorphism 
\[
\vartheta:  \left \langle \! \left \langle  M \triangleright N \right \langle \! \right \langle_A \xrightarrow{\cong} \left \langle \! \left \langle  N \triangleleft M  \right \langle \! \right \langle_B
\]
for ${}_AM_B$ and ${}_AN_B$ has a canonical choice given by compositions of the tensor-hom isomorphism.

Unfortunately, this means that the more relevant category $\mathrm{Bim}^\mathcal{C}$ of Example \ref{eg:bimc} is \emph{not} a closed bicategory by the typical internal Hom-functors of $\mathrm{Bim}$, as the hom-sets possess no $\mathcal{C}$-grading. Thus, we conclude by observing an interesting duality: on one hand Hochschild homology defines a shadow on the (not closed) bicategory $\mathrm{Bim}^\mathcal{C}$, but requires an additional choice; on the other hand, Hochschild cohomology for $\mathcal{C}$-graded structures does not require any additional choices, but does not define a coshadow.

\subsection{Homogeneous Hochschild cohomology}
\label{ss:nonhomoHH}

Following the classical setting, there should be a notion of Hochschild cohomology where the chain groups are maps $A^{\otimes n+2} \to M$ of $\mathcal{C}$-graded $(A, A)$-bimodules with homogeneous $\mathcal{C}$-degree which are not necessarily graded maps. Such a theory follows directly from the results of the previous two sections given the relevant background on 
\begin{enumerate}
	\item systems of $\mathcal{C}$-grading shifts, and
	\item composition of homogeneous functions preserving $\mathcal{C}$-shifting degree.
\end{enumerate}
For a complete discussion of these ideas, consult Sections 4.2--6 of \cite{naisse2020odd}. We will give a brief description here, omitting all proofs.

\begin{definition}
\label{def:shiftingsyst}
A \emph{$\mathcal{C}$-grading shift} $\varphi$ is a collection of maps
\[
\varphi = \{\varphi^{X \to Y}: \mathsf{D}^{X \to Y} \subset \mathrm{Hom}_{\mathcal{C}}(X, Y) \to \mathrm{Hom}_{\mathcal{C}}(X, Y)\}_{X, Y \in \mathrm{Ob}(\mathcal{C})}.
\]
We make the abbreviation $\varphi(g):= \varphi^{X \to Y}(g)$ whenever $g\in \mathsf{D}^{X \to Y}$. Then, a \emph{$\mathcal{C}$-shifting system} $S = (\mathcal{I}, \Phi, Z)$ consists of a monoid $\mathcal{I}$ with composition $\bullet: \mathcal{I} \times \mathcal{I} \to \mathcal{I}$, unit element $\mathbf{e} \in \mathcal{I}$, a collection of \emph{$\mathcal{C}$-grading shifts} $\Phi = \{\varphi_i\}_{i \in \mathcal{I}}$, and a choice of wide subcategory $Z$ of $\mathcal{C}$ such that
\begin{itemize}
	\item $\Sigma(\mathcal{C})$ is a wide subcategory of $Z$, where $\Sigma(\mathcal{C})$ is the wide subcategory of $\mathcal{C}$ whose Hom-sets are all commuting endomorphisms of $\mathcal{C}$,  $Z(\mathrm{End}_\mathcal{C}(X))$;
	\item $\varphi_\mathbf{e}$ is the \emph{neutral shift}, which has $\varphi_e^{X \to Y} = \mathrm{Id}_{\mathsf{D}_\mathbf{e}^{X \to Y}}$;
	\item For each pair of $\mathcal{C}$-grading shifts $\varphi_j^{Y \to Z}$ and $\varphi_i^{X \to Y}$, we have that
	\[
	\mathsf{D}_{i \bullet j}^{X \to Z} = \mathsf{D}_j^{Y \to Z} \circ \mathsf{D}_i^{X \to Y}
	\subset
	\mathrm{Hom}_\mathcal{C}(Y,Z) \circ \mathrm{Hom}_\mathcal{C}(X,Y)
	\subset
	\mathrm{Hom}_\mathcal{C}(X,Z)
	\]
	and the diagram
	\[
	\begin{tikzcd}
	\mathsf{D}_i^{X \to Y} \times \mathsf{D}_j^{Y \to Z}
	\arrow[d, "\varphi_i \times \varphi_j"]
	\arrow[r, "\circ"]
	&
	\mathsf{D}_{i \bullet j}^{Z \to X}
	\arrow[d, "\varphi_{i \bullet j}"]
	\\
	\mathrm{Hom}_\mathcal{C}(X, Y) \times \mathrm{Hom}_\mathcal{C}(Y, Z)
	\arrow[r, "\circ"]
	&
	\mathrm{Hom}_\mathcal{C}(X, Z)
	\end{tikzcd}
	\]
	commutes;
	\item There is a subset $\mathcal{I}_\mathrm{Id} \subset \mathcal{I}$ such that for each pair $X, Y \in \mathrm{Ob}(\mathcal{C})$, there is a partition of $\mathrm{Hom}_\mathcal{C}(X, Y)$ into $\bigsqcup_{i \in \mathcal{I}_\mathrm{Id}} \mathsf{D}_i^{X \to Y}$ and $\varphi_i = \mathrm{Id}_{\mathsf{D}_i^{X \to Y}}$ for each $i \in \mathcal{I}_{\mathrm{Id}}$.
\end{itemize}
It is also possible that $I$ contains an absoribing element $0$: in this case, we require that $\varphi_0$ is the \emph{null shift}, with $\mathsf{D}_0^{X\to Y} = \emptyset$.
\end{definition}

Then, for each $i \in \mathcal{I}$, there is a \emph{grading shift functor}
\[
\varphi_i: \mathrm{Mod}^\mathcal{C} \to \mathrm{Mod}^\mathcal{C}
\]
which sends elements in degree $g \in \mathsf{D}_i$ to elements in degree $\varphi_i(g)$, and elements not in $\mathsf{D}_i$ to zero. Then, the \emph{identity shift functor} is given by 
\[
\varphi_{\mathrm{Id}}:= \bigoplus_{i \in \mathcal{I}_\mathrm{Id}} \varphi_i.
\]
In particular, notice that the neutral shift $\varphi_{\mathbf{e}}$ preserves only $Z$ (see Remark 4.10 in \cite{naisse2020odd}). The identity shift functor, on the other hand, has $\varphi_{\mathrm{Id}}(M) \cong M$. We write 
\[
\widetilde{\mathcal{I}} := \mathcal{I} \sqcup \mathrm{Id}.
\]

To find a $\mathcal{C}$-shifting system \emph{compatible} with a choice of associator for $\mathcal{C}$ means to provide a collection of maps 
\[
\beta_{i, j}^{X \to Y \to Z}: \mathsf{D}_i^{X \to Y} \times \mathsf{D}_j^{Y \to Z} \to \Bbbk^\times
\]
satisfying the condition of Definition 4.11 in \cite{naisse2020odd}. These maps extend to canonical isomorphisms
\[
\beta_{i, j}: \varphi_i(M) \otimes \varphi_j(N) \to \varphi_{i \bullet j} (M \otimes N)
\]
defined by $m \otimes n \mapsto \beta_{i, j}(\abs{m}, \abs{n}) m \otimes n$, where $m$ and $n$ are homogeneous elements of the $\mathcal{C}$-graded $\Bbbk$-modules $M$ and $N$ respectively.

Suppose $S = (\mathcal{I}, \Phi, Z)$ is a $\mathcal{C}$-shifting system compatible with $\alpha$ through $\beta$. We assume that all $\mathcal{C}$-graded algebras $A$ are supported by $Z$ in the sense that $A_g = 0$ whenever $g\not \in \mathrm{Hom}_Z$; thus $\varphi_\mathbf{e}(A) \cong A$ for all $\mathcal{C}$-graded algebras. Then, the function of $Z$ is that we can define shifted bimodules $\varphi_i(M)$ for any $\mathcal{C}$-graded $(A, B)$-bimodule $M$ to be the $\mathcal{C}$-graded $\Bbbk$-module $M$ with $\mathcal{C}$-grading shifted by $\varphi_i$ and left and right actions given by the diagrams
\[
\begin{tikzcd}
	A \otimes \varphi_i(M) \arrow[r, dashrightarrow, "\varphi_i(\rho_L)"] \arrow[d, "\cong"'] & \varphi_i(M) \arrow[dd, equal]
	\\
	\varphi_\mathbf{e}(A) \otimes \varphi_i(M) \arrow[d, "\beta_{\mathbf{e}, i}"']&
	\\
	\varphi_{\mathbf{e} \bullet i}(A \otimes M) \arrow[r, "\rho_L"] & \varphi_{\mathbf{e} \bullet i}(M)
\end{tikzcd}
\qquad
\text{and}
\qquad
\begin{tikzcd}
	\varphi_i(M) \otimes B \arrow[r, dashrightarrow, "\varphi_i(\rho_R)"] \arrow[d, "\cong"'] & \varphi_i(M) \arrow[dd, equal]
	\\
	\varphi_i(M) \otimes \varphi_\mathbf{e}(B) \arrow[d, "\beta_{i, \mathbf{e}}"'] & 
	\\
	\varphi_{i \bullet \mathbf{e}} (M \otimes B) \arrow[r, "\rho_R"] & \varphi_{i \bullet \mathbf{e}}(M)
	\end{tikzcd}
\]
In particular, it is easy to show that if $M$ is a $\mathcal{C}$-graded $(A, B)$-bimodule and $N$ is a $\mathcal{C}$-graded $(B, C)$-bimodule, then the maps $\beta$ induce an isomorphism of $\mathcal{C}$-graded $(A, C)$-bimodules
\[
\varphi_i(M) \otimes_B \varphi_j(N) \cong \varphi_{i \bullet j} (M \otimes_B N).
\]

\begin{definition}
\label{def:homogeneousmaps}
Suppose $M$ and $N$ are $\mathcal{C}$-graded $(A, B)$-bimodules. A map $f: M \to N$ is called \emph{homogeneous of degree $i\in \widetilde{\mathcal{I}}$} if for all $m\in M$
\begin{itemize}
	\item $f(m) = 0$ whenever $\abs{m} \not\in \mathsf{D}_i$,
	\item $\abs{f(m)} = \varphi_i(\abs{m})$,
	\item $\rho_L(a, f(m)) = \beta_{\mathbf{e}, i}(\abs{a}, \abs{m}) f(\rho_L(a, m))$ for each homogeneous $a \in A$, and
	\item $\rho_R(f(m), b) = \beta_{i ,\mathbf{e}}(\abs{m}, \abs{b}) f(\rho_R(m, b))$ for each homogeneous $b\in B$.
\end{itemize}
\end{definition}

The main note here is that if $f: M \to N$ is homogeneous of degree $i$, then the induced map $\widetilde{f}: \varphi_i(M) \to N$ defined by $\widetilde{f}(\varphi_i(m)) := f(m)$ is a graded map.

Suppose $f:M \to L$ and $g:N \to L$ are maps of $\mathcal{C}$-graded $(A, B)$- and $(B, C)$-modules respectively, and that each can be written as the sum of homogeneous parts; \textit{e.g.}, $f = \sum_{j \in J \subset \widetilde{\mathcal{I}}}$. Then there is a natural tensor product 
\[
g \otimes f = \sum_{k \in \widetilde{\mathcal{I}}} (f \otimes g)_k : M \otimes_B N \to L
\]
where
\[
(f \otimes g)_k(m \otimes n) = \sum_{i \bullet j = k} \beta_{\abs{f_{i}}, \abs{g_{j}}}(\abs{m}, \abs{n}) f_{i}(m) \otimes g_{j}(n)
\]
for any homogeneous elements $m \in M$ and $n \in N$. 

The final step in defining a category of $\mathcal{C}$-graded bimodules with homogeneous maps, rather than graded maps, is defining a composition of homogeneous maps which is itself homogeneous. To do this, it is necessary to pass to a natural specialization of the $\mathcal{C}$-shifting system.

\begin{definition}
A \emph{$\mathcal{C}$-shifting 2-system} is a $\mathcal{C}$-shifting system $S = \{\mathcal{I}, \Phi, Z\}$ such that $\mathcal{I}$ is additionally equipped with an associative vertical composition map
\[
\circ: \mathcal{I} \times \mathcal{I} \to \mathcal{I}
\]
such that, for each $i, i', j, j' \in \mathcal{I}$, 
\begin{itemize}
	\item $\mathbf{e} \circ \mathbf{e} = \mathbf{e}$,
	\item	$\mathsf{D}_{j \circ i} = \mathsf{D}_i \cap \varphi_i^{-1} (\mathsf{D}_j)$,
	\item $\varphi_{j \circ i} = \varphi_j|_{\varphi_i(\mathsf{D}_i) \cap \mathsf{D}_j} \circ \varphi_i|_{\mathsf{D}_{j \circ i}}$, and
	\item $\varphi_{(j' \circ i') \bullet (j \circ i)} = \varphi_{(j' \bullet j) \circ (i' \bullet i)}$.
\end{itemize}
\end{definition}

To find a $\mathcal{C}$-shifting 2-system \emph{compatible} with a choice of associator $\alpha$ and choice of underlying (compatible by way of $\beta$) $\mathcal{C}$-shifting system means to provide two more collections:
\[
\gamma_{i, j}^{X \to Y}: \mathsf{D}_i \to \Bbbk^\times
\qquad
\text{and}
\qquad
\Xi_{\substack{i, i' \\ j, j'}}^{X \to Y \to Z} \in \Bbbk^\times
\]
which satisfy the conditions of Definition 4.30 of \cite{naisse2020odd}. The value of these maps is that they, in analogy with the $\beta$ maps, extend to canonical isomorphisms
\[
\varphi_j \circ \varphi_i (M) \to \varphi_{j \circ i}(M)
\]
given by $m \mapsto \gamma_{i, j}(\abs{m}) m$ and
\[
\varphi_{(j' \circ i') \bullet (j \circ i)}(M) \to \varphi_{(j' \bullet j) \circ (i' \bullet i)}(M)
\]
given by $m \mapsto \Xi_{\substack{i, i' \\ j, j'}} (\abs{m}) m$.

Suppose $A$ and $B$ are two $\mathcal{C}$-graded algebras. Finally, we denote by $\mathrm{BIM}(A, B)^\mathcal{C}$ the category whose objects are $\mathcal{C}$-graded $(A, B)$-bimodules, and whose morphisms are homogeneous maps (rather than just graded maps). The composition in this category is called \emph{$\mathcal{C}$-graded composition}, and is defined by 
\begin{equation}
\label{eq:cgradedcomp}
(g \circ_{\mathcal{C}} f)(m) := \gamma_{i, j}(\abs{m})^{-1} (g \circ f)(m).
\end{equation}
Notice that $\mathrm{Bim}(A, B)^\mathcal{C}$ is a full subcategory of $\mathrm{BIM}(A, B)^\mathcal{C}$. In particular, the hom-sets in this category satisfy
\[
\mathrm{Hom}_{\mathrm{BIM}(A, B)^\mathcal{C}} (M, N)
=
\bigoplus_{i \in \widetilde{\mathcal{I}}} \mathrm{Hom}_{\mathrm{Bim}(A, B)^\mathcal{C}} (\varphi_i(M), N).
\]
From this observation, Theorem \ref{thm:cohored} immediately implies the following.

\begin{corollary}
Suppose $A$ is a $\mathcal{C}$-graded algebra and $M$ is a $\mathcal{C}$-graded $(A, A)$-bimodule. Then
\[
\mathrm{Hom}_{\mathrm{BIM}(I_\mathcal{C}, I_\mathcal{C})^\mathcal{C}}(A^{\otimes n}, M)
\cong
\mathrm{Hom}_{\mathrm{BIM}(A,A)^\mathcal{C}}(A^{\otimes n+2}, M).
\]
\end{corollary}

More importantly, we obtain a notion of Hochschild cohomology for $\mathcal{C}$-graded algebras in higher generality, defining $\mathsf{HH}^{k, i}(A; M)$ as the $k$th cohomology group of the cochain complex 
\[
\begin{tikzcd}
	0 
	\arrow[r] 
	&
	\mathrm{Hom}_{\mathrm{BIM}(A,A)^\mathcal{C}} (A^{\otimes 2}, M)
	\arrow[r, "- \circ \partial_1"]
	&
	\mathrm{Hom}_{\mathrm{BIM}(A,A)^\mathcal{C}} (A^{\otimes 3}, M)
	\arrow[r, "- \circ \partial_2"]
	&
	\mathrm{Hom}_{\mathrm{BIM}(A,A)^\mathcal{C}} (A^{\otimes 4}, M)
	\arrow[r, "- \circ \partial_3"]
	&
	\cdots
\end{tikzcd}
\]
restricting to those maps of homogeneous degree $i\in \widetilde{\mathcal{I}} = \mathcal{I} \sqcup \{\mathrm{Id}\}$. This makes sense since $\partial_k$ is grading preserving. Thus, we see that Hochschild cohomology is $\mathbb{Z} \times \widetilde{\mathcal{I}}$-graded; in particular, the Hochschild cohomology of the previous two subsections was $\mathsf{HH}^{k, \mathrm{Id}} (A; M)$. We conclude that Theorem \ref{thm2} follows by definition.

\subsubsection{Derived Hom and Hochschild cohomology}

Briefly, we note that the Hochschild cohomology in the $\mathcal{C}$-graded setting also has a description in terms of a derived $\mathrm{Hom}$. Fix two $\mathcal{C}$-graded  DG-$(A, B)$-bimodules $(M, d_M)$ and $(N, d_N)$---for ease of exposition, we assume that each has $\mathcal{C}$-grading preserving differentials. Define their \emph{Hom complex} as the $\mathbb{Z} \times \widetilde{\mathcal{I}}$-graded chain complex $\mathrm{Hom}^\bullet(M, N) = (\mathrm{Hom}^k(M, N), D_k)$ where
\[
\mathrm{Hom}^k(M, N) = \bigoplus_{i \in \widetilde{\mathcal{I}}} \left( \prod_{k = p + q} \mathrm{Hom}_{\mathrm{Bim}(A, B)^\mathcal{C}} (\varphi_i(M^{-q}), N^p) \right)
\]
and
\[
D(f) = d_M \circ_\mathcal{C} f - (-1)^n f \circ_\mathcal{C} d_N.
\]
Thus, if $f$ has homogeneous degree $(n, j) \in \mathbb{Z} \times \widetilde{\mathcal{I}}$, then $D(f)$ has homogeneous degree $(n+1, j)$.

As in \cite{naisse2020odd}, we say that a $\mathcal{C}$-graded DG-$(A,A)$-bimodule is \emph{relatively projective} if it is a direct summand of a direct sum of shifted copies (in both $\mathcal{C}$-degree and homological degree) of $A$ (trivially interpreted as a free DG-$(A,A)$-bimodule). Then, a $\mathcal{C}$-graded DG-bimodule is \emph{cofibrant} if it is a direct summand of the inverse limit of a filtration
\[
0 = F_0 \subset F_1 \subset F_2 \subset \cdots \subset F_r \subset F_{r+1} \subset \cdots,
\]
where each $F_{r+1}/F_r$ is isomorphic to a relatively projective DG-bimodule. That is, any ($\mathcal{C}$-graded) DG-module with \emph{property $\mathrm{(P)}$}, as in \cite{MR1258406}, is homotopy equivalent to a cofibrant ($\mathcal{C}$-graded) DG-module. Naisse and Putyra provide the following for us.

\begin{proposition}[Proposition 4.27 in \cite{naisse2020odd}]
For any $\mathcal{C}$-graded DG-bimodule $M$, there exists a cofibrant DG-module $\mathbf{p}(M)$ with surjective quasi-isomorphism
\[
\mathbf{p}(M) \twoheadrightarrow M.
\]
\end{proposition}

Notice that if $M = A$, then $\mathbf{p}(A)$ can be taken as $\mathcal{B}(A)$. Finally, if $M$ and $N$ are two $\mathcal{C}$-graded DG-$(A, B)$-bimodules, the \emph{derived Hom} is defined as
\[
R\mathrm{Hom}(M, N) := \mathrm{Hom}^\bullet(\mathbf{p}(M), N)
\]
for $\mathbf{p}(M)$ any projective resolution of $M$. One can verify that $R\mathrm{Hom}(M, N)$ is independent of the choice of cofibrant DG-module up to quasi-isomorphism via the standard arguments. Since $\mathcal{B}(A)$ is a projective resolution of $A$, we obtain the following corollary immediately.

\begin{corollary}
If $M$ is a $\mathcal{C}$-graded $(A, A)$-bimodule, then
\[
H^\bullet(R\mathrm{Hom}(A, M)) \cong \mathsf{HH}^\bullet(A, M).
\]
\end{corollary}

\section{Odd arc algebras and the grading category $\mathcal{G}$}
\label{s:odd}

The formalism of grading categories was first defined and successfully utilized by Naisse and Putyra \cite{naisse2020odd} to give the first known extension of odd Khovanov homology to tangles (see also \cite{MR4190457} and \cite{schelstraete2023odd}). Subsections \ref{ss:chroncob}  and \ref{ss:oddarcs} are background (found in, say, \cite{naisse2020odd} or \cite{putyra20152categorychronologicalcobordismsodd}) in which we define the linearized category of chronological cobordisms and the odd and unified arc algebras (the latter of which were first considered in \cite{naisse2017odd}). In Subsection \ref{ss:tracecat}, we recall Naisse and Putyra's grading category $(\mathcal{G}, \alpha)$, and show that their grading category admits a looper in the sense of Definition \ref{def:looper}.

\subsection{Chronological cobordisms}
\label{ss:chroncob}

\begin{definition}
A \emph{chronological cobordism} between two closed 1-manifolds $S_0$ and $S_1$ is a cobordism $W: S_0 \to S_1$ embedded in $\mathbb{R}^2 \times [0,1]$ so that
\begin{enumerate}
	\item $W$ is collared; \textit{i.e.}, there is an $\epsilon >0$ so that 
	\[
	W \cap (\mathbb{R}^2 \times [0, \epsilon]) = S_0 \times [0, \epsilon]
	 \qquad \text{and} \qquad
	W \cap (\mathbb{R}^2 \times [1- \epsilon, 1]) = S^1 \times [1-\epsilon, 1]
	\]
	and
	\item the height function $h: W \to [0,1]$ given by projection onto the third coordinate is a \emph{separative} Morse function; \textit{i.e.}, $h$ is Morse and $h^{-1}(\{c\})$ contains exactly one critical point of $h$ whenever $c$ is a critical value of $h$.
\end{enumerate}
In this paper, we assume that every chronological cobordism comes with a \emph{framing}, by which we mean a choice of orientation on a basis for each unstable manifold $W_p \subset W$ whenever $p$ is a critical point of $h$ of index one or two.
\end{definition}

It is standard to visualize the framing by an arrow through critical points, which is then adapted to planar diagrams. For example:
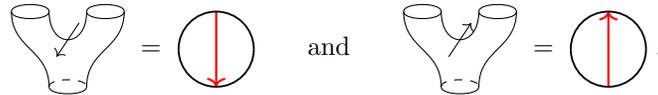
\[
\tikz[baseline={([yshift=-.5ex]current bounding box.center)}, scale=.5]{
	\draw  (1,2) .. controls (1,3) and (0,3) .. (0,4);
	\draw  (2,2) .. controls (2,3) and (3,3) .. (3,4);
	\draw (1,4) .. controls (1,3) and (2,3) .. (2,4);
	\draw (0,4) .. controls (0,3.75) and (1,3.75) .. (1,4);
	\draw (0,4) .. controls (0,4.25) and (1,4.25) .. (1,4);
	\draw (2,4) .. controls (2,3.75) and (3,3.75) .. (3,4);
	\draw (2,4) .. controls (2,4.25) and (3,4.25) .. (3,4);
	\draw (1,2) .. controls (1,1.75) and (2,1.75) .. (2,2);
	\draw[dashed] (1,2) .. controls (1,2.25) and (2,2.25) .. (2,2);
        \draw[->] (1.8,3.7) -- (1.2,2.8);
} ~=~ \tikz[baseline={([yshift=-.5ex]current bounding box.center)}, scale = .5]{
    \draw[red,thick,->] (0,1) -- (0,-1);
    \draw[knot] (-1,0) arc (180:360:1 and 1);
    \draw[knot] (1,0) arc (0:180:1 and 1);
    } \qquad \text{and} \qquad 
    \tikz[baseline={([yshift=-.5ex]current bounding box.center)}, scale=.5]{
	\draw  (1,2) .. controls (1,3) and (0,3) .. (0,4);
	\draw  (2,2) .. controls (2,3) and (3,3) .. (3,4);
	\draw (1,4) .. controls (1,3) and (2,3) .. (2,4);
	\draw (0,4) .. controls (0,3.75) and (1,3.75) .. (1,4);
	\draw (0,4) .. controls (0,4.25) and (1,4.25) .. (1,4);
	\draw (2,4) .. controls (2,3.75) and (3,3.75) .. (3,4);
	\draw (2,4) .. controls (2,4.25) and (3,4.25) .. (3,4);
	\draw (1,2) .. controls (1,1.75) and (2,1.75) .. (2,2);
	\draw[dashed] (1,2) .. controls (1,2.25) and (2,2.25) .. (2,2);
        \draw[<-] (1.8,3.7) -- (1.2,2.8);
} ~=~ \tikz[baseline={([yshift=-.5ex]current bounding box.center)}, scale = .5]{
    \draw[red,thick,<-] (0,1) -- (0,-1);
    \draw[knot] (-1,0) arc (180:360:1 and 1);
    \draw[knot] (1,0) arc (0:180:1 and 1);
    }~.
\]
Notice that each chronological cobordism admits a unique handle decomposition. Thus, each framed chronological cobordism can be described as a sequence of the following \emph{elementary chronological cobordisms}, with identity tubes elsewhere.
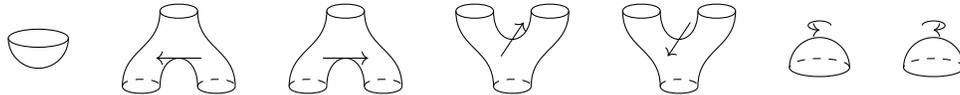
\[
\tikz[baseline={([yshift=-.5ex]current bounding box.center)}, scale=.2]{
    \draw [domain=180:360] plot ({2*cos(\x)}, {2*sin(\x)});
    \draw (-2,0) arc (180:360:2 and 0.6);
    \draw (2,0) arc (0:180:2 and 0.6);
}
\qquad
\tikz[baseline={([yshift=-.5ex]current bounding box.center)}, scale=.5]{
	\draw (0,0) .. controls (0,1) and (1,1) .. (1,2);
	\draw (1,0) .. controls (1,1) and (2,1) .. (2,0);
	\draw (3,0) .. controls (3,1) and (2,1) .. (2,2);
	\draw (0,0) .. controls (0,-.25) and (1,-.25) .. (1,0);
	\draw[dashed] (0,0) .. controls (0,.25) and (1,.25) .. (1,0);
	\draw (2,0) .. controls (2,-.25) and (3,-.25) .. (3,0);
	\draw[dashed] (2,0) .. controls (2,.25) and (3,.25) .. (3,0);
	\draw (1,2) .. controls (1,1.75) and (2,1.75) .. (2,2);
	\draw (1,2) .. controls (1,2.25) and (2,2.25) .. (2,2);
        \draw[<-] (0.9,0.75) -- (2.1,0.75);
}
\qquad
\tikz[baseline={([yshift=-.5ex]current bounding box.center)}, scale=.5]{
	\draw (0,0) .. controls (0,1) and (1,1) .. (1,2);
	\draw (1,0) .. controls (1,1) and (2,1) .. (2,0);
	\draw (3,0) .. controls (3,1) and (2,1) .. (2,2);
	\draw (0,0) .. controls (0,-.25) and (1,-.25) .. (1,0);
	\draw[dashed] (0,0) .. controls (0,.25) and (1,.25) .. (1,0);
	\draw (2,0) .. controls (2,-.25) and (3,-.25) .. (3,0);
	\draw[dashed] (2,0) .. controls (2,.25) and (3,.25) .. (3,0);
	\draw (1,2) .. controls (1,1.75) and (2,1.75) .. (2,2);
	\draw (1,2) .. controls (1,2.25) and (2,2.25) .. (2,2);
        \draw[->] (0.9,0.75) -- (2.1,0.75);
}
\qquad
\tikz[baseline={([yshift=-.5ex]current bounding box.center)}, scale=.5]{
	\draw  (1,2) .. controls (1,3) and (0,3) .. (0,4);
	\draw  (2,2) .. controls (2,3) and (3,3) .. (3,4);
	\draw (1,4) .. controls (1,3) and (2,3) .. (2,4);
	\draw (0,4) .. controls (0,3.75) and (1,3.75) .. (1,4);
	\draw (0,4) .. controls (0,4.25) and (1,4.25) .. (1,4);
	\draw (2,4) .. controls (2,3.75) and (3,3.75) .. (3,4);
	\draw (2,4) .. controls (2,4.25) and (3,4.25) .. (3,4);
	\draw (1,2) .. controls (1,1.75) and (2,1.75) .. (2,2);
	\draw[dashed] (1,2) .. controls (1,2.25) and (2,2.25) .. (2,2);
        \draw[<-] (1.8,3.7) -- (1.2,2.8);
}
\qquad
\tikz[baseline={([yshift=-.5ex]current bounding box.center)}, scale=.5]{
	\draw  (1,2) .. controls (1,3) and (0,3) .. (0,4);
	\draw  (2,2) .. controls (2,3) and (3,3) .. (3,4);
	\draw (1,4) .. controls (1,3) and (2,3) .. (2,4);
	\draw (0,4) .. controls (0,3.75) and (1,3.75) .. (1,4);
	\draw (0,4) .. controls (0,4.25) and (1,4.25) .. (1,4);
	\draw (2,4) .. controls (2,3.75) and (3,3.75) .. (3,4);
	\draw (2,4) .. controls (2,4.25) and (3,4.25) .. (3,4);
	\draw (1,2) .. controls (1,1.75) and (2,1.75) .. (2,2);
	\draw[dashed] (1,2) .. controls (1,2.25) and (2,2.25) .. (2,2);
        \draw[->] (1.8,3.7) -- (1.2,2.8);
}
\qquad
\tikz[baseline={([yshift=-.5ex]current bounding box.center)}, scale=.2]{
        \draw[domain=0:180] plot ({2*cos(\x)}, {2*sin(\x)});
        \draw (-2,0) arc (180:360:2 and 0.6);
        \draw[dashed] (2,0) arc (0:180:2 and 0.6);
        \draw[->] (0,2.8) [partial ellipse=0:270:5ex and 2ex];
}
\qquad
\tikz[baseline={([yshift=-.5ex]current bounding box.center)}, scale=.2]{
        \draw[domain=0:180] plot ({2*cos(\x)}, {2*sin(\x)});
        \draw (-2,0) arc (180:360:2 and 0.6);
        \draw[dashed] (2,0) arc (0:180:2 and 0.6);
        \draw[<-] (0,2.8) [partial ellipse=-90:180:5ex and 2ex];
}
\]
Technically, we need to include the swap cobordism as well---that is, the identity cobordism which simply transposes 1-manifolds.

Two chronological cobordisms are considered equivalent if they can be related by a diffeotopy $H_t$, for $t\in [0,1]$, so that projection of $H_t(W)$ onto the third coordinate is a separative Morse function at each time $t$. A diffeotopy breaking this condition is called a change of chronology:

\begin{definition}
A \emph{change of chronology} is a diffeotopy $H_t$ such that the projection of $H_t(W)$ onto the third coordinate is a generic homotopy of Morse functions, together with a smooth choice of framings on $H_t(W)$. Two changes of chronology between equivalent chronological cobordisms are equivalent if they are homotopic in the space of oriented Igusa functions after composing with the equivalences of cobordisms; for a thorough description, consult \cite{putyra20152categorychronologicalcobordismsodd}.
\end{definition}

We write $H: W_0 \Rightarrow W_1$ for a change of chronology $H$ between chronological cobordisms $W_0$ and $W_1$. There are two different ways of composing changes of chronology:
\begin{enumerate}
	\item Given a sequence of cobordisms $A \xrightarrow{W} B \xrightarrow{W'} C$ and changes of chronology $H$ on $W$ and $H'$ on $W'$, denote by $H' \circ H$ the change of chronology on $W' \circ W$. 
	\item Given a sequence of changes of chronology $W \xRightarrow{H} W' \xRightarrow{H'} W''$, denote their composition by $H' \star H$.
\end{enumerate}

Finally, there is a special family of changes of chronology we consider frequently.

\begin{definition}
A change of chronology $H$ on a chronological cobordism $W$ is called \emph{locally vertical} if there is a finite collection of cylinders $\{C_i\}_i$ in $\mathbb{R}^2 \times I$ such that $H$ is the identity on $W - \bigcup_i C_i$.
\end{definition}

The main utility of locally vertical changes of chronology is that they are unique up to homotopy.

\begin{proposition}[Proposition 4.4 of \cite{putyra20152categorychronologicalcobordismsodd}]
\label{putyrahammer}
If $H$ and $H'$ are locally vertical changes of chronology (with respect to the same cylinders) with the same source and target, then they are homotopic in the space of framed diffeotopies.
\end{proposition}

\subsubsection{Linearized chronological cobordism category}

\begin{definition}
Let $\textbf{ChCob}$ denote the strict 2-category whose 
\begin{itemize}
	\item objects are finite disjoint collections of circles in $\mathbb{R}^2$,
	\item 1-morphisms $\mathrm{Hom}(S_0, S_1)$ are classes of embedded chronological cobordisms from $S_0$ to $S_1$, with composition $\circ_1$ defined by stacking, and
	\item 2-morphisms are given by classes of changes of chronology up to equivalence with vertical and horizontal compositions of 2-morphisms given by $\star$ and $\circ$, defined above.
\end{itemize}
Proceeding, we drop the subscripts on $\circ_i$. Each 1-morphism is endowed with a $(\mathbb{Z} \times \mathbb{Z})$-grading, defined as
\[
\deg(W) := \abs{W} := (\# \text{births} - \# \text{merges}, \# \text{deaths} - \# \text{splits}).
\]
\end{definition}

Mirroring the fact that all cobordisms can be decomposed into sequences of elementary chronological cobordisms, any change of chronology can be decomposed into a sequence of \emph{elementary changes of chronology}; these are pictured in Figure \ref{fig:elementaryCoC} and are exactly those pairs of cobordisms described in the commutation chart (\textit{i.e.}, Figure 2) of \cite{MR3071132}. 

\begin{figure}
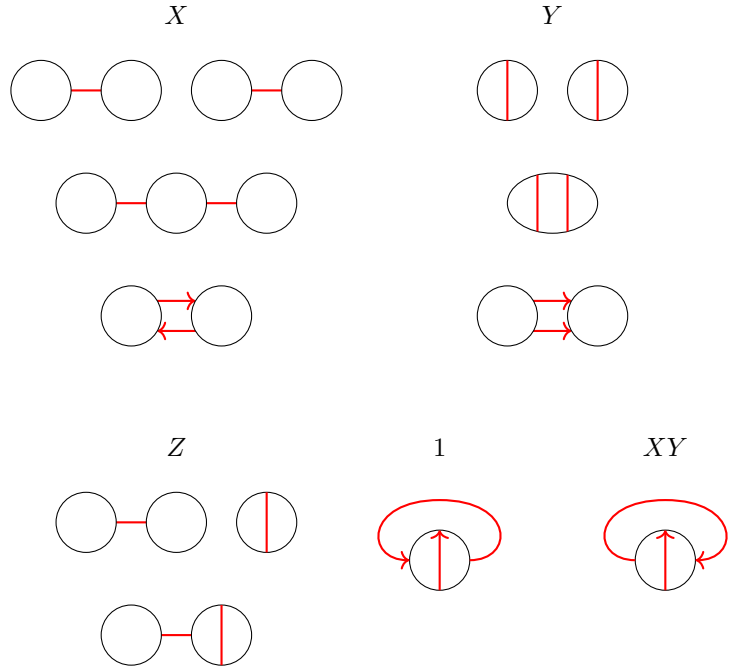

\[
\tikz{
    \node at (-3,1) {$\tikz{
  	\node at (0,0) {$X$};
  	\draw[red, thick] (-1.8,-1) -- (-.6,-1);
  	\draw[red, thick] (.6,-1) -- (1.8,-1);
  	\draw[fill=white] (-1.8,-1) circle (.4cm);
  	\draw[fill=white] (-.6,-1) circle (.4cm);
  	\draw[fill=white] (.6,-1) circle (.4cm);
  	\draw[fill=white] (1.8,-1) circle (.4cm);
  	\draw[red, thick] (-1.2,-2.5) -- (1.2,-2.5);
  	\draw[fill=white] (-1.2,-2.5) circle (.4cm);
  	\draw[fill=white] (0,-2.5) circle (.4cm);
  	\draw[fill=white] (1.2,-2.5) circle (.4cm);
  	\draw[red, thick,->] (-.6,-3.8) -- (.254,-3.8);
  	\draw[red, thick,->] (.6,-4.2) -- (-.254,-4.2);
  	\draw[fill=white] (-.6,-4) circle (.4cm);
  	\draw[fill=white] (.6,-4) circle (.4cm);
}$};
    \node at (2,1) {$\tikz{
  	\node at (0,0) {$Y$};
  	\draw[fill=white] (-.6,-1) circle (.4cm);
  	\draw[fill=white] (.6,-1) circle (.4cm);
  	\draw[red, thick] (-.6,-1.4) -- (-.6,-.6);
  	\draw[red, thick] (.6,-1.4) -- (.6,-.6);
  	\draw[fill=white] (0,-2.5) ellipse (.6cm and .4cm);
	\begin{scope}
	   	\clip (0,-2.5) ellipse (.6cm and .4cm);
  		\draw[red, thick] (-.2,-2.9) -- (-.2,-2.1);
  		\draw[red, thick] (.2,-2.9) -- (.2,-2.1);
	\end{scope}
  	\draw[red, thick,->] (-.6,-3.8) -- (.254,-3.8);
  	\draw[red, thick,->] (-.6,-4.2) -- (.254,-4.2);
  	\draw[fill=white] (-.6,-4) circle (.4cm);
  	\draw[fill=white] (.6,-4) circle (.4cm);
}$};
    \node at (-3, -4) {$\tikz{
  	\node at (0,0) {$Z$};
  	\draw[red, thick] (-1.2,-1) -- (0,-1);
  	\draw[fill=white] (-1.2,-1) circle (.4cm);
  	\draw[fill=white] (0,-1) circle (.4cm);
  	\draw[fill=white] (1.2,-1) circle (.4cm);
  	\draw[red, thick] (1.2,-1.4) -- (1.2,-.6);
  	\draw[red, thick] (-.6,-2.5) -- (.6,-2.5);
  	\draw[fill=white] (-.6,-2.5) circle (.4cm);
  	\draw[fill=white] (.6,-2.5) circle (.4cm);
  	\draw[red, thick] (.6,-2.9) -- (.6,-2.1);
  }$};
    \node at (2,-4) {$\tikz{
  	\node at (-1.5,0) {$1$};
  	\draw[fill=white] (-1.5,-1.5) circle (.4cm);
  	\draw[red, thick,->] (-1.5,-1.9) -- (-1.5,-1.1);
  	\draw[red, thick,->] (-1.1,-1.5) .. controls (-.5,-1.5) and (-.5,-.7) .. (-1.5,-.7)
  		.. controls (-2.5,-.7) and (-2.5,-1.5) .. (-1.9,-1.5);
  	\node at (1.5,0) {$XY$};
  	\draw[fill=white] (1.5,-1.5) circle (.4cm);
  	\draw[red, thick,->] (1.5,-1.9) -- (1.5,-1.1);
  	\draw[red, thick,->] (1.1,-1.5) .. controls (.5,-1.5) and (.5,-.7) .. (1.5,-.7)
  		.. controls (2.5,-.7) and (2.5,-1.5) .. (1.9,-1.5);
  	\draw[opacity=0] (.6,-2.5) circle (.4cm);
  }$};
}
\]
\caption{This is the collection of elementary changes of chronologies, together with their evaluation by $\iota$. Notice that taking $X=Z=1$ and $Y=-1$ yields the commutation chart of~\cite{MR3071132}. Framings are omitted if evaluation by $\iota$ does not depend on them. Finally, for those elementary cobordims $H$ which $\iota(H) = Z$, it is assumed that $H$ takes a merge followed by a split to a split followed by a merge. If the opposite is true, $\iota(H) = Z^{-1}$.}
\label{fig:elementaryCoC}
\end{figure}

Consider the commutataive ring
\[
R := \mathbb{Z}[X, Y, Z^{\pm1}] \big/ (X^2 = Y^2 = 1).
\]
As in \cite{naisse2020odd} and \cite{spyropoulos2024}, we denote by $\iota$ the map that associates to each elementary change of chronology the unit in $R$ as pictured in Figure \ref{fig:elementaryCoC}. Putyra proves this association is well-defined in \cite{putyra20152categorychronologicalcobordismsodd}; $\iota$ then extends to any change of chronology, since 
\[
\iota(H' \circ H) = \iota(H')\iota(H)
\qquad\text{and}\qquad
\iota(H' \star H) = \iota(H') \iota(H).
\]

\begin{definition}
The linearized category of cobordisms $R\textbf{ChCob}$ is the $R$-linear monoidal category whose objects are the same as $\textbf{ChCob}$, and whose morphisms are $R$-linear combinations of morphisms in $\textbf{ChCob}$ modulo the relation that 
\[
W' = \iota(H) W
\]
for each change of chronology $H: W\Rightarrow W'$. The monoidal product is given by juxtaposition, where we specify the chronology on $W \otimes W'$ by applying $W$ first and then $W'$ (that is, left-then-right juxtaposition).
\end{definition}

\begin{remark}
In Figure \ref{fig:elementaryCoC}, there is a choice made for the two ``ladybug'' configurations. There is another choice which swaps the values associated to these configurations. This amounts to picking the ``type'' of odd Khovanov homology obtained in the end: as presented, we obtain the type-$Y$ theory. Consult \cite{putyra20152categorychronologicalcobordismsodd} and \cite{migdail2024functoriality} for more on the differences between the (ultimately isomorphic) type-$Y$ and type-$X$ theories.
\end{remark}

Finally, it is helpful to observe that $R\textbf{ChCob}$ admits a combinatorial description. First, we have the following \emph{change of framing} moves.
\begin{equation}
\label{eq:changeofframingmoves}
\tikz[baseline={([yshift=-.5ex]current bounding box.center)}, scale=.5]{
	\draw (0,0) .. controls (0,1) and (1,1) .. (1,2);
	\draw (1,0) .. controls (1,1) and (2,1) .. (2,0);
	\draw (3,0) .. controls (3,1) and (2,1) .. (2,2);
	\draw (0,0) .. controls (0,-.25) and (1,-.25) .. (1,0);
	\draw[dashed] (0,0) .. controls (0,.25) and (1,.25) .. (1,0);
	\draw (2,0) .. controls (2,-.25) and (3,-.25) .. (3,0);
	\draw[dashed] (2,0) .. controls (2,.25) and (3,.25) .. (3,0);
	\draw (1,2) .. controls (1,1.75) and (2,1.75) .. (2,2);
	\draw (1,2) .. controls (1,2.25) and (2,2.25) .. (2,2);
        \draw[<-] (0.9,0.75) -- (2.1,0.75);
}
 = X~
 \tikz[baseline={([yshift=-.5ex]current bounding box.center)}, scale=.5]{
	\draw (0,0) .. controls (0,1) and (1,1) .. (1,2);
	\draw (1,0) .. controls (1,1) and (2,1) .. (2,0);
	\draw (3,0) .. controls (3,1) and (2,1) .. (2,2);
	\draw (0,0) .. controls (0,-.25) and (1,-.25) .. (1,0);
	\draw[dashed] (0,0) .. controls (0,.25) and (1,.25) .. (1,0);
	\draw (2,0) .. controls (2,-.25) and (3,-.25) .. (3,0);
	\draw[dashed] (2,0) .. controls (2,.25) and (3,.25) .. (3,0);
	\draw (1,2) .. controls (1,1.75) and (2,1.75) .. (2,2);
	\draw (1,2) .. controls (1,2.25) and (2,2.25) .. (2,2);
        \draw[->] (0.9,0.75) -- (2.1,0.75);
}
\qquad
\tikz[baseline={([yshift=-.5ex]current bounding box.center)}, scale=.5]{
	\draw  (1,2) .. controls (1,3) and (0,3) .. (0,4);
	\draw  (2,2) .. controls (2,3) and (3,3) .. (3,4);
	\draw (1,4) .. controls (1,3) and (2,3) .. (2,4);
	\draw (0,4) .. controls (0,3.75) and (1,3.75) .. (1,4);
	\draw (0,4) .. controls (0,4.25) and (1,4.25) .. (1,4);
	\draw (2,4) .. controls (2,3.75) and (3,3.75) .. (3,4);
	\draw (2,4) .. controls (2,4.25) and (3,4.25) .. (3,4);
	\draw (1,2) .. controls (1,1.75) and (2,1.75) .. (2,2);
	\draw[dashed] (1,2) .. controls (1,2.25) and (2,2.25) .. (2,2);
        \draw[<-] (1.8,3.7) -- (1.2,2.8);
}
= Y~
\tikz[baseline={([yshift=-.5ex]current bounding box.center)}, scale=.5]{
	\draw  (1,2) .. controls (1,3) and (0,3) .. (0,4);
	\draw  (2,2) .. controls (2,3) and (3,3) .. (3,4);
	\draw (1,4) .. controls (1,3) and (2,3) .. (2,4);
	\draw (0,4) .. controls (0,3.75) and (1,3.75) .. (1,4);
	\draw (0,4) .. controls (0,4.25) and (1,4.25) .. (1,4);
	\draw (2,4) .. controls (2,3.75) and (3,3.75) .. (3,4);
	\draw (2,4) .. controls (2,4.25) and (3,4.25) .. (3,4);
	\draw (1,2) .. controls (1,1.75) and (2,1.75) .. (2,2);
	\draw[dashed] (1,2) .. controls (1,2.25) and (2,2.25) .. (2,2);
        \draw[->] (1.8,3.7) -- (1.2,2.8);
}
\qquad
\tikz[baseline={([yshift=-.5ex]current bounding box.center)}, scale=.2]{
        \draw[domain=0:180] plot ({2*cos(\x)}, {2*sin(\x)});
        \draw (-2,0) arc (180:360:2 and 0.6);
        \draw[dashed] (2,0) arc (0:180:2 and 0.6);
        \draw[->] (0,2.8) [partial ellipse=0:270:5ex and 2ex];
}
~= Y~
\tikz[baseline={([yshift=-.5ex]current bounding box.center)}, scale=.2]{
        \draw[domain=0:180] plot ({2*cos(\x)}, {2*sin(\x)});
        \draw (-2,0) arc (180:360:2 and 0.6);
        \draw[dashed] (2,0) arc (0:180:2 and 0.6);
        \draw[<-] (0,2.8) [partial ellipse=-90:180:5ex and 2ex];
}
\end{equation}
There are also three \emph{handle cancelation moves}, as follows.
\begin{equation}
\label{eq:handlecancelationmoves}
\tikz[scale=.5, baseline={([yshift=-.5ex]current bounding box.center)}]{
	\draw (1,-2) .. controls (1,.-2.25) and (2,-2.25) .. (2,-2);
	\draw[dashed] (1,-2) .. controls (1,-1.75) and (2,-1.75) .. (2,-2);
	\draw (1,-2) -- (1,2);
	\draw (2,-2) -- (2,2);
	\draw (1,2) .. controls (1,1.75) and (2,1.75) .. (2,2);
	\draw (1,2) .. controls (1,2.25) and (2,2.25) .. (2,2);
}
\quad
=
\quad
\tikz[scale=.5, baseline={([yshift=-.5ex]current bounding box.center)}]{
	\draw (2,0) .. controls (2,.-.25) and (3,-.25) .. (3,0);
	\draw[dashed] (2,0) .. controls (2,.25) and (3,.25) .. (3,0);
	\draw (2,0) -- (2,2);
	\draw (3,0) -- (3,2);
	\draw (2,2) .. controls (2,1.75) and (3,1.75) .. (3,2);
	\draw[dashed] (2,2) .. controls (2,2.25) and (3,2.25) .. (3,2);
	\draw (0,2) .. controls (0,1) and (1,1) .. (1,2);
	\draw (0,2) .. controls (0,1.75) and (1,1.75) .. (1,2);
	\draw[dashed] (0,2) .. controls (0,2.25) and (1,2.25) .. (1,2);
	\draw (0,2) .. controls (0,3) and (1,3) .. (1,4);
	\draw (1,2) .. controls (1,3) and (2,3) .. (2,2);
	\draw (3,2) .. controls (3,3) and (2,3) .. (2,4);
	\draw (1,4) .. controls (1,3.75) and (2,3.75) .. (2,4);
	\draw (1,4) .. controls (1,4.25) and (2,4.25) .. (2,4);
	\draw [<-] (1.15,2.15) -- (1.85,2.15);
}
\quad
=
\quad
\tikz[xscale=.5,yscale=-.5, baseline={([yshift=-.5ex]current bounding box.center)}]{
	\draw (2,0) .. controls (2,.-.25) and (3,-.25) .. (3,0);
	\draw (2,0) .. controls (2,.25) and (3,.25) .. (3,0);
	\draw (2,0) -- (2,2);
	\draw (3,0) -- (3,2);
	\draw[dashed]  (2,2) .. controls (2,1.75) and (3,1.75) .. (3,2);
	\draw(2,2) .. controls (2,2.25) and (3,2.25) .. (3,2);
	\draw (0,2) .. controls (0,1) and (1,1) .. (1,2);
	\draw[dashed] (0,2) .. controls (0,1.75) and (1,1.75) .. (1,2);
	\draw (0,2) .. controls (0,2.25) and (1,2.25) .. (1,2);
	\draw[->] (.5,.5) [partial ellipse=-180:90:3ex and 1ex];
	\draw (0,2) .. controls (0,3) and (1,3) .. (1,4);
	\draw (1,2) .. controls (1,3) and (2,3) .. (2,2);
	\draw (3,2) .. controls (3,3) and (2,3) .. (2,4);
	\draw[dashed]  (1,4) .. controls (1,3.75) and (2,3.75) .. (2,4);
	\draw (1,4) .. controls (1,4.25) and (2,4.25) .. (2,4);
	\draw [->] (1.25,2.45) -- (1.75,1.85);
}
\quad
=
\quad
\tikz[xscale=-.5,yscale=-.5, baseline={([yshift=-.5ex]current bounding box.center)}]{
	\draw (2,0) .. controls (2,.-.25) and (3,-.25) .. (3,0);
	\draw (2,0) .. controls (2,.25) and (3,.25) .. (3,0);
	\draw (2,0) -- (2,2);
	\draw (3,0) -- (3,2);
	\draw[dashed]  (2,2) .. controls (2,1.75) and (3,1.75) .. (3,2);
	\draw(2,2) .. controls (2,2.25) and (3,2.25) .. (3,2);
	\draw (0,2) .. controls (0,1) and (1,1) .. (1,2);
	\draw[dashed] (0,2) .. controls (0,1.75) and (1,1.75) .. (1,2);
	\draw (0,2) .. controls (0,2.25) and (1,2.25) .. (1,2);
	\draw[->] (.5,.5) [partial ellipse=-180:90:3ex and 1ex];
	\draw (0,2) .. controls (0,3) and (1,3) .. (1,4);
	\draw (1,2) .. controls (1,3) and (2,3) .. (2,2);
	\draw (3,2) .. controls (3,3) and (2,3) .. (2,4);
	\draw[dashed]  (1,4) .. controls (1,3.75) and (2,3.75) .. (2,4);
	\draw (1,4) .. controls (1,4.25) and (2,4.25) .. (2,4);
	\draw [->] (1.75,2.45) -- (1.25,1.85);
}
\end{equation}
Finally, we have all moves which take the form
\begin{equation}
\label{eq:slide1}
\tikz[baseline={([yshift=-.5ex]current bounding box.center)}, scale=.35]{
	\draw (0,0) .. controls (0,.-.25) and (1,-.25) .. (1,0);
	\draw[dashed] (0,0) .. controls (0,.25) and (1,.25) .. (1,0);
	\draw (0,0) -- (0,4);
	\draw (1,0) -- (1,4);
	\draw (0,4) .. controls (0,3.75) and (1,3.75) .. (1,4);
	\draw (0,4) .. controls (0,4.25) and (1,4.25) .. (1,4);
	\draw (2,0) .. controls (2,.-.25) and (3,-.25) .. (3,0);
	\draw[dashed] (2,0) .. controls (2,.25) and (3,.25) .. (3,0);
	\draw (2,0) -- (2,4);
	\draw (3,0) -- (3,4);
	\draw (2,4) .. controls (2,3.75) and (3,3.75) .. (3,4);
	\draw (2,4) .. controls (2,4.25) and (3,4.25) .. (3,4);
	\draw (4,0) .. controls (4,.-.25) and (5,-.25) .. (5,0);
	\draw[dashed] (4,0) .. controls (4,.25) and (5,.25) .. (5,0);
	\draw (4,0) -- (4,4);
	\draw (5,0) -- (5,4);
	\draw (4,4) .. controls (4,3.75) and (5,3.75) .. (5,4);
	\draw (4,4) .. controls (4,4.25) and (5,4.25) .. (5,4);
	\filldraw [fill=white, draw=black,rounded corners] (-.5,.5) rectangle (2.5,1.5) node[midway] { $W'$};
	\filldraw [fill=white, draw=black,rounded corners] (2.5,2.5) rectangle (5.5,3.5) node[midway] { $W$};
} 
= \lambda(\abs{W}, \abs{W'})
\tikz[baseline={([yshift=-.5ex]current bounding box.center)}, scale=.35]{
	\draw (0,0) .. controls (0,.-.25) and (1,-.25) .. (1,0);
	\draw[dashed] (0,0) .. controls (0,.25) and (1,.25) .. (1,0);
	\draw (0,0) -- (0,4);
	\draw (1,0) -- (1,4);
	\draw (0,4) .. controls (0,3.75) and (1,3.75) .. (1,4);
	\draw (0,4) .. controls (0,4.25) and (1,4.25) .. (1,4);
	\draw (2,0) .. controls (2,.-.25) and (3,-.25) .. (3,0);
	\draw[dashed] (2,0) .. controls (2,.25) and (3,.25) .. (3,0);
	\draw (2,0) -- (2,4);
	\draw (3,0) -- (3,4);
	\draw (2,4) .. controls (2,3.75) and (3,3.75) .. (3,4);
	\draw (2,4) .. controls (2,4.25) and (3,4.25) .. (3,4);
	\draw (4,0) .. controls (4,.-.25) and (5,-.25) .. (5,0);
	\draw[dashed] (4,0) .. controls (4,.25) and (5,.25) .. (5,0);
	\draw (4,0) -- (4,4);
	\draw (5,0) -- (5,4);
	\draw (4,4) .. controls (4,3.75) and (5,3.75) .. (5,4);
	\draw (4,4) .. controls (4,4.25) and (5,4.25) .. (5,4);
	\filldraw [fill=white, draw=black,rounded corners] (-.5,2.5) rectangle (2.5,3.5) node[midway] { $W'$};
	\filldraw [fill=white, draw=black,rounded corners] (2.5,.5) rectangle (5.5,1.5) node[midway] { $W$};
}
\end{equation}
and
\begin{equation}
\label{eq:slide2}
\tikz[baseline={([yshift=-.5ex]current bounding box.center)}, scale=.35]{
	\draw (0,0) .. controls (0,.-.25) and (1,-.25) .. (1,0);
	\draw[dashed] (0,0) .. controls (0,.25) and (1,.25) .. (1,0);
	\draw (0,0) -- (0,4);
	\draw (1,0) -- (1,4);
	\draw (0,4) .. controls (0,3.75) and (1,3.75) .. (1,4);
	\draw (0,4) .. controls (0,4.25) and (1,4.25) .. (1,4);
	\draw (2,0) .. controls (2,.-.25) and (3,-.25) .. (3,0);
	\draw[dashed] (2,0) .. controls (2,.25) and (3,.25) .. (3,0);
	\draw (2,0) -- (2,4);
	\draw (3,0) -- (3,4);
	\draw (2,4) .. controls (2,3.75) and (3,3.75) .. (3,4);
	\draw (2,4) .. controls (2,4.25) and (3,4.25) .. (3,4);
	\filldraw [fill=white, draw=black,rounded corners] (-.5,.5) rectangle (3.5,1.5) node[midway] { $W'$};
} \ 
\tikz[baseline={([yshift=-.5ex]current bounding box.center)}, scale=.35]{
	\draw (0,0) .. controls (0,.-.25) and (1,-.25) .. (1,0);
	\draw[dashed] (0,0) .. controls (0,.25) and (1,.25) .. (1,0);
	\draw (0,0) -- (0,4);
	\draw (1,0) -- (1,4);
	\draw (0,4) .. controls (0,3.75) and (1,3.75) .. (1,4);
	\draw (0,4) .. controls (0,4.25) and (1,4.25) .. (1,4);
	\draw (2,0) .. controls (2,.-.25) and (3,-.25) .. (3,0);
	\draw[dashed] (2,0) .. controls (2,.25) and (3,.25) .. (3,0);
	\draw (2,0) -- (2,4);
	\draw (3,0) -- (3,4);
	\draw (2,4) .. controls (2,3.75) and (3,3.75) .. (3,4);
	\draw (2,4) .. controls (2,4.25) and (3,4.25) .. (3,4);
	\filldraw [fill=white, draw=black,rounded corners] (-.5,2.5) rectangle (3.5,3.5) node[midway] { $W$};
}
= \lambda(\abs{W}, \abs{W'})
\tikz[baseline={([yshift=-.5ex]current bounding box.center)}, scale=.35]{
	\draw (0,0) .. controls (0,.-.25) and (1,-.25) .. (1,0);
	\draw[dashed] (0,0) .. controls (0,.25) and (1,.25) .. (1,0);
	\draw (0,0) -- (0,4);
	\draw (1,0) -- (1,4);
	\draw (0,4) .. controls (0,3.75) and (1,3.75) .. (1,4);
	\draw (0,4) .. controls (0,4.25) and (1,4.25) .. (1,4);
	\draw (2,0) .. controls (2,.-.25) and (3,-.25) .. (3,0);
	\draw[dashed] (2,0) .. controls (2,.25) and (3,.25) .. (3,0);
	\draw (2,0) -- (2,4);
	\draw (3,0) -- (3,4);
	\draw (2,4) .. controls (2,3.75) and (3,3.75) .. (3,4);
	\draw (2,4) .. controls (2,4.25) and (3,4.25) .. (3,4);
	\filldraw [fill=white, draw=black,rounded corners] (-.5,2.5) rectangle (3.5,3.5) node[midway] { $W'$};
} \ \tikz[baseline={([yshift=-.5ex]current bounding box.center)}, scale=.35]{
	\draw (0,0) .. controls (0,.-.25) and (1,-.25) .. (1,0);
	\draw[dashed] (0,0) .. controls (0,.25) and (1,.25) .. (1,0);
	\draw (0,0) -- (0,4);
	\draw (1,0) -- (1,4);
	\draw (0,4) .. controls (0,3.75) and (1,3.75) .. (1,4);
	\draw (0,4) .. controls (0,4.25) and (1,4.25) .. (1,4);
	\draw (2,0) .. controls (2,.-.25) and (3,-.25) .. (3,0);
	\draw[dashed] (2,0) .. controls (2,.25) and (3,.25) .. (3,0);
	\draw (2,0) -- (2,4);
	\draw (3,0) -- (3,4);
	\draw (2,4) .. controls (2,3.75) and (3,3.75) .. (3,4);
	\draw (2,4) .. controls (2,4.25) and (3,4.25) .. (3,4);
	\filldraw [fill=white, draw=black,rounded corners] (-.5,.5) rectangle (3.5,1.5) node[midway] { $W$};
}
\end{equation}
where $\lambda:\mathbb{Z}^2 \times \mathbb{Z}^2 \to R$ is the bilinear map given by
\[
\lambda((a,b), (c,d)) = X^{ac} Y^{bd} Z^{ad - bc}.
\]
Let $x, y, z \in \mathbb{Z}^2$. Bilinearity here means that $\lambda(x, y + z) = \lambda(x, y) \lambda(x, z)$ and $\lambda(x + y, z) = \lambda(x, z) \lambda(y, z)$. Note that $\lambda(x, y)^{-1} = \lambda(y, x)$.

\subsection{Odd arc algebras}
\label{ss:oddarcs}

Let $V = R \langle v_+ \rangle \oplus R \langle v_- \rangle$ be the free $\mathbb{Z}^2$-graded $R$-module generated by $v_+$ and $v_-$ with gradings 
\[
\deg_{\mathbb{Z}^2}(v_+) = (1,0) 
\qquad \text{and} \qquad
\deg_{\mathbb{Z}^2}(v_-) = (0,-1).
\]
In \cite{putyra20152categorychronologicalcobordismsodd}, Putyra defines a so-called \emph{chronological TQFT}, which takes the form of a functor $\mathcal{F}$ from $R\textbf{ChCob}$ to the category of $\mathbb{Z}^2$-graded $R$-modules. On objects, we set
\[
\mathcal{F}(\underbrace{\bigcirc \sqcup \cdots \sqcup \bigcirc}_{n}) = V^{\otimes n}
\]
and on each elementary chronological cobordism, we define
\begin{align*}
\mathcal{F}\left(\tikz[baseline={([yshift=-.5ex]current bounding box.center)}, scale=.5]{
	\draw (0,0) .. controls (0,1) and (1,1) .. (1,2);
	\draw (1,0) .. controls (1,1) and (2,1) .. (2,0);
	\draw (3,0) .. controls (3,1) and (2,1) .. (2,2);
	\draw (0,0) .. controls (0,-.25) and (1,-.25) .. (1,0);
	\draw[dashed] (0,0) .. controls (0,.25) and (1,.25) .. (1,0);
	\draw (2,0) .. controls (2,-.25) and (3,-.25) .. (3,0);
	\draw[dashed] (2,0) .. controls (2,.25) and (3,.25) .. (3,0);
	\draw (1,2) .. controls (1,1.75) and (2,1.75) .. (2,2);
	\draw (1,2) .. controls (1,2.25) and (2,2.25) .. (2,2);
        \draw[<-] (0.9,0.75) -- (2.1,0.75);
}\right) : V \otimes V \rightarrow V &= 
\begin{cases}
v_+ \otimes v_+ \mapsto v_+, & v_+ \otimes v_- \mapsto v_-, \\
v_- \otimes v_- \mapsto 0, & v_- \otimes v_+ \mapsto XZ v_-,
\end{cases}
\\
\mathcal{F}\left(\tikz[baseline={([yshift=-.5ex]current bounding box.center)}, scale=.5]{
	\draw  (1,2) .. controls (1,3) and (0,3) .. (0,4);
	\draw  (2,2) .. controls (2,3) and (3,3) .. (3,4);
	\draw (1,4) .. controls (1,3) and (2,3) .. (2,4);
	\draw (0,4) .. controls (0,3.75) and (1,3.75) .. (1,4);
	\draw (0,4) .. controls (0,4.25) and (1,4.25) .. (1,4);
	\draw (2,4) .. controls (2,3.75) and (3,3.75) .. (3,4);
	\draw (2,4) .. controls (2,4.25) and (3,4.25) .. (3,4);
	\draw (1,2) .. controls (1,1.75) and (2,1.75) .. (2,2);
	\draw[dashed] (1,2) .. controls (1,2.25) and (2,2.25) .. (2,2);
        \draw[<-] (1.8,3.7) -- (1.2,2.8);
}\right) : V  \rightarrow V \otimes V &= 
\begin{cases}
v_+ \mapsto v_- \otimes v_+ + YZ v_+ \otimes v_-,  &\\
v_- \mapsto v_- \otimes v_- , &
\end{cases}
\\
\mathcal{F}\left(\tikz[baseline={([yshift=-.5ex]current bounding box.center)}, scale=.5]{
	\draw (1,2) .. controls (1,1) and (2,1) .. (2,2);
	\draw (1,2) .. controls (1,1.75) and (2,1.75) .. (2,2);
	\draw (1,2) .. controls (1,2.25) and (2,2.25) .. (2,2);
}\right) : R  \rightarrow V  &= 
\begin{cases}
1 \mapsto v_+, & 
\end{cases}
\\
\mathcal{F}\left(\tikz[baseline={([yshift=-.5ex]current bounding box.center)}, scale=.5]{
	\draw (1,0) .. controls (1,1) and (2,1) .. (2,0);
	\draw (1,0) .. controls (1,-.25) and (2,-.25) .. (2,0);
	\draw[dashed] (1,0) .. controls (1,.25) and (2,.25) .. (2,0);
        \draw[->] (1.5,1.1) [partial ellipse=0:270:3ex and 1ex];
}\right) : V  \rightarrow R &= 
\begin{cases}
v_+ \mapsto 0, & \\
v_- \mapsto 1, &
\end{cases}
\end{align*}
applying the change of framing local relations to obtain a complete description.

In keeping with the notation of \cite{naisse2020odd}, let $B_m^n$ denote the set of isotopy classes (fixing endpoints) of flat tangles from $2m$ fixed points on the horizontal line $\mathbb{R} \times \{0\}$ to $2n$ fixed points on $\mathbb{R} \times \{1\}$. We write $B^n:= B_0^n$ to denote the set of \emph{crossingless matchings}. We similarly define $B_n$. If $B^\bullet := \sqcup_{n\ge 0} B^n$, we write $\abs{a} = n$ for $a \in B^\bullet$ whenever $a\in B^n$. There are two notions of composition.
\begin{itemize}
	\item \emph{Stacking}: for $t\in B_m^n$ and $s \in B_n^p$, write $s\circ t$ to denote the flat tangle in $B_m^p$ obtained by stacking $s$ on top of $t$ and then shrinking to the unit square. For each $n$, there is a unit to stacking denoted by $1_n$ given by the tangle comprising of $2n$ vertical strands.
	\item \emph{Juxtaposing}: for $t\in B_{m_1}^{n_1}$ and $s \in B_{m_2}^{n_2}$, write $t \otimes s \in B_{m_1 + m_2}^{n_1 + n_2}$ for the juxtaposition of $t$ on the left and $s$ on the right and then shrinking to the unit square. A unit for juxtaposition is given by the empty tangle $\emptyset \in B_0^0$.
\end{itemize}

We will study cobordisms of flat tangles. For any $t\in B_m^n$, we denote by $\mathbbm{1}_t$ the identity cobordism given by $t \times [0,1] \subset \mathbb{R}\times [0,1]^2$. Abusing notation, we also denote the vertical stacking of cobordisms by $\circ$, and the horizontal stacking of cobordisms by $\otimes$.

Now, notice that given two crossingless matchings $a\in B^m$ and $b\in B_n$, and any flat tangle $t\in B_m^n$, the composition $b \circ t \circ a$ is a closed 1-manifold. Denote by $\overline{b}$ the reflection of $b\in B_n$ about the line $\mathbb{R}\times \{1/2\}$. Then, given a flat tangle $s \in B_n^p$ and another crossingless matching $c \in B_p$, there is a unique, minimal cobordism
\[
(c \circ s \circ \overline{b}) \circ ( b\circ t \circ a) \to c \circ (s \circ t) \circ a
\]
obtained by contracting symmetric arcs of $\overline{b} \circ b$ using saddles to obtain $1_n$. A chronology is fixed on this cobordism by adding saddle points from right to left and choosing the ``upwards'' framing. The resulting cobordism is denoted by $W_{abc} (t, s)$. Note that this cobordism has Euler characteristic $-n$

\begin{example}
If 
\[
a = \overline{c} = 
\tikz[baseline={([yshift=-.5ex]current bounding box.center)}, scale=.5, thick]{
	\draw (0,0) .. controls (0,-.5) and (1,-.5) .. (1,0);
	\draw (2,0) .. controls (2,-.5) and (3,-.5) .. (3,0);
}
\qquad
\overline{b} = 
\tikz[baseline={([yshift=-.5ex]current bounding box.center)}, scale=.5, thick]{
	\draw (0,0) .. controls (0,-1) and (3,-1) .. (3,0);
	\draw (1,0) .. controls (1,-.5) and (2,-.5) .. (2,0);
}
\]
then $W_{abc}(1_2, 1_2)$ is given by the following movie.
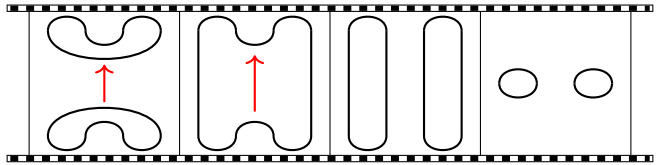
\[
\tikz[baseline={([yshift=-.5ex]current bounding box.center)}, scale=.5]{
	\draw (.5,0) -- (.5,4);
	\node at(2.5,2) {\tikz[scale=.5,  thick]{
		\draw (1,3.4) .. controls (1,2.4) and (4,2.4) .. (4,3.4);
		\draw (2,3.4) .. controls (2,2.9) and (3,2.9) .. (3,3.4);
		\draw (1,3.4) .. controls (1,3.9) and (2,3.9) .. (2,3.4);
		\draw (3,3.4) .. controls (3,3.9) and (4,3.9) .. (4,3.4);
		\draw[red, ->] (2.5,1.5) -- (2.5,2.5);
		%
		\draw (1,0.6) .. controls (1,0.1) and (2,0.1) .. (2,0.6);
		\draw (3,0.6) .. controls (3,0.1) and (4,0.1) .. (4,0.6);
		\draw (1,0.6) .. controls (1,1.6) and (4,1.6) .. (4,0.6);
		\draw (2,0.6) .. controls (2,1.1) and (3,1.1) .. (3,0.6);
	}};
	\draw (4.5,0) -- (4.5,4);
	\node at(6.5,2) {\tikz[scale=.5,  thick]{
		\draw (2,3.4) .. controls (2,2.9) and (3,2.9) .. (3,3.4);
		\draw (1,3.4) .. controls (1,3.9) and (2,3.9) .. (2,3.4);
		\draw (3,3.4) .. controls (3,3.9) and (4,3.9) .. (4,3.4);
		\draw[red,->] (2.5,1.25) -- (2.5,2.75);
		%
		\draw (1,0.6) .. controls (1,0.1) and (2,0.1) .. (2,0.6);
		\draw (3,0.6) .. controls (3,0.1) and (4,0.1) .. (4,0.6);
		\draw (2,0.6) .. controls (2,1.1) and (3,1.1) .. (3,0.6);
        \draw (1,0.6) -- (1,3.4);
        \draw (4,0.6) -- (4,3.4);
	}};
	\draw (8.5,0) -- (8.5,4);
	\node at(10.5,2) {\tikz[scale=.5,  thick]{
		%
		\draw (1,3.4) .. controls (1,3.9) and (2,3.9) .. (2,3.4);
		\draw (3,3.4) .. controls (3,3.9) and (4,3.9) .. (4,3.4);
		%
		%
		\draw (1,0.6) .. controls (1,0.1) and (2,0.1) .. (2,0.6);
		\draw (3,0.6) .. controls (3,0.1) and (4,0.1) .. (4,0.6);
        \draw (1,0.6) -- (1,3.4);
        \draw (4,0.6) -- (4,3.4);
        \draw (2,0.6) -- (2,3.4);
        \draw (3,0.6) -- (3,3.4);
	}};
	\draw (12.5,0) -- (12.5,4);
	\node at(14.5,2) {\tikz[scale=.5,  thick]{
		\draw (1,0) .. controls (1,0.5) and (2,0.5) .. (2,0);
		\draw (3,0) .. controls (3,0.5) and (4,0.5) .. (4,0);
		%
		%
		\draw (1,0) .. controls (1,-0.5) and (2,-0.5) .. (2,0);
		\draw (3,0) .. controls (3,-0.5) and (4,-0.5) .. (4,0);
	}};
	\draw (16.5,0) -- (16.5,4);
 	\draw[double,double distance=2pt,line cap=rect] (0,0) -- (17,0);
	\draw[dashed, line width=2pt] (0,0) -- (17,0);
	\draw[double,double distance=2pt,line cap=rect] (0,4) -- (17,4);
	\draw[dashed, line width=2pt] (0,4) -- (17,4);
 }
\]
\end{example}

For $t\in B_m^n$, we define the \emph{arc space}
\[
\mathcal{F}(t) := \bigoplus_{a \in B^m, b\in B_n} \mathcal{F}(b \circ t \circ a).
\]
Then, given another flat tangle $s\in B_n^p$, there is a composition map $\mu[t, s]$ defined as follows:
\begin{align*}
&\mu[t,s]:\mathcal{F}(b \circ t \circ a) \otimes \mathcal{F}(c \circ s \circ b') \to \mathcal{F}(c \circ (s\circ t) \circ a)\\ &\text{by}~\mu[t,s] = \begin{cases} 0 & \text{if}~ b' \not= \overline{b}, \\ \mathcal{F}(W_{abc}(t,s)) & \text{if}~ b=\overline{b}. \end{cases}
\end{align*}
The \emph{unified arc algebra} $H^n$ is defined as the arc space on $t=1_n$:
\begin{definition}
The \textit{unified arc algebra}, which we still denote $H^n$, is the unified arc space
\[
H^n = \mathcal{F}(1_n) = \bigoplus_{a\in B^m, b\in B_m} \mathcal{F}(a1_nb)
\]
with multiplication $\mu[1_n,1_n]$. The \emph{arc algebra} of Khovanov (see, \textit{e.g.}, \cite{MR1928174}) is the algebra obtained by setting $X = Y = Z = 1$. The \emph{odd arc algebra} of Naisse and Vaz \cite{naisse2017odd} is obtained by setting $X = Z = 1$ and $Y = -1$.
\end{definition}

For example, $H^1 \cong V$ with $\mu(v_+, v_+) = v_+$, $\mu(v_+, v_-) = v_-$, $\mu(v_-, v_+) = XZ v_-$, and $\mu(v_-, v_-) = 0$. The following observation by Naisse and Vaz implies that we cannot import Khovanov's arguments \cite{MR1928174} to obtain a tangle theory, and motivates the construction of grading categories in this setting.

\begin{proposition}[Proposition 3.2 of \cite{naisse2017odd}]
The unified (or odd) arc algebra $H^n$ is not associative for any $n\ge 2$.
\end{proposition}

Other problems include the fact that $H^n$ is not unital, and that the composition maps $\mu[t,s]$ fail to preserve $\mathbb{Z}^2$ grading. These can all be resolved using grading categories.

\subsection{The grading category $\mathcal{G}$ and its trace}
\label{ss:tracecat}

Let $s_{abc}(t,s)$ denote the $\mathbb{Z}^2$-degree of $W_{abc}(t,s)$.

\begin{lemma}
\label{lem:gassoc}
For any triple of flat tangles $r\in B_{\abs{c}}^{\abs{d}}$, $s\in B_{\abs{b}}^{\abs{c}}$, $t\in B_{\abs{a}}^{\abs{b}}$, we have that
\[
s_{abc}(t,s) + s_{acd}(s \circ t, r) = s_{bcd}(s,r) + s_{abd}(t, r\circ s)
\]
\end{lemma}

\begin{proof}
This follows from the minimality condition on the Euler characteristic of the canonical cobordisms $W_{abc}(t,s)$.
\end{proof}

\begin{definition}
Define $\mathcal{G}$ to be the category whose objects are crossingless matchings $a\in B^\bullet$ and for any $a,b\in \mathrm{Ob}(\mathcal{G})$,
\[
\mathrm{Hom}_\mathcal{G}(a, b) = B_m^n \times \mathbb{Z}^2
\]
with composition defined by
\[
(s, q) \circ (t,p) = (s \circ t, p + q + s_{abc}(t,s))
\]
for $(t,p) \in \mathrm{Hom}_\mathcal{G}(a,b)$ and $(s,q) \in \mathrm{Hom}_\mathcal{G}(b,c)$. This composition is associative by Lemma \ref{lem:gassoc}. For each $a\in \mathrm{Ob}(\mathcal{G})$, the distinguished identity element $\mathrm{Id}_a \in \mathrm{End}_\mathcal{G}(a)$ is given by
\[
\mathrm{Id}_a = (1_m, (m, 0))
\]
where $\abs{a} = m$.
\end{definition}

To turn $\mathcal{G}$ into a grading category, define $\alpha: \mathcal{G}^{[3]}\to R^\times$ as follows. For 
\[
\left(a \xrightarrow{(t_1,p_1)} b \xrightarrow{(t_2, p_2)} c \xrightarrow{(t_3, p_3)} d \right) \in \mathcal{G}^{[3]}
\]
let
\[
\alpha((t_1,p_1), (t_2, p_2), (t_3, p_3)) = \alpha_1 \alpha_2
\]
where 
\[
\alpha_1(t_1,p_1), (t_2, p_2), (t_3, p_3)) = \iota(H_\alpha)
\]
for $H_\alpha$ the locally vertical change of chronology
\[
H_\alpha: W_{acd}(t_3, t_2 \circ t_1) \circ (\mathbbm{1}_{d \circ t_3 \circ c} \otimes W_{abc}(t_1, t_2))
\Rightarrow
W_{abd} (t_3 \circ t_2, t_1) \circ (W_{bcd}(t_2, t_3) \otimes \mathbbm{1}_{b \circ t \circ a})
\]
and 
\[
\alpha_2(t_1,p_1), (t_2, p_2), (t_3, p_3)) = \lambda(s_{bcd}(t_2, t_3), p_1).
\]
Note that $\alpha_1$ depends only on the first coordinates of the morphisms involved.

Schematics in terms of leveled, binary trees are very helpful for ensuing computations. Each trivalent vertex corresponds to a canonical cobordism $W_{abc}(t,s)$; in this language, we can describe $\alpha$ as coming from
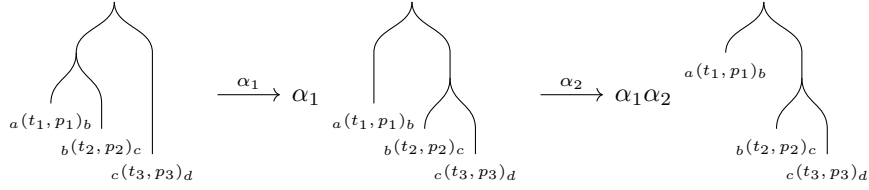
\[
\begin{tikzcd}
\tikz[scale=.45, baseline={([yshift=-.5ex]current bounding box.center)}, x=1.5cm, y=1.5cm]{
	\draw (0,1) node[below]{\tiny${}_a(t_1\text{,}\, p_1)_b$} .. controls (0,1.5) and (.5,1.5) .. (.5,2) .. controls (.5,2.5) and (1.25,2.5) .. (1.25,3); 
	\draw (1,.5) node[below]{\tiny${}_b(t_2\text{,}\, p_2)_c$} -- (1,1) .. controls (1,1.5) and (.5,1.5) .. (.5,2);
	\draw (2, 0) node[below]{\tiny${}_c(t_3\text{,}\, p_3)_d$} -- (2,2) .. controls (2,2.5) and (1.25,2.5) .. (1.25,3);
} \arrow[r, "\alpha_1"]
&
\alpha_1
\tikz[scale=.45, baseline={([yshift=-.5ex]current bounding box.center)}, x=1.5cm, y=1.5cm]{
	\draw (0,1) node[below]{\tiny${}_a(t_1\text{,}\, p_1)_b$} -- (0,2) .. controls (0,2.5) and (.75,2.5) .. (.75,3);
	\draw (1,.5) node[below]{\tiny${}_b(t_2\text{,}\, p_2)_c$} .. controls (1,1) and (1.5,1) .. (1.5,1.5);
	\draw (2, 0) node[below]{\tiny${}_c(t_3\text{,}\, p_3)_d$} -- (2,.5) .. controls (2,1) and (1.5,1) .. (1.5,1.5) -- (1.5,2) .. controls (1.5,2.5) and (.75,2.5) .. (.75,3);
} 
\arrow[r, "\alpha_2"]
&
\alpha_1 \alpha_2
\tikz[scale=.45, baseline={([yshift=-.5ex]current bounding box.center)}, x=1.5cm, y=1.5cm]{
	\draw (0,2) node[below]{\tiny${}_a(t_1\text{,}\, p_1)_b$} .. controls (0,2.5) and (.75,2.5) .. (.75,3);
	\draw (1,.5) node[below]{\tiny${}_b(t_2\text{,}\, p_2)_c$} .. controls (1,1) and (1.5,1) .. (1.5,1.5);
	\draw (2, 0) node[below]{\tiny${}_c(t_3\text{,}\, p_3)_d$} -- (2,.5) .. controls (2,1) and (1.5,1) .. (1.5,1.5) -- (1.5,2) .. controls (1.5,2.5) and (.75,2.5) .. (.75,3);
}
\end{tikzcd}
\]
where we write, \textit{e.g.}, ${}_a(t_1\text{,}\, p_1)_b$ to remember that $(t_1, p_1): a\to b$. 

\begin{proposition}[Proposition 5.4 of \cite{naisse2020odd}]
$(\mathcal{G}, \alpha)$ is a grading category.
\end{proposition}

We omit the proof, but note that the forthcoming proof of Proposition \ref{prop:looperforG} is very similar in spirit.

\subsubsection{The trace of $\mathcal{G}$}

The trace of $\mathcal{G}$ is the set flat tangles in the annulus $S^1 \times I$ paired with an element of $\mathbb{Z}^2$. Denote the annular closure of a flat tangle $t$ by $\tilde{t}$. The quotient map
\[
\mathrm{tr}: \coprod_{a\in \mathrm{Ob}(\mathcal{G})} \mathrm{End}_\mathcal{G}(a) \to \mathrm{Tr}(\mathcal{G})
\]
is given by
\[
\mathrm{tr}_a(t,p) = (\tilde{t}, p + \abs{W_a(t)})
\]
where $W_a(t)$ is the canonical cobordism obtained by contracting symmetric arcs of $\overline{a} \circ a$. To see this, recall that the defining quality of $\mathrm{tr}$ is that the diagram in Equation (\ref{eq:cohotrace}) commutes. When $\mathcal{C} = \mathcal{G}$, that diagram takes the following form.
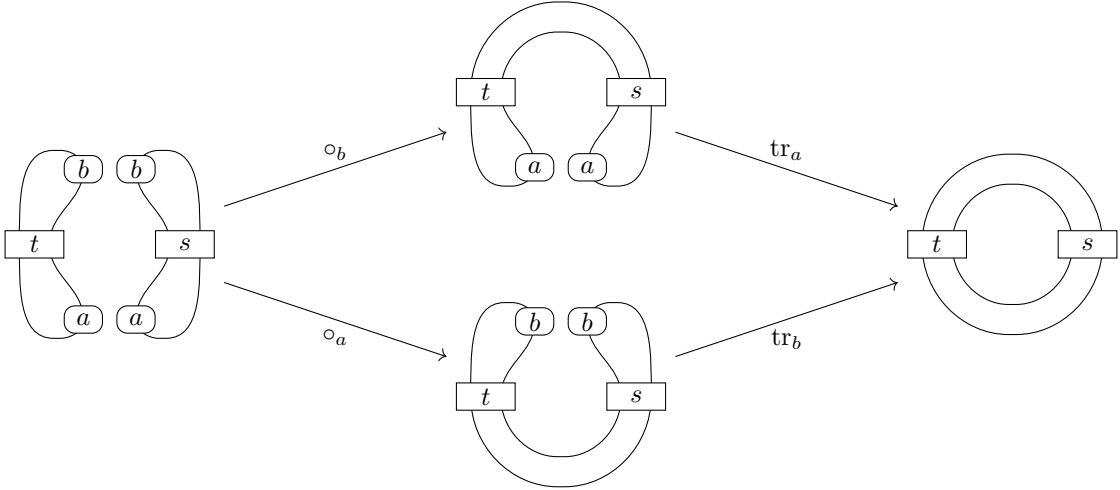
\[
\tikz{
	\node(L) at (0,0) {$
	\tikz{
	\draw[rounded corners=0.5mm] (-0.8, 0) to[out=90, in=-90] (-0.35,1);
	\draw[rounded corners=0.5mm] (-1.2, 0) to[out=90, in=180] (-0.7, 1.25) to [out=0, in=90] (-0.35,1);
	\draw[rounded corners=0.5mm] (0.8, 0) to[out=90, in=-90] (0.35,1);
	\draw[rounded corners=0.5mm] (1.2, 0) to[out=90, in=0] (0.7, 1.25) to [out=180, in=90] (0.35,1);
	\draw[rounded corners=0.5mm] (-0.8, 0) to[out=-90, in=90] (-0.35,-1);
	\draw[rounded corners=0.5mm] (-1.2, 0) to[out=-90, in=180] (-0.7, -1.25) to [out=0, in=-90] (-0.35,-1);
	\draw[rounded corners=0.5mm] (0.8, 0) to[out=-90, in=90] (0.35,-1);
	\draw[rounded corners=0.5mm] (1.2, 0) to[out=-90, in=0] (0.7, -1.25) to [out=180, in=-90] (0.35,-1);
	\node[draw, fill=white] at (-1,0) {\phantom{{}=={}}};
	\node at (-1,0) {$t$};
	\node[draw, fill=white] at (1,0) {\phantom{{}=={}}};
	\node at (1,0) {$s$};
	\node[draw, rounded corners, fill=white] at (-0.35,1) {\phantom{{}={}}};
	\node at (-0.35,1) {$b$};
	\node[draw, rounded corners, fill=white] at (0.35,1) {\phantom{{}={}}};
	\node at (0.35,1) {$b$};
	\node[draw, rounded corners, fill=white] at (-0.35,-1) {\phantom{{}={}}};
	\node at (-0.35,-1) {$a$};
	\node[draw, rounded corners, fill=white] at (0.35,-1) {\phantom{{}={}}};
	\node at (0.35,-1) {$a$};
	}
	$};
	\node(U) at (6,2) {$
\tikz{
	\draw[rounded corners=0.5mm] (-0.8, 0) to[out=90, in=180] (0, 0.8) to[out=0, in=90] (0.8,0);
	\draw[rounded corners=0.5mm] (-1.2, 0) to[out=90, in=180] (0, 1.2) to[out=0, in=90] (1.2,0);
	\draw[rounded corners=0.5mm] (-0.8, 0) to[out=-90, in=90] (-0.35,-1);
	\draw[rounded corners=0.5mm] (-1.2, 0) to[out=-90, in=180] (-0.7, -1.25) to [out=0, in=-90] (-0.35,-1);
	\draw[rounded corners=0.5mm] (0.8, 0) to[out=-90, in=90] (0.35,-1);
	\draw[rounded corners=0.5mm] (1.2, 0) to[out=-90, in=0] (0.7, -1.25) to [out=180, in=-90] (0.35,-1);
	\node[draw, fill=white] at (-1,0) {\phantom{{}=={}}};
	\node at (-1,0) {$t$};
	\node[draw, fill=white] at (1,0) {\phantom{{}=={}}};
	\node at (1,0) {$s$};
	\node[draw, rounded corners, fill=white] at (-0.35,-1) {\phantom{{}={}}};
	\node at (-0.35,-1) {$a$};
	\node[draw, rounded corners, fill=white] at (0.35,-1) {\phantom{{}={}}};
	\node at (0.35,-1) {$a$};
	}
	$};
	\node(D) at (6,-2) {$
\tikz{
	\draw[rounded corners=0.5mm] (-0.8, 0) to[out=-90, in=180] (0, -0.8) to[out=0, in=-90] (0.8,0);
	\draw[rounded corners=0.5mm] (-1.2, 0) to[out=-90, in=180] (0, -1.2) to[out=0, in=-90] (1.2,0);
	\draw[rounded corners=0.5mm] (-0.8, 0) to[out=90, in=-90] (-0.35,1);
	\draw[rounded corners=0.5mm] (-1.2, 0) to[out=90, in=180] (-0.7, 1.25) to [out=0, in=90] (-0.35,1);
	\draw[rounded corners=0.5mm] (0.8, 0) to[out=90, in=-90] (0.35,1);
	\draw[rounded corners=0.5mm] (1.2, 0) to[out=90, in=0] (0.7, 1.25) to [out=180, in=90] (0.35,1);
	\node[draw, fill=white] at (-1,0) {\phantom{{}=={}}};
	\node at (-1,0) {$t$};
	\node[draw, fill=white] at (1,0) {\phantom{{}=={}}};
	\node at (1,0) {$s$};
	\node[draw, rounded corners, fill=white] at (-0.35,1) {\phantom{{}={}}};
	\node at (-0.35,1) {$b$};
	\node[draw, rounded corners, fill=white] at (0.35,1) {\phantom{{}={}}};
	\node at (0.35,1) {$b$};
	}
	$};
	\node(R) at (12, 0) {$
	\tikz{
	\draw[rounded corners=0.5mm] (-0.8, 0) to[out=90, in=180] (0, 0.8) to[out=0, in=90] (0.8,0);
	\draw[rounded corners=0.5mm] (-1.2, 0) to[out=90, in=180] (0, 1.2) to[out=0, in=90] (1.2,0);
	\draw[rounded corners=0.5mm] (-0.8, 0) to[out=-90, in=180] (0, -0.8) to[out=0, in=-90] (0.8,0);
	\draw[rounded corners=0.5mm] (-1.2, 0) to[out=-90, in=180] (0, -1.2) to[out=0, in=-90] (1.2,0);
	\node[draw, fill=white] at (-1,0) {\phantom{{}=={}}};
	\node at (-1,0) {$t$};
	\node[draw, fill=white] at (1,0) {\phantom{{}=={}}};
	\node at (1,0) {$s$};
	}
	$};
	\draw[->] (L) to node[above]{$\circ_b$} (U);
	\draw[->] (L) to node[below]{$\circ_a$} (D);
	\draw[->] (U) to node[above]{$\mathrm{tr}_a$} (R);
	\draw[->] (D) to node[below]{$\mathrm{tr}_b$}(R);
}
\]
The upper path takes
\[
((t, p), (s, q)) \mapsto \underbrace{(s\circ t, p + q + s_{aba}(t,s))}_{\in \mathrm{End}_\mathcal{G}(a)} \mapsto \underbrace{(\widetilde{s \circ t}, p + q + s_{aba}(t,s) + \abs{W_a(s\circ t)})}_{\in \mathrm{Tr}(\mathcal{G})}
\]
while the lower path takes
\[
((t, p), (s, q)) \mapsto \underbrace{(t\circ s, q + p + s_{bab}(s,t))}_{\in \mathrm{End}_\mathcal{G}(b)} \mapsto \underbrace{(\widetilde{t \circ s}, q + p + s_{bab}(s,t) + \abs{W_b(t\circ s)})}_{\in \mathrm{Tr}(\mathcal{G})}.
\]
Clearly $\widetilde{s\circ t} = \widetilde{t \circ s}$. Moreover,
\[
s_{aba}(t,s) + \abs{W_a(s\circ t)} = s_{bab}(s,t) + \abs{W_b(t \circ s)}
\]
follows from (an argument completely analagous to) Lemma \ref{lem:gassoc}, and we conclude that the diagram commutes.

We claim that $(\mathcal{G}, \alpha)$ admits a looper. Define $\varepsilon: \Omega_2\mathcal{G} \to R^\times$ as follows. For
\[
\left(
\begin{tikzcd}
	a \arrow[r, "(t\text{,}\, p)", bend left=10mm] & b \arrow[l, "(s\text{,}\, q)", bend left=10mm]
\end{tikzcd}
\right)
\in \Omega_2 \mathcal{G}
\]
let
\[
\varepsilon((t, p), (s, q)) = \varepsilon_1 \varepsilon_2
\]
where
\[
\varepsilon_1((t,p), (s,q)) = \iota(H_\varepsilon)
\]
for $H_\varepsilon$ the locally vertical change of chronology
\[
H_\varepsilon: W_a(s\circ t) \circ W_{aba}(t, s) \Rightarrow W_b(t\circ s) \circ W_{bab}(s,t)
\]
and
\[
\varepsilon_2((t,p), (s,q)) = \lambda(p, q).
\]
In terms of tree schematics, we explain $\varepsilon$ as coming from 
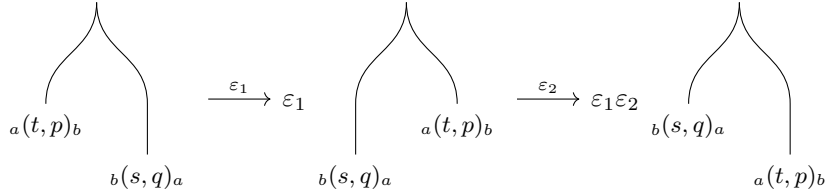
\[
\begin{tikzcd}
\tikz[scale=.45, baseline={([yshift=-.5ex]current bounding box.center)}, x=3cm, y=3cm]{
	\draw (0,1) node[below]{\small${}_a(t\text{,}\, p)_b$} .. controls (0,1.5) and (.5,1.5) .. (.5,2); 
	\draw (1,.5) node[below]{\small${}_b(s\text{,}\, q)_a$} -- (1,1) .. controls (1,1.5) and (.5,1.5) .. (.5,2);
}
\arrow[r, "\varepsilon_1"]
&
\varepsilon_1
\tikz[scale=.45, baseline={([yshift=-.5ex]current bounding box.center)}, x=3cm, y=3cm]{
	\draw(0, 0.5) node[below]{\small${}_b(s\text{,}\, q)_a$} -- (0,1) .. controls (0,1.5) and (.5,1.5) .. (.5,2);
	\draw (1, 1) node[below]{\small${}_a(t\text{,}\, p)_b$} .. controls (1,1.5) and (.5,1.5) .. (.5,2);
}
\arrow[r, "\varepsilon_2"]
&
\varepsilon_1 \varepsilon_2
\tikz[scale=.45, baseline={([yshift=-.5ex]current bounding box.center)}, x=3cm, y=3cm]{
	\draw (0,1) node[below]{\small${}_b(s\text{,}\, q)_a$} .. controls (0,1.5) and (.5,1.5) .. (.5,2); 
	\draw (1,.5) node[below]{\small${}_a(t\text{,}\, p)_b$} -- (1,1) .. controls (1,1.5) and (.5,1.5) .. (.5,2);
}
\end{tikzcd}
\]
Again, note that $\varepsilon_1$ depends only on the first coordinates of the morphisms involved.

\begin{proposition}
\label{prop:looperforG}
$\varepsilon$ is a looper for $(\mathcal{G}, \alpha)$.
\end{proposition}

\begin{proof}
First, note that $\varepsilon((t,p), (s,q))^{-1} = \varepsilon((s, q), (t, p))$. On one hand
\[
\varepsilon_2((s, q), (t, p)) = \lambda(q, p) = \lambda(p, q)^{-1} = \varepsilon_2((t, p), (s, q))^{-1}.
\]
On the other hand, consider the sequence of changes of chronology
\[
\begin{tikzcd}
W_a(s\circ t) \circ W_{aba}(t, s) \arrow[rr, Rightarrow, bend right=7mm, "\mathrm{Id}"']
\arrow[r, Rightarrow, "H_\varepsilon"]
&
W_b(t\circ s) \circ W_{bab}(s, t) 
\arrow[r, Rightarrow, "H_\varepsilon' "]
&
W_a(s\circ t) \circ W_{aba}(t, s).
\end{tikzcd}
\]
Since $H_\varepsilon$ and $H_\varepsilon'$ are locally vertical changes of chronology, we see that $H_\varepsilon' \star H_\varepsilon$ is homotopic to the identity change of chronology. Thus, using Proposition \ref{putyrahammer}
\[
\iota(H_\varepsilon') \iota(H_\varepsilon) = \iota(H_\varepsilon' \star H_\varepsilon) = \iota(\mathrm{Id}) = 1.
\]
Since $R$ is a commutative ring, we conclude that $\varepsilon_1((t,p), (s, q))^{-1} = \varepsilon_1((s, q), (t, p))$, which satisfies part (i) of Definition \ref{def:looper}.

To show (ii) (\textit{i.e.}, $\alpha$-$\varepsilon$ coherence), we break our proof into two parts: ``$\alpha_1$-$\varepsilon_1$ coherence'' and ``$\alpha_2$-$\varepsilon_2$ coherence.'' Suppose that
\[
\left(
\tikz[baseline={([yshift=-.5ex]current bounding box.center)}, scale=0.8,x=1cm, y=1cm]{
	\node(x) at (0,-1) {$a$};
	\node(z) at (1.73205/2, 0.5) {$b$};
	\node(y) at (-1.73205/2, 0.5) {$c$};
	\draw[->] (x) to node(f)[below, sloped]{\small$(t_, p_1)$} (y);
	\draw[->] (y) to node(g)[above]{\small$(t_2, p_2)$} (z);
	\draw[->] (z) to node(h)[below, sloped]{\small$(t_3, p_3)$} (x);
	}
\right)
\in \Omega_3\mathcal{G}.
\]
Considering just the first coordinates of the morphisms involved (as $\alpha_1$ and $\varepsilon_1$ depend only on these entries), $\alpha$-$\varepsilon$ coherence translates to the following hexagon of schematics
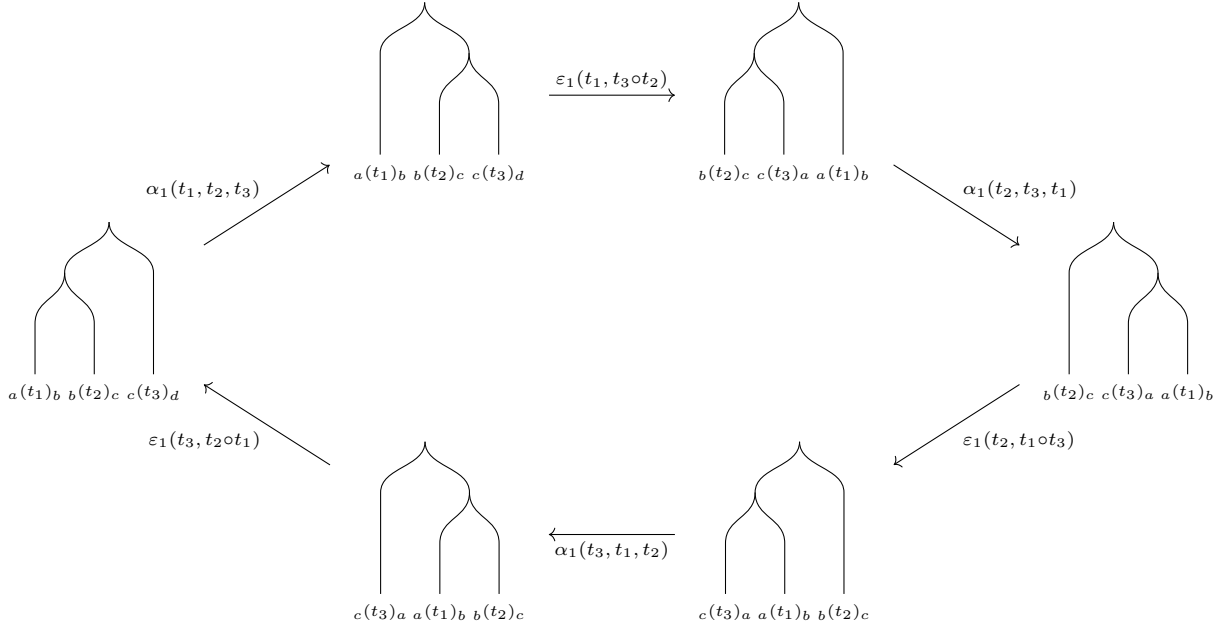
\[
\begin{tikzcd}[column sep=huge, row sep=tiny]
&
\tikz[scale=.45, baseline={([yshift=-.5ex]current bounding box.center)}, x=1.75cm, y=1.5cm]{
	\draw (0,0) node[below]{\tiny${}_a(t_1)_b$} -- (0,2) .. controls (0,2.5) and (.75,2.5) .. (.75,3);
	\draw (1,0) node[below]{\tiny${}_b(t_2)_c$} -- (1,1) .. controls (1,1.5) and (1.5,1.5) .. (1.5,2);
	\draw (2, 0) node[below]{\tiny${}_c(t_3)_d$} -- (2,1) .. controls (2,1.5) and (1.5,1.5) .. (1.5,2) -- (1.5,2) .. controls (1.5,2.5) and (.75,2.5) .. (.75,3);
}
\arrow[r, "\varepsilon_1(t_1\text{,}\, t_3 \circ t_2)"]
&
\tikz[scale=.45, baseline={([yshift=-.5ex]current bounding box.center)}, x=1.75cm, y=1.5cm]{
	\draw (0,0) node[below]{\tiny${}_b(t_2)_c$} -- (0,1) .. controls (0,1.5) and (.5,1.5) .. (.5,2) .. controls (.5,2.5) and (1.25,2.5) .. (1.25,3); 
	\draw (1,0) node[below]{\tiny${}_c(t_3)_a$} -- (1,1) .. controls (1,1.5) and (.5,1.5) .. (.5,2);
	\draw (2, 0) node[below]{\tiny${}_a(t_1)_b$} -- (2,2) .. controls (2,2.5) and (1.25,2.5) .. (1.25,3);
} 
\arrow[dr, "\alpha_1(t_2\text{,}\, t_3\text{,}\, t_1)"]
&
\\
\tikz[scale=.45, baseline={([yshift=-.5ex]current bounding box.center)}, x=1.75cm, y=1.5cm]{
	\draw (0,0) node[below]{\tiny${}_a(t_1)_b$} -- (0,1) .. controls (0,1.5) and (.5,1.5) .. (.5,2) .. controls (.5,2.5) and (1.25,2.5) .. (1.25,3); 
	\draw (1,0) node[below]{\tiny${}_b(t_2)_c$} -- (1,1) .. controls (1,1.5) and (.5,1.5) .. (.5,2);
	\draw (2, 0) node[below]{\tiny${}_c(t_3)_d$} -- (2,2) .. controls (2,2.5) and (1.25,2.5) .. (1.25,3);
}
\arrow[ur, "\alpha_1(t_1\text{,}\, t_2\text{,}\, t_3)"]
&
&
&
\tikz[scale=.45, baseline={([yshift=-.5ex]current bounding box.center)}, x=1.75cm, y=1.5cm]{
	\draw (0,0) node[below]{\tiny${}_b(t_2)_c$} -- (0,2) .. controls (0,2.5) and (.75,2.5) .. (.75,3);
	\draw (1,0) node[below]{\tiny${}_c(t_3)_a$} -- (1,1) .. controls (1,1.5) and (1.5,1.5) .. (1.5,2);
	\draw (2, 0) node[below]{\tiny${}_a(t_1)_b$} -- (2,1) .. controls (2,1.5) and (1.5,1.5) .. (1.5,2) -- (1.5,2) .. controls (1.5,2.5) and (.75,2.5) .. (.75,3);
}
\arrow[dl, "\varepsilon_1(t_2\text{,}\, t_1 \circ t_3)"]
\\
&
\tikz[scale=.45, baseline={([yshift=-.5ex]current bounding box.center)}, x=1.75cm, y=1.5cm]{
	\draw (0,0) node[below]{\tiny${}_c(t_3)_a$} -- (0,2) .. controls (0,2.5) and (.75,2.5) .. (.75,3);
	\draw (1,0) node[below]{\tiny${}_a(t_1)_b$} -- (1,1) .. controls (1,1.5) and (1.5,1.5) .. (1.5,2);
	\draw (2, 0) node[below]{\tiny${}_b(t_2)_c$} -- (2,1) .. controls (2,1.5) and (1.5,1.5) .. (1.5,2) -- (1.5,2) .. controls (1.5,2.5) and (.75,2.5) .. (.75,3);
}
\arrow[ul, "\varepsilon_1(t_3\text{,}\, t_2 \circ t_1)"]
&
\tikz[scale=.45, baseline={([yshift=-.5ex]current bounding box.center)}, x=1.75cm, y=1.5cm]{
	\draw (0,0) node[below]{\tiny${}_c(t_3)_a$} -- (0,1) .. controls (0,1.5) and (.5,1.5) .. (.5,2) .. controls (.5,2.5) and (1.25,2.5) .. (1.25,3); 
	\draw (1,0) node[below]{\tiny${}_a(t_1)_b$} -- (1,1) .. controls (1,1.5) and (.5,1.5) .. (.5,2);
	\draw (2, 0) node[below]{\tiny${}_b(t_2)_c$} -- (2,2) .. controls (2,2.5) and (1.25,2.5) .. (1.25,3);
}
\arrow[l, "\alpha_1(t_3\text{,}\, t_1\text{,}\, t_2)"]
&
\end{tikzcd}
\]
where each arrow corresponds to a change of chronology. Again, by Proposition \ref{putyrahammer}, we have that the change of chronology equal to the composition of all six of these changes of chronology is homotopic to the identity change of chronology. Thus, by multiplicativity of $\iota$ with respect to $\star$-composition, we have that
\[
\alpha_1(t_1, t_2, t_3) \varepsilon_1(t_1, t_3 \circ t_2)
		\alpha_1(t_2, t_3, t_1) \varepsilon_1(t_2, t_1 \circ t_3)
		\alpha_1(t_3, t_1, t_2) \varepsilon_1(t_3, t_2 \circ t_1)
		=1
\]
which concludes the proof of ``$\alpha_1$-$\varepsilon_1$ coherence.'' For the other, simply note that the contributions made by $\alpha_2$ are
\begin{align*}
& \alpha_2((t_1, p_1), (t_2, p_2), (t_3, p_3)) = \lambda(s_{bca}(t_2, t_3), p_1),
\\
& \alpha_2((t_2, p_2), (t_3, p_3), (t_1, p_1)) = \lambda(s_{cab}(t_3, t_1), p_2),~\text{and}
\\
& \alpha_2((t_3, p_3), (t_1, p_1), (t_2, p_2)) = \lambda(s_{abc}(t_1, t_2), p_3).
\end{align*}
Similarly, the contributions made by $\varepsilon_2$ can be listed as 
\begin{align*}
\varepsilon_2((t_1, p_1), (t_3, p_3) \circ (t_2, p_2))  & = \lambda(p_1, s_{bca}(t_2, t_3) + p_2 + p_3)
\\ &= \lambda(p_1, s_{bca}(t_2, t_3)) \lambda(p_1, p_2) \lambda(p_1, p_3), 
\end{align*}
\begin{align*}
\varepsilon_2( (t_2, p_2), (t_1, p_1) \circ (t_3, p_3) ) & = \lambda (p_2, s_{cab}(t_3, t_1) + p_3 + p_1)
\\ &= \lambda(p_2, s_{cab}(t_3, t_1)) \lambda(p_2, p_3) \lambda(p_2, p_1), 
\end{align*}
\begin{align*}
\varepsilon_2((t_3, p_3), (t_2, p_2) \circ (t_1, p_1)) & = \lambda (p_3, s_{abc}(t_1, t_2) + p_1 + p_2)
\\ &= \lambda(p_3, s_{abc}(t_1, t_2)) \lambda(p_3, p_1) \lambda(p_3, p_2),
\end{align*}
applying bilinearity in the second coordinate. Since $\lambda(x,y)^{-1} = \lambda(y,x)$, we see that the $\alpha_2$ and $\varepsilon_2$ contributions cancel with each other, concluding the proof.
\end{proof}

\begin{remark}
In the proof that $\alpha$ is an associator, the $\alpha_1$ contribution cancels with the $\alpha_2$ contribution, thus all terms in the looper for $(\mathcal{G}, \alpha)$ are essential.
\end{remark}

\section{Up-to-unit funtoriality}
\label{s:utu_fun}

In this section, we compute the zeroth (degree-preserving) Hochschild cohomology of an arbitrary arc algebra $H^n$ and the $\mathcal{G}$-graded automorphisms of flat tangles to give an alternative, relatively quick proof of up-to-unit functoriality for unified Khovanov homology with respect to tangle cobordisms. This argument can be thought of as an update to Khovanov's \cite{MR2171235}; indeed, this entire section is largely modeled after his paper. Our argument implies that Putyra's unified theory for links (over $R$ rather than $\mathbb{Z}[\pi] / \pi^2=1$, as was achieved by Migdail-Wehrli \cite{migdail2024functoriality}) is functorial with respect to link cobordism, and it is also the first proof that the Naisse-Putyra tangle theory (and thus a tangle theory for odd Khovanov homology) is invariant up to unit (or sign) with respect to tangle cobordisms.

The source category for this claim is (an oriented, chronological version of) the 2-tangle 2-category, which we denote by $\mathbf{Tang}$. Following Khovanov \cite{MR2171235}, we use the combinatorial realization of this category. The construction presented here should be attributed to Baez-Langford \cite{MR2020556}, Carter-Rieger-Saito \cite{MR1238875, MR1445361, MR1487374}, Fischer \cite{MR1290200}, Kharlamov-Turaev \cite{MR1386661}, and Roseman \cite{roseman1998reidemeister}. The objects of this category are elements of $\bigcup_{n\ge 0}\{+, -\}^{2n}$, thought of as disjoint points aligned on a horizontal line. We write $\abs{x}$ to denote half of the cardinality of an object. A 1-morphism $x \to y$ is a rigid, oriented tangle from $\abs{x}$ points to $\abs{y}$ points which is generic in the sense that it can be decomposed uniquely into a sequence of the following generating 1-morphisms. 
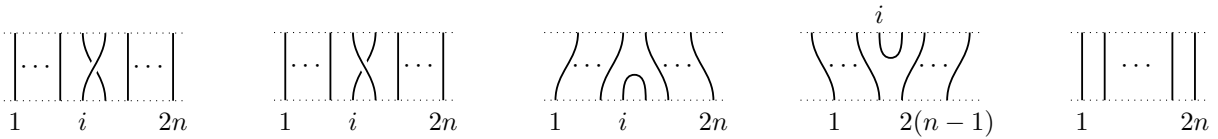
\[
\tikz[baseline={([yshift=-.5ex]current bounding box.center)}, scale=0.3, y=0.75cm]
{   
    \draw[dotted] (-0.5,4) -- (7.5,4);
    \draw[knot] (0,0) node[below]{$\phantom{(}1\phantom{)}$} -- (0,4);
    \node at (1,2) {$\cdots$};
    \draw[knot] (2,0) -- (2,4);
    \draw[knot] (4,0) to[out=90,in=-90] (3,4);
    \draw[knot, overcross] (3,0) to[out=90, in=-90] (4,4) node[above]{$\phantom{i}$};
    \node[below] at (3,0) {$\phantom{(}i\phantom{)}$};
    \draw[knot] (5,0) -- (5,4);
    \node at (6,2) {$\cdots$};
    \draw[knot] (7,0) node[below]{$\phantom{(}2n\phantom{)}$} -- (7,4);
    \draw[dotted] (-0.5,0) -- (7.5,0);
}
\qquad
\tikz[baseline={([yshift=-.5ex]current bounding box.center)}, scale=0.3, y=0.75cm]
{   
    \draw[dotted] (-0.5,4) -- (7.5,4);
    \draw[knot] (0,0) node[below]{$\phantom{(}1\phantom{)}$} -- (0,4);
    \node at (1,2) {$\cdots$};
    \draw[knot] (2,0) -- (2,4);
    \draw[knot] (3,0) node[below]{$\phantom{(}i\phantom{)}$} to[out=90, in=-90] (4,4) node[above]{$\phantom{i}$};
    \draw[knot, overcross] (4,0) to[out=90,in=-90] (3,4);
    \draw[knot] (5,0) -- (5,4);
    \node at (6,2) {$\cdots$};
    \draw[knot] (7,0) node[below]{$\phantom{(}2n\phantom{)}$} -- (7,4);
    \draw[dotted] (-0.5,0) -- (7.5,0);
}
\qquad
\tikz[baseline={([yshift=-.5ex]current bounding box.center)}, scale=0.3, y=0.75cm]
{   
    \draw[dotted] (-0.5,4) -- (7.5,4);
    \draw[knot] (0,0) node[below]{$\phantom{(}1\phantom{)}$} to[out=90, in=-90] (1,4) node[above]{$\phantom{i}$};
    \node at (1.5,2) {$\cdots$};
    \draw[knot] (2,0) to[out=90, in=-90] (3,4);
    \draw[knot] (5,0) to[out=90, in=-90] (4,4);
    \node at (5.5, 2) {$\cdots$};
    \draw[knot] (7,0) node[below]{$\phantom{(}2n\phantom{)}$} to[out=90, in=-90] (6,4);
    \draw[knot] (3,0) node[below]{$\phantom{(}i\phantom{)}$}   to[out=90, in=180] (3.5, 1.5) to[out=0, in=90] (4,0);
    \draw[dotted] (-0.5,0) -- (7.5,0);
}
\qquad
\tikz[baseline={([yshift=-.5ex]current bounding box.center)}, scale=0.3, y=0.75cm]
{   
    \draw[dotted] (-0.5,4) -- (7.5,4);
    \draw[knot] (0,4) to[out=-90, in=90] (1,0) node[below]{$\phantom{(}1\phantom{)}$};
    \node at (1.5,2) {$\cdots$};
    \draw[knot] (2,4) to[out=-90, in=90] (3,0);
    \draw[knot] (5,4) to[out=-90, in=90] (4,0);
    \node at (5.5, 2) {$\cdots$};
    \draw[knot] (7,4) to[out=-90, in=90] (6,0) node[below]{$2(n-1)$};
    \draw[knot] (3,4) node[above]{$i$} to[out=-90, in=180] (3.5, 2.5) to[out=0, in=-90] (4,4);
    \draw[dotted] (-0.5,0) -- (7.5,0);
}
\qquad
\tikz[baseline={([yshift=-.5ex]current bounding box.center)}, scale=0.3, y=0.75cm]
{   
    \draw[dotted] (0.5,4) -- (6.5,4);
    \draw[knot] (1,0) node[below]{$\phantom{(}1\phantom{)}$} -- (1,4);    
    \draw[knot] (2,0) -- (2,4) node[above]{$\phantom{i}$};
    \node at (3.5,2) {$\cdots$};
    \draw[knot] (5,0) -- (5,4);
    \draw[knot] (6,0) node[below]{$\phantom{(}2n\phantom{)}$} -- (6,4);
    \draw[dotted] (0.5,0) -- (6.5,0);
}
\]
These generating 1-morphisms are denoted by $\sigma_{i,n}$, $\overline{\sigma}_{i,n}$, $\cap_{i, n}$, $\cup_{i, n-1}$, and $1_n$ respectively. The 1-morphisms $\cap_{i, n}$ and $\cup_{i, n-1}$ are interchangeably referred to as U-turns, turnbacks, and (respectively) caps and cups. A ``$+$'' sign means that the strand of that tangle is oriented toward that point, and the ``$-$'' sign means that the strand is oriented away from the point (therefore, there are many objects with no 1-morphisms between them: a strand of a tangle must be oriented from a ``$-$'' sign and toward a ``$+$'' sign). The 2-morphisms are movies of tangle diagrams, which encode framed chronological cobordisms. They are generated by the Morse movies (births, saddles, and deaths), the Reidemeister movies, and three movies encoding isotopy for generic tangle diagrams called the T-movie, the H-Movie, and the N-movie. We picture each of these below.
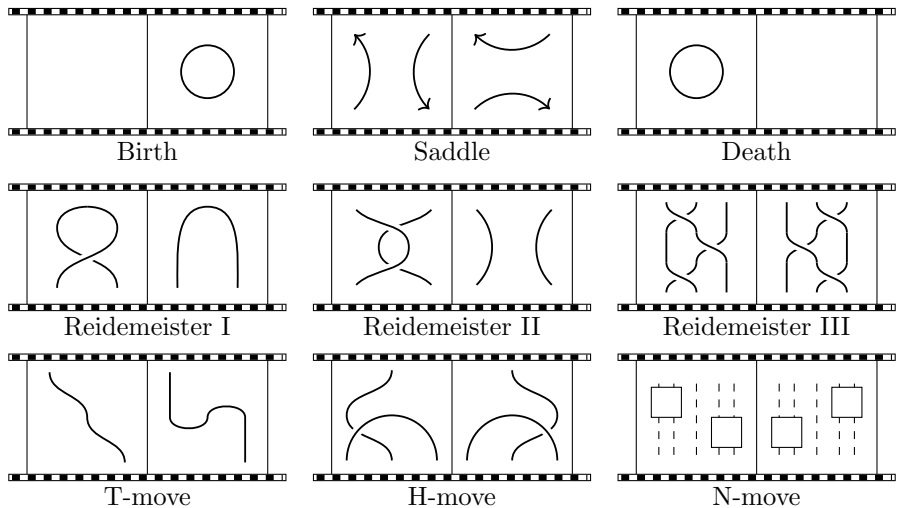
\[
\tikz[baseline={([yshift=-.5ex]current bounding box.center)}, scale=.4]{
	\draw (.5,0) -- (.5,4);
	\draw (4.5,0) -- (4.5,4);
	\draw (8.5,0) -- (8.5,4);
 	\draw[double,double distance=2pt,line cap=rect] (0,0) -- (9,0);
	\draw[dashed, line width=2pt] (0,0) -- (9,0);
	\draw[double,double distance=2pt,line cap=rect] (0,4) -- (9,4);
	\draw[dashed, line width=2pt] (0,4) -- (9,4);
        \node at (6.5,2) {$\tikz{
            \draw[knot] (0,0) circle (10pt);
        }$};
        \node[below] at (4.5,0) {Birth};
 }
\quad
\tikz[baseline={([yshift=-.5ex]current bounding box.center)}, scale=.4]{
	\draw (.5,0) -- (.5,4);
	\draw (4.5,0) -- (4.5,4);
	\draw (8.5,0) -- (8.5,4);
 	\draw[double,double distance=2pt,line cap=rect] (0,0) -- (9,0);
	\draw[dashed, line width=2pt] (0,0) -- (9,0);
	\draw[double,double distance=2pt,line cap=rect] (0,4) -- (9,4);
	\draw[dashed, line width=2pt] (0,4) -- (9,4);
        \node at (2.5,2) {$\tikz[scale=0.5]{
            \draw[knot, ->] (0,0) to[out=45, in=-45] (0,2);
            \draw[knot, <-] (2,0) to[out=135, in=-135] (2,2);
        }$};
        \node at (6.5,2) {$\tikz[scale=0.5]{
            \draw[knot, ->] (0,0) to[out=45, in=135] (2,0);
            \draw[knot, <-] (0,2) to[out=-45, in=-135] (2,2);
        }$};
        \node[below] at (4.5,0) {Saddle};
 }
\quad
\tikz[baseline={([yshift=-.5ex]current bounding box.center)}, scale=.4]{
	\draw (.5,0) -- (.5,4);
	\draw (4.5,0) -- (4.5,4);
	\draw (8.5,0) -- (8.5,4);
 	\draw[double,double distance=2pt,line cap=rect] (0,0) -- (9,0);
	\draw[dashed, line width=2pt] (0,0) -- (9,0);
	\draw[double,double distance=2pt,line cap=rect] (0,4) -- (9,4);
	\draw[dashed, line width=2pt] (0,4) -- (9,4);
        \node at (2.5,2) {$\tikz{
            \draw[knot] (0,0) circle (10pt);
        }$};
        \node[below] at (4.5,0) {Death};
 }
\]
\[
\tikz[baseline={([yshift=-.5ex]current bounding box.center)}, scale=.4]{
	\draw (.5,0) -- (.5,4);
	\draw (4.5,0) -- (4.5,4);
	\draw (8.5,0) -- (8.5,4);
 	\draw[double,double distance=2pt,line cap=rect] (0,0) -- (9,0);
	\draw[dashed, line width=2pt] (0,0) -- (9,0);
	\draw[double,double distance=2pt,line cap=rect] (0,4) -- (9,4);
	\draw[dashed, line width=2pt] (0,4) -- (9,4);
        \node at (2.5,2) {$\tikz[scale=0.8]{
            \draw[knot] (0.5,1.35) to[out=180, in=90] (0,1) to[out=-90, in=90] (1,0);
            \draw[knot, overcross] (0,0) to[out=90, in=-90] (1,1) to[out=90, in=0] (0.5,1.35);
        }$};
        \node at (6.5, 2) {$\tikz[scale=0.8]{
            \draw[knot] (0,0) to[out=90,in=180] (0.5, 1.35) to[out=0, in=90] (1,0);
        }$};
        \node[below] at (4.5,0) {Reidemeister I};
 }
\quad
\tikz[baseline={([yshift=-.5ex]current bounding box.center)}, scale=.4]{
	\draw (.5,0) -- (.5,4);
	\draw (4.5,0) -- (4.5,4);
	\draw (8.5,0) -- (8.5,4);
 	\draw[double,double distance=2pt,line cap=rect] (0,0) -- (9,0);
	\draw[dashed, line width=2pt] (0,0) -- (9,0);
	\draw[double,double distance=2pt,line cap=rect] (0,4) -- (9,4);
	\draw[dashed, line width=2pt] (0,4) -- (9,4);
        \node at (2.5,2) {$\tikz{
            \draw[knot] (1,0) to[out=135, in=-90] (0.3, 0.5) to[out=90, in=-135] (1,1);
            \draw[knot, overcross] (0,0) to[out=45, in=-90] (0.7, 0.5) to[out=90, in=-45] (0,1);
        }$};
        \node at (6.5,2) {$\tikz{
            \draw[knot] (1,0) to[out=135, in=-135] (1,1);
            \draw[knot, overcross] (0,0) to[out=45, in=-45] (0,1);
        }$};
        \node[below] at (4.5,0) {Reidemeister II};
 }
\quad
\tikz[baseline={([yshift=-.5ex]current bounding box.center)}, scale=.4]{
	\draw (.5,0) -- (.5,4);
	\draw (4.5,0) -- (4.5,4);
	\draw (8.5,0) -- (8.5,4);
 	\draw[double,double distance=2pt,line cap=rect] (0,0) -- (9,0);
	\draw[dashed, line width=2pt] (0,0) -- (9,0);
	\draw[double,double distance=2pt,line cap=rect] (0,4) -- (9,4);
	\draw[dashed, line width=2pt] (0,4) -- (9,4);
        \node at (2.5,2) {$\tikz[scale=0.4, y=1cm]{
            \draw[knot] (0,0) to[out=90, in=-90] (1,1);
            \draw[knot, overcross] (1,0) to[out=90, in=-90] (0,1);
            \draw[knot] (2,0) -- (2,1);
            \draw[knot] (0,1) -- (0,2);
            \draw[knot] (1,1) to[out=90, in=-90] (2,2);
            \draw[knot, overcross] (2,1) to[out=90, in=-90] (1,2);
            \draw[knot] (0,2) to[out=90, in=-90] (1,3);
            \draw[knot, overcross] (1,2) to[out=90, in=-90] (0,3);
            \draw[knot] (2,2) -- (2,3);
        }$};
        \node at (6.5,2) {$\tikz[scale=0.4, y=1cm]{
            \draw[knot] (0,0) -- (0,1);
            \draw[knot] (1,0) to[out=90, in=-90] (2,1);
            \draw[knot, overcross] (2,0) to[out=90, in=-90] (1,1);
            \draw[knot] (2,1) -- (2,2);
            \draw[knot] (0,1) to[out=90, in=-90] (1,2);
            \draw[knot, overcross] (1,1) to[out=90, in=-90] (0,2);
            \draw[knot] (0,2) -- (0,3);
            \draw[knot] (1,2) to[out=90, in=-90] (2,3);
            \draw[knot, overcross] (2,2) to[out=90, in=-90] (1,3);
        }$};
        \node[below] at (4.5,0) {Reidemeister III};
 }
\]
\[
\tikz[baseline={([yshift=-.5ex]current bounding box.center)}, scale=.4]{
	\draw (.5,0) -- (.5,4);
	\draw (4.5,0) -- (4.5,4);
	\draw (8.5,0) -- (8.5,4);
 	\draw[double,double distance=2pt,line cap=rect] (0,0) -- (9,0);
	\draw[dashed, line width=2pt] (0,0) -- (9,0);
	\draw[double,double distance=2pt,line cap=rect] (0,4) -- (9,4);
	\draw[dashed, line width=2pt] (0,4) -- (9,4);
        \node at (2.5,2) {$\tikz[y=1.2cm]{
            \draw[knot] (1,0) to[out=90, in=-90] (0.5,0.5) to[out=90, in=-90] (0,1);
        }$};
        \node at (6.5,2) {$\tikz[y=1.2cm]{
            \draw[knot] (1,0) -- (1,0.5) to[out=90, in=90] (0.5, 0.5) to[out=-90, in=-90] (0,0.5) -- (0,1);
        }$};
        \node[below] at (4.5,0) {T-move};
 }
\quad
\tikz[baseline={([yshift=-.5ex]current bounding box.center)}, scale=.4]{
	\draw (.5,0) -- (.5,4);
	\draw (4.5,0) -- (4.5,4);
	\draw (8.5,0) -- (8.5,4);
 	\draw[double,double distance=2pt,line cap=rect] (0,0) -- (9,0);
	\draw[dashed, line width=2pt] (0,0) -- (9,0);
	\draw[double,double distance=2pt,line cap=rect] (0,4) -- (9,4);
	\draw[dashed, line width=2pt] (0,4) -- (9,4);
        \node at (2.5,2) {$\tikz[scale=0.6]{
            \draw[knot] (1,0) to[out=90, in=-90] (0,1) to[out=90, in=-90] (1,2);
            \draw[knot, overcross] (0,0) to[out=90,in=180] (1,1) to[out=0, in=90] (2,0);
        }$};
        \node at (6.5,2) {$\tikz[scale=0.6]{
            \draw[knot] (1,0) to[out=90, in=-90] (2,1) to[out=90, in=-90] (1,2);
            \draw[knot, overcross] (0,0) to[out=90,in=180] (1,1) to[out=0, in=90] (2,0);
        }$};
        \node[below] at (4.5,0) {H-move};
 }
\quad
\tikz[baseline={([yshift=-.5ex]current bounding box.center)}, scale=.4]{
	\draw (.5,0) -- (.5,4);
	\draw (4.5,0) -- (4.5,4);
	\draw (8.5,0) -- (8.5,4);
 	\draw[double,double distance=2pt,line cap=rect] (0,0) -- (9,0);
	\draw[dashed, line width=2pt] (0,0) -- (9,0);
	\draw[double,double distance=2pt,line cap=rect] (0,4) -- (9,4);
	\draw[dashed, line width=2pt] (0,4) -- (9,4);
        \node at (2.5,2) {$\tikz{
            \draw[dashed] (0,0) -- (0,1);
            \draw[dashed] (0.2,0) -- (0.2,1);
            \draw[dashed] (0.5,0) -- (0.5,1);
            \draw[dashed] (0.8,0) -- (0.8,1);
            \draw[dashed] (1,0) -- (1,1);
            \draw[fill=white] (-0.1,0.5) rectangle (0.3,0.9);
            \draw[fill=white] (0.7,0.1) rectangle (1.1, 0.5);
        }$};
        \node at (6.5,2) {$\tikz{
            \draw[dashed] (0,0) -- (0,1);
            \draw[dashed] (0.2,0) -- (0.2,1);
            \draw[dashed] (0.5,0) -- (0.5,1);
            \draw[dashed] (0.8,0) -- (0.8,1);
            \draw[dashed] (1,0) -- (1,1);
            \draw[fill=white] (-0.1,0.1) rectangle (0.3,0.5);
            \draw[fill=white] (0.7,0.5) rectangle (1.1, 0.9);
        }$};
        \node[below] at (4.5,0) {N-move};
 }
\]
Each of the generating 2-morphisms above present a family of versions, any of which are obtained by picking orientations, changing the direction of the film, reflecting about the horizontal or vertical axis, changing between positive and negative crossings, and picking framings on saddles and deaths (the chronology is determined by the movies themselves). The N-movie in particular has many versions: any generating 1-morphism may be placed in the boxes, and there can be any number of strands between the two boxes.

Finally, the relations between 2-morphisms (movies) are given by the \emph{Carter-Saito} movie moves. See pages 14-18 in \cite{MR2171235} for a list of 30 of the moves (one of which has two parts). We will also draw movie moves as they appear in the proof of up-to-unit functoriality (Subsection \ref{ss:conclusion}). Move 31 is described by the relation that, for horizontally composable 2-morphisms $\alpha:f \Rightarrow f'$ and $\beta: g \Rightarrow g'$, there is an equality
\[
(\alpha \cdot \mathrm{Id}) (\mathrm{Id} \cdot \beta) = (\mathrm{Id} \cdot \beta) (\alpha \cdot \mathrm{Id})
\]
of 2-morphisms from $f g$ to $f' g'$. 

We spend Subsections \ref{ss:unicent}--\ref{ss:NPinvariant} preparing for the proof of Theorem \ref{thm3}, which is given in Subsection \ref{ss:conclusion}. In Subsections \ref{ss:unicent} and \ref{ss:khosadjuns} we show that Khovanov's arguments lift nicely to the $\mathcal{G}$-graded setting. In Subsection \ref{ss:geobims} we give requisite background on $\mathcal{G}$-grading shifts---just enough, we hope, so that an unfamiliar reader can understand the definition of the Naisse-Putyra tangle invariant provided in Subsection \ref{ss:NPinvariant}.

\subsection{The $\mathcal{G}$-graded center of $H^n$}
\label{ss:unicent}

Using results from Subsection \ref{ss:ccenter}, we investigate the $\mathcal{G}$-graded center of $H^n$ and list the implication relevant to us.

\begin{proposition}
\label{prop:centerofH}
For each $n\ge 1$, the $\mathcal{G}$-graded center of $H^n$ is isomorphic to the ring $R$; thus, $Z^\mathcal{G}_*(H^n)$ is the group of units of $R$.
\end{proposition}

\begin{proof}
Recall that $Z^\mathcal{G}(H^n)$ consists of elements with nonhomogeneous $\mathcal{G}$-degree entirely supported in the gradings $\prod_{a\in \mathrm{Ob}(\mathcal{G})} \mathrm{Id}_a$. Note that the identity morphisms of $\mathcal{G}$ take the form
\[
\mathrm{Id}_a = (1_n, (n, 0))
\]
for any crossingless matching $a$ on $2n$ points. Since $H^n = \bigoplus_{a,b \in B^n} \mathcal{F}(a 1_n \overline{b})$, there is a unique homogeneous idempotent 
\[
e_a := \otimes_n v_+ \in \mathcal{F}(a 1_n \overline{a})
\]
supported in the grading $\mathrm{Id}_a$. Therefore, an arbitrary nonhomogeneous element satisfying the grading criteria for $Z^\mathcal{G}(H^n)$ has the form
\[
z = \sum_{a\in \mathrm{Ob}(\mathcal{G})} v_a e_a
\]
for any $v_a \in R$.

Take any homogeneous element $x \in \mathcal{F}(a 1_n \overline{b}) \subset H^n$. Then, for any $z \in Z^\mathcal{G}(H^n)$
\[
\mu(z, x) = v_a \mu(e_a, x) \qquad \text{and} \qquad \mu(x,z)= v_b \mu(x, e_b).
\]
This means that $v_a = v_b$ for each $a, b\in \mathrm{Ob}(\mathcal{G})$. Thus, if $z$ is central, then $z = r \sum_{a} e_a$ for some $r \in R$. Since $\{e_a\}_{a\in \mathrm{Ob}(\mathcal{G})}$ are the units of $H^n$, it follows that 
\[
\mu(x,z) = \zeta(\abs{x}) \mu(z,x)
\]
thus $Z^\mathcal{G}(H^n) \cong R$ as $R$-modules. We refer readers to Proposition 6.2 from \cite{naisse2020odd} for a proof that the units of $H^n$ are indeed $\{e_a\}_{a\in \mathrm{Ob}(\mathcal{G})}$.
\end{proof}

Proposition \ref{prop:centerofH} and the results of Subsection \ref{ss:ccenter} combine to give the following.

\begin{corollary}
\label{cor:centerofH}
If $M$ is an invertible complex of $\mathcal{G}$-graded $(H^n, H^n)$-bimodules, then the only $\mathcal{G}$-graded automorphisms of $M$ are given by $u \cdot \mathrm{Id}$ for $u$ any unit of $R$.
\end{corollary}

\begin{remark}
Note that, in \cite{MR2078414}, Khovanov proved that the derived center of $H^n$ is isomoprhic to the cohomology ring of the $(n,n)$-Springer fiber. Morevoer, in \cite{naisse2017odd}, Naisse and Vaz proved that their ``odd center'' (which can be seen to be the specialization of the $\mathcal{G}$-graded center to the odd setting) applied to the odd arc algebras results in the oddification of the cohomology of the $(n,n)$-Springer variety (see \cite{MR3257552}). This prompts the question: how is $Z^\mathcal{G}(H^n)$ related to the cohomologies of the $(n,n)$-Springer variety?
\end{remark}

\subsection{Complexes of geometric $\mathcal{G}$-graded bimodules}
\label{ss:geobims}

To define categories of complexes of $\mathcal{G}$-graded bimodules, we need to define $\mathcal{G}$-grading shifts. See Subsection \ref{ss:nonhomoHH} for a summary of the formal theory, introduced originally in Sections 4.2--6 of \cite{naisse2020odd}. The monoid $\mathcal{I}$ underlying the $\mathcal{G}$-shifting system of \cite{naisse2020odd} consists of pairs $(W, v)$ of a chronological cobordism $W$ and $v\in \mathbb{Z} \times \mathbb{Z}$ together with a formal identity element $\mathbf{e}$ and absorbing element $0$:
\[
\mathcal{I} = \{(W, v)\}_{W, v} \sqcup \{\mathbf{e}, 0\}.
\]
We define $W_1^{v_1} \bullet W_2^{v_2} = (W_1 \bullet W_2)^{v_1 + v_2}$ to be the horizontal ($W_1$ and then $W_2$) stacking of cobordisms if they are compatible, and $0$ otherwise.

A flat tangle $t \in B_m^n$ can be interpreted as a morphism from any crossingless matching $a$ with $\abs{a} = m$ to any other $b$ with $\abs{b} = n$. Whenever $W$ is a cobordism $t \to s$, we set 
\[
\varphi_{(W, v)}(t, p) = (s, p + v + \abs{\mathbbm{1}_a W \mathbbm{1}_{\overline{b}}})
\]
where $\mathbbm{1}_a W \mathbbm{1}_{\overline{b}}$ is the coboridms $W$ turned into a cobordism without corners by gluing the identity cobordisms of $a$ and $b$ on either end. We pick the wide subcategory $Z$ of $\mathcal{G}$ (see Definition \ref{def:shiftingsyst}) to be the one whose morphisms consist of all identity tangles between potentially distinct crossingless matchings:
\[
\{(1_n: a\to b, p): n\ge 0, a, b\in B^n, p\in \mathbb{Z} \times \mathbb{Z}\}.
\]
Moreover, the components of the identity shift functor are 
\[
\mathcal{I}_{\mathrm{Id}} = \{(\mathbbm{1}_t, (0,0)) : t~\text{is a planar tangle}\} \subset \mathcal{I}
\]
where $\mathbbm{1}_t$ is the identity cobordism on $t$. We extend the shifting system to a 2-system by defining vertical composition $W_2^{v_2} \circ W_1^{v_1} = (W_2 \circ W_1)^{v_1 + v_2}$ if the intermediate flat tangles align, and zero otherwise. We omit a description of the witness maps $\beta$, $\gamma$, and $\Xi$, though they arise naturally; see Sections 5.1 and 5.2 of \cite{naisse2020odd}.

This $\mathcal{G}$-grading shifting 2-system is interesting for a few reasons. First, changes of chronology $H: W \Rightarrow W'$ induce natural transformations of grading shifting functors $\varphi_H: \varphi_W \Rightarrow \varphi_{W'}$ given by $\varphi_H(M): \varphi_W(M) \to \varphi_{W'}(M)$ where
\[
\varphi_W(m) \mapsto \iota({}_{a}H_b)^{-1} \varphi_{W'}(m)
\]
for any $\mathcal{G}$-graded $R$-module $M$ and homogeneous $m$ with $\abs{m}: a \to b$, where ${}_aH_b$ is the change of chronology $\mathbbm{1}_a W \mathbbm{1}_{\overline{b}} \Rightarrow \mathbbm{1}_a W' \mathbbm{1}_{\overline{b}}$ induced by $H$. It follows from the multiplicativity of $\iota$ that 
\begin{equation}
\label{eq:compcoc}
\varphi_{H'} \circ \varphi_H \cong \varphi_{H' \star H}
\end{equation}
for any two changes of chronology $H: W \Rightarrow W'$ and $H': W' \Rightarrow W''$. Suppose $W: t \to s$ is a cobordism with corners and let $m(t,s)$ denote the Euler characteristic of a minimal cobordism $t \to s$. Then, we define
\[
\widehat{\chi}(W) := \left(\frac{m(t,s) - \chi(W)}{2}, \frac{m(t,s) - \chi(W)}{2}\right) \in \mathbb{Z} \times \mathbb{Z}.
\]
Using (\ref{eq:compcoc}), it follows that any grading shift $\varphi_{(W, v)}$ admits a left inverse up to natural isomorphism given by
\[
\varphi_{(W, v)}^{-1} := \varphi_{(\overline{W}, -v - \widehat{\chi}(\overline{W} \circ W))}
\]
for $\overline{W}$ the mirror image of $W$ with respect to the horizontal plane, in the sense that
\[
\varphi_{(W, v)}^{-1} \circ \varphi_{(W, v)} \cong \varphi_{\mathbbm{1}_t}.
\]
For more details, see Subsection 5.3 of \cite{naisse2020odd}.

As one useful application, note that an elementary saddle cobordism $\tikz[baseline={([yshift=-.5ex]current bounding box.center)}, scale=0.45]
{
    \begin{scope}[rotate=90]
	\draw[dotted] (3,-2) circle(0.707);
	\draw (2.5,-1.5) .. controls (2.75,-1.75) and (3.25,-1.75) .. (3.5,-1.5);
	\draw (2.5,-2.5) .. controls  (2.75,-2.25) and (3.25,-2.25) .. (3.5,-2.5);
        \draw[red,thick] (3,-1.7) -- (3,-2.3);
    \end{scope}
} : \tikz[baseline={([yshift=-.5ex]current bounding box.center)}, scale=0.45]
{
    \begin{scope}[rotate=90]
	\draw[dotted] (3,-2) circle(0.707);
	\draw (2.5,-1.5) .. controls (2.75,-1.75) and (3.25,-1.75) .. (3.5,-1.5);
	\draw (2.5,-2.5) .. controls  (2.75,-2.25) and (3.25,-2.25) .. (3.5,-2.5);
    \end{scope}
} \to \tikz[baseline={([yshift=-.5ex]current bounding box.center)}, scale=0.45]
{
	\draw[dotted] (3,-2) circle(0.707);
	\draw (2.5,-1.5) .. controls (2.75,-1.75) and (3.25,-1.75) .. (3.5,-1.5);
	\draw (2.5,-2.5) .. controls  (2.75,-2.25) and (3.25,-2.25) .. (3.5,-2.5);
}$ induces the graded map 
\[
\mathcal{F}\left(\tikz[baseline={([yshift=-.5ex]current bounding box.center)}, scale=0.45]
{
    \begin{scope}[rotate=90]
	\draw[dotted] (3,-2) circle(0.707);
	\draw (2.5,-1.5) .. controls (2.75,-1.75) and (3.25,-1.75) .. (3.5,-1.5);
	\draw (2.5,-2.5) .. controls  (2.75,-2.25) and (3.25,-2.25) .. (3.5,-2.5);
        \draw[red,thick] (3,-1.7) -- (3,-2.3);
    \end{scope}
}\right) : \varphi_{\tikz[baseline={([yshift=-.5ex]current bounding box.center)}, scale=0.45]
{
    \begin{scope}[rotate=90]
	\draw[dotted] (3,-2) circle(0.707);
	\draw (2.5,-1.5) .. controls (2.75,-1.75) and (3.25,-1.75) .. (3.5,-1.5);
	\draw (2.5,-2.5) .. controls  (2.75,-2.25) and (3.25,-2.25) .. (3.5,-2.5);
        \draw[red,thick] (3,-1.7) -- (3,-2.3);
    \end{scope}
}} \mathcal{F}\left(\tikz[baseline={([yshift=-.5ex]current bounding box.center)}, scale=0.45]
{
    \begin{scope}[rotate=90]
	\draw[dotted] (3,-2) circle(0.707);
	\draw (2.5,-1.5) .. controls (2.75,-1.75) and (3.25,-1.75) .. (3.5,-1.5);
	\draw (2.5,-2.5) .. controls  (2.75,-2.25) and (3.25,-2.25) .. (3.5,-2.5);
    \end{scope}
}\right) \to \mathcal{F}\left(\tikz[baseline={([yshift=-.5ex]current bounding box.center)}, scale=0.45]
{
	\draw[dotted] (3,-2) circle(0.707);
	\draw (2.5,-1.5) .. controls (2.75,-1.75) and (3.25,-1.75) .. (3.5,-1.5);
	\draw (2.5,-2.5) .. controls  (2.75,-2.25) and (3.25,-2.25) .. (3.5,-2.5);
}\right).
\]
Consider the isomorphism induced by the change of chronology
\[
\varphi_H: \mathrm{Id} \Rightarrow \varphi_{\tikz[baseline={([yshift=-.5ex]current bounding box.center)}, scale=0.45]
{
    \begin{scope}[rotate=90]
	\draw[dotted] (3,-2) circle(0.707);
	\draw (2.5,-1.5) .. controls (2.75,-1.75) and (3.25,-1.75) .. (3.5,-1.5);
	\draw (2.5,-2.5) .. controls  (2.75,-2.25) and (3.25,-2.25) .. (3.5,-2.5);
        \draw[red,thick] (3,-1.7) -- (3,-2.3);
    \end{scope}
}}^{-1} \circ \varphi_{\tikz[baseline={([yshift=-.5ex]current bounding box.center)}, scale=0.45]
{
    \begin{scope}[rotate=90]
	\draw[dotted] (3,-2) circle(0.707);
	\draw (2.5,-1.5) .. controls (2.75,-1.75) and (3.25,-1.75) .. (3.5,-1.5);
	\draw (2.5,-2.5) .. controls  (2.75,-2.25) and (3.25,-2.25) .. (3.5,-2.5);
        \draw[red,thick] (3,-1.7) -- (3,-2.3);
    \end{scope}
}}.
\]
Then, precomposing with $\varphi_H$, the saddle can be reinterpreted as the following graded map.
\[
\mathcal{F}\left(\tikz[baseline={([yshift=-.5ex]current bounding box.center)}, scale=0.45]
{
    \begin{scope}[rotate=90]
	\draw[dotted] (3,-2) circle(0.707);
	\draw (2.5,-1.5) .. controls (2.75,-1.75) and (3.25,-1.75) .. (3.5,-1.5);
	\draw (2.5,-2.5) .. controls  (2.75,-2.25) and (3.25,-2.25) .. (3.5,-2.5);
        \draw[red,thick] (3,-1.7) -- (3,-2.3);
    \end{scope}
}\right) \circ \varphi_H : \mathcal{F}\left(\tikz[baseline={([yshift=-.5ex]current bounding box.center)}, scale=0.45]
{
    \begin{scope}[rotate=90]
	\draw[dotted] (3,-2) circle(0.707);
	\draw (2.5,-1.5) .. controls (2.75,-1.75) and (3.25,-1.75) .. (3.5,-1.5);
	\draw (2.5,-2.5) .. controls  (2.75,-2.25) and (3.25,-2.25) .. (3.5,-2.5);
    \end{scope}
}\right) \to \varphi_{\tikz[baseline={([yshift=-.5ex]current bounding box.center)}, scale=0.45]
{
    \begin{scope}[rotate=90]
	\draw[dotted] (3,-2) circle(0.707);
	\draw (2.5,-1.5) .. controls (2.75,-1.75) and (3.25,-1.75) .. (3.5,-1.5);
	\draw (2.5,-2.5) .. controls  (2.75,-2.25) and (3.25,-2.25) .. (3.5,-2.5);
        \draw[red,thick] (3,-1.7) -- (3,-2.3);
    \end{scope}
}}^{-1} \mathcal{F}\left(\tikz[baseline={([yshift=-.5ex]current bounding box.center)}, scale=0.45]
{
	\draw[dotted] (3,-2) circle(0.707);
	\draw (2.5,-1.5) .. controls (2.75,-1.75) and (3.25,-1.75) .. (3.5,-1.5);
	\draw (2.5,-2.5) .. controls  (2.75,-2.25) and (3.25,-2.25) .. (3.5,-2.5);
}\right).
\]
We compute that $\varphi_{\tikz[baseline={([yshift=-.5ex]current bounding box.center)}, scale=0.45]
{
    \begin{scope}[rotate=90]
	\draw[dotted] (3,-2) circle(0.707);
	\draw (2.5,-1.5) .. controls (2.75,-1.75) and (3.25,-1.75) .. (3.5,-1.5);
	\draw (2.5,-2.5) .. controls  (2.75,-2.25) and (3.25,-2.25) .. (3.5,-2.5);
        \draw[red,thick] (3,-1.7) -- (3,-2.3);
    \end{scope}
}}^{-1} = \varphi_{\left(\tikz[baseline={([yshift=-.5ex]current bounding box.center)}, scale=0.45]
{
	\draw[dotted] (3,-2) circle(0.707);
	\draw (2.5,-1.5) .. controls (2.75,-1.75) and (3.25,-1.75) .. (3.5,-1.5);
	\draw (2.5,-2.5) .. controls  (2.75,-2.25) and (3.25,-2.25) .. (3.5,-2.5);
        \draw[red,thick] (3,-1.7) -- (3,-2.3);
}\,,~ (1,1)\right)}$ since $\tikz[baseline={([yshift=-.5ex]current bounding box.center)}, scale=0.45]
{
	\draw[dotted] (3,-2) circle(0.707);
	\draw (2.5,-1.5) .. controls (2.75,-1.75) and (3.25,-1.75) .. (3.5,-1.5);
	\draw (2.5,-2.5) .. controls  (2.75,-2.25) and (3.25,-2.25) .. (3.5,-2.5);
        \draw[red,thick] (3,-1.7) -- (3,-2.3);
} \circ \tikz[baseline={([yshift=-.5ex]current bounding box.center)}, scale=0.45]
{
    \begin{scope}[rotate=90]
	\draw[dotted] (3,-2) circle(0.707);
	\draw (2.5,-1.5) .. controls (2.75,-1.75) and (3.25,-1.75) .. (3.5,-1.5);
	\draw (2.5,-2.5) .. controls  (2.75,-2.25) and (3.25,-2.25) .. (3.5,-2.5);
        \draw[red,thick] (3,-1.7) -- (3,-2.3);
    \end{scope}
}$ produces a tube. In general,
\[
\varphi_{\left(\tikz[baseline={([yshift=-.5ex]current bounding box.center)}, scale=0.45]
{
    \begin{scope}[rotate=90]
	\draw[dotted] (3,-2) circle(0.707);
	\draw (2.5,-1.5) .. controls (2.75,-1.75) and (3.25,-1.75) .. (3.5,-1.5);
	\draw (2.5,-2.5) .. controls  (2.75,-2.25) and (3.25,-2.25) .. (3.5,-2.5);
        \draw[red,thick] (3,-1.7) -- (3,-2.3);
    \end{scope}
}\,,~ (u,v)\right)}^{-1} = \varphi_{\left(\tikz[baseline={([yshift=-.5ex]current bounding box.center)}, scale=0.45]
{
	\draw[dotted] (3,-2) circle(0.707);
	\draw (2.5,-1.5) .. controls (2.75,-1.75) and (3.25,-1.75) .. (3.5,-1.5);
	\draw (2.5,-2.5) .. controls  (2.75,-2.25) and (3.25,-2.25) .. (3.5,-2.5);
        \draw[red,thick] (3,-1.7) -- (3,-2.3);
}\,,~ (1-u,1-v)\right)}.
\]

The $\mathbb{Z} \times \mathbb{Z}$-grading shift functor $\{v_1, v_2\}$ is defined by setting
\[
M\{v_1, v_2\} := \bigoplus_{\text{planar tangles}~t} \varphi_{\left(\mathbbm{1}_t, (v_2, v_1)\right)} (M).
\]
We note that non-ambiguous cobordisms (\textit{e.g.}, saddles which can be realized as either a split or a merge) admit grading shift functors which are naturally isomorphic to grading shift functors which are  nontrivial in the $\mathbb{Z} \times \mathbb{Z}$-degree only. We list them below:
\[
\varphi_{\tikz[baseline={([yshift=-.5ex]current bounding box.center)}, scale=0.3]{
	\draw (1,2) .. controls (1,1) and (2,1) .. (2,2);
	\draw (1,2) .. controls (1,1.75) and (2,1.75) .. (2,2);
	\draw (1,2) .. controls (1,2.25) and (2,2.25) .. (2,2);
}} \cong \{1,0\}
\qquad
\varphi_{\tikz[baseline={([yshift=-.5ex]current bounding box.center)}, scale=.3]{
	\draw (0,0) .. controls (0,1) and (1,1) .. (1,2);
	\draw (1,0) .. controls (1,1) and (2,1) .. (2,0);
	\draw (3,0) .. controls (3,1) and (2,1) .. (2,2);
	\draw (0,0) .. controls (0,-.25) and (1,-.25) .. (1,0);
	\draw[dashed] (0,0) .. controls (0,.25) and (1,.25) .. (1,0);
	\draw (2,0) .. controls (2,-.25) and (3,-.25) .. (3,0);
	\draw[dashed] (2,0) .. controls (2,.25) and (3,.25) .. (3,0);
	\draw (1,2) .. controls (1,1.75) and (2,1.75) .. (2,2);
	\draw (1,2) .. controls (1,2.25) and (2,2.25) .. (2,2);
}} \cong \{-1,0\}
\qquad
\varphi_{\tikz[baseline={([yshift=-.5ex]current bounding box.center)}, scale=.3]{
	\draw  (1,2) .. controls (1,3) and (0,3) .. (0,4);
	\draw  (2,2) .. controls (2,3) and (3,3) .. (3,4);
	\draw (1,4) .. controls (1,3) and (2,3) .. (2,4);
	\draw (0,4) .. controls (0,3.75) and (1,3.75) .. (1,4);
	\draw (0,4) .. controls (0,4.25) and (1,4.25) .. (1,4);
	\draw (2,4) .. controls (2,3.75) and (3,3.75) .. (3,4);
	\draw (2,4) .. controls (2,4.25) and (3,4.25) .. (3,4);
	\draw (1,2) .. controls (1,1.75) and (2,1.75) .. (2,2);
	\draw[dashed] (1,2) .. controls (1,2.25) and (2,2.25) .. (2,2);
}} \cong \{0, -1\}
\qquad
\varphi_{\tikz[baseline={([yshift=-.5ex]current bounding box.center)}, scale=.3]{
	\draw (1,0) .. controls (1,1) and (2,1) .. (2,0);
	\draw (1,0) .. controls (1,-.25) and (2,-.25) .. (2,0);
	\draw[dashed] (1,0) .. controls (1,.25) and (2,.25) .. (2,0);
}} \cong \{0,1\}
\]
for births, merges, splits, and deaths respectively. 

\begin{definition}
A \emph{geometric $(H^m, H^n)$-bimodule} is a $(H^m, H^n)$-bimodule which is isomorphic to a finite direct sum of bimodules $\varphi_{(W, v)} \mathcal{F}(t)$ for flat tangles $t$ and $\mathcal{G}$-grading shifts $\varphi_{(W, v)}$. Denote by $\mathcal{K}_m^n$ the category whose objects are bounded complexes of geometric $(H^m, H^n)$-bimodules (with grading-preserving differential) up to chain homotopy, and whose morphisms are grading-preserving homomorphisms of complexes up to chain homotopy. Finally, denote by $\mathsf{K}_m^n$ the result of localizing the category $\mathcal{K}_m^n$ along quasi-isomorphisms.
\end{definition}

Thus we have the following corollary of Proposition \ref{prop:centerofH}.

\begin{corollary}
\label{cor:invs}
If $f: M \to N$ is an isomorphism of invertible objects in $\mathcal{K}_n^n$, any other isomorphism differs only up to multiplication by a unit in $R$.
\end{corollary}

\subsection{Khovanov adjunctions}
\label{ss:khosadjuns}

In \cite{MR2171235}, Khovanov showed that the functors 
\[
\mathcal{F}(\cap_{i,n}) \otimes_{H^{n-1}}  - \qquad \text{and} \qquad \mathcal{F}(\cup_{i, n-1}) \otimes_{H^n} -
\]
are biadjoint up to grading shifts. We will show that a generalization of this statement exists in the $\mathcal{G}$-graded setting.

\begin{proposition}
\label{prop:Kadjun}
Suppose $M$ is an object of $\mathcal{K}_{n-1}^{n-1}$ and $N$ is an object of $\mathcal{K}_{n-1}^n$. Then,
\begin{equation}
\label{eq:adjun1}
\mathrm{Hom}_{\mathcal{K}_{n-1}^n}\left(
\tikz[baseline={([yshift=-.5ex]current bounding box.center)}, scale=.3, y=0.5cm]
{   
    \draw[dotted] (0.5,4) -- (6.5,4);
    \draw[knot] (1,4) to[out=-90, in=90] (2,1);
    \draw[knot] (2,4) to[out=-90, in=90] (3,1);
    \draw[knot] (5,4) to[out=-90, in=90] (4,1);
    \draw[knot] (6,4) to[out=-90, in=90] (5,1);
    \draw[knot] (3,4) to[out=-90, in=180] (3.5, 2.5) to[out=0, in=-90] (4,4);
    \draw[knot] (1.5, -3) rectangle (5.5, 1);
    \node at (3.5, -1) {$M$};
    \draw[knot] (2,-3) -- (2,-4);
    \draw[knot] (3,-3) -- (3,-4);
    \draw[knot] (4,-3) -- (4,-4);
    \draw[knot] (5,-3) -- (5,-4);
    \draw[dotted] (0.5,-4) -- (6.5,-4);
}
,\,
\varphi_{\left(
\tikz[baseline={([yshift=-.5ex]current bounding box.center)}, scale=.2, y=0.5cm]
{   
    \draw[dotted] (0.5,4) -- (6.5,4);
    \draw[knot] (1,4) -- (1,1);
    \draw[knot] (2,4) -- (2,1);
    \draw[knot] (3,4) -- (3,1);
    \draw[knot] (4,4) -- (4,1);
    \draw[knot] (5,4) -- (5,1);
    \draw[knot] (6,4) -- (6,1);
    \draw[knot,red] (3,2.5) -- (4,2.5);
    \draw[knot] (0.75, -3) rectangle (6.25, 1);
    \draw[knot] (2,-3) -- (2,-4);
    \draw[knot] (3,-3) -- (3,-4);
    \draw[knot] (4,-3) -- (4,-4);
    \draw[knot] (5,-3) -- (5,-4);
    \draw[dotted] (0.5,-4) -- (6.5,-4);
}
,\, (1,1)\right)}
\tikz[baseline={([yshift=-.5ex]current bounding box.center)}, scale=.3, y=0.5cm]
{   
    \draw[dotted] (0.5,4) -- (6.5,4);
    \draw[knot] (1,4) -- (1,2);
    \draw[knot] (2,4) -- (2,2);
    \draw[knot] (3,4) -- (3,2);
    \draw[knot] (4,4) -- (4,2);
    \draw[knot] (5,4) -- (5,2);
    \draw[knot] (6,4) -- (6,2);
    \draw[knot] (0.75, -2) rectangle (6.25, 2);
    \node at (3.5, 0) {$N$};
    \draw[knot] (2,-2) -- (2,-4);
    \draw[knot] (3,-2) -- (3,-4);
    \draw[knot] (4,-2) -- (4,-4);
    \draw[knot] (5,-2) -- (5,-4);
    \draw[dotted] (0.5,-4) -- (6.5,-4);
}
\right)
\cong
\mathrm{Hom}_{\mathcal{K}_{n-1}^{n-1}} \left(
\tikz[baseline={([yshift=-.5ex]current bounding box.center)}, scale=.3, y=0.5cm]
{   
    \draw[dotted] (0.5,4) -- (6.5,4);
    \draw[knot] (2,4) -- (2,2);
    \draw[knot] (3,4) -- (3,2);
    \draw[knot] (4,4) -- (4,2);
    \draw[knot] (5,4) -- (5,2);
    \draw[knot] (1.5, -2) rectangle (5.5, 2);
    \node at (3.5, 0) {$M$};
    \draw[knot] (2,-2) -- (2,-4);
    \draw[knot] (3,-2) -- (3,-4);
    \draw[knot] (4,-2) -- (4,-4);
    \draw[knot] (5,-2) -- (5,-4);
    \draw[dotted] (0.5,-4) -- (6.5,-4);
}
,\,
\tikz[baseline={([yshift=-.5ex]current bounding box.center)}, scale=.3, y=0.5cm]
{   
    \draw[dotted] (0.5,4) -- (6.5,4);
    \draw[knot] (1,1) to[out=90, in=-90] (2,4);
    \draw[knot] (2,1) to[out=90, in=-90] (3,4);
    \draw[knot] (3,1) to[out=90, in=180] (3.5, 2.5) to[out=0, in=90] (4,1);
    \draw[knot] (5,1) to[out=90, in=-90] (4,4);
    \draw[knot] (6,1) to[out=90, in=-90] (5,4);
    \draw[knot] (0.75, -3) rectangle (6.25, 1);
    \node at (3.5, -1) {$N$};
    \draw[knot] (2,-3) -- (2,-4);
    \draw[knot] (3,-3) -- (3,-4);
    \draw[knot] (4,-3) -- (4,-4);
    \draw[knot] (5,-3) -- (5,-4);
    \draw[dotted] (0.5,-4) -- (6.5,-4);
}
\right).
\end{equation}
Similarly, for $L$ an object of $\mathcal{K}_n^{n-1}$,
\begin{equation}
\label{eq:adjun2}
\mathrm{Hom}_{\mathcal{K}_{n}^{n-1}}\left(
\tikz[baseline={([yshift=-.5ex]current bounding box.center)}, scale=.3, y=0.5cm]
    {   
    \draw[dotted] (0.5,4) -- (6.5,4);
    \draw[knot] (1,-4) to[out=90, in=-90] (2,-1);
    \draw[knot] (2,-4) to[out=90, in=-90] (3,-1);
    \draw[knot] (5,-4) to[out=90, in=-90] (4,-1);
    \draw[knot] (6,-4) to[out=90, in=-90] (5,-1);
    \draw[knot] (3,-4) to[out=90, in=180] (3.5, -2.5) to[out=0, in=90] (4,-4);
    \draw[knot] (1.5, 3) rectangle (5.5, -1);
    \node at (3.5, 1) {$M$};
    \draw[knot] (2,3) -- (2,4);
    \draw[knot] (3,3) -- (3,4);
    \draw[knot] (4,3) -- (4,4);
    \draw[knot] (5,3) -- (5,4);
    \draw[dotted] (0.5,-4) -- (6.5,-4);
    }
,\,
\varphi_{\left(
    \tikz[baseline={([yshift=-.5ex]current bounding box.center)}, scale=.2, y=0.5cm]
    {   
    \draw[dotted] (0.5,4) -- (6.5,4);
    \draw[knot] (1,-4) -- (1,-1);
    \draw[knot] (2,-4) -- (2,-1);
    \draw[knot] (3,-4) -- (3,-1);
    \draw[knot] (4,-4) -- (4,-1);
    \draw[knot] (5,-4) -- (5,-1);
    \draw[knot] (6,-4) -- (6,-1);
    \draw[knot,red] (3,-2.5) -- (4,-2.5);
    \draw[knot] (0.75, 3) rectangle (6.25, -1);
    \draw[knot] (2,3) -- (2,4);
    \draw[knot] (3,3) -- (3,4);
    \draw[knot] (4,3) -- (4,4);
    \draw[knot] (5,3) -- (5,4);
    \draw[dotted] (0.5,-4) -- (6.5,-4);
    }, \, (1,1)\right)}
\tikz[baseline={([yshift=-.5ex]current bounding box.center)}, scale=.3, y=0.5cm]
{   
    \draw[dotted] (0.5,4) -- (6.5,4);
    \draw[knot] (2,4) -- (2,2);
    \draw[knot] (3,4) -- (3,2);
    \draw[knot] (4,4) -- (4,2);
    \draw[knot] (5,4) -- (5,2);
    \draw[knot] (0.75, -2) rectangle (6.25, 2);
    \node at (3.5, 0) {$L$};
    \draw[knot] (2,-2) -- (2,-4);
    \draw[knot] (3,-2) -- (3,-4);
    \draw[knot] (4,-2) -- (4,-4);
    \draw[knot] (5,-2) -- (5,-4);
    \draw[knot] (1,-2) -- (1,-4);
    \draw[knot] (6,-2) -- (6,-4);
    \draw[dotted] (0.5,-4) -- (6.5,-4);
}
\right)
\cong
\mathrm{Hom}_{\mathcal{K}_{n-1}^{n-1}}\left(
\tikz[baseline={([yshift=-.5ex]current bounding box.center)}, scale=.3, y=0.5cm]
    {   
    \draw[dotted] (0.5,4) -- (6.5,4);
    \draw[knot] (2,4) -- (2,2);
    \draw[knot] (3,4) -- (3,2);
    \draw[knot] (4,4) -- (4,2);
    \draw[knot] (5,4) -- (5,2);
    \draw[knot] (1.5, -2) rectangle (5.5, 2);
    \node at (3.5, 0) {$M$};
    \draw[knot] (2,-2) -- (2,-4);
    \draw[knot] (3,-2) -- (3,-4);
    \draw[knot] (4,-2) -- (4,-4);
    \draw[knot] (5,-2) -- (5,-4);
    \draw[dotted] (0.5,-4) -- (6.5,-4);
    }
,\,
\tikz[baseline={([yshift=-.5ex]current bounding box.center)}, scale=.3, y=0.5cm]
    {   
    \draw[dotted] (0.5,4) -- (6.5,4);
    \draw[knot] (1,-1) to[out=-90, in=90] (2,-4);
    \draw[knot] (2,-1) to[out=-90, in=90] (3,-4);
    \draw[knot] (3,-1) to[out=-90, in=180] (3.5, -2.5) to[out=0, in=-90] (4,-1);
    \draw[knot] (5,-1) to[out=-90, in=90] (4,-4);
    \draw[knot] (6,-1) to[out=-90, in=90] (5,-4);
    \draw[knot] (0.75, 3) rectangle (6.25, -1);
    \node at (3.5, 1) {$L$};
    \draw[knot] (2,3) -- (2,4);
    \draw[knot] (3,3) -- (3,4);
    \draw[knot] (4,3) -- (4,4);
    \draw[knot] (5,3) -- (5,4);
    \draw[dotted] (0.5,-4) -- (6.5,-4);
    }
\right).
\end{equation}
\end{proposition}

Notice that Khovanov's adjunction (Proposition 3 in \cite{MR2171235}) is obtained by summing the $\mathbb{Z} \times \mathbb{Z}$-degrees and subtracting the number of saddles in each of the grading shifts.

\begin{proof}
We prove (\ref{eq:adjun1}); the proof of (\ref{eq:adjun2}) is analogous. Define maps 
\[
\phi:
 \mathrm{Hom}_{\mathcal{K}_{n-1}^n} \left(\mathcal{F}(\cup_{i, n-1}) \otimes_{H^{n-1}} M, \varphi_{(W, (1,1))} N\right) 
\rightleftarrows
\mathrm{Hom}_{\mathcal{K}_{n-1}^{n-1}} \left(M, \mathcal{F}(\cap_{i,n}) \otimes_{H^n} N\right)
: \psi
\]
(for $W$ the cobordism appearing in (\ref{eq:adjun1})) as follows. We define $\phi(f)$ to be the composition
\[
\tikz[baseline={([yshift=-.5ex]current bounding box.center)}, scale=.5]
{
    \node (A) at (-2,0) {$
    \tikz[baseline={([yshift=-.5ex]current bounding box.center)}, scale=.4, y=0.5cm]
    {   
    \draw[dotted] (0.5,4) -- (6.5,4);
    \draw[knot] (2,4) -- (2,2);
    \draw[knot] (3,4) -- (3,2);
    \draw[knot] (4,4) -- (4,2);
    \draw[knot] (5,4) -- (5,2);
    \draw[knot] (1.5, -2) rectangle (5.5, 2);
    \node at (3.5, 0) {$M$};
    \draw[knot] (2,-2) -- (2,-4);
    \draw[knot] (3,-2) -- (3,-4);
    \draw[knot] (4,-2) -- (4,-4);
    \draw[knot] (5,-2) -- (5,-4);
    \draw[dotted] (0.5,-4) -- (6.5,-4);
    }
    $};
    \node (B) at (10,0) {$
    \tikz[baseline={([yshift=-.5ex]current bounding box.center)}, scale=.4, y=0.5cm]
    {   
    \draw[dotted] (0.5,4) -- (6.5,4);
    \draw[knot] (2,0) to[out=90, in=-90] (1,2) to[out=90, in=-90] (2,4);
    \draw[knot] (3,0) to[out=90, in=-90] (2,2) to[out=90, in=-90] (3,4);
    \draw[knot] (3.5,2) circle (20pt);
    \draw[knot] (4,0) to[out=90, in=-90] (5,2) to[out=90, in=-90] (4,4);
    \draw[knot] (5,0) to[out=90, in=-90] (6,2) to[out=90, in=-90] (5,4);
    \draw[dotted] (0.5,2) -- (6.5,2);
    \draw[knot] (1.5, -3) rectangle (5.5, 0);
    \node at (3.5, -1.5) {$M$};
    \draw[knot] (2,-3) -- (2,-4);
    \draw[knot] (3,-3) -- (3,-4);
    \draw[knot] (4,-3) -- (4,-4);
    \draw[knot] (5,-3) -- (5,-4);
    \draw[dotted] (0.5,-4) -- (6.5,-4) node[right]{$\{-1,0\}$};
    }
    $};
    \node (C) at (22,0) {$
    \tikz[baseline={([yshift=-.5ex]current bounding box.center)}, scale=.4, y=0.5cm]
    {   
    \draw[dotted] (0.5,4) -- (6.5,4);
    \draw[knot] (1,1) to[out=90, in=-90] (2,4);
    \draw[knot] (2,1) to[out=90, in=-90] (3,4);
    \draw[knot] (3,1) to[out=90, in=180] (3.5, 2.5) to[out=0, in=90] (4,1);
    \draw[knot] (5,1) to[out=90, in=-90] (4,4);
    \draw[knot] (6,1) to[out=90, in=-90] (5,4);
    \draw[knot] (0.75, -3) rectangle (6.25, 1);
    \node at (3.5, -1) {$N$};
    \draw[knot] (2,-3) -- (2,-4);
    \draw[knot] (3,-3) -- (3,-4);
    \draw[knot] (4,-3) -- (4,-4);
    \draw[knot] (5,-3) -- (5,-4);
    \draw[dotted] (0.5,-4) -- (6.5,-4);
    }
    $};
    \draw[->] (A) to node[pos=0.5, above]{$\tikz[baseline={([yshift=-.5ex]current bounding box.center)}, scale=0.5]{
	\draw[-] (1,2) .. controls (1,1) and (2,1) .. (2,2);
	\draw[-] (1,2) .. controls (1,1.75) and (2,1.75) .. (2,2);
	\draw[-] (1,2) .. controls (1,2.25) and (2,2.25) .. (2,2);
    }$} (B);
    \draw[->] (B) to node[pos=0.5, above]{$\mathrm{Id}_{\cap_{i, n}} \otimes f$} (C);
}
\]
This is a $\mathcal{G}$-grading preserving map since the composition of grading shifts 
\[
\{-1,0\} \circ \varphi_{
    \left(
    \tikz[baseline={([yshift=-.5ex]current bounding box.center)}, scale=.25, y=0.5cm]
    {   
    \draw[dotted] (0.5,4) -- (6.5,4);
    \draw[knot, red] (3, 1.5) -- (4,1.5);
    \draw[knot] (1,1) to[out=90, in=-90] (2,4);
    \draw[knot] (2,1) to[out=90, in=-90] (3,4);
    \draw[knot] (3,1) to[out=90, in=180] (3.5, 2.5) to[out=0, in=90] (4,1);
    \draw[knot] (5,1) to[out=90, in=-90] (4,4);
    \draw[knot] (6,1) to[out=90, in=-90] (5,4);
    \draw[knot] (0.75, -3) rectangle (6.25, 1);
    \node at (3.5, -1) {$N$};
    \draw[knot] (2,-3) -- (2,-4);
    \draw[knot] (3,-3) -- (3,-4);
    \draw[knot] (4,-3) -- (4,-4);
    \draw[knot] (5,-3) -- (5,-4);
    \draw[dotted] (0.5,-4) -- (6.5,-4);
    }
,\, (1,1)
\right)
}
\]
is naturally isomorphic to the identity shift (because the cobordism involved is the split map). Conversely, define $\psi(g)$ to be the composition
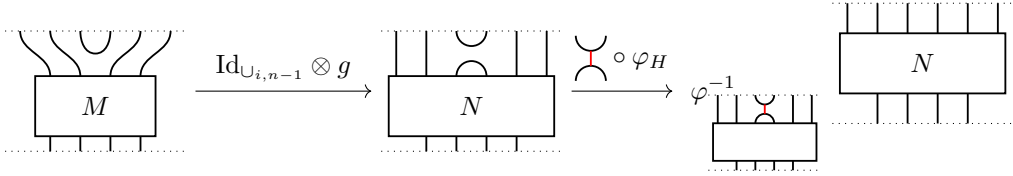
\[
\tikz[baseline={([yshift=-.5ex]current bounding box.center)}, scale=.5]
{
    \node (A) at (-2,0) {$
    \tikz[baseline={([yshift=-.5ex]current bounding box.center)}, scale=.4, y=0.5cm]
    {   
    \draw[dotted] (0.5,4) -- (6.5,4);
    \draw[knot] (1,4) to[out=-90, in=90] (2,1);
    \draw[knot] (2,4) to[out=-90, in=90] (3,1);
    \draw[knot] (5,4) to[out=-90, in=90] (4,1);
    \draw[knot] (6,4) to[out=-90, in=90] (5,1);
    \draw[knot] (3,4) to[out=-90, in=180] (3.5, 2.5) to[out=0, in=-90] (4,4);
    \draw[knot] (1.5, -3) rectangle (5.5, 1);
    \node at (3.5, -1) {$M$};
    \draw[knot] (2,-3) -- (2,-4);
    \draw[knot] (3,-3) -- (3,-4);
    \draw[knot] (4,-3) -- (4,-4);
    \draw[knot] (5,-3) -- (5,-4);
    \draw[dotted] (0.5,-4) -- (6.5,-4);
    }
    $};
    \node (B) at (8,0) {$
    \tikz[baseline={([yshift=-.5ex]current bounding box.center)}, scale=.4, y=0.5cm]
    {   
    \draw[dotted] (0.5,4) -- (6.5,4);
    \draw[knot] (1,4) -- (1,1);
    \draw[knot] (2,4) -- (2,1);
    \draw[knot] (3,4) to[out=-90, in=180] (3.5,3) to[out=0, in=-90] (4,4);
    \draw[knot] (3,1) to[out=90, in=180] (3.5,2) to[out=0, in=90] (4,1);
    \draw[knot] (5,4) -- (5,1);
    \draw[knot] (6,4) -- (6,1);
    \draw[knot] (0.75, -3) rectangle (6.25, 1);
    \node at (3.5, -1) {$N$};
    \draw[knot] (2,-3) -- (2,-4);
    \draw[knot] (3,-3) -- (3,-4);
    \draw[knot] (4,-3) -- (4,-4);
    \draw[knot] (5,-3) -- (5,-4);
    \draw[dotted] (0.5,-4) -- (6.5,-4);
    }
    $};
    \node (C) at (18,0) {$
    \tikz[baseline={([yshift=-.5ex]current bounding box.center)}, scale=.4, y=0.5cm]
    {   
    \draw[dotted] (0.5,4) -- (6.5,4);
    \draw[knot] (1,4) -- (1,2);
    \draw[knot] (2,4) -- (2,2);
    \draw[knot] (3,4) -- (3,2);
    \draw[knot] (4,4) -- (4,2);
    \draw[knot] (5,4) -- (5,2);
    \draw[knot] (6,4) -- (6,2);
    \draw[knot] (0.75, -2) rectangle (6.25, 2);
    \node at (3.5, 0) {$N$};
    \draw[knot] (2,-2) -- (2,-4);
    \draw[knot] (3,-2) -- (3,-4);
    \draw[knot] (4,-2) -- (4,-4);
    \draw[knot] (5,-2) -- (5,-4);
    \draw[dotted] (0.5,-4) -- (6.5,-4);
    \draw (0.5,-4) node[left]{$\varphi_{
    \tikz[baseline={([yshift=-.5ex]current bounding box.center)}, scale=.25, y=0.5cm]
    {   
    \draw[dotted] (0.5,4) -- (6.5,4);
    \draw[knot] (1,4) -- (1,1);
    \draw[knot] (2,4) -- (2,1);
    \draw[knot] (3,4) to[out=-90, in=180] (3.5,3) to[out=0, in=-90] (4,4);
    \draw[knot] (3,1) to[out=90, in=180] (3.5,2) to[out=0, in=90] (4,1);
    \draw[knot] (5,4) -- (5,1);
    \draw[knot] (6,4) -- (6,1);
    \draw[knot,red] (3.5,3) -- (3.5,2);
    \draw[knot] (0.75, -3) rectangle (6.25, 1);
    \draw[knot] (2,-3) -- (2,-4);
    \draw[knot] (3,-3) -- (3,-4);
    \draw[knot] (4,-3) -- (4,-4);
    \draw[knot] (5,-3) -- (5,-4);
    \draw[dotted] (0.5,-4) -- (6.5,-4);
    }
    }^{-1}$};
    }
    $};
    \draw[->] (A) to node[pos=0.5, above]{$
    \mathrm{Id}_{\cup_{i, n-1}}  \otimes g
    $} (B);
    \draw[->] (B) to node[pos=0.5, above]{$
    \tikz[baseline={([yshift=-.5ex]current bounding box.center)}, scale=.4, y=0.5cm]
    {   
    \draw[knot, -] (3,4) to[out=-90, in=180] (3.5,3) to[out=0, in=-90] (4,4);
    \draw[knot, -] (3,1) to[out=90, in=180] (3.5,2) to[out=0, in=90] (4,1);
    \draw[knot,red, -] (3.5,3) -- (3.5,2);
    }
    \circ \varphi_H
    $} (C);
}
\]
where $\varphi_H: \mathrm{Id} \Rightarrow \varphi_{
\tikz[baseline={([yshift=-.5ex]current bounding box.center)}, scale=.25, y=0.5cm]
    {   
    \draw[dotted] (0.5,4) -- (6.5,4);
    \draw[knot] (1,4) -- (1,1);
    \draw[knot] (2,4) -- (2,1);
    \draw[knot] (3,4) to[out=-90, in=180] (3.5,3) to[out=0, in=-90] (4,4);
    \draw[knot] (3,1) to[out=90, in=180] (3.5,2) to[out=0, in=90] (4,1);
    \draw[knot] (5,4) -- (5,1);
    \draw[knot] (6,4) -- (6,1);
    \draw[knot,red] (3.5,3) -- (3.5,2);
    \draw[knot] (0.75, -3) rectangle (6.25, 1);
    \draw[knot] (2,-3) -- (2,-4);
    \draw[knot] (3,-3) -- (3,-4);
    \draw[knot] (4,-3) -- (4,-4);
    \draw[knot] (5,-3) -- (5,-4);
    \draw[dotted] (0.5,-4) -- (6.5,-4);
    }}^{-1} \circ
    \varphi_{
    \tikz[baseline={([yshift=-.5ex]current bounding box.center)}, scale=.25, y=0.5cm]
	{   
    \draw[dotted] (0.5,4) -- (6.5,4);
    \draw[knot] (1,4) -- (1,1);
    \draw[knot] (2,4) -- (2,1);
    \draw[knot] (3,4) -- (3,1);
    \draw[knot] (4,4) -- (4,1);
    \draw[knot] (5,4) -- (5,1);
    \draw[knot] (6,4) -- (6,1);
    \draw[knot,red] (3,2.5) -- (4,2.5);
    \draw[knot] (0.75, -3) rectangle (6.25, 1);
    \node at (3.5, -1) {$N$};
    \draw[knot] (2,-3) -- (2,-4);
    \draw[knot] (3,-3) -- (3,-4);
    \draw[knot] (4,-3) -- (4,-4);
    \draw[knot] (5,-3) -- (5,-4);
    \draw[dotted] (0.5,-4) -- (6.5,-4);
	}
    }$. This map is graded correctly since 
    \[
    \varphi_
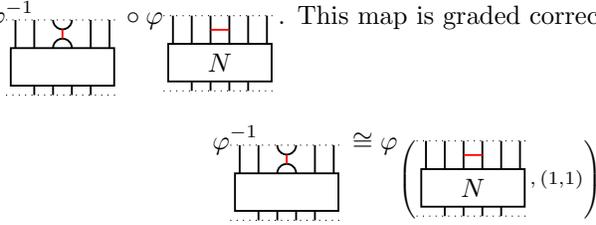
{
    \tikz[baseline={([yshift=-.5ex]current bounding box.center)}, scale=.25, y=0.5cm]
    {   
    \draw[dotted] (0.5,4) -- (6.5,4);
    \draw[knot] (1,4) -- (1,1);
    \draw[knot] (2,4) -- (2,1);
    \draw[knot] (3,4) to[out=-90, in=180] (3.5,3) to[out=0, in=-90] (4,4);
    \draw[knot] (3,1) to[out=90, in=180] (3.5,2) to[out=0, in=90] (4,1);
    \draw[knot] (5,4) -- (5,1);
    \draw[knot] (6,4) -- (6,1);
    \draw[knot,red] (3.5,3) -- (3.5,2);
    \draw[knot] (0.75, -3) rectangle (6.25, 1);
    \draw[knot] (2,-3) -- (2,-4);
    \draw[knot] (3,-3) -- (3,-4);
    \draw[knot] (4,-3) -- (4,-4);
    \draw[knot] (5,-3) -- (5,-4);
    \draw[dotted] (0.5,-4) -- (6.5,-4);
    }}^{-1}
    \cong
    \varphi_{\left(
    \tikz[baseline={([yshift=-.5ex]current bounding box.center)}, scale=.25, y=0.5cm]
{   
    \draw[dotted] (0.5,4) -- (6.5,4);
    \draw[knot] (1,4) -- (1,1);
    \draw[knot] (2,4) -- (2,1);
    \draw[knot] (3,4) -- (3,1);
    \draw[knot] (4,4) -- (4,1);
    \draw[knot] (5,4) -- (5,1);
    \draw[knot] (6,4) -- (6,1);
    \draw[knot,red] (3,2.5) -- (4,2.5);
    \draw[knot] (0.75, -3) rectangle (6.25, 1);
    \node at (3.5, -1) {$N$};
    \draw[knot] (2,-3) -- (2,-4);
    \draw[knot] (3,-3) -- (3,-4);
    \draw[knot] (4,-3) -- (4,-4);
    \draw[knot] (5,-3) -- (5,-4);
    \draw[dotted] (0.5,-4) -- (6.5,-4);
}
,\, (1,1)\right)}
    \]
    using results of the previous subsection. It remains to prove that $\psi(\phi(f)) = f$ and $\phi(\psi(g)) = g$. We compute $\psi(\phi(f))$ as the composition
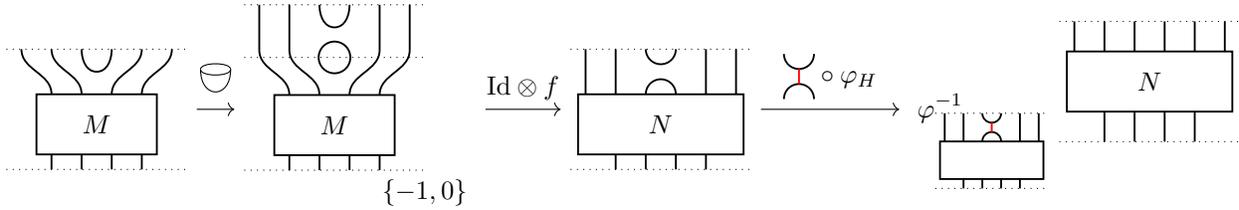
\[
\tikz[baseline={([yshift=-.5ex]current bounding box.center)}, scale=.5]
{
    \node (A) at (0,0) {$
    \tikz[baseline={([yshift=-.5ex]current bounding box.center)}, scale=.4, y=0.5cm]
    {   
    \draw[dotted] (0.5,4) -- (6.5,4);
    \draw[knot] (1,4) to[out=-90, in=90] (2,1);
    \draw[knot] (2,4) to[out=-90, in=90] (3,1);
    \draw[knot] (5,4) to[out=-90, in=90] (4,1);
    \draw[knot] (6,4) to[out=-90, in=90] (5,1);
    \draw[knot] (3,4) to[out=-90, in=180] (3.5, 2.5) to[out=0, in=-90] (4,4);
    \draw[knot] (1.5, -3) rectangle (5.5, 1);
    \node at (3.5, -1) {$M$};
    \draw[knot] (2,-3) -- (2,-4);
    \draw[knot] (3,-3) -- (3,-4);
    \draw[knot] (4,-3) -- (4,-4);
    \draw[knot] (5,-3) -- (5,-4);
    \draw[dotted] (0.5,-4) -- (6.5,-4);
    }
    $};
    \node (B) at (7,0) {$
    \tikz[baseline={([yshift=-.5ex]current bounding box.center)}, scale=.4, y=0.5cm]
    {   
    \draw[dotted] (0.5,7) -- (6.5,7);
    \draw[knot] (3,7) to[out=-90, in=180] (3.5, 5.5) to[out=0, in=-90] (4,7);
    \draw[knot] (1,7) -- (1,4);
    \draw[knot] (2,7) -- (2,4);
    \draw[knot] (5,7) -- (5,4);
    \draw[knot] (6,7) -- (6,4);
    \draw[knot] (3.5,3.5) circle (15pt);
    \draw[dotted] (0.5,3.5) -- (6.5,3.5);
    \draw[knot] (1,4) to[out=-90, in=90] (2,1);
    \draw[knot] (2,4) to[out=-90, in=90] (3,1);
    \draw[knot] (5,4) to[out=-90, in=90] (4,1);
    \draw[knot] (6,4) to[out=-90, in=90] (5,1);
    \draw[knot] (1.5, -3) rectangle (5.5, 1);
    \node at (3.5, -1) {$M$};
    \draw[knot] (2,-3) -- (2,-4);
    \draw[knot] (3,-3) -- (3,-4);
    \draw[knot] (4,-3) -- (4,-4);
    \draw[knot] (5,-3) -- (5,-4);
    \draw[dotted] (0.5,-4) -- (6.5,-4) node[below]{$\{-1,0\}$};
    }
    $};
    \node (C) at (15,0) {$
    \tikz[baseline={([yshift=-.5ex]current bounding box.center)}, scale=.4, y=0.5cm]
    {   
    \draw[dotted] (0.5,4) -- (6.5,4);
    \draw[knot] (1,4) -- (1,1);
    \draw[knot] (2,4) -- (2,1);
    \draw[knot] (3,4) to[out=-90, in=180] (3.5,3) to[out=0, in=-90] (4,4);
    \draw[knot] (3,1) to[out=90, in=180] (3.5,2) to[out=0, in=90] (4,1);
    \draw[knot] (5,4) -- (5,1);
    \draw[knot] (6,4) -- (6,1);
    \draw[knot] (0.75, -3) rectangle (6.25, 1);
    \node at (3.5, -1) {$N$};
    \draw[knot] (2,-3) -- (2,-4);
    \draw[knot] (3,-3) -- (3,-4);
    \draw[knot] (4,-3) -- (4,-4);
    \draw[knot] (5,-3) -- (5,-4);
    \draw[dotted] (0.5,-4) -- (6.5,-4);
    }
    $};
    \node (D) at (26,0) {$
    \tikz[baseline={([yshift=-.5ex]current bounding box.center)}, scale=.4, y=0.5cm]
    {   
    \draw[dotted] (0.5,4) -- (6.5,4);
    \draw[knot] (1,4) -- (1,2);
    \draw[knot] (2,4) -- (2,2);
    \draw[knot] (3,4) -- (3,2);
    \draw[knot] (4,4) -- (4,2);
    \draw[knot] (5,4) -- (5,2);
    \draw[knot] (6,4) -- (6,2);
    \draw[knot] (0.75, -2) rectangle (6.25, 2);
    \node at (3.5, 0) {$N$};
    \draw[knot] (2,-2) -- (2,-4);
    \draw[knot] (3,-2) -- (3,-4);
    \draw[knot] (4,-2) -- (4,-4);
    \draw[knot] (5,-2) -- (5,-4);
    \draw[dotted] (0.5,-4) -- (6.5,-4);
    \draw (0.5,-4) node[left]{$\varphi_{
    \tikz[baseline={([yshift=-.5ex]current bounding box.center)}, scale=.25, y=0.5cm]
    {   
    \draw[dotted] (0.5,4) -- (6.5,4);
    \draw[knot] (1,4) -- (1,1);
    \draw[knot] (2,4) -- (2,1);
    \draw[knot] (3,4) to[out=-90, in=180] (3.5,3) to[out=0, in=-90] (4,4);
    \draw[knot] (3,1) to[out=90, in=180] (3.5,2) to[out=0, in=90] (4,1);
    \draw[knot] (5,4) -- (5,1);
    \draw[knot] (6,4) -- (6,1);
    \draw[knot,red] (3.5,3) -- (3.5,2);
    \draw[knot] (0.75, -3) rectangle (6.25, 1);
    \draw[knot] (2,-3) -- (2,-4);
    \draw[knot] (3,-3) -- (3,-4);
    \draw[knot] (4,-3) -- (4,-4);
    \draw[knot] (5,-3) -- (5,-4);
    \draw[dotted] (0.5,-4) -- (6.5,-4);
    }
    }^{-1}$};
    }
    $};
    \draw[->] (A) to node[pos=0.5, above]{$\tikz[baseline={([yshift=-.5ex]current bounding box.center)}, scale=0.4]{
	\draw[-] (1,2) .. controls (1,1) and (2,1) .. (2,2);
	\draw[-] (1,2) .. controls (1,1.75) and (2,1.75) .. (2,2);
	\draw[-] (1,2) .. controls (1,2.25) and (2,2.25) .. (2,2);
    }$} (B);
    \draw[->] (B) ton node[pos=0.5, above]{$\mathrm{Id}\otimes f$} (C);
    \draw[->] (C) to node[pos=0.5, above]{$
    \tikz[baseline={([yshift=-.5ex]current bounding box.center)}, scale=.4, y=0.5cm]
    {   
    \draw[knot, -] (3,4) to[out=-90, in=180] (3.5,3) to[out=0, in=-90] (4,4);
    \draw[knot, -] (3,1) to[out=90, in=180] (3.5,2) to[out=0, in=90] (4,1);
    \draw[knot,red, -] (3.5,3) -- (3.5,2);
    }
    \circ \varphi_H
    $} (D);
}
\]
Then, notice that $f$ and the saddle occur in disjoint cylinders. Thus, these maps can slide past one another as long as we compensate by a grading shift induced by this locally vertical change of chronology. Since
\[
\varphi_{
    \left(
    \tikz[baseline={([yshift=-.5ex]current bounding box.center)}, scale=.25, y=0.5cm]
    {   
    \draw[dotted] (0.5,4) -- (6.5,4);
    \draw[knot, red] (3, 1.5) -- (4,1.5);
    \draw[knot] (1,1) to[out=90, in=-90] (2,4);
    \draw[knot] (2,1) to[out=90, in=-90] (3,4);
    \draw[knot] (3,1) to[out=90, in=180] (3.5, 2.5) to[out=0, in=90] (4,1);
    \draw[knot] (5,1) to[out=90, in=-90] (4,4);
    \draw[knot] (6,1) to[out=90, in=-90] (5,4);
    \draw[knot] (0.75, -3) rectangle (6.25, 1);
    \node at (3.5, -1) {$N$};
    \draw[knot] (2,-3) -- (2,-4);
    \draw[knot] (3,-3) -- (3,-4);
    \draw[knot] (4,-3) -- (4,-4);
    \draw[knot] (5,-3) -- (5,-4);
    \draw[dotted] (0.5,-4) -- (6.5,-4);
    }
,\, (1,1)
\right)}
\cong \{1,0\} \cong
\varphi_{\tikz[baseline={([yshift=-.5ex]current bounding box.center)}, scale=.3]{
	\draw (0,0) .. controls (0,1) and (1,1) .. (1,2);
	\draw (1,0) .. controls (1,1) and (2,1) .. (2,0);
	\draw (3,0) .. controls (3,1) and (2,1) .. (2,2);
	\draw (0,0) .. controls (0,-.25) and (1,-.25) .. (1,0);
	\draw[dashed] (0,0) .. controls (0,.25) and (1,.25) .. (1,0);
	\draw (2,0) .. controls (2,-.25) and (3,-.25) .. (3,0);
	\draw[dashed] (2,0) .. controls (2,.25) and (3,.25) .. (3,0);
	\draw (1,2) .. controls (1,1.75) and (2,1.75) .. (2,2);
	\draw (1,2) .. controls (1,2.25) and (2,2.25) .. (2,2);
}}^{-1}
\]
it follows that this composition is equivalent to 
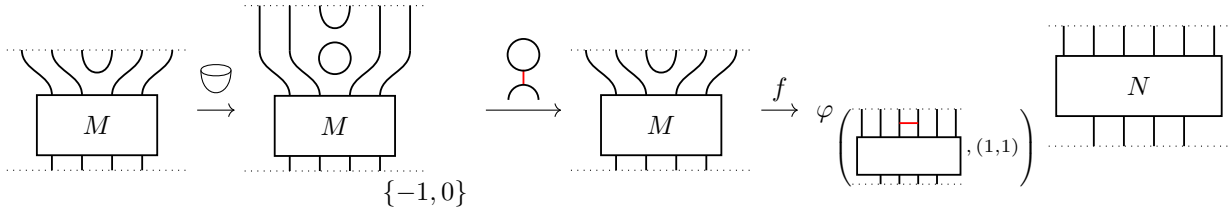
\[
\tikz[baseline={([yshift=-.5ex]current bounding box.center)}, scale=.5]
{
    \node (A) at (0,0) {$
    \tikz[baseline={([yshift=-.5ex]current bounding box.center)}, scale=.4, y=0.5cm]
    {   
    \draw[dotted] (0.5,4) -- (6.5,4);
    \draw[knot] (1,4) to[out=-90, in=90] (2,1);
    \draw[knot] (2,4) to[out=-90, in=90] (3,1);
    \draw[knot] (5,4) to[out=-90, in=90] (4,1);
    \draw[knot] (6,4) to[out=-90, in=90] (5,1);
    \draw[knot] (3,4) to[out=-90, in=180] (3.5, 2.5) to[out=0, in=-90] (4,4);
    \draw[knot] (1.5, -3) rectangle (5.5, 1);
    \node at (3.5, -1) {$M$};
    \draw[knot] (2,-3) -- (2,-4);
    \draw[knot] (3,-3) -- (3,-4);
    \draw[knot] (4,-3) -- (4,-4);
    \draw[knot] (5,-3) -- (5,-4);
    \draw[dotted] (0.5,-4) -- (6.5,-4);
    }
    $};
    \node (B) at (7,0) {$
    \tikz[baseline={([yshift=-.5ex]current bounding box.center)}, scale=.4, y=0.5cm]
    {   
    \draw[dotted] (0.5,7) -- (6.5,7);
    \draw[knot] (3,7) to[out=-90, in=180] (3.5, 5.5) to[out=0, in=-90] (4,7);
    \draw[knot] (1,7) -- (1,4);
    \draw[knot] (2,7) -- (2,4);
    \draw[knot] (5,7) -- (5,4);
    \draw[knot] (6,7) -- (6,4);
    \draw[knot] (3.5,3.5) circle (15pt);
    \draw[knot] (1,4) to[out=-90, in=90] (2,1);
    \draw[knot] (2,4) to[out=-90, in=90] (3,1);
    \draw[knot] (5,4) to[out=-90, in=90] (4,1);
    \draw[knot] (6,4) to[out=-90, in=90] (5,1);
    \draw[knot] (1.5, -3) rectangle (5.5, 1);
    \node at (3.5, -1) {$M$};
    \draw[knot] (2,-3) -- (2,-4);
    \draw[knot] (3,-3) -- (3,-4);
    \draw[knot] (4,-3) -- (4,-4);
    \draw[knot] (5,-3) -- (5,-4);
    \draw[dotted] (0.5,-4) -- (6.5,-4) node[below]{$\{-1,0\}$};
    }
    $};
    \node (C) at (15,0) {$
    \tikz[baseline={([yshift=-.5ex]current bounding box.center)}, scale=.4, y=0.5cm]
    {   
    \draw[dotted] (0.5,4) -- (6.5,4);
    \draw[knot] (1,4) to[out=-90, in=90] (2,1);
    \draw[knot] (2,4) to[out=-90, in=90] (3,1);
    \draw[knot] (5,4) to[out=-90, in=90] (4,1);
    \draw[knot] (6,4) to[out=-90, in=90] (5,1);
    \draw[knot] (3,4) to[out=-90, in=180] (3.5, 2.5) to[out=0, in=-90] (4,4);
    \draw[knot] (1.5, -3) rectangle (5.5, 1);
    \node at (3.5, -1) {$M$};
    \draw[knot] (2,-3) -- (2,-4);
    \draw[knot] (3,-3) -- (3,-4);
    \draw[knot] (4,-3) -- (4,-4);
    \draw[knot] (5,-3) -- (5,-4);
    \draw[dotted] (0.5,-4) -- (6.5,-4);
    }
    $};
    \node (D) at (24.5,0) {$
    \tikz[baseline={([yshift=-.5ex]current bounding box.center)}, scale=.4, y=0.5cm]
    {   
    \draw[dotted] (0.5,4) -- (6.5,4);
    \draw[knot] (1,4) -- (1,2);
    \draw[knot] (2,4) -- (2,2);
    \draw[knot] (3,4) -- (3,2);
    \draw[knot] (4,4) -- (4,2);
    \draw[knot] (5,4) -- (5,2);
    \draw[knot] (6,4) -- (6,2);
    \draw[knot] (0.75, -2) rectangle (6.25, 2);
    \node at (3.5, 0) {$N$};
    \draw[knot] (2,-2) -- (2,-4);
    \draw[knot] (3,-2) -- (3,-4);
    \draw[knot] (4,-2) -- (4,-4);
    \draw[knot] (5,-2) -- (5,-4);
    \draw[dotted] (0.5,-4) -- (6.5,-4);
    \draw (0.5,-4) node[left]{$\varphi_{\left(
    \tikz[baseline={([yshift=-.5ex]current bounding box.center)}, scale=.25, y=0.5cm]
{   
    \draw[dotted] (0.5,4) -- (6.5,4);
    \draw[knot] (1,4) -- (1,1);
    \draw[knot] (2,4) -- (2,1);
    \draw[knot] (3,4) -- (3,1);
    \draw[knot] (4,4) -- (4,1);
    \draw[knot] (5,4) -- (5,1);
    \draw[knot] (6,4) -- (6,1);
    \draw[knot,red] (3,2.5) -- (4,2.5);
    \draw[knot] (0.75, -3) rectangle (6.25, 1);
    \draw[knot] (2,-3) -- (2,-4);
    \draw[knot] (3,-3) -- (3,-4);
    \draw[knot] (4,-3) -- (4,-4);
    \draw[knot] (5,-3) -- (5,-4);
    \draw[dotted] (0.5,-4) -- (6.5,-4);
}
,\, (1,1)\right)}$};
    }
    $};
    \draw[->] (A) to node[pos=0.5, above]{$\tikz[baseline={([yshift=-.5ex]current bounding box.center)}, scale=0.4]{
	\draw[-] (1,2) .. controls (1,1) and (2,1) .. (2,2);
	\draw[-] (1,2) .. controls (1,1.75) and (2,1.75) .. (2,2);
	\draw[-] (1,2) .. controls (1,2.25) and (2,2.25) .. (2,2);
    }$} (B);
    \draw[->] (B) to node[pos=0.5, above]{$
    \tikz[baseline={([yshift=-.5ex]current bounding box.center)}, scale=.4, y=0.5cm]
    {   
    \draw[knot,red, -] (3.5,3) -- (3.5,2);
    \draw[knot] (3.5,4) circle (15pt);
    \draw[knot, -] (3,1) to[out=90, in=180] (3.5,2) to[out=0, in=90] (4,1);
    }
    $} (C);
    \draw[->] (C) to node[pos=0.5, above]{$f$} (D);
}
\]
Merges following a birth cancel in the linearized chronological cobordism category, so the desired result follows. To conclude the proof, we compute $\phi(\psi(g))$ as the composition
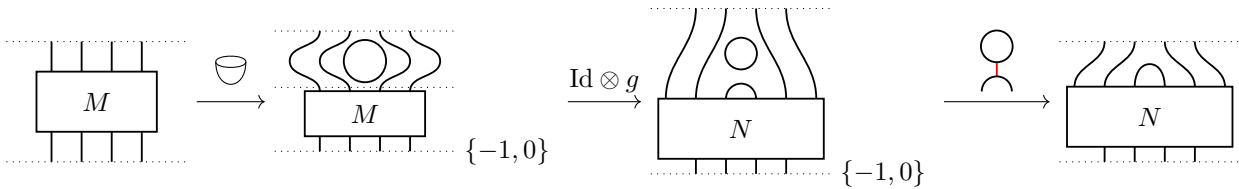
\[
\tikz[baseline={([yshift=-.5ex]current bounding box.center)}, scale=.5]
{
    \node (A) at (0,0) {$
    \tikz[baseline={([yshift=-.5ex]current bounding box.center)}, scale=.4, y=0.5cm]
    {   
    \draw[dotted] (0.5,4) -- (6.5,4);
    \draw[knot] (2,4) -- (2,2);
    \draw[knot] (3,4) -- (3,2);
    \draw[knot] (4,4) -- (4,2);
    \draw[knot] (5,4) -- (5,2);
    \draw[knot] (1.5, -2) rectangle (5.5, 2);
    \node at (3.5, 0) {$M$};
    \draw[knot] (2,-2) -- (2,-4);
    \draw[knot] (3,-2) -- (3,-4);
    \draw[knot] (4,-2) -- (4,-4);
    \draw[knot] (5,-2) -- (5,-4);
    \draw[dotted] (0.5,-4) -- (6.5,-4);
    }
    $};
    \node (B) at (8.5,0) {$
    \tikz[baseline={([yshift=-.5ex]current bounding box.center)}, scale=.4, y=0.5cm]
    {   
    \draw[dotted] (0.5,4) -- (6.5,4);
    \draw[knot] (2,0) to[out=90, in=-90] (1,2) to[out=90, in=-90] (2,4);
    \draw[knot] (3,0) to[out=90, in=-90] (2,2) to[out=90, in=-90] (3,4);
    \draw[knot] (3.5,2) circle (20pt);
    \draw[knot] (4,0) to[out=90, in=-90] (5,2) to[out=90, in=-90] (4,4);
    \draw[knot] (5,0) to[out=90, in=-90] (6,2) to[out=90, in=-90] (5,4);
    \draw[dotted] (0.5,0.25) -- (6.5,0.25);
    \draw[knot] (1.5, -3) rectangle (5.5, 0);
    \node at (3.5, -1.5) {$M$};
    \draw[knot] (2,-3) -- (2,-4);
    \draw[knot] (3,-3) -- (3,-4);
    \draw[knot] (4,-3) -- (4,-4);
    \draw[knot] (5,-3) -- (5,-4);
    \draw[dotted] (0.5,-4) -- (6.5,-4) node[right]{$\{-1,0\}$};
    }
    $};
    \node (C) at (18.5,0) {$
    \tikz[baseline={([yshift=-.5ex]current bounding box.center)}, scale=.4, y=0.5cm]
    {   
    \draw[dotted] (0.5,7) -- (6.5,7);
    \draw[knot] (2,7) to[out=-90, in=90] (1,1);
    \draw[knot] (3,7) to[out=-90, in=90] (2,1);
    \draw[knot] (3.5, 4) circle (15pt);
    \draw[knot] (3,1) to[out=90, in=180] (3.5,2) to[out=0, in=90] (4,1);
    \draw[knot] (4,7) to[out=-90, in=90] (5,1);
    \draw[knot] (5,7) to[out=-90, in=90] (6,1);
    \draw[knot] (0.75, -3) rectangle (6.25, 1);
    \node at (3.5, -1) {$N$};
    \draw[knot] (2,-3) -- (2,-4);
    \draw[knot] (3,-3) -- (3,-4);
    \draw[knot] (4,-3) -- (4,-4);
    \draw[knot] (5,-3) -- (5,-4);
    \draw[dotted] (0.5,-4) -- (6.5,-4) node[right]{$\{-1,0\}$};
    }
    $};
    \node (D) at (28,0) {$
        \tikz[baseline={([yshift=-.5ex]current bounding box.center)}, scale=.4, y=0.5cm]
    {   
    \draw[dotted] (0.5,4) -- (6.5,4);
    \draw[knot] (1,1) to[out=90, in=-90] (2,4);
    \draw[knot] (2,1) to[out=90, in=-90] (3,4);
    \draw[knot] (3,1) to[out=90, in=180] (3.5, 2.5) to[out=0, in=90] (4,1);
    \draw[knot] (5,1) to[out=90, in=-90] (4,4);
    \draw[knot] (6,1) to[out=90, in=-90] (5,4);
    \draw[knot] (0.75, -3) rectangle (6.25, 1);
    \node at (3.5, -1) {$N$};
    \draw[knot] (2,-3) -- (2,-4);
    \draw[knot] (3,-3) -- (3,-4);
    \draw[knot] (4,-3) -- (4,-4);
    \draw[knot] (5,-3) -- (5,-4);
    \draw[dotted] (0.5,-4) -- (6.5,-4);
    }
    $};
    \draw[->] (A) to node[pos=0.5, above]{$\tikz[baseline={([yshift=-.5ex]current bounding box.center)}, scale=0.4]{
	\draw[-] (1,2) .. controls (1,1) and (2,1) .. (2,2);
	\draw[-] (1,2) .. controls (1,1.75) and (2,1.75) .. (2,2);
	\draw[-] (1,2) .. controls (1,2.25) and (2,2.25) .. (2,2);
    }$} (B);
    \draw[->] (B) ton node[pos=0.5, above]{$\mathrm{Id}\otimes g$} (C);
    \draw[->] (C) to node[pos=0.5, above]{$
    \tikz[baseline={([yshift=-.5ex]current bounding box.center)}, scale=.4, y=0.5cm]
    {   
    \draw[knot,red, -] (3.5,3) -- (3.5,2);
    \draw[knot] (3.5,4) circle (15pt);
    \draw[knot, -] (3,1) to[out=90, in=180] (3.5,2) to[out=0, in=90] (4,1);
    }
    $} (D);
}
\]
Since $g$ is a $\mathcal{G}$-graded map, it commutes with the birth up to a grading shift functor which is isomorphic to the identity shift, yielding the following composite.
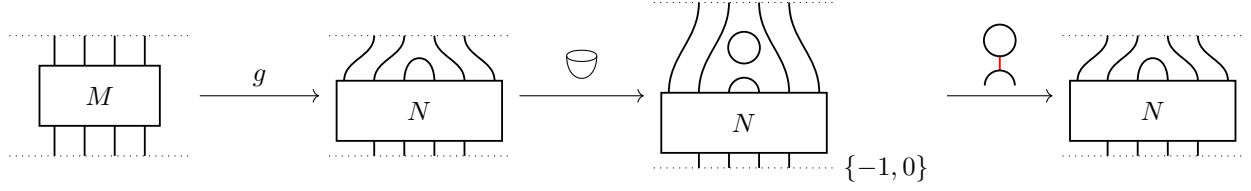
\[
\tikz[baseline={([yshift=-.5ex]current bounding box.center)}, scale=.5]
{
    \node (A) at (0,0) {$
    \tikz[baseline={([yshift=-.5ex]current bounding box.center)}, scale=.4, y=0.5cm]
    {   
    \draw[dotted] (0.5,4) -- (6.5,4);
    \draw[knot] (2,4) -- (2,2);
    \draw[knot] (3,4) -- (3,2);
    \draw[knot] (4,4) -- (4,2);
    \draw[knot] (5,4) -- (5,2);
    \draw[knot] (1.5, -2) rectangle (5.5, 2);
    \node at (3.5, 0) {$M$};
    \draw[knot] (2,-2) -- (2,-4);
    \draw[knot] (3,-2) -- (3,-4);
    \draw[knot] (4,-2) -- (4,-4);
    \draw[knot] (5,-2) -- (5,-4);
    \draw[dotted] (0.5,-4) -- (6.5,-4);
    }
    $};
    \node (B) at (8.5,0) {$
        \tikz[baseline={([yshift=-.5ex]current bounding box.center)}, scale=.4, y=0.5cm]
    {   
    \draw[dotted] (0.5,4) -- (6.5,4);
    \draw[knot] (1,1) to[out=90, in=-90] (2,4);
    \draw[knot] (2,1) to[out=90, in=-90] (3,4);
    \draw[knot] (3,1) to[out=90, in=180] (3.5, 2.5) to[out=0, in=90] (4,1);
    \draw[knot] (5,1) to[out=90, in=-90] (4,4);
    \draw[knot] (6,1) to[out=90, in=-90] (5,4);
    \draw[knot] (0.75, -3) rectangle (6.25, 1);
    \node at (3.5, -1) {$N$};
    \draw[knot] (2,-3) -- (2,-4);
    \draw[knot] (3,-3) -- (3,-4);
    \draw[knot] (4,-3) -- (4,-4);
    \draw[knot] (5,-3) -- (5,-4);
    \draw[dotted] (0.5,-4) -- (6.5,-4);
    }
    $};
    \node (C) at (18.5,0) {$
    \tikz[baseline={([yshift=-.5ex]current bounding box.center)}, scale=.4, y=0.5cm]
    {   
    \draw[dotted] (0.5,7) -- (6.5,7);
    \draw[knot] (2,7) to[out=-90, in=90] (1,1);
    \draw[knot] (3,7) to[out=-90, in=90] (2,1);
    \draw[knot] (3.5, 4) circle (15pt);
    \draw[knot] (3,1) to[out=90, in=180] (3.5,2) to[out=0, in=90] (4,1);
    \draw[knot] (4,7) to[out=-90, in=90] (5,1);
    \draw[knot] (5,7) to[out=-90, in=90] (6,1);
    \draw[knot] (0.75, -3) rectangle (6.25, 1);
    \node at (3.5, -1) {$N$};
    \draw[knot] (2,-3) -- (2,-4);
    \draw[knot] (3,-3) -- (3,-4);
    \draw[knot] (4,-3) -- (4,-4);
    \draw[knot] (5,-3) -- (5,-4);
    \draw[dotted] (0.5,-4) -- (6.5,-4) node[right]{$\{-1,0\}$};
    }
    $};
    \node (D) at (28,0) {$
        \tikz[baseline={([yshift=-.5ex]current bounding box.center)}, scale=.4, y=0.5cm]
    {   
    \draw[dotted] (0.5,4) -- (6.5,4);
    \draw[knot] (1,1) to[out=90, in=-90] (2,4);
    \draw[knot] (2,1) to[out=90, in=-90] (3,4);
    \draw[knot] (3,1) to[out=90, in=180] (3.5, 2.5) to[out=0, in=90] (4,1);
    \draw[knot] (5,1) to[out=90, in=-90] (4,4);
    \draw[knot] (6,1) to[out=90, in=-90] (5,4);
    \draw[knot] (0.75, -3) rectangle (6.25, 1);
    \node at (3.5, -1) {$N$};
    \draw[knot] (2,-3) -- (2,-4);
    \draw[knot] (3,-3) -- (3,-4);
    \draw[knot] (4,-3) -- (4,-4);
    \draw[knot] (5,-3) -- (5,-4);
    \draw[dotted] (0.5,-4) -- (6.5,-4);
    }
    $};
    \draw[->] (A) to node[pos=0.5, above]{$g$} (B);
    \draw[->] (B) to node[pos=0.5, above]{$\tikz[baseline={([yshift=-.5ex]current bounding box.center)}, scale=0.4]{
	\draw[-] (1,2) .. controls (1,1) and (2,1) .. (2,2);
	\draw[-] (1,2) .. controls (1,1.75) and (2,1.75) .. (2,2);
	\draw[-] (1,2) .. controls (1,2.25) and (2,2.25) .. (2,2);
    }$} (C);
    \draw[->] (C) to node[pos=0.5, above]{$
    \tikz[baseline={([yshift=-.5ex]current bounding box.center)}, scale=.4, y=0.5cm]
    {   
    \draw[knot,red, -] (3.5,3) -- (3.5,2);
    \draw[knot] (3.5,4) circle (15pt);
    \draw[knot, -] (3,1) to[out=90, in=180] (3.5,2) to[out=0, in=90] (4,1);
    }
    $} (D);
}
\]
By the same argument as before, we obtain $\phi(\psi(g)) = g$.
\end{proof}

\begin{corollary}
\label{cor:cup_cap}
The only $\mathcal{G}$-graded endomorphisms of the bimodules $\mathcal{F}(\cup_{i, n-1})$ and $\mathcal{F}(\cap_{i, n})$ are multiplications by elements of $R$. Thus, the only $\mathcal{G}$-graded automorphisms of these bimodules are multiplications by the units of $R$.
\end{corollary}

\begin{proof}
To prove the first claim, we apply (\ref{eq:adjun1}) from Proposition \ref{prop:Kadjun} for $M = H^{n-1}$ and $N= \mathcal{F}(\cup_{i, n-1})$. Again, notice that the relevant grading shift is
\[
\varphi_{\left(
\tikz[baseline={([yshift=-.5ex]current bounding box.center)}, scale=.3, y=0.7cm]
{   
    \draw[dotted] (0.5,4) -- (6.5,4);
    \draw[knot, red] (3, 3.5) -- (4,3.5);
    \draw[knot] (1,4) to[out=-90, in=90] (2,0);
    \draw[knot] (2,4) to[out=-90, in=90] (3,0);
    \draw[knot] (5,4) to[out=-90, in=90] (4,0);
    \draw[knot] (6,4) to[out=-90, in=90] (5,0);
    \draw[knot] (3,4) to[out=-90, in=180] (3.5, 2.5) to[out=0, in=-90] (4,4);
    \draw[dotted] (0.5,0) -- (6.5,0);
},\, (1,1)\right)} \cong \{1,0\}.
\]
pictured for $n = 3$. Then, we compute
\begin{align*}
\mathrm{Hom}_{\mathcal{K}_{n-1}^n}\left(
\tikz[baseline={([yshift=-.5ex]current bounding box.center)}, scale=.3, y=0.7cm]
{   
    \draw[dotted] (0.5,4) -- (6.5,4);
    \draw[knot] (1,4) to[out=-90, in=90] (2,0);
    \draw[knot] (2,4) to[out=-90, in=90] (3,0);
    \draw[knot] (5,4) to[out=-90, in=90] (4,0);
    \draw[knot] (6,4) to[out=-90, in=90] (5,0);
    \draw[knot] (3,4) to[out=-90, in=180] (3.5, 2.5) to[out=0, in=-90] (4,4);
    \draw[dotted] (0.5,0) -- (6.5,0);
},\,
\tikz[baseline={([yshift=-.5ex]current bounding box.center)}, scale=.3, y=0.7cm]
{   
    \draw[dotted] (0.5,4) -- (6.5,4);
    \draw[knot] (1,4) to[out=-90, in=90] (2,0);
    \draw[knot] (2,4) to[out=-90, in=90] (3,0);
    \draw[knot] (5,4) to[out=-90, in=90] (4,0);
    \draw[knot] (6,4) to[out=-90, in=90] (5,0);
    \draw[knot] (3,4) to[out=-90, in=180] (3.5, 2.5) to[out=0, in=-90] (4,4);
    \draw[dotted] (0.5,0) -- (6.5,0);
}
\right) & \cong
\mathrm{Hom}_{\mathcal{K}_{n-1}^{n-1}}\left(
H^{n-1},\,
\tikz[baseline={([yshift=-.5ex]current bounding box.center)}, scale=.3, y=0.7cm]
{   
    \draw[dotted] (0.5,4) -- (6.5,4);
    \draw[knot] (2,0) to[out=90, in=-90] (1,2) to[out=90, in=-90] (2,4);
    \draw[knot] (3,0) to[out=90, in=-90] (2,2) to[out=90, in=-90] (3,4);
    \draw[knot] (3.5,2) circle (20pt);
    \draw[knot] (4,0) to[out=90, in=-90] (5,2) to[out=90, in=-90] (4,4);
    \draw[knot] (5,0) to[out=90, in=-90] (6,2) to[out=90, in=-90] (5,4);
    \draw[dotted] (0.5,0) -- (6.5,0);
}
\{-1, 0\}
\right)
\\ & \cong \mathrm{Hom}_{\mathcal{K}_{n-1}^{n-1}}(H^{n-1}, H^{n-1} \oplus H^{n-1}\{-1,-1\})
\\ & \cong \mathrm{Hom}_{\mathcal{K}_{n-1}^{n-1}}(H^{n-1}, H^{n-1})
\\ & \cong R.
\end{align*}
The first line is from the adjunction, after rearranging the $\mathbb{Z} \times \mathbb{Z}$-grading shift. The second isomorphism follows from removing the free loop, and the fourth comes from Corollary \ref{cor:centerofH}. The third comes from the observation that
\[
\mathrm{Hom}_{\mathcal{K}_{n-1}^{n-1}}(H^{n-1}, H^{n-1}\{-1, -1\}) = 0.
\]
The second claim follows by applying the same arguments for (\ref{eq:adjun2}) of Proposition \ref{prop:Kadjun}, taking $M = H^{n-1}$ and $L = \mathcal{F}(\cap_{i,n)}$.
\end{proof}

Actually, in general,
\[
\mathrm{Hom}_{\mathcal{K}_n^n}(H^n, H^n\{k,k\}) = 0
\]
whenever $k\le -1$. Thus, iterating the argument of the above proof, we have the following more general statement.

\begin{corollary}
\label{cor:cups_caps}
If $M$ is the tensor product of bimodules 
\[
M = \mathcal{F}(\cap_{i_1, n_1-1}) \otimes_{H^{n_1-1}} \cdots \otimes_{H^{n_j-1}} \mathcal{F}(\cap_{i_{j+1}, n_{j+1}-1})
\]
then the only $\mathcal{G}$-graded endomorphisms $M$ are multiplications by elements of $R$, and the only $\mathcal{G}$-graded automorphisms of $M$ are multiplications by the units of $R$. The same holds if $M$ is a tensor product of such bimodules and invertible complexes. Everything stated also holds if we replace $\cap$ with $\cup$ everywhere.
\end{corollary}

Actually, whenever $\mathcal{F}(\cup_{i, n-1}) \otimes_{H^{n-1}} M = N$ or $M \otimes_{H^{n-1}} \mathcal{F}(\cap_{i, n}) = L$, the grading shifts of Proposition \ref{prop:Kadjun} are naturally isomorphic to $\{1,0\}$. Suppose $t$ is a flat tangle without closed components and fix a decomposition of $t$ into a product of cups and caps. To prove that $\mathcal{G}$-graded endomorphisms of $\mathcal{F}(t)$ are multiplications by elements of $R$, we will need one more adjunction statement. It turns out to be slightly more natural than the previous one.

\begin{proposition}
\label{prop:KadjunB}
Suppose $M$ is an object of $\mathcal{K}_n^n$ and $N$ is an object of $\mathcal{K}_{n-1}^n$. Then
\begin{equation}
\label{eq:adjunB1}
\mathrm{Hom}_{\mathcal{K}_{n-1}^n} \left(
\tikz[baseline={([yshift=-.5ex]current bounding box.center)}, scale=.3, y=0.5cm]
    {   
    \draw[dotted] (0.5,4) -- (6.5,4);
    \draw[knot] (1,3) -- (1,4);
    \draw[knot] (2,3) -- (2,4);
    \draw[knot] (3,3) -- (3,4);
    \draw[knot] (4,3) -- (4,4);
    \draw[knot] (5,3) -- (5,4);
    \draw[knot] (6,3) -- (6,4);
    \draw[knot] (0.5, -1) rectangle (6.5, 3);
    \node at (3.5, 1) {$M$};
    \draw[knot] (1,-1) to[out=-90, in=90] (2,-4);
    \draw[knot] (2,-1) to[out=-90, in=90] (3,-4);
    \draw[knot] (3,-1) to[out=-90, in=180] (3.5, -2.5) to[out=0, in=-90] (4,-1);
    \draw[knot] (5,-1) to[out=-90, in=90] (4,-4);
    \draw[knot] (6,-1) to[out=-90, in=90] (5,-4);
    \draw[dotted] (0.5,-4) -- (6.5,-4);
    }
~, \, 
\tikz[baseline={([yshift=-.5ex]current bounding box.center)}, scale=.3, y=0.5cm]
{   
    \draw[dotted] (0.5,4) -- (6.5,4);
    \draw[knot] (1,4) -- (1,2);
    \draw[knot] (2,4) -- (2,2);
    \draw[knot] (3,4) -- (3,2);
    \draw[knot] (4,4) -- (4,2);
    \draw[knot] (5,4) -- (5,2);
    \draw[knot] (6,4) -- (6,2);
    \draw[knot] (0.75, -2) rectangle (6.25, 2);
    \node at (3.5, 0) {$N$};
    \draw[knot] (2,-2) -- (2,-4);
    \draw[knot] (3,-2) -- (3,-4);
    \draw[knot] (4,-2) -- (4,-4);
    \draw[knot] (5,-2) -- (5,-4);
    \draw[dotted] (0.5,-4) -- (6.5,-4);
}
\right)
\cong
\mathrm{Hom}_{\mathcal{K}_n^n} \left(
\varphi_{
\tikz[baseline={([yshift=-.5ex]current bounding box.center)}, scale=.2, y=0.5cm]
    {   
    \draw[knot, red] (3,-2.5) -- (4,-2.5);
    \draw[dotted] (0.5,4) -- (6.5,4);
    \draw[knot] (1,3) -- (1,4);
    \draw[knot] (2,3) -- (2,4);
    \draw[knot] (3,3) -- (3,4);
    \draw[knot] (4,3) -- (4,4);
    \draw[knot] (5,3) -- (5,4);
    \draw[knot] (6,3) -- (6,4);
    \draw[knot] (0.5, -1) rectangle (6.5, 3);
    \draw[knot] (1,-1) -- (1,-4);
    \draw[knot] (2,-1) -- (2,-4);
    \draw[knot] (3,-1) -- (3,-4);
    \draw[knot] (4,-1) -- (4,-4);
    \draw[knot] (5,-1) -- (5,-4);
    \draw[knot] (6,-1) -- (6,-4);
    \draw[dotted] (0.5,-4) -- (6.5,-4);
    }
}~
\tikz[baseline={([yshift=-.5ex]current bounding box.center)}, scale=.3, y=0.5cm]
    {
    \draw[dotted] (0.5,4) -- (6.5,4);   
    \draw[knot] (1,2) -- (1,4);
    \draw[knot] (2,2) -- (2,4);
    \draw[knot] (3,2) -- (3,4);
    \draw[knot] (4,2) -- (4,4);
    \draw[knot] (5,2) -- (5,4);
    \draw[knot] (6,2) -- (6,4);
    \draw[knot] (0.5, -2) rectangle (6.5, 2);
    \node at (3.5, 0) {$M$};
    \draw[knot] (1,-2) -- (1,-4);
    \draw[knot] (2,-2) -- (2,-4);
    \draw[knot] (3,-2) -- (3,-4);
    \draw[knot] (4,-2) -- (4,-4);
    \draw[knot] (5,-2) -- (5,-4);
    \draw[knot] (6,-2) -- (6,-4);
    \draw[dotted] (0.5,-4) -- (6.5,-4);
    }
~, \, 
\tikz[baseline={([yshift=-.5ex]current bounding box.center)}, scale=.3, y=0.5cm]
    {   
    \draw[dotted] (0.5,4) -- (6.5,4);
    \draw[knot] (1,-4) to[out=90, in=-90] (2,-1);
    \draw[knot] (2,-4) to[out=90, in=-90] (3,-1);
    \draw[knot] (5,-4) to[out=90, in=-90] (4,-1);
    \draw[knot] (6,-4) to[out=90, in=-90] (5,-1);
    \draw[knot] (3,-4) to[out=90, in=180] (3.5, -2.5) to[out=0, in=90] (4,-4);
    \draw[knot] (0.5, 3) rectangle (6.5, -1);
    \node at (3.5, 1) {$N$};
    \draw[knot] (1,3) -- (1,4);
    \draw[knot] (2,3) -- (2,4);
    \draw[knot] (3,3) -- (3,4);
    \draw[knot] (4,3) -- (4,4);
    \draw[knot] (5,3) -- (5,4);
    \draw[knot] (6,3) -- (6,4);
    \draw[dotted] (0.5,-4) -- (6.5,-4);
    }
\right) .
\end{equation}
Similarly, for $L$ an object of $\mathcal{K}_n^{n-1}$,
\begin{equation}
\label{eq:adjunB2}
\mathrm{Hom}_{\mathcal{K}_n^{n-1}}\left(
\tikz[baseline={([yshift=-.5ex]current bounding box.center)}, scale=.3, y=0.5cm]
    {   
    \draw[dotted] (0.5,4) -- (6.5,4);
    \draw[knot] (2,4) to[out=-90, in=90] (1,1);
    \draw[knot] (3,4) to[out=-90, in=90] (2,1);
    \draw[knot] (3,1) to[out=90, in=180] (3.5, 2.5) to[out=0, in=90] (4,1);
    \draw[knot] (4,4) to[out=-90, in=90] (5,1);
    \draw[knot] (5,4) to[out=-90, in=90] (6,1);
    \draw[knot] (0.5, -3) rectangle (6.5, 1);
    \node at (3.5, -1) {$M$};
    \draw[knot] (1,-3) -- (1,-4);
    \draw[knot] (2,-3) -- (2,-4);
    \draw[knot] (3,-3) -- (3,-4);
    \draw[knot] (4,-3) -- (4,-4);
    \draw[knot] (5,-3) -- (5,-4);
    \draw[knot] (6,-3) -- (6,-4);
    \draw[dotted] (0.5,-4) -- (6.5,-4);
    }
~, \,
\tikz[baseline={([yshift=-.5ex]current bounding box.center)}, scale=.3, y=0.5cm]
{   
    \draw[dotted] (0.5,4) -- (6.5,4);
    \draw[knot] (2,4) -- (2,2);
    \draw[knot] (3,4) -- (3,2);
    \draw[knot] (4,4) -- (4,2);
    \draw[knot] (5,4) -- (5,2);
    \draw[knot] (0.75, -2) rectangle (6.25, 2);
    \node at (3.5, 0) {$L$};
    \draw[knot] (2,-2) -- (2,-4);
    \draw[knot] (3,-2) -- (3,-4);
    \draw[knot] (4,-2) -- (4,-4);
    \draw[knot] (5,-2) -- (5,-4);
    \draw[knot] (1,-2) -- (1,-4);
    \draw[knot] (6,-2) -- (6,-4);
    \draw[dotted] (0.5,-4) -- (6.5,-4);
}
\right)
\cong
\mathrm{Hom}_{\mathcal{K}_n^n}\left(
\varphi_{
\tikz[baseline={([yshift=-.5ex]current bounding box.center)}, scale=.2, y=0.5cm]
    {   
    \draw[knot, red] (3,2.5) -- (4,2.5);
    \draw[dotted] (0.5,4) -- (6.5,4);
    \draw[knot] (1,1) -- (1,4);
    \draw[knot] (2,1) -- (2,4);
    \draw[knot] (3,1) -- (3,4);
    \draw[knot] (4,1) -- (4,4);
    \draw[knot] (5,1) -- (5,4);
    \draw[knot] (6,1) -- (6,4);
    \draw[knot] (0.5, -3) rectangle (6.5, 1);
    \draw[knot] (1,-3) -- (1,-4);
    \draw[knot] (2,-3) -- (2,-4);
    \draw[knot] (3,-3) -- (3,-4);
    \draw[knot] (4,-3) -- (4,-4);
    \draw[knot] (5,-3) -- (5,-4);
    \draw[knot] (6,-3) -- (6,-4);
    \draw[dotted] (0.5,-4) -- (6.5,-4);
    }
}~
\tikz[baseline={([yshift=-.5ex]current bounding box.center)}, scale=.3, y=0.5cm]
    {
    \draw[dotted] (0.5,4) -- (6.5,4);   
    \draw[knot] (1,2) -- (1,4);
    \draw[knot] (2,2) -- (2,4);
    \draw[knot] (3,2) -- (3,4);
    \draw[knot] (4,2) -- (4,4);
    \draw[knot] (5,2) -- (5,4);
    \draw[knot] (6,2) -- (6,4);
    \draw[knot] (0.5, -2) rectangle (6.5, 2);
    \node at (3.5, 0) {$M$};
    \draw[knot] (1,-2) -- (1,-4);
    \draw[knot] (2,-2) -- (2,-4);
    \draw[knot] (3,-2) -- (3,-4);
    \draw[knot] (4,-2) -- (4,-4);
    \draw[knot] (5,-2) -- (5,-4);
    \draw[knot] (6,-2) -- (6,-4);
    \draw[dotted] (0.5,-4) -- (6.5,-4);
    }
~, \, 
\tikz[baseline={([yshift=-.5ex]current bounding box.center)}, scale=.3, y=0.5cm]
{   
    \draw[dotted] (0.5,4) -- (6.5,4);
    \draw[knot] (1,4) to[out=-90, in=90] (2,1);
    \draw[knot] (2,4) to[out=-90, in=90] (3,1);
    \draw[knot] (3,4) to[out=-90, in=180] (3.5,2.5) to[out=0, in=-90] (4,4);
    \draw[knot] (5,4) to[out=-90, in=90] (4,1);
    \draw[knot] (6,4) to[out=-90, in=90] (5,1);
    \draw[knot] (0.75, -3) rectangle (6.25, 1);
    \node at (3.5, -1) {$L$};
    \draw[knot] (1,-3) -- (1,-4);
    \draw[knot] (2,-3) -- (2,-4);
    \draw[knot] (3,-3) -- (3,-4);
    \draw[knot] (4,-3) -- (4,-4);
    \draw[knot] (5,-3) -- (5,-4);
    \draw[knot] (6,-3) -- (6,-4);
    \draw[dotted] (0.5,-4) -- (6.5,-4);
}
\right) .
\end{equation}
\end{proposition}

\begin{proof}
Left to the reader (use the same ideas presented in the proof of Proposition \ref{prop:Kadjun}).
\end{proof}

Now, notice that whenever $M = N \otimes_{H^{n-1}} \mathcal{F}(\cap_{i, n})$ or $M = \mathcal{F}(\cup_{i, n-1}) \otimes_{H^{n-1}} L$, the grading shifts of Proposition \ref{prop:KadjunB} are $\{0,-1\}$. In particular, we compute
\begin{align*}
\mathrm{Hom}_{\mathcal{K}_1^1}\left(
\tikz[baseline={([yshift=-.5ex]current bounding box.center)}, scale=.25, y=0.4cm]
    {   
    \draw[dotted] (0.5,4) -- (5.5,4);
    \draw[dotted] (0.5,-4) -- (5.5,-4);
    \draw[knot] (2,4) to[out=-90, in=180] (3, 2) to[out=0, in=-90] (4,4);
    \draw[knot] (2,-4) to[out=90, in=180] (3, -2) to[out=0, in=90] (4,-4);
    }
,\, 
\tikz[baseline={([yshift=-.5ex]current bounding box.center)}, scale=.25, y=0.4cm]
    {   
    \draw[dotted] (0.5,4) -- (5.5,4);
    \draw[dotted] (0.5,-4) -- (5.5,-4);
    \draw[knot] (2,4) to[out=-90, in=180] (3, 2) to[out=0, in=-90] (4,4);
    \draw[knot] (2,-4) to[out=90, in=180] (3, -2) to[out=0, in=90] (4,-4);
    }
\right)
	&\cong
	\mathrm{Hom}_{\mathcal{K}_0^1}\left(
\tikz[baseline={([yshift=-.5ex]current bounding box.center)}, scale=.25, y=0.4cm]
    {   
    \draw[dotted] (0.5,4) -- (5.5,4);
    \draw[dotted] (0.5,-4) -- (5.5,-4);
    \draw[knot] (2,4) to[out=-90, in=180] (3, 2) to[out=0, in=-90] (4,4);
    \draw[knot] (3,-1) circle (25pt);
    } \{0, 1\}
    ,\, 
\tikz[baseline={([yshift=-.5ex]current bounding box.center)}, scale=.25, y=0.4cm]
    {   
    \draw[dotted] (0.5,4) -- (5.5,4);
    \draw[dotted] (0.5,-4) -- (5.5,-4);
    \draw[knot] (2,4) to[out=-90, in=180] (3, 0) to[out=0, in=-90] (4,4);
    }
	\right)
	\\ &\cong
	\mathrm{Hom}_{\mathcal{K}_0^1}\left(
\tikz[baseline={([yshift=-.5ex]current bounding box.center)}, scale=.25, y=0.4cm]
    {   
    \draw[dotted] (0.5,4) -- (5.5,4);
    \draw[dotted] (0.5,-4) -- (5.5,-4);
    \draw[knot] (2,4) to[out=-90, in=180] (3, 0) to[out=0, in=-90] (4,4);
    } \{1,1\}
    \oplus
\tikz[baseline={([yshift=-.5ex]current bounding box.center)}, scale=.25, y=0.4cm]
    {   
    \draw[dotted] (0.5,4) -- (5.5,4);
    \draw[dotted] (0.5,-4) -- (5.5,-4);
    \draw[knot] (2,4) to[out=-90, in=180] (3, 0) to[out=0, in=-90] (4,4);
    }
	,\,
\tikz[baseline={([yshift=-.5ex]current bounding box.center)}, scale=.25, y=0.4cm]
    {   
    \draw[dotted] (0.5,4) -- (5.5,4);
    \draw[dotted] (0.5,-4) -- (5.5,-4);
    \draw[knot] (2,4) to[out=-90, in=180] (3, 0) to[out=0, in=-90] (4,4);
    }
	\right)
	\\ &\cong
	\mathrm{Hom}_{\mathcal{K}_0^1}\left(
\tikz[baseline={([yshift=-.5ex]current bounding box.center)}, scale=.25, y=0.4cm]
    {   
    \draw[dotted] (0.5,4) -- (5.5,4);
    \draw[dotted] (0.5,-4) -- (5.5,-4);
    \draw[knot] (2,4) to[out=-90, in=180] (3, 0) to[out=0, in=-90] (4,4);
    }
	,\,
\tikz[baseline={([yshift=-.5ex]current bounding box.center)}, scale=.25, y=0.4cm]
    {   
    \draw[dotted] (0.5,4) -- (5.5,4);
    \draw[dotted] (0.5,-4) -- (5.5,-4);
    \draw[knot] (2,4) to[out=-90, in=180] (3, 0) to[out=0, in=-90] (4,4);
    }
	\right)
	\\ &\cong
	R.
\end{align*}
The first line follows from an application of (\ref{eq:adjunB1}), and the last line comes from Corollary \ref{cor:cup_cap}. The third isomorphism follows from grading considerations. This argument works equally well in higher generality:

\begin{corollary}
\label{cor:totturns}
For any $i, j = 1, \ldots, 2n -1$, 
\[
\mathrm{End}_{\mathcal{K}_n^n}\left(
\tikz[baseline={([yshift=-.5ex]current bounding box.center)}, scale=.3, y=0.5cm]
    {   
    \draw[dotted] (-0.5,4) -- (10.5,4);
    \draw[dotted] (-0.5,-4) -- (10.5,-4);
    \draw[knot] (0,-4) node[below]{$1$} -- (0,4);
        \node at (1,0) {$\cdots$};
    \draw[knot] (2,-4) -- (2,4);
    \draw[knot] (3,4) node[above]{$i$} to[out=-90,in=180] (3.5, 2.5) to[out=0,in=-90] (4,4);
    \draw[knot] (3,-4) to[out=90,in=-90] (5,4);
        \node at (5,0) {$\cdots$};
    \draw[knot] (5,-4) to[out=90,in=-90] (7,4);
    \draw[knot] (6,-4) node[below]{$j$} to[out=90,in=180] (6.5,-2.5) to[out=0,in=90] (7,-4);
    \draw[knot] (8,-4) -- (8,4);
        \node at (9,0) {$\cdots$};
    \draw[knot] (10,-4) node[below]{$2n$} -- (10,4);
    }
\right)
=R.
\]
\end{corollary}

\begin{proof}
The adjunction (\ref{eq:adjunB1}) tells us that
\[
\mathrm{End}_{\mathcal{K}_n^n} \left(
\tikz[baseline={([yshift=-.5ex]current bounding box.center)}, scale=.25, y=0.5cm]
    {   
    \draw[dotted] (-0.5,4) -- (10.5,4);
    \draw[dotted] (-0.5,-4) -- (10.5,-4);
    \draw[knot] (0,-4) node[below]{$1$} -- (0,4);
        \node at (1,0) {\small$\cdots$};
    \draw[knot] (2,-4) -- (2,4);
    \draw[knot] (3,4) node[above]{$i$} to[out=-90,in=180] (3.5, 2.5) to[out=0,in=-90] (4,4);
    \draw[knot] (3,-4) to[out=90,in=-90] (5,4);
        \node at (5,0) {\small$\cdots$};
    \draw[knot] (5,-4) to[out=90,in=-90] (7,4);
    \draw[knot] (6,-4) node[below]{$j$} to[out=90,in=180] (6.5,-2.5) to[out=0,in=90] (7,-4);
    \draw[knot] (8,-4) -- (8,4);
        \node at (9,0) {\small$\cdots$};
    \draw[knot] (10,-4) node[below]{$2n$} -- (10,4);
    }
\right)
	=
\mathrm{Hom}_{\mathcal{K}_{n-1}^n}\left(
\tikz[baseline={([yshift=-.5ex]current bounding box.center)}, scale=.25, y=0.5cm]
    {   
    \draw[dotted] (-0.5,4) -- (10.5,4);
    \draw[dotted] (-0.5,-4) -- (10.5,-4);
    \draw[knot] (0,-4) node[below]{$1$} -- (0,4);
        \node at (1,0) {\small$\cdots$};
    \draw[knot] (2,-4) -- (2,4);
    \draw[knot] (3,4) node[above]{$i$} to[out=-90,in=180] (3.5, 2.5) to[out=0,in=-90] (4,4);
    \draw[knot] (3,-4) to[out=90,in=-90] (5,4);
        \node at (5,0) {\small$\cdots$};
    \draw[knot] (5,-4) to[out=90,in=-90] (7,4);
	\draw[knot] (6.75, -2) circle (20pt);
    \draw[knot] (8,-4) -- (8,4);
        \node at (9,0) {\small$\cdots$};
    \draw[knot] (10,-4) node[below]{$2n$} -- (10,4);
    } \{0,1\}
,\, 
\tikz[baseline={([yshift=-.5ex]current bounding box.center)}, scale=.25, y=0.5cm]
    {   
    \draw[dotted] (-0.5,4) -- (10.5,4);
    \draw[dotted] (-0.5,-4) -- (10.5,-4);
    \draw[knot] (0,-4) node[below]{$1$} -- (0,4);
        \node at (1,0) {\small$\cdots$};
    \draw[knot] (2,-4) -- (2,4);
    \draw[knot] (3,4) node[above]{$i$} to[out=-90,in=180] (3.5, 2.5) to[out=0,in=-90] (4,4);
    \draw[knot] (3,-4) to[out=90,in=-90] (5,4);
        \node at (5,0) {\small$\cdots$};
    \draw[knot] (5,-4) to[out=90,in=-90] (7,4);
    \draw[knot] (8,-4) -- (8,4);
        \node at (9,0) {\small$\cdots$};
    \draw[knot] (10,-4) node[below]{$2n$} -- (10,4);
    }
\right).
\]
Then, after applying the delooping isomorphism, Corollary \ref{cor:cup_cap} gives the desired result, since 
\begin{align*}
\mathrm{Hom}_{\mathcal{K}_{n-1}^n}\left(
\tikz[baseline={([yshift=-.5ex]current bounding box.center)}, scale=.3, y=0.7cm]
{   
    \draw[dotted] (0.5,4) -- (6.5,4);
    \draw[knot] (1,4) to[out=-90, in=90] (2,0);
    \draw[knot] (2,4) to[out=-90, in=90] (3,0);
    \draw[knot] (5,4) to[out=-90, in=90] (4,0);
    \draw[knot] (6,4) to[out=-90, in=90] (5,0);
    \draw[knot] (3,4) to[out=-90, in=180] (3.5, 2.5) to[out=0, in=-90] (4,4);
    \draw[dotted] (0.5,0) -- (6.5,0);
} \{1,1\} ,\,
\tikz[baseline={([yshift=-.5ex]current bounding box.center)}, scale=.3, y=0.7cm]
{   
    \draw[dotted] (0.5,4) -- (6.5,4);
    \draw[knot] (1,4) to[out=-90, in=90] (2,0);
    \draw[knot] (2,4) to[out=-90, in=90] (3,0);
    \draw[knot] (5,4) to[out=-90, in=90] (4,0);
    \draw[knot] (6,4) to[out=-90, in=90] (5,0);
    \draw[knot] (3,4) to[out=-90, in=180] (3.5, 2.5) to[out=0, in=-90] (4,4);
    \draw[dotted] (0.5,0) -- (6.5,0);
}
\right) & \cong
\mathrm{Hom}_{\mathcal{K}_{n-1}^n}\left(
\tikz[baseline={([yshift=-.5ex]current bounding box.center)}, scale=.3, y=0.7cm]
{   
    \draw[dotted] (0.5,4) -- (6.5,4);
    \draw[knot] (1,4) to[out=-90, in=90] (2,0);
    \draw[knot] (2,4) to[out=-90, in=90] (3,0);
    \draw[knot] (5,4) to[out=-90, in=90] (4,0);
    \draw[knot] (6,4) to[out=-90, in=90] (5,0);
    \draw[knot] (3,4) to[out=-90, in=180] (3.5, 2.5) to[out=0, in=-90] (4,4);
    \draw[dotted] (0.5,0) -- (6.5,0);
} ,\,
\tikz[baseline={([yshift=-.5ex]current bounding box.center)}, scale=.3, y=0.7cm]
{   
    \draw[dotted] (0.5,4) -- (6.5,4);
    \draw[knot] (1,4) to[out=-90, in=90] (2,0);
    \draw[knot] (2,4) to[out=-90, in=90] (3,0);
    \draw[knot] (5,4) to[out=-90, in=90] (4,0);
    \draw[knot] (6,4) to[out=-90, in=90] (5,0);
    \draw[knot] (3,4) to[out=-90, in=180] (3.5, 2.5) to[out=0, in=-90] (4,4);
    \draw[dotted] (0.5,0) -- (6.5,0);
} \{-1,-1\}
\right)
\\ & \cong
\mathrm{Hom}_{\mathcal{K}_{n-1}^n}\left(
\tikz[baseline={([yshift=-.5ex]current bounding box.center)}, scale=.3, y=0.7cm]
{   
    \draw[dotted] (0.5,4) -- (6.5,4);
    \draw[knot] (1,4) to[out=-90, in=90] (2,0);
    \draw[knot] (2,4) to[out=-90, in=90] (3,0);
    \draw[knot] (5,4) to[out=-90, in=90] (4,0);
    \draw[knot] (6,4) to[out=-90, in=90] (5,0);
    \draw[knot] (3,4) to[out=-90, in=180] (3.5, 2.5) to[out=0, in=-90] (4,4);
    \draw[dotted] (0.5,0) -- (6.5,0);
} ,\,
\tikz[baseline={([yshift=-.5ex]current bounding box.center)}, scale=.3, y=0.7cm]
{   
    \draw[dotted] (0.5,4) -- (6.5,4);
    \draw[knot] (1,4) to[out=-90, in=90] (2,0);
    \draw[knot] (2,4) to[out=-90, in=90] (3,0);
    \draw[knot] (5,4) to[out=-90, in=90] (4,0);
    \draw[knot] (6,4) to[out=-90, in=90] (5,0);
    \draw[knot] (3,4) to[out=-90, in=180] (3.5, 2.5) to[out=0, in=-90] (4,4);
    \draw[dotted] (0.5,0) -- (6.5,0);
} \{1-2, 0-1\}
\right)
\\ & \cong
\mathrm{Hom}_{\mathcal{K}_{n-1}^{n-1}}\left(
H^{n-1},\,
\tikz[baseline={([yshift=-.5ex]current bounding box.center)}, scale=.3, y=0.7cm]
{   
    \draw[dotted] (0.5,4) -- (6.5,4);
    \draw[knot] (2,0) to[out=90, in=-90] (1,2) to[out=90, in=-90] (2,4);
    \draw[knot] (3,0) to[out=90, in=-90] (2,2) to[out=90, in=-90] (3,4);
    \draw[knot] (3.5,2) circle (20pt);
    \draw[knot] (4,0) to[out=90, in=-90] (5,2) to[out=90, in=-90] (4,4);
    \draw[knot] (5,0) to[out=90, in=-90] (6,2) to[out=90, in=-90] (5,4);
    \draw[dotted] (0.5,0) -- (6.5,0);
}
\{-2, -1\}
\right)
\\ & \cong
\mathrm{Hom}_{\mathcal{K}_{n-1}^{n-1}}\left(H^{n-1}, H^{n-1}\{-1,-1\} \oplus H^{n-1}\{-2,-2\}\right) = 0
\end{align*}
after an application of (\ref{eq:adjun1}) from Proposition \ref{prop:Kadjun}.
\end{proof}

In even further generality, we can use Propositions \ref{prop:Kadjun} and \ref{prop:KadjunB} in tandem (as demonstrated in the proof of Corollary \ref{cor:totturns}) to prove the following.

\begin{corollary}
\label{cor:endFT}
If $t$ is a flat tangle without closed components, then the only $\mathcal{G}$-graded endomorphisms of $\mathcal{F}(t)$ are multiplications by elements of $R$, so that the only such automorphisms are of the form $u \cdot \mathrm{Id}$ for $u$ a unit of $R$.
\end{corollary}

There are other adjunctions similar to Propositions \ref{prop:Kadjun} and \ref{prop:KadjunB}. We do not list them all, but here is one more we use to prove invariance under Carter-Saito movie moves 29 and 30.

\begin{proposition}
\label{prop:KadjunC}
Suppose $M$ and $N$ are objects of $\mathcal{K}_n^n$. Then
\begin{equation}
\label{eq:adjun29-30}
\mathrm{Hom}_{\mathcal{K}_n^n} \left(
\varphi_{
\tikz[baseline={([yshift=-.5ex]current bounding box.center)}, scale=.2, y=0.5cm]
    {   
    \draw[dotted] (0.5,4) -- (6.5,4);
    \draw[knot, red] (3.5, -1) -- (3.5, -2.5);
    \draw[knot] (1,3.5) -- (1,4);
    \draw[knot] (2,3.5) -- (2,4);
    \draw[knot] (3,3.5) -- (3,4);
    \draw[knot] (4,3.5) -- (4,4);
    \draw[knot] (5,3.5) -- (5,4);
    \draw[knot] (6,3.5) -- (6,4);
    \draw[knot] (0.5, 0.5) rectangle (6.5, 3.5);
    \draw[knot] (1,0.5) to[out=-90, in=90] (1,-4);
    \draw[knot] (2,0.5) to[out=-90, in=90] (2,-4);
    \draw[knot] (3,0.5) to[out=-90, in=180] (3.5, -1) to[out=0, in=-90] (4,0.5);
    \draw[knot] (3,-4) to[out=90, in=180] (3.5, -2.5) to[out=0, in=90] (4,-4);
    \draw[knot] (5,0.5) to[out=-90, in=90] (5,-4);
    \draw[knot] (6,0.5) to[out=-90, in=90] (6,-4);
    \draw[dotted] (0.5,-4) -- (6.5,-4);
    }
}~
\tikz[baseline={([yshift=-.5ex]current bounding box.center)}, scale=.3, y=0.5cm]
    {   
    \draw[dotted] (0.5,4) -- (6.5,4);
    \draw[knot] (1,3.5) -- (1,4);
    \draw[knot] (2,3.5) -- (2,4);
    \draw[knot] (3,3.5) -- (3,4);
    \draw[knot] (4,3.5) -- (4,4);
    \draw[knot] (5,3.5) -- (5,4);
    \draw[knot] (6,3.5) -- (6,4);
    \draw[knot] (0.5, -0.5) rectangle (6.5, 3.5);
    \node at (3.5, 1.5) {$M$};
    \draw[knot] (1,-0.5) to[out=-90, in=90] (1,-4);
    \draw[knot] (2,-0.5) to[out=-90, in=90] (2,-4);
    \draw[knot] (3,-0.5) to[out=-90, in=180] (3.5, -2) to[out=0, in=-90] (4,-0.5);
    \draw[knot] (3,-4) to[out=90, in=180] (3.5, -2.5) to[out=0, in=90] (4,-4);
    \draw[knot] (5,-0.5) to[out=-90, in=90] (5,-4);
    \draw[knot] (6,-0.5) to[out=-90, in=90] (6,-4);
    \draw[dotted] (0.5,-4) -- (6.5,-4);
    }
~, \, 
\tikz[baseline={([yshift=-.5ex]current bounding box.center)}, scale=.3, y=0.5cm]
    {
    \draw[dotted] (0.5,4) -- (6.5,4);   
    \draw[knot] (1,2) -- (1,4);
    \draw[knot] (2,2) -- (2,4);
    \draw[knot] (3,2) -- (3,4);
    \draw[knot] (4,2) -- (4,4);
    \draw[knot] (5,2) -- (5,4);
    \draw[knot] (6,2) -- (6,4);
    \draw[knot] (0.5, -2) rectangle (6.5, 2);
    \node at (3.5, 0) {$N$};
    \draw[knot] (1,-2) -- (1,-4);
    \draw[knot] (2,-2) -- (2,-4);
    \draw[knot] (3,-2) -- (3,-4);
    \draw[knot] (4,-2) -- (4,-4);
    \draw[knot] (5,-2) -- (5,-4);
    \draw[knot] (6,-2) -- (6,-4);
    \draw[dotted] (0.5,-4) -- (6.5,-4);
    }
\right)
\cong
\mathrm{Hom}_{\mathcal{K}_{n-1}^n} \left(
\tikz[baseline={([yshift=-.5ex]current bounding box.center)}, scale=.3, y=0.5cm]
    {   
    \draw[dotted] (0.5,4) -- (6.5,4);
    \draw[knot] (1,3) -- (1,4);
    \draw[knot] (2,3) -- (2,4);
    \draw[knot] (3,3) -- (3,4);
    \draw[knot] (4,3) -- (4,4);
    \draw[knot] (5,3) -- (5,4);
    \draw[knot] (6,3) -- (6,4);
    \draw[knot] (0.5, -1) rectangle (6.5, 3);
    \node at (3.5, 1) {$M$};
    \draw[knot] (1,-1) to[out=-90, in=90] (2,-4);
    \draw[knot] (2,-1) to[out=-90, in=90] (3,-4);
    \draw[knot] (3,-1) to[out=-90, in=180] (3.5, -2.5) to[out=0, in=-90] (4,-1);
    \draw[knot] (5,-1) to[out=-90, in=90] (4,-4);
    \draw[knot] (6,-1) to[out=-90, in=90] (5,-4);
    \draw[dotted] (0.5,-4) -- (6.5,-4);
    }
~, \, 
\tikz[baseline={([yshift=-.5ex]current bounding box.center)}, scale=.3, y=0.5cm]
    {   
    \draw[dotted] (0.5,4) -- (6.5,4);
    \draw[knot] (1,3) -- (1,4);
    \draw[knot] (2,3) -- (2,4);
    \draw[knot] (3,3) -- (3,4);
    \draw[knot] (4,3) -- (4,4);
    \draw[knot] (5,3) -- (5,4);
    \draw[knot] (6,3) -- (6,4);
    \draw[knot] (0.5, -1) rectangle (6.5, 3);
    \node at (3.5, 1) {$N$};
    \draw[knot] (1,-1) to[out=-90, in=90] (2,-4);
    \draw[knot] (2,-1) to[out=-90, in=90] (3,-4);
    \draw[knot] (3,-1) to[out=-90, in=180] (3.5, -2.5) to[out=0, in=-90] (4,-1);
    \draw[knot] (5,-1) to[out=-90, in=90] (4,-4);
    \draw[knot] (6,-1) to[out=-90, in=90] (5,-4);
    \draw[dotted] (0.5,-4) -- (6.5,-4);
    }
\right).
\end{equation}
\end{proposition}

\subsection{The Naisse-Putyra tangle invariant and the target category}
\label{ss:NPinvariant}

Given an oriented $(m, n)$-tangle $T$, define the object $\mathbf{K}(T)$ of $\mathcal{K}_m^n$ as follows. First, decompose $T$ as a composition of elementary tangles $T = T_k \cdots T_1$. Then, if $T_j$ is a cup $\cup_{i, n-1}$ or cap $\cap_{i, n}$, we put
\[
\mathbf{K}(T_j) = \mathcal{F}(T_j).
\]
If $T_j$ is a crossing, define
\begin{align*}
\mathbf{K}\left( 
\tikz[baseline={([yshift=-.5ex]current bounding box.center)}, scale=.75]
{
	\draw[dotted] (.5,.5) circle(0.707);
	\draw[knot, ->](0,0) -- (1,1);
	\fill[fill=white] (.5,.5) circle (.15);
	\draw[knot, ->](1,0) -- (0,1);
}
 \right) 
\ &:= \ 
\mathrm{Cone}\left( 
 \varphi_{\tikz[baseline={([yshift=-.5ex]current bounding box.center)}, scale=.45]
{
	\draw[dotted] (3,-2) circle(0.707);
        \draw[knot, red,thick] (3,-1.7) -- (3,-2.3);
	\draw[knot] (2.5,-1.5) .. controls (2.75,-1.75) and (3.25,-1.75) .. (3.5,-1.5);
	\draw[knot] (2.5,-2.5) .. controls  (2.75,-2.25) and (3.25,-2.25) .. (3.5,-2.5);
}}
\mathcal{F}\left(
\tikz[baseline={([yshift=-.5ex]current bounding box.center)}, scale=.75]
{
	\draw[dotted] (3,-2) circle(0.707);
	\draw[knot] (2.5,-1.5) .. controls (2.75,-1.75) and (3.25,-1.75) .. (3.5,-1.5);
	\draw[knot] (2.5,-2.5) .. controls  (2.75,-2.25) and (3.25,-2.25) .. (3.5,-2.5);
}
  \right) 
\xrightarrow{\mathcal{F}\left( \tikz[baseline={([yshift=-.5ex]current bounding box.center)}, scale=.45]
{
	\draw[dotted] (3,-2) circle(0.707);
        \draw[knot, red,thick] (3,-1.7) -- (3,-2.3);
	\draw[knot] (2.5,-1.5) .. controls (2.75,-1.75) and (3.25,-1.75) .. (3.5,-1.5);
	\draw[knot] (2.5,-2.5) .. controls  (2.75,-2.25) and (3.25,-2.25) .. (3.5,-2.5);
} \right)}
\underline{
\mathcal{F}\left(
\tikz[baseline={([yshift=-.5ex]current bounding box.center)}, scale=.75]
{
	\draw[dotted] (-2,-2) circle(0.707);
	\draw[knot] (-1.5,-1.5) .. controls (-1.75,-1.75) and (-1.75,-2.25) .. (-1.5,-2.5);
	\draw[knot] (-2.5,-1.5) .. controls (-2.25,-1.75) and (-2.25,-2.25) ..  (-2.5,-2.5);
}  
 \right)
 }
\right) \{-1,0\},~ \text{and}
\\
\mathbf{K}\left(
\tikz[baseline={([yshift=-.5ex]current bounding box.center)}, scale=.75]
{
	\draw[dotted] (.5,.5) circle(0.707);
	\draw[knot, ->](1,0) -- (0,1);
	\fill[fill=white] (.5,.5) circle (.15);
	\draw[knot, ->](0,0) -- (1,1);
}
 \right) 
\ &:= \ 
\mathrm{Cone}\left( 
\underline{
\mathcal{F}\left(
\tikz[baseline={([yshift=-.5ex]current bounding box.center)}, scale=.75]
{
	\draw[dotted] (-2,-2) circle(0.707);
	\draw[knot] (-1.5,-1.5) .. controls (-1.75,-1.75) and (-1.75,-2.25) .. (-1.5,-2.5);
	\draw[knot] (-2.5,-1.5) .. controls (-2.25,-1.75) and (-2.25,-2.25) ..  (-2.5,-2.5);
} 
  \right) 
  }
\xrightarrow{
\mathcal{F}\left( \tikz[baseline={([yshift=-.5ex]current bounding box.center)}, scale=.45]
{
    \begin{scope}[rotate=90]
	\draw[dotted] (3,-2) circle(0.707);
        \draw[knot, red,thick] (3,-1.7) -- (3,-2.3);
	\draw[knot] (2.5,-1.5) .. controls (2.75,-1.75) and (3.25,-1.75) .. (3.5,-1.5);
	\draw[knot] (2.5,-2.5) .. controls  (2.75,-2.25) and (3.25,-2.25) .. (3.5,-2.5);
    \end{scope}
} \right)
\circ \varphi_H}
\varphi_{\tikz[baseline={([yshift=-.5ex]current bounding box.center)}, scale=.45]
{
    \begin{scope}[rotate=90]
	\draw[dotted] (3,-2) circle(0.707);
        \draw[knot, red,thick] (3,-1.7) -- (3,-2.3);
	\draw[knot] (2.5,-1.5) .. controls (2.75,-1.75) and (3.25,-1.75) .. (3.5,-1.5);
	\draw[knot] (2.5,-2.5) .. controls  (2.75,-2.25) and (3.25,-2.25) .. (3.5,-2.5);
    \end{scope}
}} ^{-1}
\mathcal{F}\left(
\tikz[baseline={([yshift=-.5ex]current bounding box.center)}, scale=.75]
{
	\draw[dotted] (3,-2) circle(0.707);
	\draw[knot] (2.5,-1.5) .. controls (2.75,-1.75) and (3.25,-1.75) .. (3.5,-1.5);
	\draw[knot] (2.5,-2.5) .. controls  (2.75,-2.25) and (3.25,-2.25) .. (3.5,-2.5);
}
\right)
\right)\{0,1\},
\end{align*}
for $\varphi_H: \mathrm{Id} \Rightarrow \varphi_{\tikz[baseline={([yshift=-.5ex]current bounding box.center)}, scale=.45]
{
    \begin{scope}[rotate=90]
	\draw[dotted] (3,-2) circle(0.707);
        \draw[knot, red,thick] (3,-1.7) -- (3,-2.3);
	\draw[knot] (2.5,-1.5) .. controls (2.75,-1.75) and (3.25,-1.75) .. (3.5,-1.5);
	\draw[knot] (2.5,-2.5) .. controls  (2.75,-2.25) and (3.25,-2.25) .. (3.5,-2.5);
    \end{scope}
}} ^{-1} \circ \varphi_{\tikz[baseline={([yshift=-.5ex]current bounding box.center)}, scale=.45]
{
    \begin{scope}[rotate=90]
	\draw[dotted] (3,-2) circle(0.707);
        \draw[knot, red,thick] (3,-1.7) -- (3,-2.3);
	\draw[knot] (2.5,-1.5) .. controls (2.75,-1.75) and (3.25,-1.75) .. (3.5,-1.5);
	\draw[knot] (2.5,-2.5) .. controls  (2.75,-2.25) and (3.25,-2.25) .. (3.5,-2.5);
    \end{scope}
}}$. The underlined terms are in homological degree zero. Then, define
\[
\mathbf{K}(T) := \mathbf{K}(T_1) \otimes_{H^{n_1}} \cdots \otimes_{H^{n_{k-1}}} \mathbf{K}(T_k)
\]
for the appropriate $n_j$ between the composable tangles $T_j$ and $T_{j+1}$. In \cite{naisse2020odd} and \cite{spyropoulos2024}, it was shown that the $\mathcal{G}$-grading is a bit too sensitive to obtain a tangle invariant from $\mathbf{K}(T)$ in the category $\mathsf{K}_m^n$ (or $\mathcal{K}_m^n$).

\begin{proposition}[Lemmas 6.12--17 in \cite{naisse2020odd}]
\label{prop:Rmoves}
We have the following isomorphisms in $\mathsf{K}_m^n$:
\begin{enumerate}
	\item If $T$ and $T'$ differ by planar isotopy, then $\mathbf{K}(T) \cong \mathbf{K})(T')$;
	\item For Reidemeister I moves, 
	\[
	\mathbf{K}
\left(
\tikz[baseline={([yshift=-.5ex]current bounding box.center)}, scale=.6]
{
        \draw[dotted] (.5,.5) circle(0.707);
        \draw[knot] (1, 0.5) to[out=-90, in=-60] (0,1);
        \draw[knot, overcross] (0,0) to[out=60, in =90] (1, 0.5);

}
\right)
\cong
\mathbf{K}
\left(
\tikz[baseline={([yshift=-.5ex]current bounding box.center)}, scale=.6]
{
        \draw[dotted] (.5,.5) circle(0.707);
        \draw[knot] (0,0) to[out=60, in=-90] (.7, .5) to[out=90, in=-60] (0,1);
}
\right)
\cong
\mathbf{K}
\left(
\tikz[baseline={([yshift=-.5ex]current bounding box.center)}, scale=.6]
{
        \draw[dotted] (.5,.5) circle(0.707);
        \draw[knot] (0,0) to[out=60, in =90] (1, 0.5);
        \draw[knot, overcross] (1, 0.5) to[out=-90, in=-60] (0,1);
}
\right);
\]
	\item For Reidemeister II moves,
	\[
	\mathbf{K}
\left(
\tikz[baseline={([yshift=-.5ex]current bounding box.center)}, scale=.6]
{
        \draw[dotted] (.5,.5) circle(0.707);
        \draw[knot, ->] (0,0) to[out=30, in=-90] (0.8, 0.5);
        \draw[knot] (0.8, 0.5) to[out=90, in=-30] (0,1);
        \draw[knot, overcross] (1,0) to[out=150, in=-90] (0.3, 0.5);
        \draw[knot, overcross] (0.3, 0.5) to[out=90, in=210] (1,1);
        \draw[knot, ->] (0.3, 0.499) -- (0.3, 0.501);
}
\right)
\cong
\mathbf{K}
\left(
\tikz[baseline={([yshift=-.5ex]current bounding box.center)}, scale=.6]
{
        \draw[dotted] (.5,.5) circle(0.707);
        \draw[knot] (0, 0) to[out=30, in=-90] (0.35, 0.5);
        \draw[knot] (0.35, 0.5) to[out=90, in=-30] (0, 1);
        \draw[knot] (1, 0) to[out=150, in=-90] (0.65, 0.5);
        \draw[knot] (0.65, 0.5) to[out=90, in=210] (1, 1);
}
\right)
\{-1,1\}
\cong
\mathbf{K}
\left(
\tikz[baseline={([yshift=-.5ex]current bounding box.center)}, scale=.6]
{
        \draw[dotted] (.5,.5) circle(0.707);
        \draw[knot, ->] (1,0) to[out=150, in=-90] (0.3, 0.5);
        \draw[knot] (0.3, 0.5) to[out=90, in=210] (1,1);
        \draw[knot, overcross] (0,0) to[out=30, in=-90] (0.8, 0.5);
        \draw[knot, overcross] (0.8, 0.5) to[out=90, in=-30] (0,1);
        \draw[knot, ->] (0.8, 0.499) -- (0.8, 0.501);
}
\right)
\]
and 
\[
\mathbf{K}
\left(
\tikz[baseline={([yshift=-.5ex]current bounding box.center)}, scale=.6]
{
        \draw[dotted] (.5,.5) circle(0.707);
        \draw[knot, ->] (0,0) to[out=30, in=-90] (0.8, 0.5);
        \draw[knot] (0.8, 0.5) to[out=90, in=-30] (0,1);
        \draw[knot, overcross] (1,0) to[out=150, in=-90] (0.3, 0.5);
        \draw[knot, overcross] (0.3, 0.5) to[out=90, in=210] (1,1);
        \draw[knot, <-] (0.3, 0.499) -- (0.3, 0.501);
}
\right)
\cong
\varphi_{\left(\tikz[baseline={([yshift=-.5ex]current bounding box.center)}, scale=.4]
{
    \begin{scope}[rotate=90]
	\draw[dotted] (3,-2) circle(0.707);
	\draw[knot] (2.5,-1.5) .. controls (2.75,-1.75) and (3.25,-1.75) .. (3.5,-1.5);
	\draw[knot] (2.5,-2.5) .. controls  (2.75,-2.25) and (3.25,-2.25) .. (3.5,-2.5);
        \draw[red, knot] (3,-1.7) -- (3,-2.3);
    \end{scope}
},\, (0,1)\right)}
\mathbf{K}
\left(
\tikz[baseline={([yshift=-.5ex]current bounding box.center)}, scale=.6]
{
        \draw[dotted] (.5,.5) circle(0.707);
        \draw[knot] (0, 0) to[out=30, in=-90] (0.35, 0.5);
        \draw[knot] (0.35, 0.5) to[out=90, in=-30] (0, 1);
        \draw[knot] (1, 0) to[out=150, in=-90] (0.65, 0.5);
        \draw[knot] (0.65, 0.5) to[out=90, in=210] (1, 1);
}
\right)
\cong
\mathbf{K}
\left(
\tikz[baseline={([yshift=-.5ex]current bounding box.center)}, scale=.6]
{
        \draw[dotted] (.5,.5) circle(0.707);
        \draw[knot] (1,0) to[out=150, in=-90] (0.3, 0.5);
        \draw[knot, <-] (0.3, 0.5) to[out=90, in=210] (1,1);
        \draw[knot, overcross] (0,0) to[out=30, in=-90] (0.8, 0.5);
        \draw[knot, overcross] (0.8, 0.5) to[out=90, in=-30] (0,1);
        \draw[knot, ->] (0.8, 0.499) -- (0.8, 0.501);
}
\right);
\]
	\item We have that
	\[
	\mathbf{K}
\left(
\tikz[baseline={([yshift=-.5ex]current bounding box.center)}, scale=.6]
{
        \draw[dotted] (.5,.5) circle(0.707);
        \draw[knot, <-] (0,1) -- (0.45, 0.55);
        \draw[knot] (0.45, 0.55) -- (0.55, 0.45);
        \draw[knot] (0.55, 0.45) -- (1,0);
        \draw[knot] (0,0) -- (0.45, 0.45);
        \draw[knot, overcross] (0.45, 0.45) -- (0.55, 0.55);
        \draw[knot, ->] (0.55, 0.55) -- (1,1);
        \draw[knot, overcross] (0.5, -0.207) to[out=90, in=-90] (0.085, 0.5);
        \draw[knot, overcross] (0.085, 0.5) to[out=90, in=-90] (0.5, 1.207);
        \draw[knot, ->] (0.085, 0.499) -- (0.085, 0.501);
}
\right)
\cong
\mathbf{K}
\left(
\tikz[baseline={([yshift=-.5ex]current bounding box.center)}, scale=.6]
{
        \draw[dotted] (.5,.5) circle(0.707);
        \draw[knot, <-] (0,1) -- (0.45, 0.55);
        \draw[knot] (0.45, 0.55) -- (0.55, 0.45);
        \draw[knot] (0.55, 0.45) -- (1,0);
        \draw[knot] (0,0) -- (0.45, 0.45);
        \draw[knot, overcross] (0.45, 0.45) -- (0.55, 0.55);
        \draw[knot, ->] (0.55, 0.55) -- (1,1);
        \draw[knot, overcross] (0.5, -0.207) to[out=90, in=-90] (0.915, 0.5);
        \draw[knot, overcross] (0.915, 0.5) to[out=90, in=-90] (0.5, 1.207);
        \draw[knot, ->] (0.915, 0.499) -- (0.915, 0.501);
}
\right)
	\]
	if strands are co-oriented (other cases being similar) and
	\[
	\mathbf{K}
\left(
\tikz[baseline={([yshift=-.5ex]current bounding box.center)}, scale=.6]
{
        \draw[dotted] (.5,.5) circle(0.707);
        \draw[knot, <-] (0,1) -- (0.45, 0.55);
        \draw[knot] (0.45, 0.55) -- (0.55, 0.45);
        \draw[knot] (0.55, 0.45) -- (1,0);
        \draw[knot] (0,0) -- (0.45, 0.45);
        \draw[knot, overcross] (0.45, 0.45) -- (0.55, 0.55);
        \draw[knot, ->] (0.55, 0.55) -- (1,1);
        \draw[knot, overcross] (0.5, -0.207) to[out=90, in=-90] (0.085, 0.5);
        \draw[knot, overcross] (0.085, 0.5) to[out=90, in=-90] (0.5, 1.207);
        \draw[knot, <-] (0.085, 0.499) -- (0.085, 0.501);
}
\right)
\cong
\varphi_{
\tikz[baseline={([yshift=-.5ex]current bounding box.center)}, scale=.4]
{
        \draw[dotted] (.5,.5) circle(0.707); 
        \draw[knot, red] (0.2, 0.5) -- (0.5, 0.5);
        \draw[knot] (0,0) to[out=45, in=-45] (0,1);
        \draw[knot] (1,0) to[out=135, in=-135] (1,1);
        \draw[knot] (0.5, -0.207) -- (0.5, 1.207);
}
}
\circ
\varphi_{
\tikz[baseline={([yshift=-.5ex]current bounding box.center)}, scale=.4]
{
        \draw[dotted] (.5,.5) circle(0.707); 
        \draw[knot, red] (0.8, 0.5) -- (0.5, 0.5);
        \draw[knot] (0,0) to[out=45, in=-45] (0,1);
        \draw[knot] (1,0) to[out=135, in=-135] (1,1);
        \draw[knot] (0.5, -0.207) -- (0.5, 1.207);
}
}^{-1}
\mathbf{K}
\left(
\tikz[baseline={([yshift=-.5ex]current bounding box.center)}, scale=.6]
{
        \draw[dotted] (.5,.5) circle(0.707);
        \draw[knot, <-] (0,1) -- (0.45, 0.55);
        \draw[knot] (0.45, 0.55) -- (0.55, 0.45);
        \draw[knot] (0.55, 0.45) -- (1,0);
        \draw[knot] (0,0) -- (0.45, 0.45);
        \draw[knot, overcross] (0.45, 0.45) -- (0.55, 0.55);
        \draw[knot, ->] (0.55, 0.55) -- (1,1);
        \draw[knot, overcross] (0.5, -0.207) to[out=90, in=-90] (0.915, 0.5);
        \draw[knot, overcross] (0.915, 0.5) to[out=90, in=-90] (0.5, 1.207);
        \draw[knot, <-] (0.915, 0.499) -- (0.915, 0.501);

}
\right)
	\]
	if strands are not co-oriented (see Lemma 7.11 in \cite{spyropoulos2024});
\end{enumerate}
where orientations are omitted if they do not affect the isomorphism. In addition, if $T$ is an $(m, n)$ tangle and $S$ is a $(n, p)$ tangle, then
\[
\mathbf{K}(T) \otimes_{H^n} \mathbf{K}(S) \cong \mathbf{K}(S \circ T) 
\]
in $\mathsf{K}_m^p$.
\end{proposition}

In order to obtain the honest link invariant, we must broaden $\mathcal{K}_m^n$. Suppose $M$ is a $\mathcal{G}$-graded $R$-module and $x\in M$ is homogeneous with
\[
\abs{x}_{\mathcal{G}} = (t: a\to b, (v_1, v_2)).
\]
We define the \emph{quantum grading} of $x$ to be
\[
\abs{x}_q = v_1 + v_2 \in \mathbb{Z}.
\]
In this way, we interpret a $\mathcal{G}$-graded $R$-module to be $\mathcal{Z} \times \mathcal{G}$-graded; that is, possessing an integral quantum grading and a $\mathcal{G}$-grading, though the former is obtained from the latter. A $\mathcal{G}$-grading shift extends to a $\mathcal{Z} \times \mathcal{G}$-grading shift by declaring the quantum component of a $\mathcal{G}$-grading shift on $x$ to be
\[
\deg_q(\varphi_{(W, v)}(x)) : =  \deg_q(x) + \deg_q(W) + v_1 + v_2
\]
where
\[
\deg_q(W) := \#\text{births} + \#\text{deaths} - \#\text{saddles}.
\]
Thus, following Proposition \ref{prop:Rmoves}, we obtain an honest tangle invariant if we reinterpret $\mathbf{K}(T)$ as belonging to a new category, $\check{\mathcal{K}}_m^n$, defined as the category whose objects are bounded complexes of geometric $\mathcal{G}$-graded $(H^m, H^n)$-bimodules with $\mathcal{G}$-homogeneous differential which preserves the underlying quantum grading. The morphisms of $\check{\mathcal{K}}_m^n$ are $\mathcal{G}$-homogeneous, $q$-preserving, homomorphisms of complexes up to chain homotopy. We let $\check{\mathsf{K}}_m^n$ denote the corresponding category obtained by localizing along quasi-isomorphisms and $\mathbf{K}_q(T)$ denote $\mathbf{K}(T)$ reinterpreted as an object of $\check{\mathcal{K}}_m^n$.

\begin{theorem}[Theorem 6.21 in \cite{naisse2020odd}]
If $T$ and $T'$ are isotopic $(m, n)$-tangles, then there is an isomorphism $\mathbf{K}_q(T) \cong \mathbf{K}_q(T')$ in $\check{\mathsf{K}}_m^n$.
\end{theorem}

\begin{corollary}
If $T$ is a link, then the homology $H^*(\mathbf{K}_q(T))$ is isomorphic to unified Khovanov homology \cite{putyra20152categorychronologicalcobordismsodd}. For a tangle, if the specialization $X = Y = Z = 1$ is chosen, we obtain the (even) Khovanov complex of a tangle \cite{MR1928174}; if $X = Z = 1$ and $Y = -1$ are chosen, we obtain the odd Khovanov complex of \cite{naisse2020odd}, generalizing odd Khovanov homology for links \cite{MR3071132}.
\end{corollary}

\subsubsection{Complexes with $\mathcal{C}$-homogeneous differential}

We will sate ourselves with a brief description of $\mathcal{C}$-graded DG bimodules with $\mathcal{C}$-homogeneous differential. See Sections 4.7, 4.8, 5.5, and 6.5 of \cite{naisse2020odd} for a more complete description.

\begin{definition}
A \emph{$\mathcal{C}$-graded DG bimodule with homogeneous differential}, $(M, d_M)$, over a pair of $\mathcal{C}$-graded algebras $A$ and $B$ is a $\mathcal{Z} \times \mathcal{C}$-graded $(A, B)$-bimodule $M = \bigoplus_{n \in \mathbb{Z}, g\in \mathrm{Mor}(\mathcal{C})} M_g^n$ with differential $d_M = \sum_{j\in J \subset \widetilde{\mathcal{I}}} d_{M, j}$ where $J$ is finite, such that for each homogeneous $a \in A$, $m\in M$, and $b\in B$,
\begin{itemize}
	\item $d_{M, j}(M_g^n) \subset M_{\varphi_j(g)}^{n+1}$ if $g\in \mathsf{D}_j$, and $d_{M, j} = 0$ otherwise, 
	\item $d_{M,j} (\rho_L(a, m)) = \beta_{\mathbf{e}, j} (\abs{y}_\mathcal{C}, \abs{m}_\mathcal{C})^{-1} \rho_L(a, d_{M,j} (m))$, 
	\item$ d_{M,j} (\rho_R(m, b)) = \beta_{j, \mathbf{e}}(\abs{m}_\mathcal{C}, \abs{b}_\mathcal{C})^{-1} \rho_R(d_{M,j}(m), b)$, and
	\item $d_M \circ d_M = 0$.
\end{itemize}
A map between these DG bimodules is a homogeneous map of $\mathcal{C}$-graded bimodules which preserves the homological grading and $\mathcal{C}$-graded commutes (in the sense of (\ref{eq:cgradedcomp})) with the differentials.
\end{definition}

Denote by $\mathrm{BIM}_{DG}^\mathcal{C}(A, B)$ the additive category of such $\mathcal{C}$-graded DG bimodules. Notice that the homology is a $\mathbb{Z} \times \mathbb{Z}$-graded space, but no longer a bimodule with a $\mathcal{G}$-grading. Much of the theory of these DG bimodules depends on the existence, but not the choice, of a \emph{commutativity system} for the grading category $\mathcal{C}$ and its chosen $\mathcal{C}$-shifting 2-system. In particular, the tensor product of these bimodules is defined as before, but must be twisted by a constant determined by the commutativity system. Similarly, the definition of chain homotopy is also twisted by elements of this commutativity system; see Section 4.8 of \cite{naisse2020odd}. This does not present any difficulties: Naisse and Putyra show that there exists a natural choice of commutativity system for $\mathcal{G}$ and the associated $\mathcal{G}$-shifting 2-system; see Section 5.5 of \cite{naisse2020odd}.

\subsubsection{The target category and 2-functor}

Borrowing notation from Khovanov, define the 2-category $\mathbb{K}$ whose
\begin{itemize}
	\item objects are non-negative integers,
	\item 1-morphisms are $\mathrm{Hom}_\mathbb{K}(m, n) = \mathsf{K}_m^n$, and whose
	\item 2-morphisms are
	\[
		\mathrm{Hom}_{\mathbb{K}}(M, N) = \bigoplus_{i \in \widetilde{\mathcal{I}}} \mathrm{Hom}_{\mathsf{K}_m^n}(\varphi_i(M), N) / \{u: u~\text{is a unit of}~R\}.
	\]
\end{itemize}
To be clear, this means that all 2-morphisms are (sums of) homogeneous maps each considered up to multiplication by units of $R$. In view of Proposition \ref{prop:Rmoves}, the target category we will consider is $\check{\mathbb{K}}$, which is the same as $\mathbb{K}$ except that we replace $\mathsf{K}_m^n$ everywhere with $\check{\mathsf{K}}_m^n$.

Construct a 2-functor $\mathbf{Tang} \to \check{\mathbb{K}}$ as follows. On objects, send $x\in \{+, -\}^{2n}$ to $n$. On 1-morphisms, send generating morphisms $T$ to $\mathbf{K}_q(T)$ and extend by sending compositions of 1-morphisms to the tensor product, just as in the definition of $\mathbf{K}_q$. Finally, the 2-functor on generating 2-morphisms is defined as follows. On Morse moves, simply assign the bimodule homomorphisms induced by $\mathcal{F}$. For Reidemeister moves, assign the isomorphisms from parts (2), (3), and (4) of Proposition \ref{prop:Rmoves}. For the isotopy moves, assign the isomorphism from part (1) of Proposition \ref{prop:Rmoves}. To see a complete description of these isomorphisms, consult Section 6.4 of \cite{naisse2020odd}.

\subsection{Proof of functoriality}
\label{ss:conclusion}

We conclude by proving that the assignment $\mathbf{Tang} \to \check{\mathbb{K}}$ described above extends uniquely to a 2-functor. This proves Theorem \ref{thm3}.

Each of the movies appearing in a Carter-Saito move correspond to an element of $\mathrm{Hom}_{\check{\mathcal{K}}_m^n}(t_0, t_\infty)$ (up to grading shifts), where $t_0$ denotes the first frame of the movies, and $t_\infty$ denotes the last. We will picture the Carter-Saito moves shortly; we call the movie appearing on the top of a Carter-Saito move $\Sigma_T$ and the movie appearing on the bottom $\Sigma_B$. Our method of proof is to use results of this section to show that these morphisms differ by at most a unit of $R$. In most of the situations, the morphisms make sense as $\mathcal{G}$-graded maps, and we can argue that the morphisms agree up to a unit in $\mathrm{Hom}_{\mathsf{K}_m^n}(t_0, t_\infty)$, so that the result holds. However, moves 6, 7, 14, 23a, 25, 26, 28, and 30 require special attention. This is because they contain Reidemeister II and/or III moves, which correspond to only $q$-grading preserving isomorphisms (not $\mathcal{G}$-grading preserving isomorphisms) thus we cannot apply previous corollaries carelessly. To evade this inconvenience, the trick is to notice that ``$\mathcal{G}$-homogeneous and $q$-preserving'' is equivalent to ``$\mathcal{G}$-preserving'' in the context of maps $f:H^n \to H^n$ as $\mathcal{G}$-graded $(H^n, H^n)$-bimodules.

\begin{lemma}
\label{lem:qcenter}
Let $\mathrm{Bim}^q(H^n)$ denote the category of $\mathcal{G}$-graded $(H^n, H^n)$-bimodules whose morphisms are $\mathcal{G}$-homogeneous, $q$-preserving maps. Then, if $M$ is an invertible complex in $\check{\mathcal{K}}_n^n$,
\[
\mathrm{End}_{\check{\mathcal{K}}_n^n}(M) \cong \mathrm{End}_{\mathrm{Bim}^q(H^n)}(H^n) \cong R.
\]
That is, there are no more endomorphisms of an invertible complex after extending to $\mathcal{G}$-homogeneous maps as long as the quantum grading is required to be preserved.
\end{lemma}

\begin{proof}
We start by considering the image of
\[
f \mapsto f(1) = \sum_{a \in \mathrm{Ob}(\mathcal{G})} f(1_a)
\]
in $H^n$ for $f \in \mathrm{End}_{\mathrm{Bim}^q(H^n)}(H^n)$. Using the same arguments of the proof of Proposition \ref{prop:c-gradedcenter} (and appealing to Definition \ref{def:homogeneousmaps}), it is easy to see that this is isomorphic to something we could call the \emph{$q$-graded center of $H^n$}, which is given by
\[
Z^q(H^n) := \{
z\in H^n : \mu(x, z) = C \mu(z, x) ~\text{for each}~ x \in H^n
\}
\]
where
\[
C :=
\alpha(\mathrm{Id}_a, \mathrm{Id}_a, \abs{x})
\alpha(\abs{x}, \mathrm{Id}_b, \mathrm{Id}_b)
\beta_{i, \mathbf{e}}(\mathrm{Id}_a, \abs{x})^{-1}
\beta_{\mathbf{e}, i}(\abs{x}, \mathrm{Id}_b)
\]
for $\abs{x}: a \to b$ and $\abs{f} = i$. From here, the first isomorphism follows by the same arguments of the proof of Proposition \ref{prop:invcomplex}. We claim that the second follows from grading considerations. Notice that the grading of $x\in {}_a H^n_b \subset H^n$ is given by 
\[
\abs{x}_\mathcal{G} = (1_n: a\to b, \deg_R(x))
\]
hence $\mathcal{G}$-grading shifts for elements of $H^n$ are only potentially nontrivial in the $\mathbb{Z} \times \mathbb{Z}$-component. Thus, the $\mathcal{G}$-grading shifts which preserve quantum degree are of the form $(n, -n)$ for any $n\in \mathbb{Z}$. Since $1_a$ has $\mathcal{G}$-grading $\mathrm{Id}_a = (1_n:a\to a, (n,0))$, and thus quantum grading $n$, $f(1_a)$ must have quantum grading $n$. But a homogeneous element of $H^n$ has quantum degree $n$ if and only if it is a unit of $H^n$. Thus, the assumption that $f$ is a $\mathcal{G}$-homogeneous and quantum-degree preserving map of $(H^n, H^n)$-bimodules means that it is actually $\mathcal{G}$-preserving. Thus the $\beta$-terms are each equal to 1, and we have that
\[
Z^q(H^n) = Z^\mathcal{G}(H^n) \cong R
\]
as desired.
\end{proof}

Now we have everything necessary to work through the Carter-Saito moves. First, we note that we can ignore the choice of framing on deaths and saddles as the linearized category of chronological cobordisms $R\mathbf{ChCob}$ admits the change of framing moves (\ref{eq:changeofframingmoves}), which state that the framing can be reversed as long as we compensate by some unit (see Subsection \ref{ss:chroncob}). Thus, we can ignore framings of the chronological cobordisms involved. 

Notice that each of moves 1, 2, 3, 4, and 5 consist of a single Reidemeister move and its inverse. Moreover, the fact that $\mathcal{F}(\Sigma_T) = \mathcal{F}(\Sigma_B)$ follows for these moves as long as the isomorphism assigned to the inverse of a Reidemeister move is the inverse of the isomorphism assigned to that Reidemeister move. This is the case, so we move on.
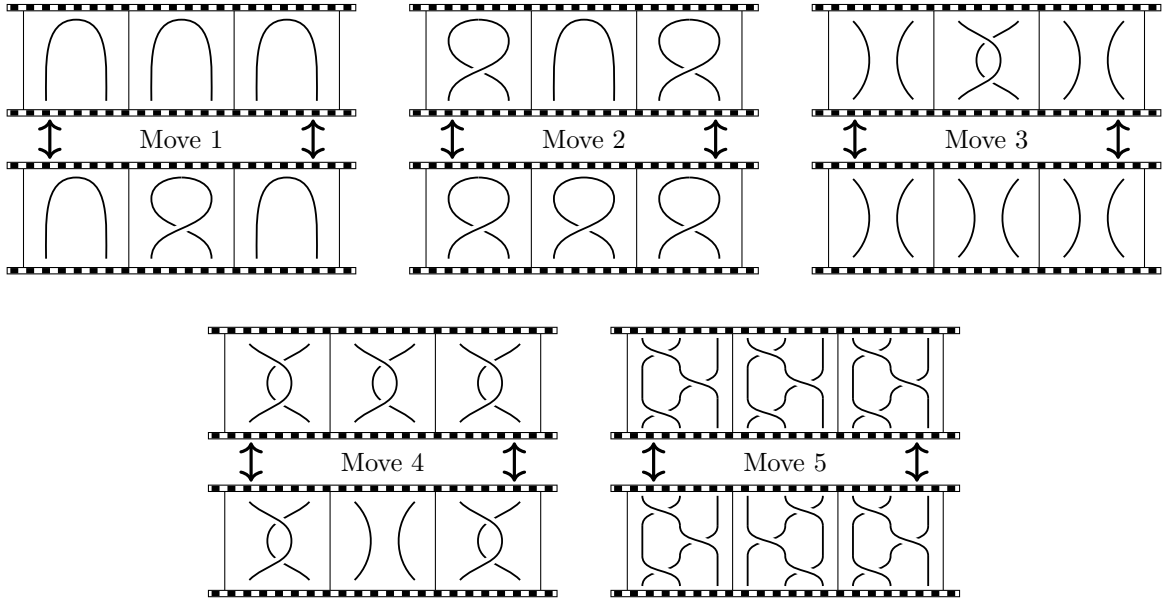
\[
\tikz[baseline={([yshift=-.5ex]current bounding box.center)}, scale=.35]{
	\draw (.5,0) -- (.5,4);
	\draw (4.5,0) -- (4.5,4);
	\draw (8.5,0) -- (8.5,4);
        \draw (12.5,0) -- (12.5,4);
 	\draw[double,double distance=2pt,line cap=rect] (0,0) -- (13,0);
	\draw[dashed, line width=2pt] (0,0) -- (13,0);
	\draw[double,double distance=2pt,line cap=rect] (0,4) -- (13,4);
	\draw[dashed, line width=2pt] (0,4) -- (13,4);
        \node at (2.5,2) {$\tikz[scale=0.8]{
        \draw[knot] (0,0) to[out=90,in=180] (0.5, 1.35) to[out=0, in=90] (1,0);
        }$};
        \node at (6.5, 2) {$\tikz[scale=0.8]{
        \draw[knot] (0,0) to[out=90,in=180] (0.5, 1.35) to[out=0, in=90] (1,0);
        }$};
        \node at (10.5,2) {$\tikz[scale=0.8]{
        \draw[knot] (0,0) to[out=90,in=180] (0.5, 1.35) to[out=0, in=90] (1,0);
        }$};
        \node at (6.5, -1) {Move 1};
        \draw[<->, very thick] (11.5, -0.25) -- (11.5,-1.75);
        \draw[<->, very thick] (1.5, -0.25) -- (1.5,-1.75);
    \begin{scope}[yshift=-6cm]
	\draw (.5,0) -- (.5,4);
	\draw (4.5,0) -- (4.5,4);
	\draw (8.5,0) -- (8.5,4);
        \draw (12.5,0) -- (12.5,4);
 	\draw[double,double distance=2pt,line cap=rect] (0,0) -- (13,0);
	\draw[dashed, line width=2pt] (0,0) -- (13,0);
	\draw[double,double distance=2pt,line cap=rect] (0,4) -- (13,4);
	\draw[dashed, line width=2pt] (0,4) -- (13,4);
        \node at (2.5,2) {$\tikz[scale=0.8]{
        \draw[knot] (0,0) to[out=90,in=180] (0.5, 1.35) to[out=0, in=90] (1,0);
        }$};
        \node at (6.5,2) {$\tikz[scale=0.8]{
        \draw[knot] (0.5,1.35) to[out=180, in=90] (0,1) to[out=-90, in=90] (1,0);
        \draw[knot, overcross] (0,0) to[out=90, in=-90] (1,1) to[out=90, in=0] (0.5,1.35);
        }$};
        \node at (10.5,2) {$\tikz[scale=0.8]{
        \draw[knot] (0,0) to[out=90,in=180] (0.5, 1.35) to[out=0, in=90] (1,0);
        }$};
    \end{scope}
 }
\qquad
\tikz[baseline={([yshift=-.5ex]current bounding box.center)}, scale=.35]{
	\draw (.5,0) -- (.5,4);
	\draw (4.5,0) -- (4.5,4);
	\draw (8.5,0) -- (8.5,4);
        \draw (12.5,0) -- (12.5,4);
 	\draw[double,double distance=2pt,line cap=rect] (0,0) -- (13,0);
	\draw[dashed, line width=2pt] (0,0) -- (13,0);
	\draw[double,double distance=2pt,line cap=rect] (0,4) -- (13,4);
	\draw[dashed, line width=2pt] (0,4) -- (13,4);
        \node at (2.5,2) {$\tikz[scale=0.8]{
        \draw[knot] (0.5,1.35) to[out=180, in=90] (0,1) to[out=-90, in=90] (1,0);
        \draw[knot, overcross] (0,0) to[out=90, in=-90] (1,1) to[out=90, in=0] (0.5,1.35);
        }$};
        \node at (6.5,2) {$\tikz[scale=0.8]{
        \draw[knot] (0,0) to[out=90,in=180] (0.5, 1.35) to[out=0, in=90] (1,0);
        }$};
        \node at (10.5,2) {$\tikz[scale=0.8]{
        \draw[knot] (0.5,1.35) to[out=180, in=90] (0,1) to[out=-90, in=90] (1,0);
        \draw[knot, overcross] (0,0) to[out=90, in=-90] (1,1) to[out=90, in=0] (0.5,1.35);
        }$};
        \node at (6.5, -1) {Move 2};
        \draw[<->, very thick] (11.5, -0.25) -- (11.5,-1.75);
        \draw[<->, very thick] (1.5, -0.25) -- (1.5,-1.75);
    \begin{scope}[yshift=-6cm]
	\draw (.5,0) -- (.5,4);
	\draw (4.5,0) -- (4.5,4);
	\draw (8.5,0) -- (8.5,4);
        \draw (12.5,0) -- (12.5,4);
 	\draw[double,double distance=2pt,line cap=rect] (0,0) -- (13,0);
	\draw[dashed, line width=2pt] (0,0) -- (13,0);
	\draw[double,double distance=2pt,line cap=rect] (0,4) -- (13,4);
	\draw[dashed, line width=2pt] (0,4) -- (13,4);
        \node at (2.5,2) {$\tikz[scale=0.8]{
        \draw[knot] (0.5,1.35) to[out=180, in=90] (0,1) to[out=-90, in=90] (1,0);
        \draw[knot, overcross] (0,0) to[out=90, in=-90] (1,1) to[out=90, in=0] (0.5,1.35);
        }$};
        \node at (6.5,2) {$\tikz[scale=0.8]{
        \draw[knot] (0.5,1.35) to[out=180, in=90] (0,1) to[out=-90, in=90] (1,0);
        \draw[knot, overcross] (0,0) to[out=90, in=-90] (1,1) to[out=90, in=0] (0.5,1.35);
        }$};
        \node at (10.5,2) {$\tikz[scale=0.8]{
        \draw[knot] (0.5,1.35) to[out=180, in=90] (0,1) to[out=-90, in=90] (1,0);
        \draw[knot, overcross] (0,0) to[out=90, in=-90] (1,1) to[out=90, in=0] (0.5,1.35);
        }$};
    \end{scope}
 }
\qquad
\tikz[baseline={([yshift=-.5ex]current bounding box.center)}, scale=.35]{
	\draw (.5,0) -- (.5,4);
	\draw (4.5,0) -- (4.5,4);
	\draw (8.5,0) -- (8.5,4);
        \draw (12.5,0) -- (12.5,4);
 	\draw[double,double distance=2pt,line cap=rect] (0,0) -- (13,0);
	\draw[dashed, line width=2pt] (0,0) -- (13,0);
	\draw[double,double distance=2pt,line cap=rect] (0,4) -- (13,4);
	\draw[dashed, line width=2pt] (0,4) -- (13,4);
        \node at (2.5,2) {$\tikz[scale=0.8, y=1.3cm]{
        \draw[knot] (1,0) to[out=135, in=-135] (1,1);
        \draw[knot, overcross] (0,0) to[out=45, in=-45] (0,1);
        }$};
        \node at (6.5,2) {$\tikz[scale=0.8, y=1.3cm]{
        \draw[knot] (1,0) to[out=135, in=-90] (0.3, 0.5) to[out=90, in=-135] (1,1);
        \draw[knot, overcross] (0,0) to[out=45, in=-90] (0.7, 0.5) to[out=90, in=-45] (0,1);
        }$};
        \node at (10.5,2) {$\tikz[scale=0.8, y=1.3cm]{
        \draw[knot] (1,0) to[out=135, in=-135] (1,1);
        \draw[knot, overcross] (0,0) to[out=45, in=-45] (0,1);
        }$};
        \node at (6.5, -1) {Move 3};
        \draw[<->, very thick] (11.5, -0.25) -- (11.5,-1.75);
        \draw[<->, very thick] (1.5, -0.25) -- (1.5,-1.75);
    \begin{scope}[yshift=-6cm]
	\draw (.5,0) -- (.5,4);
	\draw (4.5,0) -- (4.5,4);
	\draw (8.5,0) -- (8.5,4);
        \draw (12.5,0) -- (12.5,4);
 	\draw[double,double distance=2pt,line cap=rect] (0,0) -- (13,0);
	\draw[dashed, line width=2pt] (0,0) -- (13,0);
	\draw[double,double distance=2pt,line cap=rect] (0,4) -- (13,4);
	\draw[dashed, line width=2pt] (0,4) -- (13,4);
        \node at (2.5,2) {$\tikz[scale=0.8, y=1.3cm]{
        \draw[knot] (1,0) to[out=135, in=-135] (1,1);
        \draw[knot, overcross] (0,0) to[out=45, in=-45] (0,1);
        }$};
        \node at (6.5,2) {$\tikz[scale=0.8, y=1.3cm]{
        \draw[knot] (1,0) to[out=135, in=-135] (1,1);
        \draw[knot, overcross] (0,0) to[out=45, in=-45] (0,1);
        }$};
        \node at (10.5,2) {$\tikz[scale=0.8, y=1.3cm]{
        \draw[knot] (1,0) to[out=135, in=-135] (1,1);
        \draw[knot, overcross] (0,0) to[out=45, in=-45] (0,1);
        }$};
    \end{scope}
 }
\]
\[
\tikz[baseline={([yshift=-.5ex]current bounding box.center)}, scale=.35]{
	\draw (.5,0) -- (.5,4);
	\draw (4.5,0) -- (4.5,4);
	\draw (8.5,0) -- (8.5,4);
        \draw (12.5,0) -- (12.5,4);
 	\draw[double,double distance=2pt,line cap=rect] (0,0) -- (13,0);
	\draw[dashed, line width=2pt] (0,0) -- (13,0);
	\draw[double,double distance=2pt,line cap=rect] (0,4) -- (13,4);
	\draw[dashed, line width=2pt] (0,4) -- (13,4);
        \node at (2.5,2) {$\tikz[scale=0.8, y=1.3cm]{
        \draw[knot] (1,0) to[out=135, in=-90] (0.3, 0.5) to[out=90, in=-135] (1,1);
        \draw[knot, overcross] (0,0) to[out=45, in=-90] (0.7, 0.5) to[out=90, in=-45] (0,1);
        }$};
        \node at (6.5,2) {$\tikz[scale=0.8, y=1.3cm]{
        \draw[knot] (1,0) to[out=135, in=-90] (0.3, 0.5) to[out=90, in=-135] (1,1);
        \draw[knot, overcross] (0,0) to[out=45, in=-90] (0.7, 0.5) to[out=90, in=-45] (0,1);
        }$};
        \node at (10.5,2) {$\tikz[scale=0.8, y=1.3cm]{
        \draw[knot] (1,0) to[out=135, in=-90] (0.3, 0.5) to[out=90, in=-135] (1,1);
        \draw[knot, overcross] (0,0) to[out=45, in=-90] (0.7, 0.5) to[out=90, in=-45] (0,1);
        }$};
        \node at (6.5, -1) {Move 4};
        \draw[<->, very thick] (11.5, -0.25) -- (11.5,-1.75);
        \draw[<->, very thick] (1.5, -0.25) -- (1.5,-1.75);
    \begin{scope}[yshift=-6cm]
	\draw (.5,0) -- (.5,4);
	\draw (4.5,0) -- (4.5,4);
	\draw (8.5,0) -- (8.5,4);
        \draw (12.5,0) -- (12.5,4);
 	\draw[double,double distance=2pt,line cap=rect] (0,0) -- (13,0);
	\draw[dashed, line width=2pt] (0,0) -- (13,0);
	\draw[double,double distance=2pt,line cap=rect] (0,4) -- (13,4);
	\draw[dashed, line width=2pt] (0,4) -- (13,4);
        \node at (2.5,2) {$\tikz[scale=0.8, y=1.3cm]{
        \draw[knot] (1,0) to[out=135, in=-90] (0.3, 0.5) to[out=90, in=-135] (1,1);
        \draw[knot, overcross] (0,0) to[out=45, in=-90] (0.7, 0.5) to[out=90, in=-45] (0,1);
        }$};
        \node at (6.5,2) {$\tikz[scale=0.8, y=1.3cm]{
        \draw[knot] (1,0) to[out=135, in=-135] (1,1);
        \draw[knot, overcross] (0,0) to[out=45, in=-45] (0,1);
        }$};
        \node at (10.5,2) {$\tikz[scale=0.8, y=1.3cm]{
        \draw[knot] (1,0) to[out=135, in=-90] (0.3, 0.5) to[out=90, in=-135] (1,1);
        \draw[knot, overcross] (0,0) to[out=45, in=-90] (0.7, 0.5) to[out=90, in=-45] (0,1);
        }$};
    \end{scope}
 }
\qquad
\tikz[baseline={([yshift=-.5ex]current bounding box.center)}, scale=.35]{
	\draw (.5,0) -- (.5,4);
	\draw (4.5,0) -- (4.5,4);
	\draw (8.5,0) -- (8.5,4);
        \draw (12.5,0) -- (12.5,4);
 	\draw[double,double distance=2pt,line cap=rect] (0,0) -- (13,0);
	\draw[dashed, line width=2pt] (0,0) -- (13,0);
	\draw[double,double distance=2pt,line cap=rect] (0,4) -- (13,4);
	\draw[dashed, line width=2pt] (0,4) -- (13,4);
        \node at (2.5,2) {$\tikz[scale=0.5, y=0.8cm]{
            \draw[knot] (0,0) to[out=90, in=-90] (1,1);
            \draw[knot, overcross] (1,0) to[out=90, in=-90] (0,1);
            \draw[knot] (2,0) -- (2,1);
            \draw[knot] (0,1) -- (0,2);
            \draw[knot] (1,1) to[out=90, in=-90] (2,2);
            \draw[knot, overcross] (2,1) to[out=90, in=-90] (1,2);
            \draw[knot] (0,2) to[out=90, in=-90] (1,3);
            \draw[knot, overcross] (1,2) to[out=90, in=-90] (0,3);
            \draw[knot] (2,2) -- (2,3);
        }$};
        \node at (6.5,2) {$\tikz[scale=0.5, y=0.8cm]{
            \draw[knot] (0,0) to[out=90, in=-90] (1,1);
            \draw[knot, overcross] (1,0) to[out=90, in=-90] (0,1);
            \draw[knot] (2,0) -- (2,1);
            \draw[knot] (0,1) -- (0,2);
            \draw[knot] (1,1) to[out=90, in=-90] (2,2);
            \draw[knot, overcross] (2,1) to[out=90, in=-90] (1,2);
            \draw[knot] (0,2) to[out=90, in=-90] (1,3);
            \draw[knot, overcross] (1,2) to[out=90, in=-90] (0,3);
            \draw[knot] (2,2) -- (2,3);
        }$};
        \node at (10.5,2) {$\tikz[scale=0.5, y=0.8cm]{
            \draw[knot] (0,0) to[out=90, in=-90] (1,1);
            \draw[knot, overcross] (1,0) to[out=90, in=-90] (0,1);
            \draw[knot] (2,0) -- (2,1);
            \draw[knot] (0,1) -- (0,2);
            \draw[knot] (1,1) to[out=90, in=-90] (2,2);
            \draw[knot, overcross] (2,1) to[out=90, in=-90] (1,2);
            \draw[knot] (0,2) to[out=90, in=-90] (1,3);
            \draw[knot, overcross] (1,2) to[out=90, in=-90] (0,3);
            \draw[knot] (2,2) -- (2,3);
        }$};
        \node at (6.5, -1) {Move 5};
        \draw[<->, very thick] (11.5, -0.25) -- (11.5,-1.75);
        \draw[<->, very thick] (1.5, -0.25) -- (1.5,-1.75);
    \begin{scope}[yshift=-6cm]
	\draw (.5,0) -- (.5,4);
	\draw (4.5,0) -- (4.5,4);
	\draw (8.5,0) -- (8.5,4);
        \draw (12.5,0) -- (12.5,4);
 	\draw[double,double distance=2pt,line cap=rect] (0,0) -- (13,0);
	\draw[dashed, line width=2pt] (0,0) -- (13,0);
	\draw[double,double distance=2pt,line cap=rect] (0,4) -- (13,4);
	\draw[dashed, line width=2pt] (0,4) -- (13,4);
        \node at (2.5,2) {$\tikz[scale=0.5, y=0.8cm]{
            \draw[knot] (0,0) to[out=90, in=-90] (1,1);
            \draw[knot, overcross] (1,0) to[out=90, in=-90] (0,1);
            \draw[knot] (2,0) -- (2,1);
            \draw[knot] (0,1) -- (0,2);
            \draw[knot] (1,1) to[out=90, in=-90] (2,2);
            \draw[knot, overcross] (2,1) to[out=90, in=-90] (1,2);
            \draw[knot] (0,2) to[out=90, in=-90] (1,3);
            \draw[knot, overcross] (1,2) to[out=90, in=-90] (0,3);
            \draw[knot] (2,2) -- (2,3);
        }$};
        \node at (6.5,2) {$\tikz[scale=0.5, y=0.8cm]{
            \draw[knot] (0,0) -- (0,1);
            \draw[knot] (1,0) to[out=90, in=-90] (2,1);
            \draw[knot, overcross] (2,0) to[out=90, in=-90] (1,1);
            \draw[knot] (2,1) -- (2,2);
            \draw[knot] (0,1) to[out=90, in=-90] (1,2);
            \draw[knot, overcross] (1,1) to[out=90, in=-90] (0,2);
            \draw[knot] (0,2) -- (0,3);
            \draw[knot] (1,2) to[out=90, in=-90] (2,3);
            \draw[knot, overcross] (2,2) to[out=90, in=-90] (1,3);
        }$};
        \node at (10.5,2) {$\tikz[scale=0.5, y=0.8cm]{
            \draw[knot] (0,0) to[out=90, in=-90] (1,1);
            \draw[knot, overcross] (1,0) to[out=90, in=-90] (0,1);
            \draw[knot] (2,0) -- (2,1);
            \draw[knot] (0,1) -- (0,2);
            \draw[knot] (1,1) to[out=90, in=-90] (2,2);
            \draw[knot, overcross] (2,1) to[out=90, in=-90] (1,2);
            \draw[knot] (0,2) to[out=90, in=-90] (1,3);
            \draw[knot, overcross] (1,2) to[out=90, in=-90] (0,3);
            \draw[knot] (2,2) -- (2,3);
        }$};
    \end{scope}
 }
\]

Actually, moves 3, 4, and 5 fit into a larger family with moves 6, 12, 13, 23a, and 25. Notice that each of the movies involved with these moves consist of an isotopy or a Reidemeister move, and that at least one of $t_0$ and $t_\infty$ is a braid. The first observation tells us (by Proposition \ref{prop:Rmoves} for $\mathsf{K}_m^n$ or $\check{\mathsf{K}}_m^n$ in the case of moves 6, 23a, and 25) that $\mathcal{F}(\Sigma_T)$ and $\mathcal{F}(\Sigma_B)$ are isomorphisms. The second observation tells us that $\mathcal{F}(\Sigma_T)$ and $\mathcal{F}(\Sigma_B)$ must be automorphisms of an invertible complex. By Corollary \ref{cor:invs} or Lemma \ref{lem:qcenter} for moves 6, 23a, and 25, we conclude that $\mathcal{F}(\Sigma_T) = u \mathcal{F}(\Sigma_B)$ for some unit $u$ of $R$. 

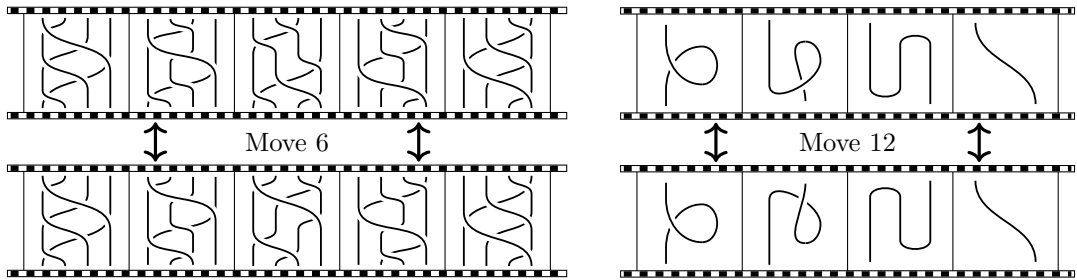
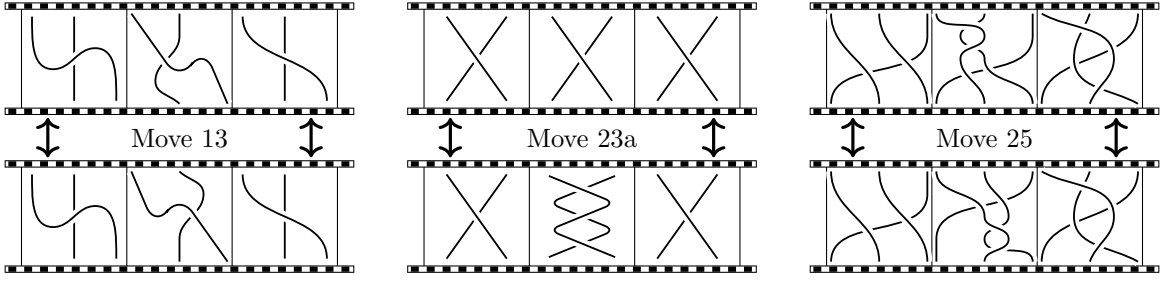
\[
\tikz[baseline={([yshift=-.5ex]current bounding box.center)}, scale=.35]{
	\draw (.5,0) -- (.5,4);
	\draw (4.5,0) -- (4.5,4);
	\draw (8.5,0) -- (8.5,4);
        \draw (12.5,0) -- (12.5,4);
        \draw (16.5,0) -- (16.5,4);
        \draw (20.5,0) -- (20.5,4);
 	\draw[double,double distance=2pt,line cap=rect] (0,0) -- (21,0);
	\draw[dashed, line width=2pt] (0,0) -- (21,0);
	\draw[double,double distance=2pt,line cap=rect] (0,4) -- (21,4);
	\draw[dashed, line width=2pt] (0,4) -- (21,4);
        \node at (2.5,2) {$\tikz[scale=0.3, y=0.8cm]{
            \draw[knot] (0,0) to[out=90, in=-90] (3,3) to[out=90, in=-90] (3,5);
            \draw[knot, overcross] (1,0) to[out=90, in=-90] (0,1) to[out=90, in=-90] (0,2) to[out=90, in=-90] (2,4) to[out=90, in=-90] (2,5);
            \draw[knot, overcross] (2,0) to[out=90, in=-90] (2,1) to[out=90, in=-90] (0,3) to[out=90, in=-90] (0,4) to[out=90, in=-90] (1,5);
            \draw[knot, overcross] (3,0) to[out=90, in=-90] (3,2) to[out=90, in=-90] (0,5);
        }$};
        \node at (6.5,2) {$\tikz[scale=0.3, y=0.65cm]{
            \draw[knot] (0,0) to[out=90, in=-90] (3,3) to[out=90, in=-90] (3,6);
            \draw[knot, overcross] (1,0) to[out=90, in=-90] (0,1) to[out=90, in=-90] (0,4) to[out=90, in=-90] (2,6);
            \draw[knot, overcross] (2,0) to[out=90, in=-90] (2,1) to[out=90, in=-90] (1,2) to[out=90, in=-90](1,3) to[out=90, in=-90] (2,4) to[out=90, in=-90] (2,5) to[out=90, in=-90] (1,6);
            \draw[knot, overcross] (3,0) to[out=90, in=-90] (3,2) to[out=90, in=-90] (0,5) to[out=90, in=-90] (0,6);
        }$};
        \node at (10.5,2) {$\tikz[scale=0.3, y=0.8cm]{
            \draw[knot] (0,0) to[out=90, in=-90] (3,3) to[out=90, in=-90] (3,5);
            \draw[knot, overcross] (1,0) to[out=90, in=-90] (0,1) to[out=90, in=-90] (0,3) to[out=90, in=-90] (2,5);
            \draw[knot, overcross] (2,0) to[out=90, in=-90] (3,1) to[out=90, in=-90] (3,2) to[out=90, in=-90](2,3) to[out=90, in=-90] (2,4) to[out=90, in=-90] (1,5);
            \draw[knot, overcross] (3,0) to[out=90, in=-90] (1,2) to[out=90, in=-90] (1,3) to[out=90, in=-90] (0,4) to[out=90, in=-90] (0,5);
        }$};
        \node at (14.5,2) {$\tikz[scale=0.3, y=0.65cm]{
            \draw[knot] (0,0) to[out=90, in=-90] (0,2) to[out=90, in=-90] (3,5) to[out=90, in=-90] (3,6);
            \draw[knot, overcross] (1,0) to[out=90, in=-90] (1,1) to[out=90, in=-90] (2,2) to[out=90, in=-90] (2,3) to[out=90, in=-90] (1,4) to[out=90, in=-90] (1,5) to[out=90, in=-90] (2,6);
            \draw[knot, overcross] (2,0) to[out=90, in=-90] (3,1) to[out=90, in=-90] (3,4) to[out=90, in=-90] (1,6);
            \draw[knot, overcross] (3,0) to[out=90, in=-90] (0,3) to[out=90, in=-90] (0,6);
        }$};
        \node at (18.5,2) {$\tikz[scale=0.3, y=0.8cm]{
            \draw[knot] (0,0) to[out=90, in=-90] (0,2) to[out=90, in=-90] (3,5);
            \draw[knot, overcross] (1,0) to[out=90, in=-90] (1,1) to[out=90, in=-90] (3,3) to[out=90, in=-90] (3,4) to[out=90, in=-90] (2,5);
            \draw[knot, overcross] (2,0) to[out=90, in=-90] (3,1) to[out=90, in=-90] (3,2) to[out=90, in=-90] (1,4) to[out=90, in=-90] (1,5);
            \draw[knot, overcross] (3,0) to[out=90, in=-90] (0,3) to[out=90, in=-90] (0,5);
        }$};
        \node at (10.5, -1) {Move 6};
        \draw[<->, very thick] (15.5, -0.25) -- (15.5,-1.75);
        \draw[<->, very thick] (5.5, -0.25) -- (5.5,-1.75);
    \begin{scope}[yshift=-6cm]
	\draw (.5,0) -- (.5,4);
	\draw (4.5,0) -- (4.5,4);
	\draw (8.5,0) -- (8.5,4);
        \draw (12.5,0) -- (12.5,4);
        \draw (16.5,0) -- (16.5,4);
        \draw (20.5,0) -- (20.5,4);
 	\draw[double,double distance=2pt,line cap=rect] (0,0) -- (21,0);
	\draw[dashed, line width=2pt] (0,0) -- (21,0);
	\draw[double,double distance=2pt,line cap=rect] (0,4) -- (21,4);
	\draw[dashed, line width=2pt] (0,4) -- (21,4);
        \node at (2.5,2) {$\tikz[scale=0.3, y=0.8cm]{
            \draw[knot] (0,0) to[out=90, in=-90] (3,3) to[out=90, in=-90] (3,5);
            \draw[knot, overcross] (1,0) to[out=90, in=-90] (0,1) to[out=90, in=-90] (0,2) to[out=90, in=-90] (2,4) to[out=90, in=-90] (2,5);
            \draw[knot, overcross] (2,0) to[out=90, in=-90] (2,1) to[out=90, in=-90] (0,3) to[out=90, in=-90] (0,4) to[out=90, in=-90] (1,5);
            \draw[knot, overcross] (3,0) to[out=90, in=-90] (3,2) to[out=90, in=-90] (0,5);
        }$};
        \node at (6.5,2) {$\tikz[scale=0.3, y=0.65cm]{
            \draw[knot] (0,0) to[out=90, in=-90] (0,1) to[out=90, in=-90] (3,4) to[out=90, in=-90] (3,6);
            \draw[knot, overcross] (1,0) to[out=90, in=-90] (2,1) to[out=90, in=-90] (2,2) to[out=90, in=-90] (1,3) to[out=90, in=-90] (1,4) to[out=90, in=-90] (2,5) to[out=90, in=-90] (2,6);
            \draw[knot, overcross] (2,0) to[out=90, in=-90] (0,2) to[out=90, in=-90] (0,5) to[out=90, in=-90] (1,6);
            \draw[knot, overcross] (3,0) to[out=90, in=-90] (3,3) to[out=90, in=-90] (0,6);
        }$};
        \node at (10.5,2) {$\tikz[scale=0.3, y=0.8cm]{
            \draw[knot] (0,0) to[out=90, in=-90] (0,1) to[out=90, in=-90] (1,2) to[out=90, in=-90] (1,3) to[out=90, in=-90] (3,5);
            \draw[knot, overcross] (1,0) to[out=90, in=-90] (2,1) to[out=90, in=-90] (2,2) to[out=90, in=-90] (3,3) to[out=90, in=-90] (3,4) to[out=90, in=-90] (2,5);
            \draw[knot, overcross] (2,0) to[out=90, in=-90] (0,2) to[out=90, in=-90] (0,4) to[out=90, in=-90] (1,5);
            \draw[knot, overcross] (3,0) to[out=90, in=-90] (3,2) to[out=90, in=-90] (0,5);
        }$};
        \node at (14.5,2) {$\tikz[scale=0.3, y=0.65cm]{
            \draw[knot] (0,0) to[out=90, in=-90] (0,3) to[out=90, in=-90] (3,6);
            \draw[knot, overcross] (1,0) to[out=90, in=-90] (3,2) to[out=90, in=-90] (3,5) to[out=90, in=-90] (2,6);
            \draw[knot, overcross] (2,0) to[out=90, in=-90] (1,1) to[out=90, in=-90] (1,2) to[out=90, in=-90] (2,3) to[out=90, in=-90] (2,4) to[out=90, in=-90] (1,5) to[out=90, in=-90] (1,6);
            \draw[knot, overcross] (3,0) to[out=90, in=-90] (3,1) to[out=90, in=-90] (0,4) to[out=90, in=-90] (0,6);
        }$};
        \node at (18.5,2) {$\tikz[scale=0.3, y=0.8cm]{
            \draw[knot] (0,0) to[out=90, in=-90] (0,2) to[out=90, in=-90] (3,5);
            \draw[knot, overcross] (1,0) to[out=90, in=-90] (1,1) to[out=90, in=-90] (3,3) to[out=90, in=-90] (3,4) to[out=90, in=-90] (2,5);
            \draw[knot, overcross] (2,0) to[out=90, in=-90] (3,1) to[out=90, in=-90] (3,2) to[out=90, in=-90] (1,4) to[out=90, in=-90] (1,5);
            \draw[knot, overcross] (3,0) to[out=90, in=-90] (0,3) to[out=90, in=-90] (0,5);
        }$};
    \end{scope}
 }
\qquad
\tikz[baseline={([yshift=-.5ex]current bounding box.center)}, scale=.35]{
	\draw (.5,0) -- (.5,4);
	\draw (4.5,0) -- (4.5,4);
	\draw (8.5,0) -- (8.5,4);
        \draw (12.5,0) -- (12.5,4);
        \draw (16.5,0) -- (16.5,4);
 	\draw[double,double distance=2pt,line cap=rect] (0,0) -- (17,0);
	\draw[dashed, line width=2pt] (0,0) -- (17,0);
	\draw[double,double distance=2pt,line cap=rect] (0,4) -- (17,4);
	\draw[dashed, line width=2pt] (0,4) -- (17,4);
        \node at (2.5,2) {$\tikz[scale=0.8, y=0.9cm]{
        \draw[knot] (0,0) to[out=90, in=-90] (0, 0.3) to[out=90, in=180] (0.6, 1.1);
        \draw[knot] (0.6, 1.1) to[out=0, in=90] (0.85, 0.75) to[out=-90, in=0] (0.6, 0.4);
        \draw[knot, overcross] (0.6, 0.4) to[out=180, in=-90] (0,1.2) -- (0,1.5);
        }$};
        \node at (6.5,2) {$\tikz[scale=0.8, y=0.9cm]{
        \draw[knot] (0.6, 1.1) to[out=180, in=90] (0.6, 0);
        \draw[knot, overcross] (0, 1.5) -- (0,0.3) to [out=-90, in=-90] (0.85, 0.75) to[out=90, in=0] (0.6, 1.1);
        }$};
        \node at (10.5,2) {$\tikz[scale=0.8, y=0.9cm]{
        \draw[knot] (0, 1.5) -- (0, 0.3) to[out=-90, in=-90] (0.5, 0.3) -- (0.5, 1.1) to[out=90, in=90] (1, 1.1) -- (1, 0);
        }$};
        \node at (14.5,2) {$\tikz[scale=0.8, y=0.9cm]{
        \draw[knot] (0,1.5) to[out=-90, in=90] (1,0);
        }$};
        \node at (8.5, -1) {Move 12};
        \draw[<->, very thick] (13.5, -0.25) -- (13.5,-1.75);
        \draw[<->, very thick] (3.5, -0.25) -- (3.5,-1.75);
    \begin{scope}[yshift=-6cm]
	\draw (.5,0) -- (.5,4);
	\draw (4.5,0) -- (4.5,4);
	\draw (8.5,0) -- (8.5,4);
        \draw (12.5,0) -- (12.5,4);
        \draw (16.5,0) -- (16.5,4);
 	\draw[double,double distance=2pt,line cap=rect] (0,0) -- (17,0);
	\draw[dashed, line width=2pt] (0,0) -- (17,0);
	\draw[double,double distance=2pt,line cap=rect] (0,4) -- (17,4);
	\draw[dashed, line width=2pt] (0,4) -- (17,4);
        \node at (2.5,2) {$\tikz[scale=0.8, y=0.9cm]{
        \draw[knot] (0,0) to[out=90, in=-90] (0, 0.3) to[out=90, in=180] (0.6, 1.1);
        \draw[knot] (0.6, 1.1) to[out=0, in=90] (0.85, 0.75) to[out=-90, in=0] (0.6, 0.4);
        \draw[knot, overcross] (0.6, 0.4) to[out=180, in=-90] (0,1.2) -- (0,1.5);
        }$};
        \node at (6.5,2) {$\tikz[scale=0.8, y=0.9cm]{
        \draw[knot] (0,0) -- (0, 1.2) to[out=90, in=90] (0.85, 0.75) to[out=-90,in=0] (0.6, 0.4);
        \draw[knot, overcross] (0.6, 0.4) to[out=180, in=-90] (0.6, 1.5);
        }$};
        \node at (10.5,2) {$\tikz[scale=0.8, y=0.9cm]{
        \draw[knot] (0,0) -- (0, 1.2) to[out=90, in=90] (0.5, 1.2) -- (0.5, 0.4) to[out=-90,in=-90] (1, 0.4) -- (1, 1.5);
        }$};
        \node at (14.5,2) {$\tikz[scale=0.8, y=0.9cm]{
        \draw[knot] (0,1.5) to[out=-90, in=90] (1,0);
        }$};
    \end{scope}
 }
\]
\[
\tikz[baseline={([yshift=-.5ex]current bounding box.center)}, scale=.35]{
	\draw (.5,0) -- (.5,4);
	\draw (4.5,0) -- (4.5,4);
	\draw (8.5,0) -- (8.5,4);
        \draw (12.5,0) -- (12.5,4);
 	\draw[double,double distance=2pt,line cap=rect] (0,0) -- (13,0);
	\draw[dashed, line width=2pt] (0,0) -- (13,0);
	\draw[double,double distance=2pt,line cap=rect] (0,4) -- (13,4);
	\draw[dashed, line width=2pt] (0,4) -- (13,4);
        \node at (2.5,2) {$\tikz[scale=0.8, y=0.7cm, x=0.7cm]{
        \draw[knot] (1,0) -- (1,2);
        \draw[knot, overcross] (2,0) to[out=90, in=0] (1.5, 1.25) to[out=180, in=0] (0.5, 0.75) to[out=180, in=-90] (0,2);
        }$};
        \node at (6.5, 2) {$\tikz[scale=0.8, y=0.75cm, x=0.8cm]{
        \draw[knot] (1,0) to[out=135, in=-90] (0.5, 0.5) to[out=90, in=-90] (1,1.5) -- (1,2);
        \draw[knot, overcross, rounded corners=3mm] (2,0) -- (1.5, 1.25) -- (1,0.5) -- (0,2);
        }$};
        \node at (10.5,2) {$\tikz[scale=0.8, y=0.7cm, x=0.7cm]{
        \draw[knot] (1,0) -- (1,2);
        \draw[knot, overcross] (2,0) to[out=90, in=-90] (0,2);
        }$};
        \node at (6.5, -1) {Move 13};
        \draw[<->, very thick] (11.5, -0.25) -- (11.5,-1.75);
        \draw[<->, very thick] (1.5, -0.25) -- (1.5,-1.75);
    \begin{scope}[yshift=-6cm]
	\draw (.5,0) -- (.5,4);
	\draw (4.5,0) -- (4.5,4);
	\draw (8.5,0) -- (8.5,4);
        \draw (12.5,0) -- (12.5,4);
 	\draw[double,double distance=2pt,line cap=rect] (0,0) -- (13,0);
	\draw[dashed, line width=2pt] (0,0) -- (13,0);
	\draw[double,double distance=2pt,line cap=rect] (0,4) -- (13,4);
	\draw[dashed, line width=2pt] (0,4) -- (13,4);
        \node at (2.5,2) {$\tikz[scale=0.8, y=0.7cm, x=0.7cm]{
        \draw[knot] (1,0) -- (1,2);
        \draw[knot, overcross] (2,0) to[out=90, in=0] (1.5, 1.25) to[out=180, in=0] (0.5, 0.75) to[out=180, in=-90] (0,2);
        }$};
        \node at (6.5,2) {$\tikz[scale=0.8, y=0.75cm, x=0.8cm]{
        \draw[knot] (1,0) to[out=90, in=-90] (1,0.5) to[out=90, in=-90] (1.5, 1.5) to[out=90, in=-45] (1,2);
        \draw[knot, overcross, rounded corners=3mm] (2,0) -- (1,1.5) -- (0.5, 0.75) -- (0,2);
        }$};
        \node at (10.5,2) {$\tikz[scale=0.8, y=0.7cm, x=0.7cm]{
        \draw[knot] (1,0) -- (1,2);
        \draw[knot, overcross] (2,0) to[out=90, in=-90] (0,2);
        }$};
    \end{scope}
 }
\qquad
\tikz[baseline={([yshift=-.5ex]current bounding box.center)}, scale=.35]{
	\draw (.5,0) -- (.5,4);
	\draw (4.5,0) -- (4.5,4);
	\draw (8.5,0) -- (8.5,4);
        \draw (12.5,0) -- (12.5,4);
 	\draw[double,double distance=2pt,line cap=rect] (0,0) -- (13,0);
	\draw[dashed, line width=2pt] (0,0) -- (13,0);
	\draw[double,double distance=2pt,line cap=rect] (0,4) -- (13,4);
	\draw[dashed, line width=2pt] (0,4) -- (13,4);
        \node at (2.5,2) {$\tikz[scale=0.8, y=1.4cm]{
        \draw[knot] (0,0) -- (1,1);
        \draw[knot, overcross] (1,0) -- (0,1);
        }$};
        \node at (6.5, 2) {$\tikz[scale=0.8, y=1.4cm]{
        \draw[knot] (0,0) -- (1,1);
        \draw[knot, overcross] (1,0) -- (0,1);
        }$};
        \node at (10.5,2) {$\tikz[scale=0.8, y=1.4cm]{
        \draw[knot] (0,0) -- (1,1);
        \draw[knot, overcross] (1,0) -- (0,1);
        }$};
        \node at (6.5, -1) {Move 23a};
        \draw[<->, very thick] (11.5, -0.25) -- (11.5,-1.75);
        \draw[<->, very thick] (1.5, -0.25) -- (1.5,-1.75);
    \begin{scope}[yshift=-6cm]
	\draw (.5,0) -- (.5,4);
	\draw (4.5,0) -- (4.5,4);
	\draw (8.5,0) -- (8.5,4);
        \draw (12.5,0) -- (12.5,4);
 	\draw[double,double distance=2pt,line cap=rect] (0,0) -- (13,0);
	\draw[dashed, line width=2pt] (0,0) -- (13,0);
	\draw[double,double distance=2pt,line cap=rect] (0,4) -- (13,4);
	\draw[dashed, line width=2pt] (0,4) -- (13,4);
        \node at (2.5,2) {$\tikz[scale=0.8, y=1.4cm]{
        \draw[knot] (0,0) -- (1,1);
        \draw[knot, overcross] (1,0) -- (0,1);
        }$};
        \node at (6.5,2) {$\tikz[scale=0.8, x=1.1cm, y=0.45cm]{
        \draw[knot, rounded corners] (0,0) -- (1,1) -- (0,2) -- (1,3);
        \draw[knot, rounded corners, overcross] (1,0) -- (0,1) -- (1,2) -- (0,3);
        }$};
        \node at (10.5,2) {$\tikz[scale=0.8, y=1.4cm]{
        \draw[knot] (0,0) -- (1,1);
        \draw[knot, overcross] (1,0) -- (0,1);
        }$};
    \end{scope}
 }
\qquad
\tikz[baseline={([yshift=-.5ex]current bounding box.center)}, scale=.35]{
	\draw (.5,0) -- (.5,4);
	\draw (4.5,0) -- (4.5,4);
	\draw (8.5,0) -- (8.5,4);
        \draw (12.5,0) -- (12.5,4);
 	\draw[double,double distance=2pt,line cap=rect] (0,0) -- (13,0);
	\draw[dashed, line width=2pt] (0,0) -- (13,0);
	\draw[double,double distance=2pt,line cap=rect] (0,4) -- (13,4);
	\draw[dashed, line width=2pt] (0,4) -- (13,4);
        \node at (2.5,2) {$\tikz[scale=0.8, x=0.8cm, y=0.75cm]{
        \draw[knot] (0,0) to[out=90, in=-90] (2,1.5) -- (2,2);
        \draw[knot, overcross] (1,0) to[out=90, in=-90] (0,2);
        \draw[knot, overcross] (2,0) to[out=90, in=-90] (1,2);
        }$};
        \node at (6.5, 2) {$\tikz[scale=0.8, x=0.8cm, y=0.75cm]{
        \draw[knot] (0,0) to[out=90, in=-90] (2,1.5) -- (2,2);
        \draw[knot, overcross] (2,0) to[out=90, in=-90] (1,1) to[out=90, in=-90] (0.5,1.5) to[out=90, in=-90] (1,2);
        \draw[knot, overcross] (1,0) to[out=90, in=-90] (0.5, 1) to[out=90, in=-90] (1, 1.5) to[out=90, in=-90] (0,2);
        }$};
        \node at (10.5,2) {$\tikz[scale=0.8, x=0.8cm, y=0.75cm]{
        \draw[knot] (0,0) to[out=90, in=-90] (2,2);
        \draw[knot, overcross] (2,0) to[out=150, in=-90] (0.67, 1) to[out=90, in=-110] (1,2);
        \draw[knot, overcross] (1,0) to[out=70, in=-90] (1.5, 1) to[out=90, in=-30] (0,2);
        }$};
        \node at (6.5, -1) {Move 25};
        \draw[<->, very thick] (11.5, -0.25) -- (11.5,-1.75);
        \draw[<->, very thick] (1.5, -0.25) -- (1.5,-1.75);
    \begin{scope}[yshift=-6cm]
	\draw (.5,0) -- (.5,4);
	\draw (4.5,0) -- (4.5,4);
	\draw (8.5,0) -- (8.5,4);
        \draw (12.5,0) -- (12.5,4);
 	\draw[double,double distance=2pt,line cap=rect] (0,0) -- (13,0);
	\draw[dashed, line width=2pt] (0,0) -- (13,0);
	\draw[double,double distance=2pt,line cap=rect] (0,4) -- (13,4);
	\draw[dashed, line width=2pt] (0,4) -- (13,4);
        \node at (2.5,2) {$\tikz[scale=0.8, x=0.8cm, y=0.75cm]{
        \draw[knot] (0,0) to[out=90, in=-90] (2,1.5) -- (2,2);
        \draw[knot, overcross] (1,0) to[out=90, in=-90] (0,2);
        \draw[knot, overcross] (2,0) to[out=90, in=-90] (1,2);
        }$};
        \node at (6.5,2) {$\tikz[scale=0.8, x=0.8cm, y=0.75cm]{
        \draw[knot] (0,0) -- (0,0.5) to[out=90, in=-90] (2,2);
        \draw[knot, overcross] (2,0) to[out=90, in=-90] (1,0.5) to[out=90, in=-90] (1.5, 1) to[out=90, in=-90] (1,2);
        \draw[knot, overcross] (1,0) to[out=90, in=-90] (1.5, 0.5) to[out=90, in=-90] (1,1) to[out=90, in=-90] (0,2);
        }$};
        \node at (10.5,2) {$\tikz[scale=0.8, x=0.8cm, y=0.75cm]{
        \draw[knot] (0,0) to[out=90, in=-90] (2,2);
        \draw[knot, overcross] (2,0) to[out=150, in=-90] (0.67, 1) to[out=90, in=-110] (1,2);
        \draw[knot, overcross] (1,0) to[out=70, in=-90] (1.5, 1) to[out=90, in=-30] (0,2);
        }$};
    \end{scope}
 }
\]

Moves 7, 11, 14, and 26 also relate movies which consist only of Reidemeister moves, and thus $\mathcal{F}(\Sigma_T)$ and $\mathcal{F}(\Sigma_B)$ are isomorphisms $\mathcal{F}(t_0) \to \mathcal{F}(t_\infty)$, but they do not start or end at braids. Notice, however, that $t_\infty$ for each of these moves is single cap composed with a single crossing ($t_0$ of move 26 is a single cap). Thus, by Corollary \ref{cor:cups_caps}, we have that $\mathcal{F}(\Sigma_T) = u \mathcal{F}(\Sigma_B)$ for some unit $u$ of $R$. More precisely, for moves 7, 14, and 26, since all adjunctions (and, in particular, those of Proposition \ref{prop:Kadjun}) hold after flattening $\mathcal{G}$-degree to $q$-degree, the arguments preceding Corollary \ref{cor:cups_caps} hold, appealing to Lemma \ref{lem:qcenter}. Thus Corollary \ref{cor:cups_caps} holds in the $\check{\mathcal{K}}_m^n$ setting.
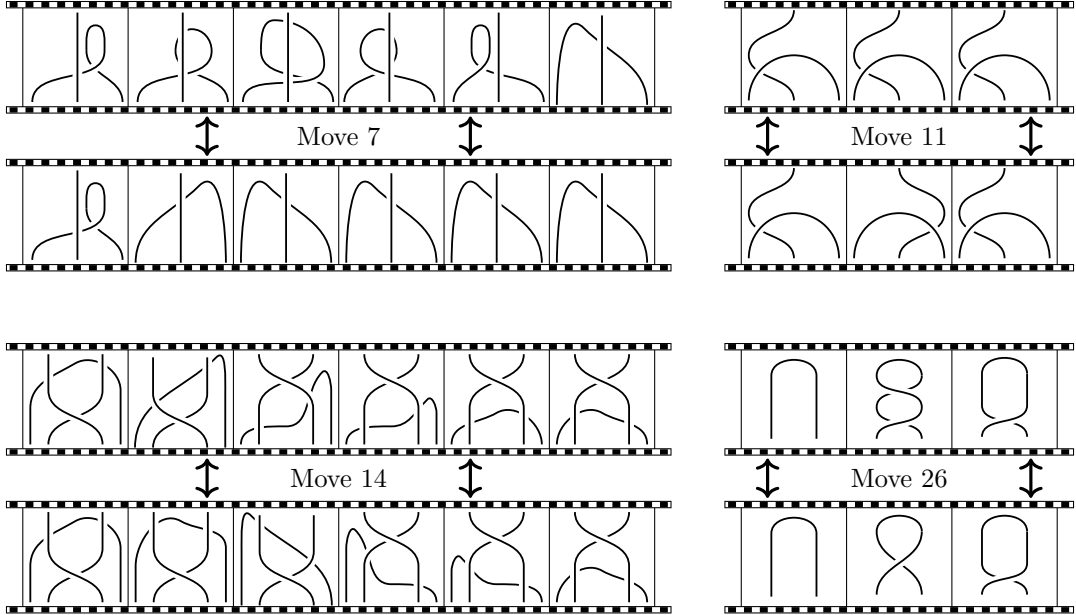
\[
\tikz[baseline={([yshift=-.5ex]current bounding box.center)}, scale=.35]{
	\draw (.5,0) -- (.5,4);
	\draw (4.5,0) -- (4.5,4);
	\draw (8.5,0) -- (8.5,4);
        \draw (12.5,0) -- (12.5,4);
        \draw (16.5,0) -- (16.5,4);
        \draw (20.5,0) -- (20.5,4);
        \draw (24.5,0) -- (24.5,4);
 	\draw[double,double distance=2pt,line cap=rect] (0,0) -- (25,0);
	\draw[dashed, line width=2pt] (0,0) -- (25,0);
	\draw[double,double distance=2pt,line cap=rect] (0,4) -- (25,4);
	\draw[dashed, line width=2pt] (0,4) -- (25,4);
        \node at (2.5,2) {$\tikz[scale=0.8, x=0.3cm, y=0.425cm]{
        \draw[knot] (3,2) to[out=-90, in=90] (5,0);
        \draw[knot, overcross] (0,0) to[out=90, in=-90] (4,2) to[out=90, in=0] (3.5, 3) to[out=180, in=90] (3,2);
        \draw[knot, overcross] (2.5, 0) -- (2.5, 3.5);
        }$};
        \node at (6.5,2) {$\tikz[scale=0.8, x=0.3cm, y=0.425cm]{
        \draw[knot] (2.1,2) to[out=-90, in=90] (5,0);
        \draw[knot, overcross] (0,0) to[out=90, in=-90] (4.1,2) to[out=90, in=0] (3, 2.85) to[out=180, in=90] (2.1,2);
        \draw[knot, overcross] (2.5, 0) -- (2.5, 3.5);
        }$};
        \node at (10.5,2) {$\tikz[scale=0.8, x=0.3cm, y=0.425cm]{
        \draw[knot] (1,2) to[out=-90, in=90] (5,0);
        \draw[knot, overcross] (0,0) to[out=90, in=-90] (4.5,1.5) to[out=90, in=0] (1.75, 3.2) to[out=180, in=90] (1,2);
        \draw[knot, overcross] (2.5, 0) -- (2.5, 3.5);
        }$};
        \node at (14.5,2) {$\tikz[scale=0.8, x=0.3cm, y=0.425cm]{
        \draw[knot] (2.9,2) to[out=-90, in=90] (0,0);
        \draw[knot, overcross] (5,0) to[out=90, in=-90] (0.9,2) to[out=90, in=180] (2, 2.85) to[out=0, in=90] (2.9,2);
        \draw[knot, overcross] (2.5, 0) -- (2.5, 3.5);
        }$};
        \node at (18.5,2) {$\tikz[scale=0.8, x=0.3cm, y=0.425cm]{
        \draw[knot] (1,2) to[out=-90, in=90] (5,0);
        \draw[knot, overcross] (0,0) to[out=90, in=-90] (2,2) to[out=90, in=0] (1.5, 3) to[out=180, in=90] (1,2);
        \draw[knot, overcross] (2.5, 0) -- (2.5, 3.5);
        }$};
        \node at (22.5,2) {$\tikz[scale=0.8, x=0.3cm, y=0.425cm]{
        \draw[knot] (0,0) to[out=90, in=135] (1.5, 3) to[out=-45, in=90] (5,0);
        \draw[knot, overcross] (2.5, 0) -- (2.5, 3.5);
        }$};
        \node at (12.5, -1) {Move 7};
        \draw[<->, very thick] (17.5, -0.25) -- (17.5,-1.75);
        \draw[<->, very thick] (7.5, -0.25) -- (7.5,-1.75);
    \begin{scope}[yshift=-6cm]
	\draw (.5,0) -- (.5,4);
	\draw (4.5,0) -- (4.5,4);
	\draw (8.5,0) -- (8.5,4);
        \draw (12.5,0) -- (12.5,4);
        \draw (16.5,0) -- (16.5,4);
        \draw (20.5,0) -- (20.5,4);
        \draw (24.5,0) -- (24.5,4);
 	\draw[double,double distance=2pt,line cap=rect] (0,0) -- (25,0);
	\draw[dashed, line width=2pt] (0,0) -- (25,0);
	\draw[double,double distance=2pt,line cap=rect] (0,4) -- (25,4);
	\draw[dashed, line width=2pt] (0,4) -- (25,4);
        \node at (2.5,2) {$\tikz[scale=0.8, x=0.3cm, y=0.425cm]{
        \draw[knot] (3,2) to[out=-90, in=90] (5,0);
        \draw[knot, overcross] (0,0) to[out=90, in=-90] (4,2) to[out=90, in=0] (3.5, 3) to[out=180, in=90] (3,2);
        \draw[knot, overcross] (2.5, 0) -- (2.5, 3.5);
        }$};
        \node at (6.5,2) {$\tikz[scale=0.8, x=0.3cm, y=0.425cm]{
        \draw[knot] (0,0) to[out=90, in=225] (3.5, 3) to[out=45, in=90] (5,0);
        \draw[knot, overcross] (2.5, 0) -- (2.5, 3.5);
        }$};
        \node at (10.5,2) {$\tikz[scale=0.8, x=0.3cm, y=0.425cm]{
        \draw[knot] (0,0) to[out=90, in=135] (1.5, 3) to[out=-45, in=90] (5,0);
        \draw[knot, overcross] (2.5, 0) -- (2.5, 3.5);
        }$};
        \node at (14.5,2) {$\tikz[scale=0.8, x=0.3cm, y=0.425cm]{
        \draw[knot] (0,0) to[out=90, in=135] (1.5, 3) to[out=-45, in=90] (5,0);
        \draw[knot, overcross] (2.5, 0) -- (2.5, 3.5);
        }$};
        \node at (18.5,2) {$\tikz[scale=0.8, x=0.3cm, y=0.425cm]{
        \draw[knot] (0,0) to[out=90, in=135] (1.5, 3) to[out=-45, in=90] (5,0);
        \draw[knot, overcross] (2.5, 0) -- (2.5, 3.5);
        }$};
        \node at (22.5,2) {$\tikz[scale=0.8, x=0.3cm, y=0.425cm]{
        \draw[knot] (0,0) to[out=90, in=135] (1.5, 3) to[out=-45, in=90] (5,0);
        \draw[knot, overcross] (2.5, 0) -- (2.5, 3.5);
        }$};
    \end{scope}
 }
\qquad 
\tikz[baseline={([yshift=-.5ex]current bounding box.center)}, scale=.35]{
	\draw (.5,0) -- (.5,4);
	\draw (4.5,0) -- (4.5,4);
	\draw (8.5,0) -- (8.5,4);
        \draw (12.5,0) -- (12.5,4);
 	\draw[double,double distance=2pt,line cap=rect] (0,0) -- (13,0);
	\draw[dashed, line width=2pt] (0,0) -- (13,0);
	\draw[double,double distance=2pt,line cap=rect] (0,4) -- (13,4);
	\draw[dashed, line width=2pt] (0,4) -- (13,4);
        \node at (2.5,2) {$\tikz[scale=0.6]{
            \draw[knot] (1,0) to[out=90, in=-90] (0,1) to[out=90, in=-90] (1,2);
            \draw[knot, overcross] (0,0) to[out=90,in=180] (1,1) to[out=0, in=90] (2,0);
        }$};
        \node at (6.5, 2) {$\tikz[scale=0.6]{
            \draw[knot] (1,0) to[out=90, in=-90] (0,1) to[out=90, in=-90] (1,2);
            \draw[knot, overcross] (0,0) to[out=90,in=180] (1,1) to[out=0, in=90] (2,0);
        }$};
        \node at (10.5,2) {$\tikz[scale=0.6]{
            \draw[knot] (1,0) to[out=90, in=-90] (0,1) to[out=90, in=-90] (1,2);
            \draw[knot, overcross] (0,0) to[out=90,in=180] (1,1) to[out=0, in=90] (2,0);
        }$};
        \node at (6.5, -1) {Move 11};
        \draw[<->, very thick] (11.5, -0.25) -- (11.5,-1.75);
        \draw[<->, very thick] (1.5, -0.25) -- (1.5,-1.75);
    \begin{scope}[yshift=-6cm]
	\draw (.5,0) -- (.5,4);
	\draw (4.5,0) -- (4.5,4);
	\draw (8.5,0) -- (8.5,4);
        \draw (12.5,0) -- (12.5,4);
 	\draw[double,double distance=2pt,line cap=rect] (0,0) -- (13,0);
	\draw[dashed, line width=2pt] (0,0) -- (13,0);
	\draw[double,double distance=2pt,line cap=rect] (0,4) -- (13,4);
	\draw[dashed, line width=2pt] (0,4) -- (13,4);
        \node at (2.5,2) {$\tikz[scale=0.6]{
            \draw[knot] (1,0) to[out=90, in=-90] (0,1) to[out=90, in=-90] (1,2);
            \draw[knot, overcross] (0,0) to[out=90,in=180] (1,1) to[out=0, in=90] (2,0);
        }$};
        \node at (6.5,2) {$\tikz[scale=0.6]{
            \draw[knot] (1,0) to[out=90, in=-90] (2,1) to[out=90, in=-90] (1,2);
            \draw[knot, overcross] (0,0) to[out=90,in=180] (1,1) to[out=0, in=90] (2,0);
        }$};
        \node at (10.5,2) {$\tikz[scale=0.6]{
            \draw[knot] (1,0) to[out=90, in=-90] (0,1) to[out=90, in=-90] (1,2);
            \draw[knot, overcross] (0,0) to[out=90,in=180] (1,1) to[out=0, in=90] (2,0);
        }$};
    \end{scope}
 }
\]
\[
\tikz[baseline={([yshift=-.5ex]current bounding box.center)}, scale=.35]{
	\draw (.5,0) -- (.5,4);
	\draw (4.5,0) -- (4.5,4);
	\draw (8.5,0) -- (8.5,4);
        \draw (12.5,0) -- (12.5,4);
        \draw (16.5,0) -- (16.5,4);
        \draw (20.5,0) -- (20.5,4);
        \draw (24.5,0) -- (24.5,4);
 	\draw[double,double distance=2pt,line cap=rect] (0,0) -- (25,0);
	\draw[dashed, line width=2pt] (0,0) -- (25,0);
	\draw[double,double distance=2pt,line cap=rect] (0,4) -- (25,4);
	\draw[dashed, line width=2pt] (0,4) -- (25,4);
        \node at (2.5,2) {$\tikz[scale=0.8, x=0.3cm, y=0.3cm]{
        \draw[knot] (0,0) -- (0,2) to[out=90, in=210] (2.5,4.5) to[out=30, in=90] (5,3) --(5,0);
        \draw[knot, overcross] (1,0) to[out=90, in=-90] (4,3) -- (4,5);
        \draw[knot, overcross] (4,0) to[out=90, in=-90] (1,3) -- (1,5);
        }$};
        \node at (6.5,2) {$\tikz[scale=0.8, x=0.3cm, y=0.3cm]{
        \draw[knot] (0,0) to[out=90, in=225] (4.5,5) to[out=45, in=90] (5,3) --(5,0);
        \draw[knot, overcross] (1,0) to[out=90, in=-90] (4,3) -- (4,5);
        \draw[knot, overcross] (4,0) to[out=90, in=-90] (1,3) -- (1,5);
        }$};
        \node at (10.5,2) {$\tikz[scale=0.8, x=0.3cm, y=0.3cm]{
        \draw[knot] (0,0) to[out=90, in=180] (2.5, 1) to[out=0, in=225] (4.5,4) to[out=45, in=90] (5,3) --(5,0);
        \draw[knot, overcross] (1,0) -- (1,2) to[out=90, in=-90] (4,5);
        \draw[knot, overcross] (4,0) -- (4,2) to[out=90, in=-90] (1,5);
        }$};
        \node at (14.5,2) {$\tikz[scale=0.8, x=0.3cm, y=0.3cm]{
        \draw[knot] (0,0) to[out=90, in=180] (2.5, 1) to[out=0, in=225] (4.5,2.5) to[out=45, in=90] (5,1.75) --(5,0);
        \draw[knot, overcross] (1,0) -- (1,2) to[out=90, in=-90] (4,5);
        \draw[knot, overcross] (4,0) -- (4,2) to[out=90, in=-90] (1,5);
        }$};
        \node at (18.5,2) {$\tikz[scale=0.8, x=0.3cm, y=0.3cm]{
        \draw[knot] (0,0) to[out=90, in=210] (2.5,1.75) to[out=30, in=90] (5,0);
        \draw[knot, overcross] (1,0) -- (1,2) to[out=90, in=-90] (4,5);
        \draw[knot, overcross] (4,0) -- (4,2) to[out=90, in=-90] (1,5);
        }$};
        \node at (22.5,2) {$\tikz[scale=0.8, x=0.3cm, y=0.3cm]{
        \draw[knot] (0,0) to[out=90, in=150] (2.5,1.75) to[out=-30, in=90] (5,0);
        \draw[knot, overcross] (1,0) -- (1,2) to[out=90, in=-90] (4,5);
        \draw[knot, overcross] (4,0) -- (4,2) to[out=90, in=-90] (1,5);
        }$};
        \node at (12.5, -1) {Move 14};
        \draw[<->, very thick] (17.5, -0.25) -- (17.5,-1.75);
        \draw[<->, very thick] (7.5, -0.25) -- (7.5,-1.75);
    \begin{scope}[yshift=-6cm]
	\draw (.5,0) -- (.5,4);
	\draw (4.5,0) -- (4.5,4);
	\draw (8.5,0) -- (8.5,4);
        \draw (12.5,0) -- (12.5,4);
        \draw (16.5,0) -- (16.5,4);
        \draw (20.5,0) -- (20.5,4);
        \draw (24.5,0) -- (24.5,4);
 	\draw[double,double distance=2pt,line cap=rect] (0,0) -- (25,0);
	\draw[dashed, line width=2pt] (0,0) -- (25,0);
	\draw[double,double distance=2pt,line cap=rect] (0,4) -- (25,4);
	\draw[dashed, line width=2pt] (0,4) -- (25,4);
        \node at (2.5,2) {$\tikz[scale=0.8, x=0.3cm, y=0.3cm]{
        \draw[knot] (0,0) -- (0,2) to[out=90, in=210] (2.5,4.5) to[out=30, in=90] (5,3) --(5,0);
        \draw[knot, overcross] (1,0) to[out=90, in=-90] (4,3) -- (4,5);
        \draw[knot, overcross] (4,0) to[out=90, in=-90] (1,3) -- (1,5);
        }$};
        \node at (6.5,2) {$\tikz[scale=0.8, x=0.3cm, y=0.3cm]{
        \draw[knot] (0,0) -- (0,2) to[out=90, in=150] (2.5,4.5) to[out=-30, in=90] (5,3) --(5,0);
        \draw[knot, overcross] (1,0) to[out=90, in=-90] (4,3) -- (4,5);
        \draw[knot, overcross] (4,0) to[out=90, in=-90] (1,3) -- (1,5);
        }$};
        \node at (10.5,2) {$\tikz[scale=0.8, x=0.3cm, y=0.3cm]{
        \draw[knot] (0,0) -- (0,3) to[out=90, in=135] (0.5,5) to[out=-45, in=90] (5,0);
        \draw[knot, overcross] (1,0) to[out=90, in=-90] (4,3) -- (4,5);
        \draw[knot, overcross] (4,0) to[out=90, in=-90] (1,3) -- (1,5);
        }$};
        \node at (14.5,2) {$\tikz[scale=0.8, x=0.3cm, y=0.3cm]{
        \draw[knot] (5,0) to[out=90, in=0] (2.5, 1) to[out=180, in=-45] (0.5,4) to[out=135, in=90] (0,3) --(0,0);
        \draw[knot, overcross] (1,0) -- (1,2) to[out=90, in=-90] (4,5);
        \draw[knot, overcross] (4,0) -- (4,2) to[out=90, in=-90] (1,5);
        }$};
        \node at (18.5,2) {$\tikz[scale=0.8, x=0.3cm, y=0.3cm]{
        \draw[knot] (5,0) to[out=90, in=0] (2.5, 1) to[out=180, in=-45] (0.5,2.5) to[out=135, in=90] (0,1.75) --(0,0);
        \draw[knot, overcross] (1,0) -- (1,2) to[out=90, in=-90] (4,5);
        \draw[knot, overcross] (4,0) -- (4,2) to[out=90, in=-90] (1,5);
        }$};
        \node at (22.5,2) {$\tikz[scale=0.8, x=0.3cm, y=0.3cm]{
        \draw[knot] (0,0) to[out=90, in=150] (2.5,1.75) to[out=-30, in=90] (5,0);
        \draw[knot, overcross] (1,0) -- (1,2) to[out=90, in=-90] (4,5);
        \draw[knot, overcross] (4,0) -- (4,2) to[out=90, in=-90] (1,5);
        }$};
    \end{scope}
 }
\qquad
\tikz[baseline={([yshift=-.5ex]current bounding box.center)}, scale=.35]{
	\draw (.5,0) -- (.5,4);
	\draw (4.5,0) -- (4.5,4);
	\draw (8.5,0) -- (8.5,4);
        \draw (12.5,0) -- (12.5,4);
 	\draw[double,double distance=2pt,line cap=rect] (0,0) -- (13,0);
	\draw[dashed, line width=2pt] (0,0) -- (13,0);
	\draw[double,double distance=2pt,line cap=rect] (0,4) -- (13,4);
	\draw[dashed, line width=2pt] (0,4) -- (13,4);
        \node at (2.5,2) {$\tikz[scale=0.6, y=0.7cm]{
        \draw[knot] (0,0) -- (0,2) to[out=90, in=180] (0.5, 2.5) to[out=0, in=90] (1,2) -- (1,0);
        }$};
        \node at (6.5, 2) {$\tikz[scale=0.6, y=0.7cm]{
        \draw[knot] (1,0) to[out=90, in=-90] (0,1) to[out=90, in=-90] (1,2);
        \draw[knot, overcross] (0,0) to[out=90, in=-90] (1,1) to[out=90, in=-90] (0,2) to[out=90, in=180] (0.5, 2.5) to[out=0, in=90] (1,2);
        }$};
        \node at (10.5,2) {$\tikz[scale=0.6, y=0.7cm]{
        \draw[knot] (1,0) to[out=90, in=-90] (0,1) -- (0,2) to[out=90, in=180] (0.5, 2.5) to[out=0, in=90] (1,2);
        \draw[knot, overcross] (0,0) to[out=90, in=-90] (1,1) -- (1,2);
        }$};
        \node at (6.5, -1) {Move 26};
        \draw[<->, very thick] (11.5, -0.25) -- (11.5,-1.75);
        \draw[<->, very thick] (1.5, -0.25) -- (1.5,-1.75);
    \begin{scope}[yshift=-6cm]
	\draw (.5,0) -- (.5,4);
	\draw (4.5,0) -- (4.5,4);
	\draw (8.5,0) -- (8.5,4);
        \draw (12.5,0) -- (12.5,4);
 	\draw[double,double distance=2pt,line cap=rect] (0,0) -- (13,0);
	\draw[dashed, line width=2pt] (0,0) -- (13,0);
	\draw[double,double distance=2pt,line cap=rect] (0,4) -- (13,4);
	\draw[dashed, line width=2pt] (0,4) -- (13,4);
        \node at (2.5,2) {$\tikz[scale=0.6, y=0.7cm]{
        \draw[knot] (0,0) -- (0,2) to[out=90, in=180] (0.5, 2.5) to[out=0, in=90] (1,2) -- (1,0);
        }$};
        \node at (6.5,2) {$\tikz[scale=0.6, y=0.7cm]{
        \draw[knot] (1,0) to[out=90, in=-90] (0,2) to[out=90, in=180] (0.5, 2.5) to[out=0, in=90] (1,2);
        \draw[knot, overcross] (0,0) to[out=90, in=-90] (1,2);
        }$};
        \node at (10.5,2) {$\tikz[scale=0.6, y=0.7cm]{
        \draw[knot] (1,0) to[out=90, in=-90] (0,1) -- (0,2) to[out=90, in=180] (0.5, 2.5) to[out=0, in=90] (1,2);
        \draw[knot, overcross] (0,0) to[out=90, in=-90] (1,1) -- (1,2);
        }$};
    \end{scope}
 }
\]

As an aside, we note that one may not even need the machinery of Lemma \ref{lem:qcenter} for these moves. For example, in move 26, $\mathcal{F}(\Sigma_B)$ is a $\mathcal{G}$-graded isomorphism clearly. Moreover, notice that $\mathcal{F}(\Sigma_T)$ consists of a RII move which is necessarily of the second form of (3) in Proposition \ref{prop:Rmoves}. Closing the top of the $\mathcal{G}$-degree shift, we see that the grading shift is canonically isomorphic to the identity shift, and thus $\mathcal{F}(\Sigma_T)$ is also a $\mathcal{G}$-graded isomorphism. This means that we need only appeal to Corollary \ref{cor:cups_caps} for this move.

Moves 8, 9, 10, 23b, and 24 consist of movies in which no frame has a crossing, so each frame change corresponds to either an isotopy or a Morse move. Moves 8 and 9 actually only consist of $T$-moves and their inverses, so we have that $\mathcal{F}(\Sigma_T)=\mathcal{F}(\Sigma_B)$. Move 10 is comprised of isotopies only, so we have that $\mathcal{F}(\Sigma_T)$ and $\mathcal{F}(\Sigma_B)$ are isomorphisms. Moreover $t_0 = t_\infty$ is a cap for move 10, so Corollary \ref{cor:cup_cap} tells us that $\mathcal{F}(\Sigma_T) = u \mathcal{F}(\Sigma_B)$ for some unit $u$ of $R$. The move 23b says that a saddle following a merge or a death following a split is the same (up to unit) as the identity---this is satisfied immediately by the definition of the unified TQFT $\mathcal{F}$. For move 24, let $t_T$ and $t_B$ denote the intermediate frames of $\Sigma_T$ and $\Sigma_B$ and decompose each as $\Sigma_T: t_0 \xrightarrow{P_T} t_T \xrightarrow{Q_T} t_\infty$ and $\Sigma_B: t_0 \xrightarrow{P_B} t_B \xrightarrow{Q_B} t_\infty$. Then $P_T$ and $Q_B$ are both saddles (between the same components), and $Q_T$ and $P_B$ are isotopies. In particular, $Q_T$ and $P_B$ are isotopies on caps, so by Corollary \ref{cor:cup_cap}, $\mathcal{F}(Q_T)$ and $\mathcal{F}(P_B)$ are multiplications by units of $R$, say $u_T$ and $u_B$. Thus $\mathcal{F}(\Sigma_T) = u_T^{-1} u_B \mathcal{F}(\Sigma_B)$ for the unit $u_T^{-1} u_B$ of $R$.

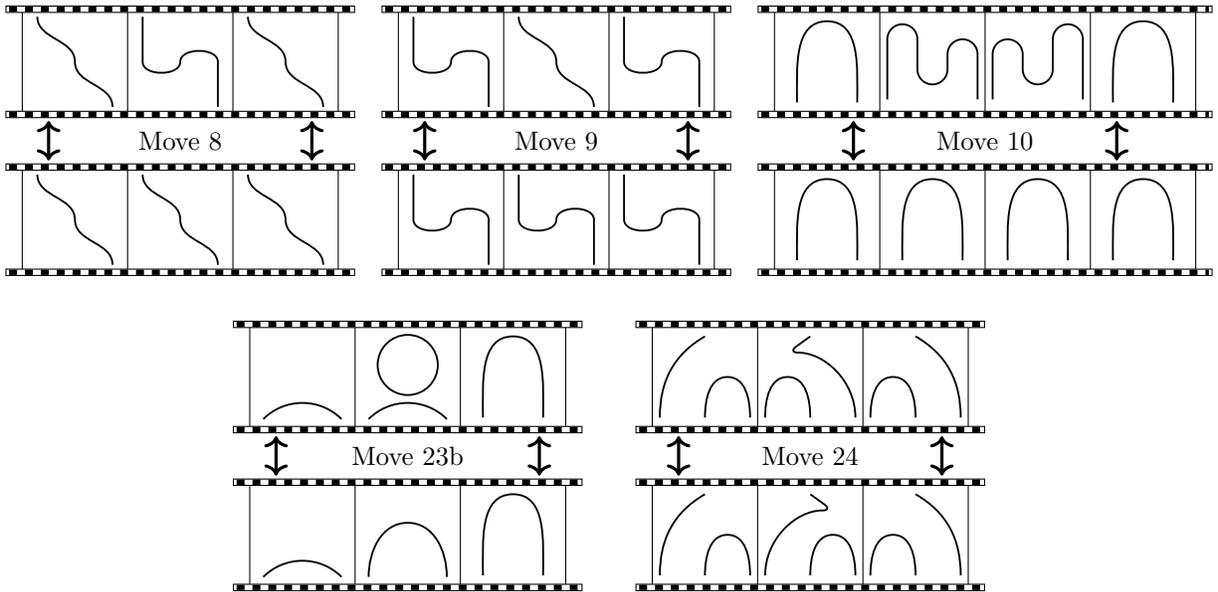
\[
\tikz[baseline={([yshift=-.5ex]current bounding box.center)}, scale=.35]{
	\draw (.5,0) -- (.5,4);
	\draw (4.5,0) -- (4.5,4);
	\draw (8.5,0) -- (8.5,4);
        \draw (12.5,0) -- (12.5,4);
 	\draw[double,double distance=2pt,line cap=rect] (0,0) -- (13,0);
	\draw[dashed, line width=2pt] (0,0) -- (13,0);
	\draw[double,double distance=2pt,line cap=rect] (0,4) -- (13,4);
	\draw[dashed, line width=2pt] (0,4) -- (13,4);
        \node at (2.5,2) {$\tikz[y=1.2cm]{
            \draw[knot] (1,0) to[out=90, in=-90] (0.5,0.5) to[out=90, in=-90] (0,1);
        }$};
        \node at (6.5, 2) {$\tikz[y=1.2cm]{
            \draw[knot] (1,0) -- (1,0.5) to[out=90, in=90] (0.5, 0.5) to[out=-90, in=-90] (0,0.5) -- (0,1);
        }$};
        \node at (10.5,2) {$\tikz[y=1.2cm]{
            \draw[knot] (1,0) to[out=90, in=-90] (0.5,0.5) to[out=90, in=-90] (0,1);
        }$};
        \node at (6.5, -1) {Move 8};
        \draw[<->, very thick] (11.5, -0.25) -- (11.5,-1.75);
        \draw[<->, very thick] (1.5, -0.25) -- (1.5,-1.75);
    \begin{scope}[yshift=-6cm]
	\draw (.5,0) -- (.5,4);
	\draw (4.5,0) -- (4.5,4);
	\draw (8.5,0) -- (8.5,4);
        \draw (12.5,0) -- (12.5,4);
 	\draw[double,double distance=2pt,line cap=rect] (0,0) -- (13,0);
	\draw[dashed, line width=2pt] (0,0) -- (13,0);
	\draw[double,double distance=2pt,line cap=rect] (0,4) -- (13,4);
	\draw[dashed, line width=2pt] (0,4) -- (13,4);
        \node at (2.5,2) {$\tikz[y=1.2cm]{
            \draw[knot] (1,0) to[out=90, in=-90] (0.5,0.5) to[out=90, in=-90] (0,1);
        }$};
        \node at (6.5,2) {$\tikz[y=1.2cm]{
            \draw[knot] (1,0) to[out=90, in=-90] (0.5,0.5) to[out=90, in=-90] (0,1);
        }$};
        \node at (10.5,2) {$\tikz[y=1.2cm]{
            \draw[knot] (1,0) to[out=90, in=-90] (0.5,0.5) to[out=90, in=-90] (0,1);
        }$};
    \end{scope}
 }
\quad
\tikz[baseline={([yshift=-.5ex]current bounding box.center)}, scale=.35]{
	\draw (.5,0) -- (.5,4);
	\draw (4.5,0) -- (4.5,4);
	\draw (8.5,0) -- (8.5,4);
        \draw (12.5,0) -- (12.5,4);
 	\draw[double,double distance=2pt,line cap=rect] (0,0) -- (13,0);
	\draw[dashed, line width=2pt] (0,0) -- (13,0);
	\draw[double,double distance=2pt,line cap=rect] (0,4) -- (13,4);
	\draw[dashed, line width=2pt] (0,4) -- (13,4);
        \node at (2.5,2) {$\tikz[y=1.2cm]{
            \draw[knot] (1,0) -- (1,0.5) to[out=90, in=90] (0.5, 0.5) to[out=-90, in=-90] (0,0.5) -- (0,1);
        }$};
        \node at (6.5, 2) {$\tikz[y=1.2cm]{
            \draw[knot] (1,0) to[out=90, in=-90] (0.5,0.5) to[out=90, in=-90] (0,1);
        }$};
        \node at (10.5,2) {$\tikz[y=1.2cm]{
            \draw[knot] (1,0) -- (1,0.5) to[out=90, in=90] (0.5, 0.5) to[out=-90, in=-90] (0,0.5) -- (0,1);
        }$};
        \node at (6.5, -1) {Move 9};
        \draw[<->, very thick] (11.5, -0.25) -- (11.5,-1.75);
        \draw[<->, very thick] (1.5, -0.25) -- (1.5,-1.75);
    \begin{scope}[yshift=-6cm]
	\draw (.5,0) -- (.5,4);
	\draw (4.5,0) -- (4.5,4);
	\draw (8.5,0) -- (8.5,4);
        \draw (12.5,0) -- (12.5,4);
 	\draw[double,double distance=2pt,line cap=rect] (0,0) -- (13,0);
	\draw[dashed, line width=2pt] (0,0) -- (13,0);
	\draw[double,double distance=2pt,line cap=rect] (0,4) -- (13,4);
	\draw[dashed, line width=2pt] (0,4) -- (13,4);
        \node at (2.5,2) {$\tikz[y=1.2cm]{
            \draw[knot] (1,0) -- (1,0.5) to[out=90, in=90] (0.5, 0.5) to[out=-90, in=-90] (0,0.5) -- (0,1);
        }$};
        \node at (6.5,2) {$\tikz[y=1.2cm]{
            \draw[knot] (1,0) -- (1,0.5) to[out=90, in=90] (0.5, 0.5) to[out=-90, in=-90] (0,0.5) -- (0,1);
        }$};
        \node at (10.5,2) {$\tikz[y=1.2cm]{
            \draw[knot] (1,0) -- (1,0.5) to[out=90, in=90] (0.5, 0.5) to[out=-90, in=-90] (0,0.5) -- (0,1);
        }$};
    \end{scope}
 }
\quad
\tikz[baseline={([yshift=-.5ex]current bounding box.center)}, scale=.35]{
	\draw (.5,0) -- (.5,4);
	\draw (4.5,0) -- (4.5,4);
	\draw (8.5,0) -- (8.5,4);
        \draw (12.5,0) -- (12.5,4);
        \draw (16.5,0) -- (16.5,4);
 	\draw[double,double distance=2pt,line cap=rect] (0,0) -- (17,0);
	\draw[dashed, line width=2pt] (0,0) -- (17,0);
	\draw[double,double distance=2pt,line cap=rect] (0,4) -- (17,4);
	\draw[dashed, line width=2pt] (0,4) -- (17,4);
        \node at (2.5,2) {$\tikz[scale=0.8]{
        \draw[knot] (0,0) to[out=90,in=180] (0.5, 1.35) to[out=0, in=90] (1,0);
        }$};
        \node at (6.5,2) {$\tikz[scale=0.8, y=1cm]{
        \draw[knot] (0,0) -- (0,1) to[out=90, in=180] (0.25, 1.25) to[out=0, in=90] (0.5, 1) -- (0.5, 0.5) to[out=-90, in=180] (0.75, 0.25) to[out=0, in=-90] (1, 0.5) -- (1,0.75) to[out=90, in=180] (1.25, 1) to[out=0, in=90] (1.5, 0.75) -- (1.5,0);
        }$};
        \node at (10.5,2) {$\tikz[scale=0.8, y=1cm]{
        \draw[knot] (0,0) -- (0,0.75) to[out=90, in=180] (0.25, 1) to[out=0, in=90] (0.5, 0.75) -- (0.5, 0.5) to[out=-90, in=180] (0.75, 0.25) to[out=0, in=-90] (1, 0.5) -- (1,1) to[out=90, in=180] (1.25, 1.25) to[out=0, in=90] (1.5, 1) -- (1.5,0);
        }$};
        \node at (14.5,2) {$\tikz[scale=0.8]{
        \draw[knot] (0,0) to[out=90,in=180] (0.5, 1.35) to[out=0, in=90] (1,0);
        }$};
        \node at (8.5, -1) {Move 10};
        \draw[<->, very thick] (13.5, -0.25) -- (13.5,-1.75);
        \draw[<->, very thick] (3.5, -0.25) -- (3.5,-1.75);
    \begin{scope}[yshift=-6cm]
	\draw (.5,0) -- (.5,4);
	\draw (4.5,0) -- (4.5,4);
	\draw (8.5,0) -- (8.5,4);
        \draw (12.5,0) -- (12.5,4);
        \draw (16.5,0) -- (16.5,4);
 	\draw[double,double distance=2pt,line cap=rect] (0,0) -- (17,0);
	\draw[dashed, line width=2pt] (0,0) -- (17,0);
	\draw[double,double distance=2pt,line cap=rect] (0,4) -- (17,4);
	\draw[dashed, line width=2pt] (0,4) -- (17,4);
        \node at (2.5,2) {$\tikz[scale=0.8]{
        \draw[knot] (0,0) to[out=90,in=180] (0.5, 1.35) to[out=0, in=90] (1,0);
        }$};
        \node at (6.5,2) {$\tikz[scale=0.8]{
        \draw[knot] (0,0) to[out=90,in=180] (0.5, 1.35) to[out=0, in=90] (1,0);
        }$};
        \node at (10.5,2) {$\tikz[scale=0.8]{
        \draw[knot] (0,0) to[out=90,in=180] (0.5, 1.35) to[out=0, in=90] (1,0);
        }$};
        \node at (14.5,2) {$\tikz[scale=0.8]{
        \draw[knot] (0,0) to[out=90,in=180] (0.5, 1.35) to[out=0, in=90] (1,0);
        }$};
    \end{scope}
 }
\]
\[
\tikz[baseline={([yshift=-.5ex]current bounding box.center)}, scale=.35]{
	\draw (.5,0) -- (.5,4);
	\draw (4.5,0) -- (4.5,4);
	\draw (8.5,0) -- (8.5,4);
        \draw (12.5,0) -- (12.5,4);
 	\draw[double,double distance=2pt,line cap=rect] (0,0) -- (13,0);
	\draw[dashed, line width=2pt] (0,0) -- (13,0);
	\draw[double,double distance=2pt,line cap=rect] (0,4) -- (13,4);
	\draw[dashed, line width=2pt] (0,4) -- (13,4);
        \node at (2.5,2) {$\tikz[scale=0.8, x=1.3cm, y=0.9cm]{
        \draw[knot] (0,0) to[out=45, in=135] (1,0);
        \draw[knot, color=white] (0.5,1) circle (0.5cm); 
        }$};
        \node at (6.5, 2) {$\tikz[scale=0.8, x=1.3cm, y=0.9cm]{
        \draw[knot] (0,0) to[out=45, in=135] (1,0);
        \draw[knot, fill=white] (0.5,1) circle (0.5cm); 
        }$};
        \node at (10.5,2) {$\tikz[scale=0.8]{
        \draw[knot] (0,0) to[out=90,in=180] (0.5, 1.35) to[out=0, in=90] (1,0);
        }$};
        \node at (6.5, -1) {Move 23b};
        \draw[<->, very thick] (11.5, -0.25) -- (11.5,-1.75);
        \draw[<->, very thick] (1.5, -0.25) -- (1.5,-1.75);
    \begin{scope}[yshift=-6cm]
	\draw (.5,0) -- (.5,4);
	\draw (4.5,0) -- (4.5,4);
	\draw (8.5,0) -- (8.5,4);
        \draw (12.5,0) -- (12.5,4);
 	\draw[double,double distance=2pt,line cap=rect] (0,0) -- (13,0);
	\draw[dashed, line width=2pt] (0,0) -- (13,0);
	\draw[double,double distance=2pt,line cap=rect] (0,4) -- (13,4);
	\draw[dashed, line width=2pt] (0,4) -- (13,4);
        \node at (2.5,2) {$\tikz[scale=0.8, x=1.3cm, y=0.9cm]{
        \draw[knot] (0,0) to[out=45, in=135] (1,0);
        \draw[knot, color=white] (0.5,1) circle (0.5cm); 
        }$};
        \node at (6.5,2) {$\tikz[scale=0.8, x=1.3cm, y=0.9cm]{
        \draw[knot, color=white] (0.5,1) circle (0.5cm); 
        \draw[knot] (0,0) to[out=90,in=180] (0.5, 1) to[out=0, in=90] (1,0);
        }$};
        \node at (10.5,2) {$\tikz[scale=0.8]{
        \draw[knot] (0,0) to[out=90,in=180] (0.5, 1.35) to[out=0, in=90] (1,0);
        }$};
    \end{scope}
 }
\qquad
\tikz[baseline={([yshift=-.5ex]current bounding box.center)}, scale=.35]{
	\draw (.5,0) -- (.5,4);
	\draw (4.5,0) -- (4.5,4);
	\draw (8.5,0) -- (8.5,4);
        \draw (12.5,0) -- (12.5,4);
 	\draw[double,double distance=2pt,line cap=rect] (0,0) -- (13,0);
	\draw[dashed, line width=2pt] (0,0) -- (13,0);
	\draw[double,double distance=2pt,line cap=rect] (0,4) -- (13,4);
	\draw[dashed, line width=2pt] (0,4) -- (13,4);
        \node at (2.5,2) {$\tikz[scale=0.6, y=0.9cm]{
        \draw[knot] (0,0) to[out=90, in=210] (1,2);
        \draw[knot] (1,0) to[out=90, in=180] (1.5, 1) to[out=0, in=90] (2,0);
        }$};
        \node at (6.5, 2) {$\tikz[scale=0.6, y=0.9cm]{
        \draw[knot, rounded corners] (2,0) to[out=90, in=0] (0.5, 1.6) -- (1,2);
        \draw[knot] (0,0) to[out=90, in=180] (0.5, 1) to[out=0, in=90] (1,0);
        }$};
        \node at (10.5,2) {$\tikz[scale=0.6, y=0.9cm]{
        \draw[knot] (2,0) to[out=90, in=-30] (1,2);
        \draw[knot] (0,0) to[out=90, in=180] (0.5, 1) to[out=0, in=90] (1,0);
        }$};
        \node at (6.5, -1) {Move 24};
        \draw[<->, very thick] (11.5, -0.25) -- (11.5,-1.75);
        \draw[<->, very thick] (1.5, -0.25) -- (1.5,-1.75);
    \begin{scope}[yshift=-6cm]
	\draw (.5,0) -- (.5,4);
	\draw (4.5,0) -- (4.5,4);
	\draw (8.5,0) -- (8.5,4);
        \draw (12.5,0) -- (12.5,4);
 	\draw[double,double distance=2pt,line cap=rect] (0,0) -- (13,0);
	\draw[dashed, line width=2pt] (0,0) -- (13,0);
	\draw[double,double distance=2pt,line cap=rect] (0,4) -- (13,4);
	\draw[dashed, line width=2pt] (0,4) -- (13,4);
        \node at (2.5,2) {$\tikz[scale=0.6, y=0.9cm]{
        \draw[knot] (0,0) to[out=90, in=210] (1,2);
        \draw[knot] (1,0) to[out=90, in=180] (1.5, 1) to[out=0, in=90] (2,0);
        }$};
        \node at (6.5,2) {$\tikz[scale=0.6, y=0.9cm]{
        \draw[knot, rounded corners] (0,0) to[out=90, in=180] (1.5, 1.6) -- (1,2);
        \draw[knot] (2,0) to[out=90, in=0] (1.5, 1) to[out=180, in=90] (1,0);
        }$};
        \node at (10.5,2) {$\tikz[scale=0.6, y=0.9cm]{
        \draw[knot] (2,0) to[out=90, in=-30] (1,2);
        \draw[knot] (0,0) to[out=90, in=180] (0.5, 1) to[out=0, in=90] (1,0);
        }$};
    \end{scope}
 }
\]

Moves 15, 16, 17, 18, 19, 20, and 22 are the so-called ``semi-local'' moves and have to do with sliding crossings or turnbacks past one another. They encode larger families of moves since the boxes could stand for a single crossing or a single cup or cap. We group moves 15, 16, 17, 18, 19 and 20, since each consists of only isotopies or Reidemeister moves, so that $\mathcal{F}(\Sigma_B)$ and $\mathcal{F}(\Sigma_T)$ are isomorphisms. For move 15, we have $\mathcal{F}(\Sigma_B) = \mathcal{F}(\Sigma_T)$ since the inverse of an isotopy is the inverse of the isomorphism assigned to the isotopy. Move 16 requires a bit of casework. In $t_0$, denote the boxes by $A$, $B$, and $C$. If all of $A$, $B$, and $C$ are crossings, then the result follows by Corollary \ref{cor:invs}. If all are U-turns, then we can appeal to Corollary \ref{cor:endFT}. If $A$ and $B$, or $B$ and $C$ are both crossings or both U-turns, then we can use either of Corollaries \ref{cor:cups_caps} or \ref{cor:totturns}. Finally, if $A$ and $C$ share the same type (that is, both are crossings or both are U-turns), but $B$ differs, the trick is to notice that the second frame of the isomorphisms has the type of the previous case, so we can apply the same argument. For move 17, if the box is a crossing, we apply Corollary \ref{cor:invs} and if the box is a U-turn, we apply Corollary \ref{cor:cups_caps}. Move 18 follows by the exact same argument. For move 19 the required corollary again depends on the input for the indeterminate box: appeal to either of Corollaries \ref{cor:cups_caps} or \ref{cor:totturns}. Similar arguments can be used for move 20 as well.

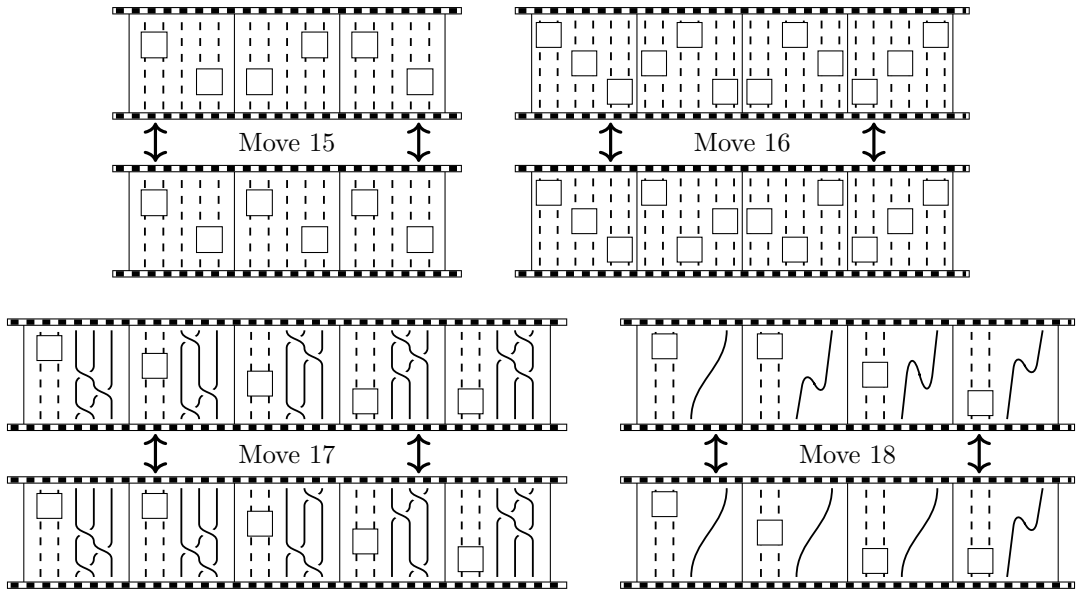
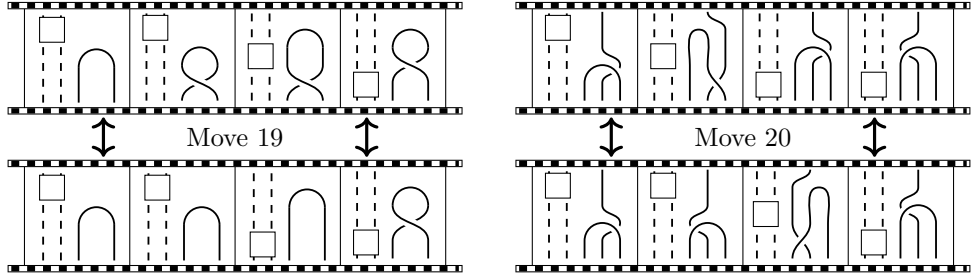
\[
\tikz[baseline={([yshift=-.5ex]current bounding box.center)}, scale=.35]{
	\draw (.5,0) -- (.5,4);
	\draw (4.5,0) -- (4.5,4);
	\draw (8.5,0) -- (8.5,4);
        \draw (12.5,0) -- (12.5,4);
 	\draw[double,double distance=2pt,line cap=rect] (0,0) -- (13,0);
	\draw[dashed, line width=2pt] (0,0) -- (13,0);
	\draw[double,double distance=2pt,line cap=rect] (0,4) -- (13,4);
	\draw[dashed, line width=2pt] (0,4) -- (13,4);
        \node at (2.5,2) {$\tikz[scale=0.35, x=0.7cm, y=0.7cm]{
        \draw[knot, dashed] (0,0) -- (0,5);
        \draw[knot, dashed] (1,0) -- (1,5);
        \draw[knot, dashed] (2,0) -- (2,5);
        \draw[knot, dashed] (3,0) -- (3,5);
        \draw[knot, dashed] (4,0) -- (4,5);
        \draw[fill=white] (-0.2,2.8) rectangle (1.2, 4.2);
        \draw[fill=white] (2.8, 0.8) rectangle (4.2, 2.2);
        }$};
        \node at (6.5, 2) {$\tikz[scale=0.35, x=0.7cm, y=0.7cm]{
        \draw[knot, dashed] (0,0) -- (0,5);
        \draw[knot, dashed] (1,0) -- (1,5);
        \draw[knot, dashed] (2,0) -- (2,5);
        \draw[knot, dashed] (3,0) -- (3,5);
        \draw[knot, dashed] (4,0) -- (4,5);
        \draw[fill=white] (-0.2,0.8) rectangle (1.2, 2.2);
        \draw[fill=white] (2.8, 2.8) rectangle (4.2, 4.2);
        }$};
        \node at (10.5,2) {$\tikz[scale=0.35, x=0.7cm, y=0.7cm]{
        \draw[knot, dashed] (0,0) -- (0,5);
        \draw[knot, dashed] (1,0) -- (1,5);
        \draw[knot, dashed] (2,0) -- (2,5);
        \draw[knot, dashed] (3,0) -- (3,5);
        \draw[knot, dashed] (4,0) -- (4,5);
        \draw[fill=white] (-0.2,2.8) rectangle (1.2, 4.2);
        \draw[fill=white] (2.8, 0.8) rectangle (4.2, 2.2);
        }$};
        \node at (6.5, -1) {Move 15};
        \draw[<->, very thick] (11.5, -0.25) -- (11.5,-1.75);
        \draw[<->, very thick] (1.5, -0.25) -- (1.5,-1.75);
    \begin{scope}[yshift=-6cm]
	\draw (.5,0) -- (.5,4);
	\draw (4.5,0) -- (4.5,4);
	\draw (8.5,0) -- (8.5,4);
        \draw (12.5,0) -- (12.5,4);
 	\draw[double,double distance=2pt,line cap=rect] (0,0) -- (13,0);
	\draw[dashed, line width=2pt] (0,0) -- (13,0);
	\draw[double,double distance=2pt,line cap=rect] (0,4) -- (13,4);
	\draw[dashed, line width=2pt] (0,4) -- (13,4);
        \node at (2.5,2) {$\tikz[scale=0.35, x=0.7cm, y=0.7cm]{
        \draw[knot, dashed] (0,0) -- (0,5);
        \draw[knot, dashed] (1,0) -- (1,5);
        \draw[knot, dashed] (2,0) -- (2,5);
        \draw[knot, dashed] (3,0) -- (3,5);
        \draw[knot, dashed] (4,0) -- (4,5);
        \draw[fill=white] (-0.2,2.8) rectangle (1.2, 4.2);
        \draw[fill=white] (2.8, 0.8) rectangle (4.2, 2.2);
        }$};
        \node at (6.5,2) {$\tikz[scale=0.35, x=0.7cm, y=0.7cm]{
        \draw[knot, dashed] (0,0) -- (0,5);
        \draw[knot, dashed] (1,0) -- (1,5);
        \draw[knot, dashed] (2,0) -- (2,5);
        \draw[knot, dashed] (3,0) -- (3,5);
        \draw[knot, dashed] (4,0) -- (4,5);
        \draw[fill=white] (-0.2,2.8) rectangle (1.2, 4.2);
        \draw[fill=white] (2.8, 0.8) rectangle (4.2, 2.2);
        }$};
        \node at (10.5,2) {$\tikz[scale=0.35, x=0.7cm, y=0.7cm]{
        \draw[knot, dashed] (0,0) -- (0,5);
        \draw[knot, dashed] (1,0) -- (1,5);
        \draw[knot, dashed] (2,0) -- (2,5);
        \draw[knot, dashed] (3,0) -- (3,5);
        \draw[knot, dashed] (4,0) -- (4,5);
        \draw[fill=white] (-0.2,2.8) rectangle (1.2, 4.2);
        \draw[fill=white] (2.8, 0.8) rectangle (4.2, 2.2);
        }$};
    \end{scope}
 }
\qquad
\tikz[baseline={([yshift=-.5ex]current bounding box.center)}, scale=.35]{
	\draw (.5,0) -- (.5,4);
	\draw (4.5,0) -- (4.5,4);
	\draw (8.5,0) -- (8.5,4);
        \draw (12.5,0) -- (12.5,4);
        \draw (16.5,0) -- (16.5,4);
 	\draw[double,double distance=2pt,line cap=rect] (0,0) -- (17,0);
	\draw[dashed, line width=2pt] (0,0) -- (17,0);
	\draw[double,double distance=2pt,line cap=rect] (0,4) -- (17,4);
	\draw[dashed, line width=2pt] (0,4) -- (17,4);
        \node at (2.5,2) {$\tikz[scale=0.35, x=0.675cm, y=0.675cm]{
        \draw[knot, dashed] (0,0) -- (0,5);
        \draw[knot, dashed] (1,0) -- (1,5);
        \draw[knot, dashed] (2,0) -- (2,5);
        \draw[knot, dashed] (3,0) -- (3,5);
        \draw[knot, dashed] (4,0) -- (4,5);
        \draw[knot, dashed] (5,0) -- (5,5);
        \draw[fill=white] (-0.2,3.4) rectangle (1.2,4.8);
        \draw[fill=white] (1.8,1.8) rectangle (3.2,3.2);
        \draw[fill=white] (3.8,0.2) rectangle (5.2,1.6);
        }$};
        \node at (6.5,2) {$\tikz[scale=0.35, x=0.675cm, y=0.675cm]{
        \draw[knot, dashed] (0,0) -- (0,5);
        \draw[knot, dashed] (1,0) -- (1,5);
        \draw[knot, dashed] (2,0) -- (2,5);
        \draw[knot, dashed] (3,0) -- (3,5);
        \draw[knot, dashed] (4,0) -- (4,5);
        \draw[knot, dashed] (5,0) -- (5,5);
        \draw[fill=white] (-0.2,1.8) rectangle (1.2,3.2);
        \draw[fill=white] (1.8,3.4) rectangle (3.2,4.8);
        \draw[fill=white] (3.8,0.2) rectangle (5.2,1.6);
        }$};
        \node at (10.5,2) {$\tikz[scale=0.35, x=0.675cm, y=0.675cm]{
        \draw[knot, dashed] (0,0) -- (0,5);
        \draw[knot, dashed] (1,0) -- (1,5);
        \draw[knot, dashed] (2,0) -- (2,5);
        \draw[knot, dashed] (3,0) -- (3,5);
        \draw[knot, dashed] (4,0) -- (4,5);
        \draw[knot, dashed] (5,0) -- (5,5);
        \draw[fill=white] (-0.2,0.2) rectangle (1.2,1.6);
        \draw[fill=white] (1.8,3.4) rectangle (3.2,4.8);
        \draw[fill=white] (3.8,1.8) rectangle (5.2,3.2);
        }$};
        \node at (14.5,2) {$\tikz[scale=0.35, x=0.675cm, y=0.675cm]{
        \draw[knot, dashed] (0,0) -- (0,5);
        \draw[knot, dashed] (1,0) -- (1,5);
        \draw[knot, dashed] (2,0) -- (2,5);
        \draw[knot, dashed] (3,0) -- (3,5);
        \draw[knot, dashed] (4,0) -- (4,5);
        \draw[knot, dashed] (5,0) -- (5,5);
        \draw[fill=white] (-0.2,0.2) rectangle (1.2,1.6);
        \draw[fill=white] (1.8,1.8) rectangle (3.2,3.2);
        \draw[fill=white] (3.8,3.4) rectangle (5.2,4.8);        }$};
        \node at (8.5, -1) {Move 16};
        \draw[<->, very thick] (13.5, -0.25) -- (13.5,-1.75);
        \draw[<->, very thick] (3.5, -0.25) -- (3.5,-1.75);
    \begin{scope}[yshift=-6cm]
	\draw (.5,0) -- (.5,4);
	\draw (4.5,0) -- (4.5,4);
	\draw (8.5,0) -- (8.5,4);
        \draw (12.5,0) -- (12.5,4);
        \draw (16.5,0) -- (16.5,4);
 	\draw[double,double distance=2pt,line cap=rect] (0,0) -- (17,0);
	\draw[dashed, line width=2pt] (0,0) -- (17,0);
	\draw[double,double distance=2pt,line cap=rect] (0,4) -- (17,4);
	\draw[dashed, line width=2pt] (0,4) -- (17,4);
        \node at (2.5,2) {$\tikz[scale=0.35, x=0.675cm, y=0.675cm]{
        \draw[knot, dashed] (0,0) -- (0,5);
        \draw[knot, dashed] (1,0) -- (1,5);
        \draw[knot, dashed] (2,0) -- (2,5);
        \draw[knot, dashed] (3,0) -- (3,5);
        \draw[knot, dashed] (4,0) -- (4,5);
        \draw[knot, dashed] (5,0) -- (5,5);
        \draw[fill=white] (-0.2,3.4) rectangle (1.2,4.8);
        \draw[fill=white] (1.8,1.8) rectangle (3.2,3.2);
        \draw[fill=white] (3.8,0.2) rectangle (5.2,1.6);
        }$};
        \node at (6.5,2) {$\tikz[scale=0.35, x=0.675cm, y=0.675cm]{
        \draw[knot, dashed] (0,0) -- (0,5);
        \draw[knot, dashed] (1,0) -- (1,5);
        \draw[knot, dashed] (2,0) -- (2,5);
        \draw[knot, dashed] (3,0) -- (3,5);
        \draw[knot, dashed] (4,0) -- (4,5);
        \draw[knot, dashed] (5,0) -- (5,5);
        \draw[fill=white] (-0.2,3.4) rectangle (1.2,4.8);
        \draw[fill=white] (1.8,0.2) rectangle (3.2,1.6);
        \draw[fill=white] (3.8,1.8) rectangle (5.2,3.2);
        }$};
        \node at (10.5,2) {$\tikz[scale=0.35, x=0.675cm, y=0.675cm]{
        \draw[knot, dashed] (0,0) -- (0,5);
        \draw[knot, dashed] (1,0) -- (1,5);
        \draw[knot, dashed] (2,0) -- (2,5);
        \draw[knot, dashed] (3,0) -- (3,5);
        \draw[knot, dashed] (4,0) -- (4,5);
        \draw[knot, dashed] (5,0) -- (5,5);
        \draw[fill=white] (-0.2,1.8) rectangle (1.2,3.2);
        \draw[fill=white] (1.8,0.2) rectangle (3.2,1.6);
        \draw[fill=white] (3.8,3.4) rectangle (5.2,4.8);
        }$};
        \node at (14.5,2) {$\tikz[scale=0.35, x=0.675cm, y=0.675cm]{
        \draw[knot, dashed] (0,0) -- (0,5);
        \draw[knot, dashed] (1,0) -- (1,5);
        \draw[knot, dashed] (2,0) -- (2,5);
        \draw[knot, dashed] (3,0) -- (3,5);
        \draw[knot, dashed] (4,0) -- (4,5);
        \draw[knot, dashed] (5,0) -- (5,5);
        \draw[fill=white] (-0.2,0.2) rectangle (1.2,1.6);
        \draw[fill=white] (1.8,1.8) rectangle (3.2,3.2);
        \draw[fill=white] (3.8,3.4) rectangle (5.2,4.8);
        }$};
    \end{scope}
 }
\]
\[
\tikz[baseline={([yshift=-.5ex]current bounding box.center)}, scale=.35]{
	\draw (.5,0) -- (.5,4);
	\draw (4.5,0) -- (4.5,4);
	\draw (8.5,0) -- (8.5,4);
        \draw (12.5,0) -- (12.5,4);
        \draw (16.5,0) -- (16.5,4);
        \draw (20.5,0) -- (20.5,4);
 	\draw[double,double distance=2pt,line cap=rect] (0,0) -- (21,0);
	\draw[dashed, line width=2pt] (0,0) -- (21,0);
	\draw[double,double distance=2pt,line cap=rect] (0,4) -- (21,4);
	\draw[dashed, line width=2pt] (0,4) -- (21,4);
        \node at (2.5,2) {$\tikz[scale=0.35, x=0.675cm, y=0.675cm]{
        \draw[knot, dashed] (0,0) -- (0,5);
        \draw[knot, dashed] (1,0) -- (1,5);
        \draw[fill=white] (-0.2, 3.3) rectangle (1.2, 4.7);
        \draw[knot] (2,0) to[out=90, in=-90] (3,1) to[out=90, in=-90] (4,2) -- (4,5);
        \draw[knot, overcross] (3,0) to[out=90,in=-90] (2,1) -- (2,2) to[out=90,in=-90] (3,3) -- (3,5);
        \draw[knot, overcross] (4,0) -- (4,1) to[out=90,in=-90] (3,2) to[out=90,in=-90] (2,3) -- (2,5);
        }$};
        \node at (6.5,2) {$\tikz[scale=0.35, x=0.675cm, y=0.675cm]{
        \draw[knot, dashed] (0,0) -- (0,5);
        \draw[knot, dashed] (1,0) -- (1,5);
        \draw[fill=white] (-0.2, 2.3) rectangle (1.2, 3.7);
        \draw[knot] (2,0) to[out=90, in=-90] (3,1) to[out=90, in=-90] (4,2) -- (4,5);
        \draw[knot, overcross] (3,0) to[out=90,in=-90] (2,1) -- (2,4) to[out=90,in=-90] (3,5);
        \draw[knot, overcross] (4,0) -- (4,1) to[out=90,in=-90] (3,2) -- (3,4) to[out=90,in=-90] (2,5);
        }$};
        \node at (10.5,2) {$\tikz[scale=0.35, x=0.675cm, y=0.675cm]{
        \draw[knot, dashed] (0,0) -- (0,5);
        \draw[knot, dashed] (1,0) -- (1,5);
        \draw[fill=white] (-0.2, 1.3) rectangle (1.2, 2.7);
        \draw[knot] (2,0) to[out=90, in=-90] (3,1) -- (3,3) to[out=90, in=-90] (4,4) -- (4,5);
        \draw[knot, overcross] (3,0) to[out=90,in=-90] (2,1) -- (2,4) to[out=90,in=-90] (3,5);
        \draw[knot, overcross] (4,0) -- (4,3) to[out=90,in=-90](3,4) to[out=90,in=-90] (2,5);
        }$};
        \node at (14.5,2) {$\tikz[scale=0.35, x=0.675cm, y=0.675cm]{
        \draw[knot, dashed] (0,0) -- (0,5);
        \draw[knot, dashed] (1,0) -- (1,5);
        \draw[fill=white] (-0.2, 0.3) rectangle (1.2, 1.7);
        \draw[knot] (2,0) -- (2,2) to[out=90, in=-90] (3,3) to[out=90, in=-90] (4,4) -- (4,5);
        \draw[knot, overcross] (3,0) -- (3,2) to[out=90,in=-90] (2,3) -- (2,4) to[out=90,in=-90] (3,5);
        \draw[knot, overcross] (4,0) -- (4,3) to[out=90,in=-90](3,4) to[out=90,in=-90] (2,5);
        }$};
        \node at (18.5,2) {$\tikz[scale=0.35, x=0.675cm, y=0.675cm]{
        \draw[knot, dashed] (0,0) -- (0,5);
        \draw[knot, dashed] (1,0) -- (1,5);
        \draw[fill=white] (-0.2, 0.3) rectangle (1.2, 1.7);
        \draw[knot] (2,0) -- (2,3) to[out=90,in=-90] (3,4) to[out=90,in=-90] (4,5);
        \draw[knot, overcross] (3,0) -- (3,2) to[out=90,in=-90] (4,3) -- (4,4) to[out=90,in=-90] (3,5);
        \draw[knot, overcross] (4,0) -- (4,2) to[out=90,in=-90] (3,3) to[out=90,in=-90] (2,4) -- (2,5);
        }$};
        \node at (10.5, -1) {Move 17};
        \draw[<->, very thick] (15.5, -0.25) -- (15.5,-1.75);
        \draw[<->, very thick] (5.5, -0.25) -- (5.5,-1.75);
    \begin{scope}[yshift=-6cm]
	\draw (.5,0) -- (.5,4);
	\draw (4.5,0) -- (4.5,4);
	\draw (8.5,0) -- (8.5,4);
        \draw (12.5,0) -- (12.5,4);
        \draw (16.5,0) -- (16.5,4);
        \draw (20.5,0) -- (20.5,4);
 	\draw[double,double distance=2pt,line cap=rect] (0,0) -- (21,0);
	\draw[dashed, line width=2pt] (0,0) -- (21,0);
	\draw[double,double distance=2pt,line cap=rect] (0,4) -- (21,4);
	\draw[dashed, line width=2pt] (0,4) -- (21,4);
        \node at (2.5,2) {$\tikz[scale=0.35, x=0.675cm, y=0.675cm]{
        \draw[knot, dashed] (0,0) -- (0,5);
        \draw[knot, dashed] (1,0) -- (1,5);
        \draw[fill=white] (-0.2, 3.3) rectangle (1.2, 4.7);
        \draw[knot] (2,0) to[out=90, in=-90] (3,1) to[out=90, in=-90] (4,2) -- (4,5);
        \draw[knot, overcross] (3,0) to[out=90,in=-90] (2,1) -- (2,2) to[out=90,in=-90] (3,3) -- (3,5);
        \draw[knot, overcross] (4,0) -- (4,1) to[out=90,in=-90] (3,2) to[out=90,in=-90] (2,3) -- (2,5);
        }$};
        \node at (6.5,2) {$\tikz[scale=0.35, x=0.675cm, y=0.675cm]{
        \draw[knot, dashed] (0,0) -- (0,5);
        \draw[knot, dashed] (1,0) -- (1,5);
        \draw[fill=white] (-0.2, 3.3) rectangle (1.2, 4.7);
        \draw[knot] (2,0) -- (2,1) to[out=90,in=-90] (3,2) to[out=90,in=-90] (4,3) -- (4,5);
        \draw[knot, overcross] (3,0) to[out=90,in=-90] (4,1) -- (4,2) to[out=90,in=-90] (3,3) -- (3,5);
        \draw[knot, overcross] (4,0) to[out=90,in=-90] (3,1) to[out=90,in=-90] (2,2) -- (2,5);
        }$};
        \node at (10.5,2) {$\tikz[scale=0.35, x=0.675cm, y=0.675cm]{
        \draw[knot, dashed] (0,0) -- (0,5);
        \draw[knot, dashed] (1,0) -- (1,5);
        \draw[fill=white] (-0.2, 2.3) rectangle (1.2, 3.7);
        \draw[knot] (2,0) -- (2,1) to[out=90,in=-90] (3,2) -- (3,4) to[out=90,in=-90] (4,5);
        \draw[knot, overcross] (3,0) to[out=90,in=-90] (4,1) -- (4,4) to[out=90,in=-90] (3,5);
        \draw[knot, overcross] (4,0) to[out=90,in=-90] (3,1) to[out=90,in=-90] (2,2) -- (2,5);
        }$};
        \node at (14.5,2) {$\tikz[scale=0.35, x=0.675cm, y=0.675cm]{
        \draw[knot, dashed] (0,0) -- (0,5);
        \draw[knot, dashed] (1,0) -- (1,5);
        \draw[fill=white] (-0.2, 1.3) rectangle (1.2, 2.7);
        \draw[knot] (2,0) -- (2,3) to[out=90,in=-90] (3,4) to[out=90,in=-90] (4,5);
        \draw[knot, overcross] (3,0) to[out=90,in=-90] (4,1) -- (4,4) to[out=90,in=-90] (3,5);
        \draw[knot, overcross] (4,0) to[out=90,in=-90] (3,1) -- (3,3) to[out=90,in=-90] (2,4) -- (2,5);
        }$};
        \node at (18.5,2) {$\tikz[scale=0.35, x=0.675cm, y=0.675cm]{
        \draw[knot, dashed] (0,0) -- (0,5);
        \draw[knot, dashed] (1,0) -- (1,5);
        \draw[fill=white] (-0.2, 0.3) rectangle (1.2, 1.7);
        \draw[knot] (2,0) -- (2,3) to[out=90,in=-90] (3,4) to[out=90,in=-90] (4,5);
        \draw[knot, overcross] (3,0) -- (3,2) to[out=90,in=-90] (4,3) -- (4,4) to[out=90,in=-90] (3,5);
        \draw[knot, overcross] (4,0) -- (4,2) to[out=90,in=-90] (3,3) to[out=90,in=-90] (2,4) -- (2,5);
        }$};
    \end{scope}
 }
\qquad
\tikz[baseline={([yshift=-.5ex]current bounding box.center)}, scale=.35]{
	\draw (.5,0) -- (.5,4);
	\draw (4.5,0) -- (4.5,4);
	\draw (8.5,0) -- (8.5,4);
        \draw (12.5,0) -- (12.5,4);
        \draw (16.5,0) -- (16.5,4);
 	\draw[double,double distance=2pt,line cap=rect] (0,0) -- (17,0);
	\draw[dashed, line width=2pt] (0,0) -- (17,0);
	\draw[double,double distance=2pt,line cap=rect] (0,4) -- (17,4);
	\draw[dashed, line width=2pt] (0,4) -- (17,4);
        \node at (2.5,2) {$\tikz[scale=0.35, x=0.675cm, y=0.675cm]{
        \draw[knot, dashed] (0,0) -- (0,5);
        \draw[knot, dashed] (1,0) -- (1,5);
        \draw[fill=white] (-0.2, 3.4) rectangle (1.2,4.8);
        \draw[knot] (2,0) to[out=90, in=-90] (4,5);
        }$};
        \node at (6.5,2) {$\tikz[scale=0.35, x=0.675cm, y=0.675cm]{
        \draw[knot, dashed] (0,0) -- (0,5);
        \draw[knot, dashed] (1,0) -- (1,5);
        \draw[fill=white] (-0.2, 3.4) rectangle (1.2,4.8);
        \draw[knot, rounded corners=2.75mm] (2,0) -- (2.5, 3) -- (3.5,1) -- (4,5);
        }$};
        \node at (10.5,2) {$\tikz[scale=0.35, x=0.675cm, y=0.675cm]{
        \draw[knot, dashed] (0,0) -- (0,5);
        \draw[knot, dashed] (1,0) -- (1,5);
        \draw[fill=white] (-0.2, 1.8) rectangle (1.2,3.2);
        \draw[knot, rounded corners=5mm] (2,0) -- (2.5, 4.5) -- (3.5,0.5) -- (4,5);
        }$};
        \node at (14.5,2) {$\tikz[scale=0.35, x=0.675cm, y=0.675cm]{
        \draw[knot, dashed] (0,0) -- (0,5);
        \draw[knot, dashed] (1,0) -- (1,5);
        \draw[fill=white] (-0.2, 0.2) rectangle (1.2,1.6);
        \draw[knot, rounded corners=2.75mm] (2,0) -- (2.5, 4) -- (3.5,2) -- (4,5);
        }$};
        \node at (8.5, -1) {Move 18};
        \draw[<->, very thick] (13.5, -0.25) -- (13.5,-1.75);
        \draw[<->, very thick] (3.5, -0.25) -- (3.5,-1.75);
    \begin{scope}[yshift=-6cm]
	\draw (.5,0) -- (.5,4);
	\draw (4.5,0) -- (4.5,4);
	\draw (8.5,0) -- (8.5,4);
        \draw (12.5,0) -- (12.5,4);
        \draw (16.5,0) -- (16.5,4);
 	\draw[double,double distance=2pt,line cap=rect] (0,0) -- (17,0);
	\draw[dashed, line width=2pt] (0,0) -- (17,0);
	\draw[double,double distance=2pt,line cap=rect] (0,4) -- (17,4);
	\draw[dashed, line width=2pt] (0,4) -- (17,4);
        \node at (2.5,2) {$\tikz[scale=0.35, x=0.675cm, y=0.675cm]{
        \draw[knot, dashed] (0,0) -- (0,5);
        \draw[knot, dashed] (1,0) -- (1,5);
        \draw[fill=white] (-0.2, 3.4) rectangle (1.2,4.8);
        \draw[knot] (2,0) to[out=90, in=-90] (4,5);
        }$};
        \node at (6.5,2) {$\tikz[scale=0.35, x=0.675cm, y=0.675cm]{
        \draw[knot, dashed] (0,0) -- (0,5);
        \draw[knot, dashed] (1,0) -- (1,5);
        \draw[fill=white] (-0.2, 1.8) rectangle (1.2,3.2);
        \draw[knot] (2,0) to[out=90, in=-90] (4,5);
        }$};
        \node at (10.5,2) {$\tikz[scale=0.35, x=0.675cm, y=0.675cm]{
        \draw[knot, dashed] (0,0) -- (0,5);
        \draw[knot, dashed] (1,0) -- (1,5);
        \draw[fill=white] (-0.2, 0.2) rectangle (1.2,1.6);
        \draw[knot] (2,0) to[out=90, in=-90] (4,5);
        }$};
        \node at (14.5,2) {$\tikz[scale=0.35, x=0.675cm, y=0.675cm]{
        \draw[knot, dashed] (0,0) -- (0,5);
        \draw[knot, dashed] (1,0) -- (1,5);
        \draw[fill=white] (-0.2, 0.2) rectangle (1.2,1.6);
        \draw[knot, rounded corners=2.75mm] (2,0) -- (2.5, 4) -- (3.5,2) -- (4,5);
        }$};
    \end{scope}
 }
\]
\[
\tikz[baseline={([yshift=-.5ex]current bounding box.center)}, scale=.35]{
	\draw (.5,0) -- (.5,4);
	\draw (4.5,0) -- (4.5,4);
	\draw (8.5,0) -- (8.5,4);
        \draw (12.5,0) -- (12.5,4);
        \draw (16.5,0) -- (16.5,4);
 	\draw[double,double distance=2pt,line cap=rect] (0,0) -- (17,0);
	\draw[dashed, line width=2pt] (0,0) -- (17,0);
	\draw[double,double distance=2pt,line cap=rect] (0,4) -- (17,4);
	\draw[dashed, line width=2pt] (0,4) -- (17,4);
        \node at (2.5,2) {$\tikz[scale=0.35, x=0.675cm, y=0.675cm]{
        \draw[knot, dashed] (0,0) -- (0,5);
        \draw[knot, dashed] (1,0) -- (1,5);
        \draw[fill=white] (-0.2, 3.4) rectangle (1.2,4.8);
        \draw[knot] (2,0) -- (2,2);
        \draw[knot] (4,0) -- (4,2);
        \draw[knot] (2,2) to[out=90, in=180] (3,3) to[out=0, in=90] (4,2);
        }$};
        \node at (6.5,2) {$\tikz[scale=0.35, x=0.675cm, y=0.675cm]{
        \draw[knot, dashed] (0,0) -- (0,5);
        \draw[knot, dashed] (1,0) -- (1,5);
        \draw[fill=white] (-0.2, 3.4) rectangle (1.2,4.8);
        \draw[knot] (2,0) to[out=90, in=-90] (4,2);
        \draw[knot, overcross] (4,0) to[out=90, in=-90] (2,2);
        \draw[knot] (2,2) to[out=90, in=180] (3,3) to[out=0, in=90] (4,2);
        }$};
        \node at (10.5,2) {$\tikz[scale=0.35, x=0.675cm, y=0.675cm]{
        \draw[knot, dashed] (0,0) -- (0,5);
        \draw[knot, dashed] (1,0) -- (1,5);
        \draw[fill=white] (-0.2, 1.8) rectangle (1.2,3.2);
        \draw[knot] (2,0) to[out=90, in=-90] (4,2) -- (4,3);
        \draw[knot, overcross] (4,0) to[out=90, in=-90] (2,2) -- (2,3);
        \draw[knot] (2,3) to[out=90, in=180] (3,4) to[out=0, in=90] (4,3);
        }$};
        \node at (14.5,2) {$\tikz[scale=0.35, x=0.675cm, y=0.675cm]{
        \draw[knot, dashed] (0,0) -- (0,5);
        \draw[knot, dashed] (1,0) -- (1,5);
        \draw[fill=white] (-0.2, 0.2) rectangle (1.2,1.6);
        \draw[knot] (2,0) -- (2,1) to[out=90, in=-90] (4,3);
        \draw[knot, overcross] (4,0) -- (4,1) to[out=90, in=-90] (2,3);
        \draw[knot] (2,3) to[out=90, in=180] (3,4) to[out=0, in=90] (4,3);
        }$};
        \node at (8.5, -1) {Move 19};
        \draw[<->, very thick] (13.5, -0.25) -- (13.5,-1.75);
        \draw[<->, very thick] (3.5, -0.25) -- (3.5,-1.75);
    \begin{scope}[yshift=-6cm]
	\draw (.5,0) -- (.5,4);
	\draw (4.5,0) -- (4.5,4);
	\draw (8.5,0) -- (8.5,4);
        \draw (12.5,0) -- (12.5,4);
        \draw (16.5,0) -- (16.5,4);
 	\draw[double,double distance=2pt,line cap=rect] (0,0) -- (17,0);
	\draw[dashed, line width=2pt] (0,0) -- (17,0);
	\draw[double,double distance=2pt,line cap=rect] (0,4) -- (17,4);
	\draw[dashed, line width=2pt] (0,4) -- (17,4);
        \node at (2.5,2) {$\tikz[scale=0.35, x=0.675cm, y=0.675cm]{
        \draw[knot, dashed] (0,0) -- (0,5);
        \draw[knot, dashed] (1,0) -- (1,5);
        \draw[fill=white] (-0.2, 3.4) rectangle (1.2,4.8);
        \draw[knot] (2,0) -- (2,2);
        \draw[knot] (4,0) -- (4,2);
        \draw[knot] (2,2) to[out=90, in=180] (3,3) to[out=0, in=90] (4,2);
        }$};
        \node at (6.5,2) {$\tikz[scale=0.35, x=0.675cm, y=0.675cm]{
        \draw[knot, dashed] (0,0) -- (0,5);
        \draw[knot, dashed] (1,0) -- (1,5);
        \draw[fill=white] (-0.2, 3.4) rectangle (1.2,4.8);
        \draw[knot] (2,0) -- (2,2);
        \draw[knot] (4,0) -- (4,2);
        \draw[knot] (2,2) to[out=90, in=180] (3,3) to[out=0, in=90] (4,2);
        }$};
        \node at (10.5,2) {$\tikz[scale=0.35, x=0.675cm, y=0.675cm]{
        \draw[knot, dashed] (0,0) -- (0,5);
        \draw[knot, dashed] (1,0) -- (1,5);
        \draw[fill=white] (-0.2, 0.2) rectangle (1.2,1.6);
        \draw[knot] (2,0) -- (2,3);
        \draw[knot] (4,0) -- (4,3);
        \draw[knot] (2,3) to[out=90, in=180] (3,4) to[out=0, in=90] (4,3);
        }$};
        \node at (14.5,2) {$\tikz[scale=0.35, x=0.675cm, y=0.675cm]{
        \draw[knot, dashed] (0,0) -- (0,5);
        \draw[knot, dashed] (1,0) -- (1,5);
        \draw[fill=white] (-0.2, 0.2) rectangle (1.2,1.6);
        \draw[knot] (2,0) -- (2,1) to[out=90, in=-90] (4,3);
        \draw[knot, overcross] (4,0) -- (4,1) to[out=90, in=-90] (2,3);
        \draw[knot] (2,3) to[out=90, in=180] (3,4) to[out=0, in=90] (4,3);
        }$};
    \end{scope}
 }
\qquad
\tikz[baseline={([yshift=-.5ex]current bounding box.center)}, scale=.35]{
	\draw (.5,0) -- (.5,4);
	\draw (4.5,0) -- (4.5,4);
	\draw (8.5,0) -- (8.5,4);
        \draw (12.5,0) -- (12.5,4);
        \draw (16.5,0) -- (16.5,4);
 	\draw[double,double distance=2pt,line cap=rect] (0,0) -- (17,0);
	\draw[dashed, line width=2pt] (0,0) -- (17,0);
	\draw[double,double distance=2pt,line cap=rect] (0,4) -- (17,4);
	\draw[dashed, line width=2pt] (0,4) -- (17,4);
        \node at (2.5,2) {$\tikz[scale=0.35, x=0.675cm, y=0.675cm]{
        \draw[knot, dashed] (0,0) -- (0,5);
        \draw[knot, dashed] (1,0) -- (1,5);
        \draw[fill=white] (-0.2, 3.4) rectangle (1.2,4.8);
        \draw[knot] (3,0) -- (3,1) to[out=90, in=-90] (4,2) to[out=90, in=-90] (3,3) -- (3,5);
        \draw[knot, overcross] (2,0) -- (2,1) to[out=90, in=180] (3,2) to[out=0, in=90] (4,1) -- (4,0);
        }$};
        \node at (6.5,2) {$\tikz[scale=0.35, x=0.675cm, y=0.675cm]{
        \draw[knot, dashed] (0,0) -- (0,5);
        \draw[knot, dashed] (1,0) -- (1,5);
        \draw[fill=white] (-0.2, 1.8) rectangle (1.2,3.2);
        \draw[knot] (3,0) to[out=90, in=-90] (4,2) -- (4,4) to[out=90, in=-90] (3,5);
        \draw[knot, overcross] (2,0) -- (2,2) to[out=90, in=180] (2.5,4) to[out=0, in=90] (3,2) to[out=-90, in=90] (4,0);
        }$};
        \node at (10.5,2) {$\tikz[scale=0.35, x=0.675cm, y=0.675cm]{
        \draw[knot, dashed] (0,0) -- (0,5);
        \draw[knot, dashed] (1,0) -- (1,5);
        \draw[fill=white] (-0.2, 0.2) rectangle (1.2,1.6);
        \draw[knot] (3,0) -- (3,2) to[out=90, in=-90] (4,3) to[out=90, in=-90] (3,4) -- (3,5);
        \draw[knot, overcross] (2,0) -- (2,2) to[out=90, in=180] (3,3) to[out=0, in=90] (4,2) -- (4,0);
        }$};
        \node at (14.5,2) {$\tikz[scale=0.35, x=0.675cm, y=0.675cm]{
        \draw[knot, dashed] (0,0) -- (0,5);
        \draw[knot, dashed] (1,0) -- (1,5);
        \draw[fill=white] (-0.2, 0.2) rectangle (1.2,1.6);
        \draw[knot] (3,0) -- (3,2) to[out=90, in=-90] (2,3) to[out=90, in=-90] (3,4) -- (3,5);
        \draw[knot, overcross] (2,0) -- (2,2) to[out=90, in=180] (3,3) to[out=0, in=90] (4,2) -- (4,0);
        }$};
        \node at (8.5, -1) {Move 20};
        \draw[<->, very thick] (13.5, -0.25) -- (13.5,-1.75);
        \draw[<->, very thick] (3.5, -0.25) -- (3.5,-1.75);
    \begin{scope}[yshift=-6cm]
	\draw (.5,0) -- (.5,4);
	\draw (4.5,0) -- (4.5,4);
	\draw (8.5,0) -- (8.5,4);
        \draw (12.5,0) -- (12.5,4);
        \draw (16.5,0) -- (16.5,4);
 	\draw[double,double distance=2pt,line cap=rect] (0,0) -- (17,0);
	\draw[dashed, line width=2pt] (0,0) -- (17,0);
	\draw[double,double distance=2pt,line cap=rect] (0,4) -- (17,4);
	\draw[dashed, line width=2pt] (0,4) -- (17,4);
        \node at (2.5,2) {$\tikz[scale=0.35, x=0.675cm, y=0.675cm]{
        \draw[knot, dashed] (0,0) -- (0,5);
        \draw[knot, dashed] (1,0) -- (1,5);
        \draw[fill=white] (-0.2, 3.4) rectangle (1.2,4.8);
        \draw[knot] (3,0) -- (3,1) to[out=90, in=-90] (4,2) to[out=90, in=-90] (3,3) -- (3,5);
        \draw[knot, overcross] (2,0) -- (2,1) to[out=90, in=180] (3,2) to[out=0, in=90] (4,1) -- (4,0);
        }$};
        \node at (6.5,2) {$\tikz[scale=0.35, x=0.675cm, y=0.675cm]{
        \draw[knot, dashed] (0,0) -- (0,5);
        \draw[knot, dashed] (1,0) -- (1,5);
        \draw[fill=white] (-0.2, 3.4) rectangle (1.2,4.8);
        \draw[knot] (3,0) -- (3,1) to[out=90, in=-90] (2,2) to[out=90, in=-90] (3,3) -- (3,5);
        \draw[knot, overcross] (2,0) -- (2,1) to[out=90, in=180] (3,2) to[out=0, in=90] (4,1) -- (4,0);
        }$};
        \node at (10.5,2) {$\tikz[scale=0.35, x=0.675cm, y=0.675cm]{
        \draw[knot, dashed] (0,0) -- (0,5);
        \draw[knot, dashed] (1,0) -- (1,5);
        \draw[fill=white] (-0.2, 1.8) rectangle (1.2,3.2);
        \draw[knot] (3,0) to[out=90, in=-90] (2,2) -- (2,4) to[out=90, in=-90] (3,5);
        \draw[knot, overcross] (4,0) -- (4,2) to[out=90, in=0] (3.5,4) to[out=180, in=90] (3,2) to[out=-90, in=90] (2,0);
        }$};
        \node at (14.5,2) {$\tikz[scale=0.35, x=0.675cm, y=0.675cm]{
        \draw[knot, dashed] (0,0) -- (0,5);
        \draw[knot, dashed] (1,0) -- (1,5);
        \draw[fill=white] (-0.2, 0.2) rectangle (1.2,1.6);
        \draw[knot] (3,0) -- (3,2) to[out=90, in=-90] (2,3) to[out=90, in=-90] (3,4) -- (3,5);
        \draw[knot, overcross] (2,0) -- (2,2) to[out=90, in=180] (3,3) to[out=0, in=90] (4,2) -- (4,0);
        }$};
    \end{scope}
 }
\]

Moves 21, 29, and 30 all involve frames $t_0$ or $t_\infty$ which have a cap on the bottom and a cup on the top. Move 21 is relatively simple using familiar arguments: each movie consists entirely of isotopies or Reidemeister moves, so  $\mathcal{F}(\Sigma_B)$ and $\mathcal{F}(\Sigma_T)$ are isomorphisms, and up-to-unit invariance follows from Corollary \ref{cor:totturns}. For moves 29 and 30, we start by noticing that
\[
\mathrm{Hom}_{\mathcal{K}_n^n}\left(
\varphi_{
\tikz[baseline={([yshift=-.5ex]current bounding box.center)}, scale=0.2, y=0.6cm]
{   
    \draw[dotted] (1.5,4) -- (5.5,4);
    \draw[knot] (2,0) -- (2,4);
    \draw[knot, red] (3.5,1.5) -- (3.5,2.5);
    \draw[knot] (3,0) to[out=90, in=180] (3.5, 1.5) to[out=0, in=90] (4,0);
    \draw[knot] (3,4) to[out=-90, in=180] (3.5, 2.5) to[out=0, in=-90] (4,4);
    \draw[knot] (5,0) -- (5,4);
    \draw[dotted] (1.5,0) -- (5.5,0);
}
}
\tikz[baseline={([yshift=-.5ex]current bounding box.center)}, scale=0.3, y=0.6cm]
{   
    \draw[dotted] (1.5,4) -- (5.5,4);
    \draw[knot] (2,0) -- (2,4);
    \draw[knot] (3,0) to[out=90, in=180] (3.5, 1.5) to[out=0, in=90] (4,0);
    \draw[knot] (3,4) to[out=-90, in=180] (3.5, 2.5) to[out=0, in=-90] (4,4);
    \draw[knot] (5,0) -- (5,4);
    \draw[dotted] (1.5,0) -- (5.5,0);
}
,\,
H^n
\right)
\cong R.
\]
To see this, apply the adjunction of Proposition \ref{prop:KadjunC} and then Corollary \ref{cor:cup_cap}. The saddle cobordism $\mathcal{F}\left(\tikz[baseline={([yshift=-.5ex]current bounding box.center)}, scale=0.2, y=0.6cm]
{   
    \draw[dotted] (1.5,4) -- (5.5,4);
    \draw[knot] (2,0) -- (2,4);
    \draw[knot, red] (3.5,1.5) -- (3.5,2.5);
    \draw[knot] (3,0) to[out=90, in=180] (3.5, 1.5) to[out=0, in=90] (4,0);
    \draw[knot] (3,4) to[out=-90, in=180] (3.5, 2.5) to[out=0, in=-90] (4,4);
    \draw[knot] (5,0) -- (5,4);
    \draw[dotted] (1.5,0) -- (5.5,0);
}\right)$ is a generator for this hom-set. For Move 29, notice that both movies are produced by a Reidemeister I move followed by a saddle, thus $\mathcal{F}(\Sigma_T)$ and $\mathcal{F}(\Sigma_B)$ are generators of 
$
\mathrm{Hom}\left(
\varphi_{
\tikz[baseline={([yshift=-.5ex]current bounding box.center)}, scale=0.2, y=0.6cm]
{   
    \draw[dotted] (2.5,4) -- (4.5,4);
    \draw[knot, red] (3.5,1.5) -- (3.5,2.5);
    \draw[knot] (3,0) to[out=90, in=180] (3.5, 1.5) to[out=0, in=90] (4,0);
    \draw[knot] (3,4) to[out=-90, in=180] (3.5, 2.5) to[out=0, in=-90] (4,4);
    \draw[dotted] (2.5,0) -- (4.5,0);
}
}
\mathcal{F}(t_0), 
\mathcal{F}(t_\infty)
\right) \cong R$ 
and must differ by a unit at most. (Actually, Reidemeister I moves are isomorphisms in the $\mathcal{G}$-graded sense, so invariance of this move holds in $\mathcal{K}_n^n$ as well as $\check{\mathcal{K}}_n^n$.) Move 30 follows from the same style of argument, though we have to work in $\check{\mathcal{K}}_n^n$ since Reidemeister II moves are involved. 

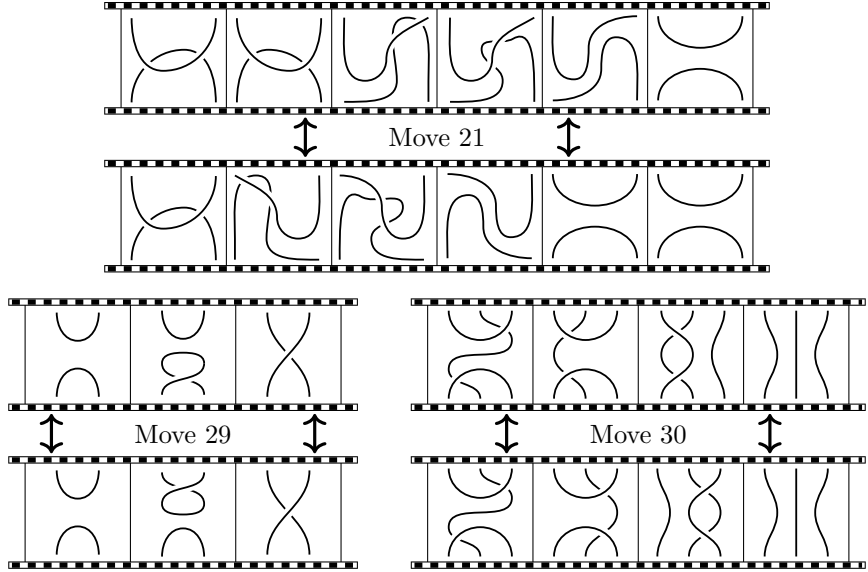
\[
\tikz[baseline={([yshift=-.5ex]current bounding box.center)}, scale=.35]{
	\draw (.5,0) -- (.5,4);
	\draw (4.5,0) -- (4.5,4);
	\draw (8.5,0) -- (8.5,4);
        \draw (12.5,0) -- (12.5,4);
        \draw (16.5,0) -- (16.5,4);
        \draw (20.5,0) -- (20.5,4);
        \draw (24.5,0) -- (24.5,4);
 	\draw[double,double distance=2pt,line cap=rect] (0,0) -- (25,0);
	\draw[dashed, line width=2pt] (0,0) -- (25,0);
	\draw[double,double distance=2pt,line cap=rect] (0,4) -- (25,4);
	\draw[dashed, line width=2pt] (0,4) -- (25,4);
        \node at (2.5,2) {$\tikz[scale=0.8, x=0.7cm, y=0.7cm]{
        \draw[knot] (0,0) to[out=90, in=180] (1.25, 1.25) to[out=0, in=90] (2,0);
        \draw[knot, overcross] (0,2) to[out=-90, in=180] (0.75, 0.75) to[out=0, in=-90] (2,2);
        }$};
        \node at (6.5,2) {$\tikz[scale=0.8, x=0.7cm, y=0.7cm]{
        \draw[knot] (0,0) to[out=90, in=180] (0.75, 1.25) to[out=0, in=90] (2,0);
        \draw[knot, overcross] (0,2) to[out=-90, in=180] (1.25, 0.75) to[out=0, in=-90] (2,2);
        }$};
        \node at (10.5,2) {$\tikz[scale=0.8, x=0.7cm, y=0.7cm]{
        \draw[knot] (0,0) to[out=0, in=-90] (1.25, 0.5) to[out=90, in=180] (1.5, 2) to[out=0, in=90] (2, 0.5) to[out=-90, in=90] (2,0);
        \draw[knot, overcross] (0,2) to[out=-90, in=180] (0.5, 0.5) to[out=0, in=-90] (1,1) to[out=90, in=210] (2,2);
        }$};
        \node at (14.5,2) {$\tikz[scale=0.8, x=0.7cm, y=0.7cm]{
        \draw[knot] (0,0) to[out=0, in=-90] (1.25, 0.5) to[out=90, in=-90] (0.75,1.25) to[out=90, in=210] (1.5, 1.5) to[out=30, in=90] (2,0);
        \draw[knot, overcross] (0,2) to[out=-90, in=180] (0.5, 0.5) to[out=0, in=-90] (1,1) to[out=90, in=210] (2,2);
        }$};
        \node at (18.5,2) {$\tikz[scale=0.8, x=0.7cm, y=0.7cm]{
        \draw[knot] (0,0) to[out=0, in=-90] (1.15,0.75) to[out=90, in=180] (1.5, 1.5) to[out=0, in=90] (2,0);
        \draw[knot] (0,2) to[out=-90, in=180] (0.5, 0.5) to[out=0, in=-90] (0.85, 1.25) to[out=90, in=180] (2,2);
        }$};
        \node at (22.5,2) {$\tikz[scale=0.8, x=0.7cm, y=0.7cm]{
        \draw[knot] (0,0) to[out=90, in=180] (1,0.75) to[out=0, in=90] (2,0);
        \draw[knot] (0,2) to[out=-90, in=180] (1,1.25) to[out=0, in=-90] (2,2);
        }$};
        \node at (12.5, -1) {Move 21};
        \draw[<->, very thick] (17.5, -0.25) -- (17.5,-1.75);
        \draw[<->, very thick] (7.5, -0.25) -- (7.5,-1.75);
    \begin{scope}[yshift=-6cm]
	\draw (.5,0) -- (.5,4);
	\draw (4.5,0) -- (4.5,4);
	\draw (8.5,0) -- (8.5,4);
        \draw (12.5,0) -- (12.5,4);
        \draw (16.5,0) -- (16.5,4);
        \draw (20.5,0) -- (20.5,4);
        \draw (24.5,0) -- (24.5,4);
 	\draw[double,double distance=2pt,line cap=rect] (0,0) -- (25,0);
	\draw[dashed, line width=2pt] (0,0) -- (25,0);
	\draw[double,double distance=2pt,line cap=rect] (0,4) -- (25,4);
	\draw[dashed, line width=2pt] (0,4) -- (25,4);
        \node at (2.5,2) {$\tikz[scale=0.8, x=0.7cm, y=0.7cm]{
        \draw[knot] (0,0) to[out=90, in=180] (1.25, 1.25) to[out=0, in=90] (2,0);
        \draw[knot, overcross] (0,2) to[out=-90, in=180] (0.75, 0.75) to[out=0, in=-90] (2,2);
        }$};
        \node at (6.5,2) {$\tikz[scale=0.8, x=0.7cm, y=0.7cm]{
        \draw[knot] (0,0) to[out=90, in=-90] (0,0.5) to[out=90, in=180] (0.5,2) to[out=0, in=90] (0.75, 0.5) to[out=-90, in=180] (2,0);
        \draw[knot, overcross] (0,2) to[out=-30, in=90] (1,1) to[out=-90, in=180] (1.5, 0.5) to[out=0, in=-90] (2,2);
        }$};
        \node at (10.5,2) {$\tikz[scale=0.8, x=0.7cm, y=0.7cm]{
        \draw[knot] (0,0) to[out=90, in=150] (0.5, 1.5) to[out=-30, in=90] (1.5, 1.25) to[out=-90, in=90] (0.75, 0.5) to[out=-90, in=180] (2,0);
        \draw[knot, overcross] (0,2) to[out=0, in=90] (1,1.25) to[out=-90, in=180] (1.5, 0.5) to[out=0, in=-90] (2,2);
        }$};
        \node at (14.5,2) {$\tikz[scale=0.8, x=0.7cm, y=0.7cm]{
        \draw[knot] (0,0) to[out=90, in=180] (0.5,1.5) to[out=0, in=90] (0.85, 0.75) to[out=-90, in=180] (2,0);
        \draw[knot] (0,2) to[out=0, in=90] (1.15, 1.25) to[out=-90, in=180] (1.5, 0.5) to[out=0, in=-90] (2,2);
        }$};
        \node at (18.5,2) {$\tikz[scale=0.8, x=0.7cm, y=0.7cm]{
        \draw[knot] (0,0) to[out=90, in=180] (1,0.75) to[out=0, in=90] (2,0);
        \draw[knot] (0,2) to[out=-90, in=180] (1,1.25) to[out=0, in=-90] (2,2);
        }$};
        \node at (22.5,2) {$\tikz[scale=0.8, x=0.7cm, y=0.7cm]{
        \draw[knot] (0,0) to[out=90, in=180] (1,0.75) to[out=0, in=90] (2,0);
        \draw[knot] (0,2) to[out=-90, in=180] (1,1.25) to[out=0, in=-90] (2,2);
        }$};
    \end{scope}
 }
\]
\[
\tikz[baseline={([yshift=-.5ex]current bounding box.center)}, scale=.35]{
	\draw (.5,0) -- (.5,4);
	\draw (4.5,0) -- (4.5,4);
	\draw (8.5,0) -- (8.5,4);
        \draw (12.5,0) -- (12.5,4);
 	\draw[double,double distance=2pt,line cap=rect] (0,0) -- (13,0);
	\draw[dashed, line width=2pt] (0,0) -- (13,0);
	\draw[double,double distance=2pt,line cap=rect] (0,4) -- (13,4);
	\draw[dashed, line width=2pt] (0,4) -- (13,4);
        \node at (2.5,2) {$\tikz[scale=0.8, x=0.7cm, y=0.7cm]{
        \draw[knot] (0,0) to[out=90, in=180] (0.5, 0.67) to[out=0, in=90] (1,0);
        \draw[knot] (0,2) to[out=-90, in=180] (0.5, 1.33) to[out=0, in=-90] (1,2);
        }$};
        \node at (6.5, 2) {$\tikz[scale=0.8, x=0.7cm, y=0.7cm]{
        \draw[knot] (1,0) to[out=90, in=-90] (0,0.75) to[out=90, in=180] (0.5,1) to[out=0, in=90] (1,0.75);
        \draw[knot, overcross] (0,0) to[out=90, in=-90] (1,0.75);
        \draw[knot] (0,2) to[out=-90, in=180] (0.5, 1.33) to[out=0, in=-90] (1,2);
        }$};
        \node at (10.5,2) {$\tikz[scale=0.8, x=0.7cm, y=0.7cm]{
        \draw[knot] (1,0) to[out=90, in=-90] (0,2);
        \draw[knot, overcross] (0,0) to[out=90, in=-90] (1,2);
        }$};
        \node at (6.5, -1) {Move 29};
        \draw[<->, very thick] (11.5, -0.25) -- (11.5,-1.75);
        \draw[<->, very thick] (1.5, -0.25) -- (1.5,-1.75);
    \begin{scope}[yshift=-6cm]
	\draw (.5,0) -- (.5,4);
	\draw (4.5,0) -- (4.5,4);
	\draw (8.5,0) -- (8.5,4);
        \draw (12.5,0) -- (12.5,4);
 	\draw[double,double distance=2pt,line cap=rect] (0,0) -- (13,0);
	\draw[dashed, line width=2pt] (0,0) -- (13,0);
	\draw[double,double distance=2pt,line cap=rect] (0,4) -- (13,4);
	\draw[dashed, line width=2pt] (0,4) -- (13,4);
        \node at (2.5,2) {$\tikz[scale=0.8, x=0.7cm, y=0.7cm]{
        \draw[knot] (0,0) to[out=90, in=180] (0.5, 0.67) to[out=0, in=90] (1,0);
        \draw[knot] (0,2) to[out=-90, in=180] (0.5, 1.33) to[out=0, in=-90] (1,2);
        }$};
        \node at (6.5,2) {$\tikz[scale=0.8, x=0.7cm, y=0.7cm]{
        \draw[knot] (0,0) to[out=90, in=180] (0.5, 0.67) to[out=0, in=90] (1,0);
        \draw[knot] (0,2) to[out=-90, in=90] (1, 1.25) to[out=-90, in=0] (0.5, 1) to[out=180, in=-90] (0,1.25);
        \draw[knot, overcross] (1,2) to[out=-90, in=90] (0,1.25);
        }$};
        \node at (10.5,2) {$\tikz[scale=0.8, x=0.7cm, y=0.7cm]{
        \draw[knot] (1,0) to[out=90, in=-90] (0,2);
        \draw[knot, overcross] (0,0) to[out=90, in=-90] (1,2);
        }$};
    \end{scope}
 }
\qquad
\tikz[baseline={([yshift=-.5ex]current bounding box.center)}, scale=.35]{
	\draw (.5,0) -- (.5,4);
	\draw (4.5,0) -- (4.5,4);
	\draw (8.5,0) -- (8.5,4);
        \draw (12.5,0) -- (12.5,4);
        \draw (16.5,0) -- (16.5,4);
 	\draw[double,double distance=2pt,line cap=rect] (0,0) -- (17,0);
	\draw[dashed, line width=2pt] (0,0) -- (17,0);
	\draw[double,double distance=2pt,line cap=rect] (0,4) -- (17,4);
	\draw[dashed, line width=2pt] (0,4) -- (17,4);
        \node at (2.5,2) {$\tikz[scale=0.6, x=0.7cm, y=0.65cm]{
        \draw[knot] (1,0) to[out=90, in=-90] (0,1) to[out=90, in=-90] (2,2) to[out=90, in=-90] (1,3);
        \draw[knot, overcross] (0,0) to[out=90, in=180] (1,1) to[out=0, in=90] (2,0);
        \draw[knot, overcross] (0,3) to[out=-90, in=180] (1,2) to[out=0, in=-90] (2,3); 
        }$};
        \node at (6.5,2) {$\tikz[scale=0.6, x=0.7cm, y=0.65cm]{
        \draw[knot] (1,0) to[out=90, in=-90] (0,1.5) to[out=90, in=-90] (1,3);
        \draw[knot, overcross] (0,0) to[out=90, in=180] (1,1) to[out=0, in=90] (2,0);
        \draw[knot, overcross] (0,3) to[out=-90, in=180] (1,2) to[out=0, in=-90] (2,3); 
        }$};
        \node at (10.5,2) {$\tikz[scale=0.6, x=0.7cm, y=0.65cm]{
        \draw[knot] (1,0) to[out=90, in=-90] (0,1.5) to[out=90, in=-90] (1,3);
        \draw[knot, overcross] (0,0) to[out=90, in=-90] (1,1.5) to[out=90, in=-90] (0,3);
        \draw[knot] (2,0) to[out=90, in=-90] (1.6, 1.5) to[out=90, in=-90] (2,3);
        }$};
        \node at (14.5,2) {$\tikz[scale=0.6, x=0.7cm, y=0.65cm]{
        \draw[knot] (0,0) to[out=90, in=-90] (0.4, 1.5) to[out=90, in=-90] (0,3);
        \draw[knot] (1,0) -- (1,3);
        \draw[knot] (2,0) to[out=90, in=-90] (1.6, 1.5) to[out=90, in=-90] (2,3);
        }$};
        \node at (8.5, -1) {Move 30};
        \draw[<->, very thick] (13.5, -0.25) -- (13.5,-1.75);
        \draw[<->, very thick] (3.5, -0.25) -- (3.5,-1.75);
    \begin{scope}[yshift=-6cm]
	\draw (.5,0) -- (.5,4);
	\draw (4.5,0) -- (4.5,4);
	\draw (8.5,0) -- (8.5,4);
        \draw (12.5,0) -- (12.5,4);
        \draw (16.5,0) -- (16.5,4);
 	\draw[double,double distance=2pt,line cap=rect] (0,0) -- (17,0);
	\draw[dashed, line width=2pt] (0,0) -- (17,0);
	\draw[double,double distance=2pt,line cap=rect] (0,4) -- (17,4);
	\draw[dashed, line width=2pt] (0,4) -- (17,4);
        \node at (2.5,2) {$\tikz[scale=0.6, x=0.7cm, y=0.65cm]{
        \draw[knot] (1,0) to[out=90, in=-90] (0,1) to[out=90, in=-90] (2,2) to[out=90, in=-90] (1,3);
        \draw[knot, overcross] (0,0) to[out=90, in=180] (1,1) to[out=0, in=90] (2,0);
        \draw[knot, overcross] (0,3) to[out=-90, in=180] (1,2) to[out=0, in=-90] (2,3); 
        }$};
        \node at (6.5,2) {$\tikz[scale=0.6, x=0.7cm, y=0.65cm]{
        \draw[knot] (1,0) to[out=90, in=-90] (2,1.5) to[out=90, in=-90] (1,3);
        \draw[knot, overcross] (0,0) to[out=90, in=180] (1,1) to[out=0, in=90] (2,0);
        \draw[knot, overcross] (0,3) to[out=-90, in=180] (1,2) to[out=0, in=-90] (2,3); 
        }$};
        \node at (10.5,2) {$\tikz[scale=0.6, x=0.7cm, y=0.65cm]{
        \draw[knot] (0,0) to[out=90, in=-90] (0.4, 1.5) to[out=90, in=-90] (0,3);
        \draw[knot] (1,0) to[out=90, in=-90] (2,1.5) to[out=90, in=-90] (1,3);
        \draw[knot, overcross] (2,0) to[out=90, in=-90] (1,1.5) to[out=90, in=-90] (2,3);
        }$};
        \node at (14.5,2) {$\tikz[scale=0.6, x=0.7cm, y=0.65cm]{
        \draw[knot] (0,0) to[out=90, in=-90] (0.4, 1.5) to[out=90, in=-90] (0,3);
        \draw[knot] (1,0) -- (1,3);
        \draw[knot] (2,0) to[out=90, in=-90] (1.6, 1.5) to[out=90, in=-90] (2,3);
        }$};
    \end{scope}
 }
\]

Moves 22, 27, and 28 involve births/deaths, so $\mathcal{F}(\Sigma_B)$ and $\mathcal{F}(\Sigma_T)$ are also not isomorphisms. In each, let $\star \in \{T, B\}$ and decompose  $\Sigma_\star$ as $t_0 \xrightarrow{P_\star} t_\star \xrightarrow{Q_\star} t_\infty$. For each of the moves in this group, $t_\star$ will denote the second frame of the movie $\Sigma_\star$. For move 22, $P_\star$ is an isotopy and $Q_\star$ is a death. Thus, $\mathcal{F}(P_\star)$ is an isomorphism. By the delooping isomorphism, $\mathcal{F}(t_\star) \cong \mathcal{F}(t_\infty) \{-1,0\} \oplus \mathcal{F}(t_\infty) \{0,1\}$, so we have that
\[
\mathrm{Hom}_{\mathcal{K}_n^n}\left(
\varphi_{\tikz[baseline={([yshift=-.5ex]current bounding box.center)}, scale=.3]{
	\draw (1,0) .. controls (1,1) and (2,1) .. (2,0);
	\draw (1,0) .. controls (1,-.25) and (2,-.25) .. (2,0);
	\draw[dashed] (1,0) .. controls (1,.25) and (2,.25) .. (2,0);
}}
\mathcal{F}(t_T),
\mathcal{F}(t_\infty)
\right)
\cong
\mathrm{Hom}_{\mathcal{K}_n^n} \left(\mathcal{F}(t_\infty)\{1,1\} \oplus \mathcal{F}(t_\infty), \mathcal{F}(t_\infty) \right)
\cong
R.
\]
For the last isomorphism, we refer to Corollary \ref{cor:totturns} and its proof if the box is a U-turn, and to Corollary \ref{cor:invs} if the box is a crossing. Then, $\mathcal{F}(Q_T)$ and $\mathcal{F}(Q_B) \circ \mathcal{F}(P_B) \circ \mathcal{F}(P_T)^{-1}$ both generate this hom-set, and thus differ only up to multiplication by a unit. It follows that $\mathcal{F}(\Sigma_T)$ and $\mathcal{F}(\Sigma_B)$ also only differ up to multiplication by a unit. For move 27, $\mathcal{F}(P_T) = \mathcal{F}(P_B)$ as both are identical births, and $\mathcal{F}(Q_T)^{-1} \circ \mathcal{F}(Q_B)$ is an isomorphism. Also, notice by the delooping isomorphism, we have that
\[
\mathrm{Hom}\left(
\varphi_{\tikz[baseline={([yshift=-.5ex]current bounding box.center)}, scale=0.3]{
	\draw (1,2) .. controls (1,1) and (2,1) .. (2,2);
	\draw (1,2) .. controls (1,1.75) and (2,1.75) .. (2,2);
	\draw (1,2) .. controls (1,2.25) and (2,2.25) .. (2,2);
}}
H^n, 
H^n \otimes V
\right)
\cong
\mathrm{Hom}\left(
H^n \{1,0\},
H^n \{1,0\} \oplus H^n \{0,-1\
\right)
\cong R.
\]
The maps $\mathcal{F}(P_T)$ and $\mathcal{F}(Q_T)^{-1} \circ \mathcal{F}(Q_B) \circ \mathcal{F}(P_B)$ are generators of this hom-set, so they differ up to multiplication by a unit, and thus $\mathcal{F}(\Sigma_T)$ and $\mathcal{F}(\Sigma_B)$ also differ only up to a unit. For move 28, $P_T$ and $P_B$ are both births. Notice that $\mathcal{F}(Q_B)^{-1} \circ \mathcal{F}(Q_T): \mathcal{F}(t_T) \to \mathcal{F}(t_B)$ is a grading preserving isomorphism: indeed, this follows in the $\mathcal{G}$-graded setting. To see this, notice that one of the movies involves a Reidemeister II move of the first type in (3) of Proposition \ref{prop:Rmoves}, and the other has a Reidemeister II move of the second type. However, since one of the components is closed, the saddle cobordism is a merge, and is canonically isomorphic to the grading shift $\{-1,1\}$. By the delooping isomorphism, we have that $\mathcal{F}(t_\star) \cong H^n \{1,0\} \oplus H^n \{0,-1\}$, so it follows that
\[
\mathrm{Hom}\left(
\varphi_{\tikz[baseline={([yshift=-.5ex]current bounding box.center)}, scale=0.3]{
	\draw (1,2) .. controls (1,1) and (2,1) .. (2,2);
	\draw (1,2) .. controls (1,1.75) and (2,1.75) .. (2,2);
	\draw (1,2) .. controls (1,2.25) and (2,2.25) .. (2,2);
}}
H^n, 
\mathcal{F}(t_\star)
\right)
\cong R.
\]
as before. Now, $\mathcal{F}(Q_B)^{-1} \circ \mathcal{F}(Q_T) \circ \mathcal{F}(P_T)$ and $\mathcal{F}(P_B)$ are generators of this hom-set, and thus differ by at most a unit. Thus $\mathcal{F}(\Sigma_T)$ and $\mathcal{F}(\Sigma_B)$ differ by at must a unit of $R$.

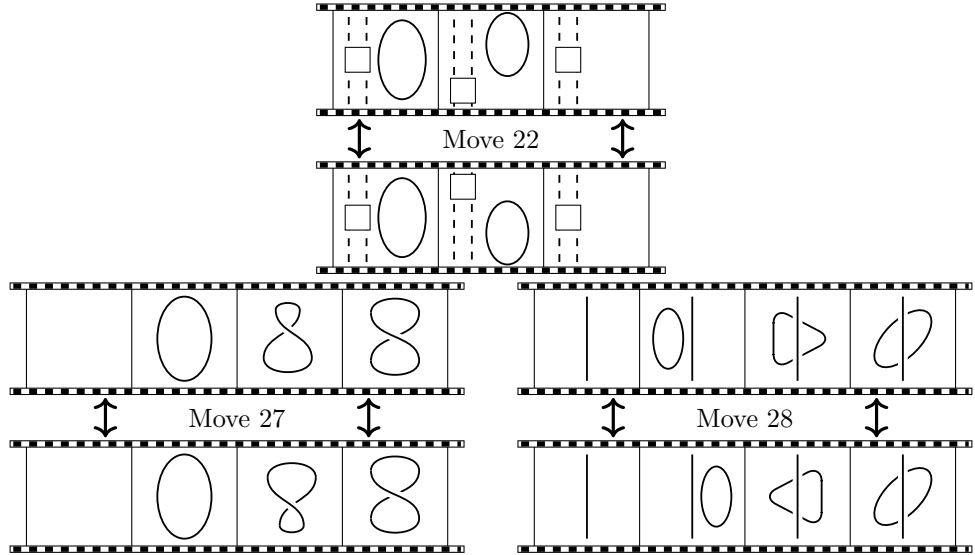
\[
\tikz[baseline={([yshift=-.5ex]current bounding box.center)}, scale=.35]{
	\draw (.5,0) -- (.5,4);
	\draw (4.5,0) -- (4.5,4);
	\draw (8.5,0) -- (8.5,4);
        \draw (12.5,0) -- (12.5,4);
 	\draw[double,double distance=2pt,line cap=rect] (0,0) -- (13,0);
	\draw[dashed, line width=2pt] (0,0) -- (13,0);
	\draw[double,double distance=2pt,line cap=rect] (0,4) -- (13,4);
	\draw[dashed, line width=2pt] (0,4) -- (13,4);
        \node at (2.5,2) {$\tikz[scale=0.35, x=0.675cm, y=0.675cm]{
        \draw[knot, dashed] (0,0) -- (0,5);
        \draw[knot, dashed] (1,0) -- (1,5);
        \draw[fill=white] (-0.2, 1.8) rectangle (1.2,3.2);
        \draw[knot] (3,2.5) ellipse (0.9cm and 1.5cm);
        }$};
        \node at (6.5, 2) {$\tikz[scale=0.35, x=0.675cm, y=0.675cm]{
        \draw[knot, dashed] (0,0) -- (0,5);
        \draw[knot, dashed] (1,0) -- (1,5);
        \draw[fill=white] (-0.2, 0.2) rectangle (1.2,1.6);
        \draw[knot, white] (3,2.5) ellipse (0.9cm and 1.2cm);
        \draw[knot] (3,3.5) ellipse (0.8cm and 1.2cm);
        }$};
        \node at (10.5,2) {$\tikz[scale=0.35, x=0.675cm, y=0.675cm]{
        \draw[knot, dashed] (0,0) -- (0,5);
        \draw[knot, dashed] (1,0) -- (1,5);
        \draw[fill=white] (-0.2, 1.8) rectangle (1.2,3.2);
        \draw[knot, white] (3,2.5) ellipse (0.9cm and 1.2cm);
        }$};
        \node at (6.5, -1) {Move 22};
        \draw[<->, very thick] (11.5, -0.25) -- (11.5,-1.75);
        \draw[<->, very thick] (1.5, -0.25) -- (1.5,-1.75);
    \begin{scope}[yshift=-6cm]
	\draw (.5,0) -- (.5,4);
	\draw (4.5,0) -- (4.5,4);
	\draw (8.5,0) -- (8.5,4);
        \draw (12.5,0) -- (12.5,4);
 	\draw[double,double distance=2pt,line cap=rect] (0,0) -- (13,0);
	\draw[dashed, line width=2pt] (0,0) -- (13,0);
	\draw[double,double distance=2pt,line cap=rect] (0,4) -- (13,4);
	\draw[dashed, line width=2pt] (0,4) -- (13,4);
        \node at (2.5,2) {$\tikz[scale=0.35, x=0.675cm, y=0.675cm]{
        \draw[knot, dashed] (0,0) -- (0,5);
        \draw[knot, dashed] (1,0) -- (1,5);
        \draw[fill=white] (-0.2, 1.8) rectangle (1.2,3.2);
        \draw[knot] (3,2.5) ellipse (0.9cm and 1.5cm);
        }$};
        \node at (6.5,2) {$\tikz[scale=0.35, x=0.675cm, y=0.675cm]{
        \draw[knot, dashed] (0,0) -- (0,5);
        \draw[knot, dashed] (1,0) -- (1,5);
        \draw[fill=white] (-0.2, 3.4) rectangle (1.2,4.8);
        \draw[knot, white] (3,2.5) ellipse (0.9cm and 1.2cm);
        \draw[knot] (3,1.5) ellipse (0.8cm and 1.2cm);
        }$};
        \node at (10.5,2) {$\tikz[scale=0.35, x=0.675cm, y=0.675cm]{
        \draw[knot, dashed] (0,0) -- (0,5);
        \draw[knot, dashed] (1,0) -- (1,5);
        \draw[fill=white] (-0.2, 1.8) rectangle (1.2,3.2);
        \draw[knot, white] (3,2.5) ellipse (0.9cm and 1.2cm);
        }$};
    \end{scope}
 }
\]
\[
\tikz[baseline={([yshift=-.5ex]current bounding box.center)}, scale=.35]{
	\draw (.5,0) -- (.5,4);
	\draw (4.5,0) -- (4.5,4);
	\draw (8.5,0) -- (8.5,4);
        \draw (12.5,0) -- (12.5,4);
        \draw (16.5,0) -- (16.5,4);
 	\draw[double,double distance=2pt,line cap=rect] (0,0) -- (17,0);
	\draw[dashed, line width=2pt] (0,0) -- (17,0);
	\draw[double,double distance=2pt,line cap=rect] (0,4) -- (17,4);
	\draw[dashed, line width=2pt] (0,4) -- (17,4);
        \node at (2.5,2) {$\tikz[scale=0.8]{
        }$};
        \node at (6.5,2) {$\tikz[scale=0.8, x=0.75cm, y=0.75cm]{
        \draw[knot] (0.5,0.5) ellipse (0.45cm and 0.7cm);
        }$};
        \node at (10.5,2) {$\tikz[scale=0.8, x=0.8cm, y=0.8cm]{
        \draw[knot] (0,0) to[out=90, in=-90] (0.75,1);
        \draw[knot, overcross] (1,0) to[out=90, in=-90] (0.25,1);
        \draw[knot] (0,0) to[out=-90, in=-90] (1,0);
        \draw[knot] (0.25,1) to[out=90, in=90] (0.75, 1);
        }$};
        \node at (14.5,2) {$\tikz[scale=0.8, x=0.8cm, y=0.8cm]{
        \draw[knot] (0,0) to[out=90, in=-90] (1,1) to[out=90, in=90] (0,1);
        \draw[knot, overcross] (0,0) to[out=-90, in=-90] (1,0) to[out=90, in=-90] (0,1);
        }$};
        \node at (8.5, -1) {Move 27};
        \draw[<->, very thick] (13.5, -0.25) -- (13.5,-1.75);
        \draw[<->, very thick] (3.5, -0.25) -- (3.5,-1.75);
    \begin{scope}[yshift=-6cm]
	\draw (.5,0) -- (.5,4);
	\draw (4.5,0) -- (4.5,4);
	\draw (8.5,0) -- (8.5,4);
        \draw (12.5,0) -- (12.5,4);
        \draw (16.5,0) -- (16.5,4);
 	\draw[double,double distance=2pt,line cap=rect] (0,0) -- (17,0);
	\draw[dashed, line width=2pt] (0,0) -- (17,0);
	\draw[double,double distance=2pt,line cap=rect] (0,4) -- (17,4);
	\draw[dashed, line width=2pt] (0,4) -- (17,4);
        \node at (2.5,2) {$\tikz[scale=0.8, x=0.8cm, y=0.8cm]{
        }$};
        \node at (6.5,2) {$\tikz[scale=0.8, x=0.8cm, y=0.8cm]{
        \draw[knot] (0.5,0.5) ellipse (0.45cm and 0.7cm);
        }$};
        \node at (10.5,2) {$\tikz[scale=0.8, x=0.8cm, y=0.8cm]{
        \draw[knot] (0.25,0) to[out=90, in=-90] (1,1);
        \draw[knot, overcross] (0.75,0) to[out=90, in=-90] (0,1);
        \draw[knot] (0.25,0) to[out=-90, in=-90] (0.75, 0);
        \draw[knot] (0,1) to[out=90,in=90] (1,1);
        }$};
        \node at (14.5,2) {$\tikz[scale=0.8, x=0.8cm, y=0.8cm]{
        \draw[knot] (0,0) to[out=90, in=-90] (1,1) to[out=90, in=90] (0,1);
        \draw[knot, overcross] (0,0) to[out=-90, in=-90] (1,0) to[out=90, in=-90] (0,1);
        }$};
    \end{scope}
 }
\qquad
\tikz[baseline={([yshift=-.5ex]current bounding box.center)}, scale=.35]{
	\draw (.5,0) -- (.5,4);
	\draw (4.5,0) -- (4.5,4);
	\draw (8.5,0) -- (8.5,4);
        \draw (12.5,0) -- (12.5,4);
        \draw (16.5,0) -- (16.5,4);
 	\draw[double,double distance=2pt,line cap=rect] (0,0) -- (17,0);
	\draw[dashed, line width=2pt] (0,0) -- (17,0);
	\draw[double,double distance=2pt,line cap=rect] (0,4) -- (17,4);
	\draw[dashed, line width=2pt] (0,4) -- (17,4);
        \node at (2.5,2) {$\tikz[scale=0.8, x=0.8cm, y=0.7cm]{
        \draw[knot] (0.5,0) -- (0.5,2);
        }$};
        \node at (6.5,2) {$\tikz[scale=0.8, x=0.8cm, y=0.7cm]{
        \draw[knot] (0.5,0) -- (0.5,2);
        \draw[knot] (0, 1) ellipse (0.25cm and 0.5cm);
        \draw[knot, color=white] (1, 1) circle (0.25cm);
        }$};
        \node at (10.5,2) {$\tikz[scale=0.8, x=0.8cm, y=0.7cm]{
        \draw[knot, rounded corners=0.2cm, color=white] (1,1) -- (1,1.75) -- (-0.25, 1) -- (1,0.25) -- cycle;
        \draw[knot, rounded corners=2.15mm] (0,1) -- (0,1.75) -- (1.25, 1) -- (0,0.25) -- cycle;
        \draw[knot, overcross] (0.5,0) -- (0.5,2);
        }$};
        \node at (14.5,2) {$\tikz[scale=0.8, x=0.8cm, y=0.7cm]{
        \draw[knot, rotate around={45:(0.5,1)}] (0.5,1) ellipse (0.6cm and 0.3cm);
        \draw[knot, overcross] (0.5,0) -- (0.5,2);
        }$};
        \node at (8.5, -1) {Move 28};
        \draw[<->, very thick] (13.5, -0.25) -- (13.5,-1.75);
        \draw[<->, very thick] (3.5, -0.25) -- (3.5,-1.75);
    \begin{scope}[yshift=-6cm]
	\draw (.5,0) -- (.5,4);
	\draw (4.5,0) -- (4.5,4);
	\draw (8.5,0) -- (8.5,4);
        \draw (12.5,0) -- (12.5,4);
        \draw (16.5,0) -- (16.5,4);
 	\draw[double,double distance=2pt,line cap=rect] (0,0) -- (17,0);
	\draw[dashed, line width=2pt] (0,0) -- (17,0);
	\draw[double,double distance=2pt,line cap=rect] (0,4) -- (17,4);
	\draw[dashed, line width=2pt] (0,4) -- (17,4);
        \node at (2.5,2) {$\tikz[scale=0.8, x=0.8cm, y=0.7cm]{
        \draw[knot] (0.5,0) -- (0.5,2);
        }$};
        \node at (6.5,2) {$\tikz[scale=0.8, x=0.8cm, y=0.7cm]{
        \draw[knot] (0.5,0) -- (0.5,2);
        \draw[knot, color=white] (0, 1) circle (0.25cm);
        \draw[knot] (1, 1) circle (0.25cm and 0.5cm);
        }$};
        \node at (10.5,2) {$\tikz[scale=0.8, x=0.8cm, y=0.7cm]{
        \draw[knot, rounded corners=0.2cm, color=white] (0,1) -- (0,1.75) -- (1.25, 1) -- (0,0.25) -- cycle;
        \draw[knot, rounded corners=2.15mm] (1,1) -- (1,1.75) -- (-0.25, 1) -- (1,0.25) -- cycle;
        \draw[knot, overcross] (0.5,0) -- (0.5,2);
        }$};
        \node at (14.5,2) {$\tikz[scale=0.8, x=0.8cm, y=0.7cm]{
        \draw[knot, rotate around={45:(0.5,1)}] (0.5,1) ellipse (0.6cm and 0.3cm);
        \draw[knot, overcross] (0.5,0) -- (0.5,2);
        }$};
    \end{scope}
 }
\]

Move 31 requires that
\[
(f \otimes \mathrm{Id}) \circ (\mathrm{Id} \otimes g) = (\mathrm{Id} \otimes g) \circ (f \otimes \mathrm{Id})
\]
for $f, g$ morphisms of $\check{\mathcal{K}}_n^n$. In general, Proposition 4.33 of \cite{naisse2020odd} states that for homogeneous maps $M_1 \xrightarrow{f_1} M_2 \xrightarrow{f_2} M_3$ and $N_1 \xrightarrow{g_1} N_2 \xrightarrow{g_2} N_3$ of $\mathcal{C}$-graded $(A,A)$-bimodules,
\[
\left[(f_2 \circ_\mathcal{C} f_1) \otimes (g_2 \circ_\mathcal{C} g_1)\right](m \otimes n)
=
\Xi_{\substack{\abs{f_1}, \abs{f_2} \\ \abs{g_1}, \abs{g_2}}}(\abs{m} \bullet \abs{n}) \left[(f_2 \otimes g_2) \circ_\mathcal{C} (f_1 \otimes g_1)\right] (m \otimes n)
\]
for units $\Xi_{\substack{\abs{f_1}, \abs{f_2} \\ \abs{g_1}, \abs{g_2}}}$, where we recall that $\circ_\mathcal{C}$ is $\mathcal{C}$-graded vertical composition (\ref{eq:cgradedcomp}). Then, since
\[
(f \circ_\mathcal{C} \mathrm{Id}) \otimes (\mathrm{Id} \circ_{\mathcal{C}} g)
=
f \otimes g
=
(\mathrm{Id} \circ_\mathcal{C} f) \otimes (g \circ_\mathcal{C} \mathrm{Id})
\]
the result follows. \hfill $\square$

\newpage

\bibliographystyle{alpha}
\bibliography{OKH_and_HH}

\end{document}